\documentclass[reqno,11pt]{amsart}

\usepackage{amscd}
\usepackage{amsmath}
\usepackage{amssymb}
\usepackage[american]{babel}
\usepackage{bbm}
\usepackage{bookmark}
\usepackage{cmap} 
\usepackage{dsfont}
\usepackage{enumerate}
\usepackage{epigraph}
\usepackage[mathscr]{euscript}
\usepackage[myheadings]{fullpage}
\usepackage{graphicx}
\usepackage[geometry]{ifsym}
\usepackage{mathabx}
\usepackage{accents}
\usepackage{mathrsfs}
\usepackage{mathtools}
\usepackage{refcount}
\usepackage{setspace}
\usepackage{stmaryrd}
\usepackage{thinsp}
\usepackage{verbatim}
\usepackage[all,2cell]{xy} \UseTwocells
\usepackage{xr}
\usepackage[multiple]{footmisc}
\usepackage{hyperref}

\usepackage{tikz}

\newcommand{\e}{\mathds{1}}

\numberwithin{equation}{subsection}

\newtheorem{thm}{Theorem}[subsection]
\newtheorem*{thm*}{Theorem}

\newtheorem{cor}[thm]{Corollary}
\newtheorem*{cor*}{Corollary}
\newtheorem{lem}[thm]{Lemma}

\newtheorem{prop}[thm]{Proposition}
\newtheorem{prop-const}[thm]{Proposition-Construction}

\newtheorem*{conjecture*}{Conjecture}

\newtheorem*{princ*}{Principle}

\theoremstyle{remark}
\newtheorem{rem}[thm]{Remark}
\newtheorem{example}[thm]{Example}

\newtheorem{counterexample}[thm]{Counterexample}
\newtheorem{defin}[thm]{Definition}

\newtheorem{notation}[thm]{Notation}
\newtheorem{warning}[thm]{Warning}
\newtheorem{variant}[thm]{Variant}

\newtheorem{construction}[thm]{Construction}
\newtheorem{terminology}[thm]{Terminology}

\newcounter{steps}[thm]
\newtheorem{step}[steps]{Step}

\newcommand{\presup}[1]{\prescript{#1}{}\!}

\newcommand{\xar}[1]{\xrightarrow{#1}}
\newcommand{\isom}{\xar{\simeq}}

\newcommand{\into}{\hookrightarrow}
\newcommand{\onto}{\twoheadrightarrow}

\newcommand{\bDelta}{\mathbf{\Delta}}

\newcommand{\bA}{{\mathbb A}}
\newcommand{\bB}{{\mathbb B}}

\newcommand{\bD}{{\mathbb D}}

\newcommand{\bG}{{\mathbb G}}

\newcommand{\bV}{{\mathbb V}}

\newcommand{\bZ}{{\mathbb Z}}

\newcommand{\sA}{{\EuScript A}}
\newcommand{\sB}{{\EuScript B}}
\newcommand{\sC}{{\EuScript C}}
\newcommand{\sD}{{\EuScript D}}
\newcommand{\sE}{{\EuScript E}}
\newcommand{\sF}{{\EuScript F}}
\newcommand{\sG}{{\EuScript G}}
\newcommand{\sH}{{\EuScript H}}
\newcommand{\sI}{{\EuScript I}}
\newcommand{\sJ}{{\EuScript J}}
\newcommand{\sK}{{\EuScript K}}
\newcommand{\sL}{{\EuScript L}}
\newcommand{\sM}{{\EuScript M}}
\newcommand{\sN}{{\EuScript N}}
\newcommand{\sO}{{\EuScript O}}
\newcommand{\sP}{{\EuScript P}}

\newcommand{\sW}{{\EuScript W}}

\newcommand{\sZ}{{\EuScript Z}}

\newcommand{\fZ}{{\mathfrak Z}}

\newcommand{\fg}{{\mathfrak g}}
\newcommand{\fh}{{\mathfrak h}}

\newcommand{\fk}{{\mathfrak k}}

\newcommand{\fz}{{\mathfrak z}}

\newcommand{\on}{\operatorname}
\newcommand{\ol}[1]{\overline{#1}{}}
\newcommand{\ul}{\underline}

\newcommand{\mathendash}{\text{\textendash}}

\newcommand{\Ker}{\on{Ker}}
\newcommand{\Coker}{\on{Coker}}

\newcommand{\Hom}{\on{Hom}}
\newcommand{\Ext}{\on{Ext}}
\newcommand{\Aut}{\on{Aut}}
\newcommand{\Spec}{\on{Spec}}
\newcommand{\Spf}{\on{Spf}}

\newcommand{\id}{\on{id}}

\newcommand{\Ad}{\on{Ad}}

\newcommand{\ind}{\on{ind}}

\newcommand{\Rep}{\mathsf{Rep}}
\newcommand{\gr}{\on{gr}}

\newcommand{\act}{\on{act}}
\newcommand{\actson}{\curvearrowright}
\newcommand{\actedon}{\curvearrowleft}
\newcommand{\coact}{\on{coact}}

\renewcommand{\dot}{\bullet}

\newcommand{\Tor}{\on{Tor}}
\newcommand{\Fun}{\on{Fun}}

\newcommand{\vph}{\varphi}
\newcommand{\vareps}{\varepsilon}

\newcommand{\Vect}{\mathsf{Vect}}
\newcommand{\Fil}{\mathsf{Fil}\,	}

\newcommand{\Res}{\on{Res}}
\newcommand{\Gr}{\on{Gr}}

\newcommand{\Whit}{\mathsf{Whit}}

\newcommand{\Op}{\on{Op}}

\renewcommand{\mod}{\mathendash\mathsf{mod}}
\newcommand{\comod}{\mathendash\mathsf{comod}}

\newcommand{\rightmod}{\mathsf{mod}\mathendash}

\newcommand{\sinf}{\!\frac{\infty}{2}} 

\newcommand{\colim}{\on{colim}}
\newcommand{\Set}{\mathsf{Set}}
\newcommand{\Cat}{\mathsf{Cat}}
\newcommand{\DGCat}{\mathsf{DGCat}}
\newcommand{\TwoDGCat}{2\mathendash\DGCat}

\newcommand{\Gpd}{\mathsf{Gpd}}
\newcommand{\TwoCat}{2\mathendash\Cat}
\renewcommand{\lim}{\on{lim}}
\newcommand{\Ind}{\mathsf{Ind}}
\newcommand{\Pro}{\mathsf{Pro}}

\newcommand{\TwoHom}{\mathsf{Hom}}
\newcommand{\TwoEnd}{\mathsf{End}}

\newcommand{\Yo}{\mathsf{Yo}}

\newcommand{\Tot}{\on{Tot}}
\newcommand{\heart}{\heartsuit}

\newcommand{\Oblv}{\on{Oblv}}
\newcommand{\Av}{\on{Av}}

\newcommand{\ld}{\check}

\newcommand{\Perf}{\mathsf{Perf}}
\newcommand{\Coh}{\mathsf{Coh}}
\newcommand{\IndCoh}{\mathsf{IndCoh}}
\newcommand{\QCoh}{\mathsf{QCoh}}
\newcommand{\Sch}{\mathsf{Sch}}
\newcommand{\IndSch}{\mathsf{IndSch}}
\newcommand{\AffSch}{\mathsf{AffSch}}
\newcommand{\PreStk}{\mathsf{PreStk}}

\newcommand{\Alg}{\mathsf{Alg}}
\newcommand{\ComAlg}{\mathsf{ComAlg}}

\newcommand{\LieAlg}{\mathsf{LieAlg}}
\newcommand{\Lie}{\on{Lie}}
\newcommand{\triv}{\on{triv}}
\newcommand{\Mor}{\mathsf{Morita}}

\newcommand*\circled[1]{\tikz[baseline=(char.base)]{
            \node[shape=circle,draw,inner sep=2pt] (char) {#1};}}

\renewcommand{\subset}{\subseteq}

\makeatletter
\newcommand{\biggg}{\bBigg@{4}}
\newcommand{\Biggg}{\bBigg@{5}}
\makeatother

\begin{document}

\frenchspacing

\setlength{\epigraphwidth}{0.4\textwidth}
\renewcommand{\epigraphsize}{\footnotesize}

\begin{abstract}

Actions of algebraic groups on DG categories provide a convenient, unifying
framework in some parts of geometric representation theory,
especially the representation theory of reductive Lie algebras.
We extend this theory to loop groups and affine Lie algebras, extending
previous work of Beraldo, Gaitsgory and the author.

Along the way, we introduce some subjects of independent interest: 
ind-coherent sheaves of infinite type
indschemes, topological DG algebras, and weak actions
of group indschemes on categories. We also present a new construction of
semi-infinite cohomology for affine Lie algebras, based on the
\emph{modular character} for loop groups. 

As an application of our methods, we establish an important technical 
result for Kac-Moody representations 
at the critical level, showing that the appropriate (``renormalized") derived 
category of representations
admits a large class of symmetries coming from the adjoint action of
the loop group and from the center of the enveloping algebra. 

\end{abstract}

\title{Homological methods in semi-infinite contexts}

\author{Sam Raskin}

\address{The University of Texas at Austin, 
Department of Mathematics, 
RLM 8.100, 2515 Speedway Stop C1200, 
Austin, TX 78712}

\email{sraskin@math.utexas.edu}

\maketitle

\setcounter{tocdepth}{1}
\tableofcontents

\section{Introduction}\label{s:intro}

\subsection{}

Let $k$ be a field of characteristic zero, and let $K = k((t))$. 
Suppose $G$ is a split reductive group over $k$, and let 
$G(K)$ denote the corresponding algebraic loop group 
(see e.g. \cite{hitchin} \S 7.11.2 for the definition).

There is a robust theory of actions of $G$ and $G(K)$ on suitable DG categories:
see \cite{beraldo-*/!}, \cite{dmod}, and \cite{paris-notes} for discussion, or \cite{hitchin} \S 7 and \cite{fg2} for
more classical approaches.
A basic example is that if $G$ (resp. $G(K)$) acts on a suitable space
$X$, then $G$ (resp. $G(K)$) will act on the DG category $D(X)$ of $D$-modules on 
$X$. We let $G\mod$ (resp. $G(K)\mod$) denote the (2-)category of 
DG categories with a $G$ (resp. $G(K)$) action.

According to \cite{fg2} and \cite{quantum-langlands-summary}, $G(K)\mod$
is analogous to the category of smooth representations of a $p$-adic reductive group.
In particular, one expects a robust local geometric Langlands program for which
$G(K)\mod$ is one side. 
We refer to \cite{paris-notes} for further discussion. 

\subsection{}

Group actions on categories inherently involve higher categorical data.

However, a special 
property of working specifically with $D$-modules (as opposed to $\ell$-adic
sheaves, say), is that many objects of $G\mod$ can be constructed using
1-categorical methods; notably, if $A$ is an algebra equipped 
with a $G$-action and a Harish-Chandra
datum $i:\fg \to A$, then the DG category $A\mod$ of $A$-modules 
defines an object of $G\mod$. Many interesting examples arise this way, so
sometimes theorems can be proved by reducing to a 1-categorical setting where
traditional representation theoretic methods apply.

One of the primary objectives of this text, achieved in \S \ref{s:hc}, is to develop
a parallel theory for $G(K)$. Such a theory is certainly not surprising.
However, a number of difficulties appear, related to the infinite-dimensional
(or ``semi-infinite") nature of $G(K)$. 
The body of this text addresses those problems and develops
the requisite theory.

There are several motivations for wanting to link Harish-Chandra data
for $G(K)$ with $G(K)\mod$:

\begin{itemize}

\item Construct interesting objects of $G(K)\mod$.

\item Prove theorems about objects of $G(K)\mod$ using 1-categorical methods.

\end{itemize}

Below, we give an example of the latter, in the spirit of \cite{fg2}.
An example of the former will be given in \cite{weilrep}.

We gave another example of the former in \cite{paris-notes}, where
we outlined a proof of the quantum local 
Langlands conjectures in the abelian case. We emphasize that the argument from
\emph{loc. cit}. was not complete: the construction of the bimodule 
implementing local Langlands (for a torus) relied on heuristics about 
Harish-Chandra data. The methods from \S \ref{s:hc} can be used
to complete the argument in \cite{paris-notes}.

\begin{rem}

We warn that the theory of Harish-Chandra data for $G(K)$ is much more
subtle than for $G$; there are some additional, potentially subtle 
conditions to check to
verify that a purported Harish-Chandra datum is in fact a Harish-Chandra
datum. These subtleties are closely related to the length of this paper.
We refer to the beginning of \S \ref{s:hc} for further discussion.

\end{rem}
 
\subsection{}

The genesis of this paper is \cite{localization}, where we use
ideas from the theory of group actions on categories to study 
affine Beilinson-Bernstein localization at critical level.
To use those methods, we needed the following construction.
Let $\widehat{\fg}_{crit}\mod$ be the DG category of critical level
Kac-Moody representations; see e.g. \S \ref{s:critical} 
for a suitable definition.

The notes \cite{km-indcoh} sketch the construction 
of an action of $G(K)$ (``at critical level") on $\widehat{\fg}_{crit}\mod$. 

On the other hand, it is known that the (completed, twisted) enveloping algebra
$U(\widehat{\fg}_{crit})$ has a large center, typically denoted
$\fZ$. 

It is natural to expect that $\widehat{\fg}_{crit}\mod$ has commuting
actions of $G(K)$ and a suitable category $\fZ\mod$ of (discrete) $\fZ$-modules. 
We formulate this result precisely and prove it in Theorem \ref{t:opers}.

Let us describe the difficulty explicitly. 
The construction from \cite{km-indcoh} is geometric, 
and builds from the formal completion of
$G(K)$ along compact open subgroup. However, the center $\fZ$ is not
visible at this level, and it requires thinking about the enveloping
algebra $U(\widehat{\fg}_{crit})$ in an essential way. 
So the difficulty is that \emph{we wish to mix 
homotopically complicated categorical methods,
such as $G(K)$ actions on categories, with concrete constructions involving
some flavors of topological algebras}. Making such arguments work is one
of the key goals of this paper.

\begin{rem}

In local geometric Langlands, the object $\widehat{\fg}_{crit}\mod \in 
G(K)\mod$ plays a fundamental role. The above action of its center
plays a key role because of the Feigin-Frenkel isomorphism
\cite{ff-critical}. This isomorphism identifies
$\Spf(\fZ)$ with $\Op_{\ld{G}}$, the indscheme of opers on the
punctured disc for the Langlands dual group $\ld{G}$. Opers
are certain $\ld{G}$-local systems (in the de Rham sense) on the 
punctured disc; in particular, $\widehat{\fg}_{crit}\mod \in G(K)\mod$
``spectrally decomposes" in the categorical
sense over the stack of such local systems, because it decomposes
already over opers (by Theorem \ref{t:opers} and the Feigin-Frenkel
isomorphism). Such a decomposition is an
instance of a general prediction of local geometric Langlands,
and is highly non-trivial in that arbitrary ramification occurs.

We refer to the introduction of \cite{fg2}
for further discussion of the (conjectural, but central) role of 
structures in local geometric Langlands. 

\end{rem}

\subsection{}

To summarize: the theory of
Harish-Chandra data in \S \ref{s:hc} provides a way to study certain
categories with $G(K)$-actions using $1$-categorical methods, and the bulk
of this paper is used to develop such a theory of Harish-Chandra data.
One application of these methods is given in \S \ref{s:critical},
and a more interesting one is given in \cite{localization}.

However, this paper contains several other ideas that may 
be of independent interest to the reader. We discuss these below.

\subsection{Topological algebras}

In \S \ref{s:tens}-\ref{s:ren}, we develop a theory of 
topological DG algebras, which is essentially a derived version of some
parts of \cite{beilinson-top-alg}.

We refer to the main players as \emph{$\overset{\rightarrow}{\otimes}$-algebras}.
Let $\Vect$ denotes the DG category of
(chain complexes of) vector spaces, and $\Pro\Vect$ its pro-category 
(see \S \ref{s:tens} for more discussion). An $\overset{\rightarrow}{\otimes}$-algebra
is $A \in \Pro\Vect$ plus a suitable algebra structure on $A$
for a certain monoidal structure $\overset{\rightarrow}{\otimes}$ on 
$\Pro\Vect$, which is a natural derived version of a construction in 
\cite{beilinson-top-alg}. A typical example of such an $A$ is the completed
enveloping algebra of a Tate Lie algebra.

There is an associated DG category of $A$-modules that we denote $A\mod_{naive}$.
By definition, objects of $A\mod_{naive}$ are ``discrete" $A$-modules: 
the forgetful functor maps $A\mod_{naive}$ to $\Vect$, not to $\Pro\Vect$.

\begin{rem}

In \S \ref{ss:comonad}, we observe that $\overset{\rightarrow}{\otimes}$-algebras
are simply a dual description of (accessible, DG, but possibly non-continuous)
comonads on $\Vect$; the category $A\mod_{naive}$ corresponds to the corresponding
category of comodules.

\end{rem}

\subsection{}\label{ss:intro-ren}

One of the persistent technical problems in this paper is 
that of \emph{renormalization} for $\overset{\rightarrow}{\otimes}$-algebras. 

In \cite{dmod-aff-flag}, Frenkel-Gaitsgory realized that for many (connective)
$\overset{\rightarrow}{\otimes}$-algebras $A$, the category $A\mod_{naive}$
is only a reasonable thing to consider on its bounded below subcategory
$A\mod_{naive}^+$. In many examples, there is a better DG category $A\mod_{ren}$ that
is equipped with a $t$-structure for which $A\mod_{ren}^+ = A\mod_{naive}^+$.
Typically, $A\mod_{ren}$ is compactly generated, while $A\mod_{naive}$ might
not be. In general, $A\mod_{ren}$ may be nicely behaved (e.g., admitting
various ``obvious" symmetries) while $A\mod_{naive}$ may be severely pathological
and misleading.

One downside to $A\mod_{ren}$ is that the forgetful functor $A\mod_{ren} \to \Vect$
is not conservative. In other words, objects of $A\mod_{ren}$ cannot quite be thought of
as vector spaces with extra structure. This introduces a number of technical
difficulties; for instance, a functor to $A\mod_{ren}$ that ``looks" $t$-exact
(meaning its composition down to $\Vect$ is) requires an argument to justify
why it is actually $t$-exact (if indeed it is!).

Let us give some examples of renormalization. First, if $A \in \Vect^{\heart}$
is a usual commutative algebra of finite type over $k$, then 
$A\mod_{naive}$ certainly equals $A\mod$. However, as in \cite{indcoh}, there
is an alternative category $A\mod_{ren} \coloneqq \IndCoh(\Spec(A))$, satisfying
the above properties. 
This category has a functor $\Gamma^{\IndCoh}(\Spec(A),-):A\mod_{ren} \to \Vect$
that is not conservative unless $A$ is a regular ring.

Second, for $A = U(\fg((t)))$ the (completed) enveloping algebra, a ``correct"
DG category $\fg((t))\mod$ was introduced in \cite{dmod-aff-flag} \S 23.
Here the distinction is not like ``$\QCoh$ vs. $\IndCoh$": the renormalized
derived category is the only one that should be considered (e.g., the
naive one does not admit a strong $G(K)$-action).
(In \cite{whit} \S 1.20, we tried to write an informal introduction to the 
renormalization procedure in this case.)

Finally, we remark that these issues are generally compounded
by considering actions of group schemes (and group indschemes)
that are not of finite type. 

\subsection{Ind-coherent sheaves in infinite type}

This theory is the subject of \S \ref{s:indcoh}. 

We introduce a DG category $\IndCoh^*(S)$ for $S$ any reasonable indscheme,
as defined in \S \ref{ss:indsch} (following \cite{hitchin}). 
This class includes arbitrary quasi-compact quasi-separated eventually
coconnective DG schemes; in particular, any qcqs classical scheme.
This category compactly generated for any such indscheme.

The superscript $*$ in the notation follows \cite{dmod}, where a similar
construction was given for $D$-modules. The notation indicates that
the construction $S \mapsto \IndCoh^*(S)$ is covariantly functorial. 

There is a formally dual operation: one can define 
$\IndCoh^!(S) \coloneqq \TwoHom_{\DGCat_{cont}}(\IndCoh^*(S),\Vect)$
as the formally \emph{dual} DG category.
Then $S\mapsto \IndCoh^!(S)$ is contravariantly functorial, and pullback functors
are denoted as upper-$!$. For $S$ locally almost of finite type,
Serre duality provides a canonical equivalence $\IndCoh^*(S) \simeq \IndCoh^!(S)$
(see \cite{indcoh}), accounting for the bivariant functoriality in that setting.

\subsection{}

We also develop some theory of $\IndCoh^*(S)$ for $S$ allowed to be stacky.
The theory in this case is somewhat more subtle. We refer to \S \ref{s:indcoh}
for further discussion.

\subsection{}

There has recently been a great deal of interest among some 
geometric representation theorists
in $\IndCoh$ in infinite type situations: there are various conjectures for
Koszul duality and geometric Langlands type statements in such settings.

Although the 
theory we present in \S \ref{s:indcoh} is not quite as exhaustive as its $D$-module
analogue in \cite{dmod}, we hope that it provides useful foundations and some
clarity in such geometric representation theoretic situations.

\subsection{Weak $G(K)$-actions on categories}\label{ss:intro-weak}

As discussed above, there is a theory of $G(K)$-actions on DG categories
$\sC \in \DGCat_{cont}$. By definition, this is a $D^*(G(K))$-module structure
on $\sC$, where $D^*(G(K))$ is the monoidal DG category of $*$-D-modules on 
$G(K)$ in the sense of \cite{beraldo-*/!} and \cite{dmod}.
Because of the appearance of $D$-modules, such actions are often
called \emph{strong} $G(K)$-actions.

In \S \ref{s:weak-pro} and \ref{s:weak-ind}, we introduce a theory of
\emph{weak} $G(K)$-actions on $\sC \in \DGCat_{cont}$. These are goverened
by a certain category that we denote $G(K)\mod_{weak}$.

The weak theory is more subtle than the strong theory. The main technical
difficulty is that the forgetful functor $G(K)\mod_{weak} \to \DGCat_{cont}$
is \emph{not} conservative; that is, we cannot think of weak $G(K)$-module categories
as DG categories with extra structure. This phenomenon is obviously 
parallel to the discussion of \S \ref{ss:intro-ren}, but occurs a categorical
level higher, leading to additional technical complications.
We remark that to have various reasonable functoriality properties
and universal properties, it is essential to work with this
definition (e.g., to have the universal property for
$\fg((t))\mod$ from \S \ref{ss:intro-str-weak} below).

\begin{rem}

The definition of $G(K)\mod_{weak}$ is essentially designed so that
the functor of weak invariants with respect to any compact open subgroup
\emph{is} conservative. In other words, $G(K)\mod_{weak}$ can be realized
as the category of modules over the suitable (weak) Hecke category with
respect to any compact open subgroup. Continuing the analogy
$G(K)\mod_{weak} \sim \IndCoh(S)$ from above, this description is
analogous to finding a Koszul dual description of $\IndCoh(S)$. 
We remark that describing $G(K)\mod_{weak}$ in terms of a Hecke
category breaks symmetry (because of the choice of compact open subgroup)
and can be inconvenient; much of \S \ref{s:weak-ind} is about restoring this
symmetry.

\end{rem}

\subsection{}

A key theorem of Gaitsgory for weak $G$-actions on categories is that if
$\sC \in G\mod_{weak}$, there is a canonical equivalence
$\sC_{G,w} \isom \sC^{G,w}$ between the weak invariants and coinvariants
for this action; see \cite{shvcat} and \cite{paris-notes}, 
or the discussion in \S \ref{s:weak-pro} below.

For $G(K)$, there is an additional twist. In Proposition \ref{p:chi},
we show that in a suitable sense, $G(K)$ admits a canonical \emph{modular character} 
such that invariants and coinvariants differ by a twist
along it. 

\subsection{}\label{ss:intro-str-weak}

In \S \ref{s:strong}, we show how to construct weak actions of $G(K)$
from strong actions. (Due to the technical issues discussed in 
\S \ref{ss:intro-weak}, the construction is not quite trivial.)

From this perspective, $\fg((t))\mod \in G(K)\mod$ has a universal
property: giving a functor $\sC \to \fg((t))\mod \in G(K)\mod$ is equivalent
to giving a functor $\sC \to \Vect \in G(K)\mod_{weak}$. 
Such phenomena are certainly well-known in the finite-dimensional setting,
but were not previously available for $G(K)$.

\subsection{}

We draw the reader's attention to one important
method that we use. 

As above, the functor $G(K)\mod_{weak} \to \DGCat_{cont}$
is not conservative, roughly in analogy with 
$\Gamma^{\IndCoh}(S,-):\IndCoh(S) \to \Vect$ for a singular
affine scheme $S = \Spec(A)$. 
In the latter case, we can lift objects of 
$\QCoh(S)^+$ canonically to $\IndCoh(S)$ by identifying
$\IndCoh(S)^+ \isom \QCoh(S)^+$ in the standard way.

Although such ideas obviously do not work for $G(K)\mod_{weak}$
directly, we give
some way of canonically lifting 
$\sC \in \IndCoh^*(G(K))\mod$ to an object of $G(K)\mod_{weak}$
that we call \emph{canonical renormalization}. It appears in 
\S \ref{ss:can-renorm} for group schemes
and \S \ref{ss:can-renorm-tate} for group indschemes.
This theory plays a key role in \S \ref{s:hc}, and some
crucial methods are developed there for constructing non-trivial
examples. 

\subsection{Semi-infinite cohomology}

In \S \ref{s:sinf}, we show how the modular character discussed above
naturally leads to the theory of semi-infinite cohomology for
affine Lie algebras as introduced in \cite{feigin-sinf}.
In our perspective, the modular character naturally gives rise to a central
extension of the loop group, 
which we identify with the standard \emph{Tate} central extension.

This is a general principle: \emph{semi-infinite features of
loop group geometry are best explained by their modular characters}.  

There have been various previous attempts to give conceptual
constructions of semi-infinite cohomology: see 
\cite{voronov}, \cite{arkhipov-sinf-tate}, and \cite{positselski} for
example. Our construction emphasizes the connection to the
higher categorical representation theory of the loop group.

In addition, our construction makes evident some functoriality properties
of semi-infinite cohomology that appear subtle from the Cliffordian perspective
and have been missing in the literature.
For example, our Remark \ref{r:sinf-strong-eq} 
plays a key role
in the proof of the linkage principle for affine $\sW$-algebras 
(in the positive level case) given 
recently in \cite{gurbir-w}.
Similarly, Proposition \ref{p:inv-coinv-w/str}
(and some results of \S \ref{s:sinf}) fill in \S 1.3 from \cite{kl-extn}. 

\subsection{Harish-Chandra bimodules}

Finally, we briefly want to draw the interested reader's attention to the
fact that \S \ref{s:hc} implicitly provides tools to
study the category of affine Harish-Chandra bimodules.

For a level $\kappa$ of $\fg$, define $\sH\sC_{G,\kappa}^{\on{aff}}$ as 
$\TwoEnd_{G(K)\mod_{\kappa}}(\widehat{\fg}_{\kappa}\mod)$,
i.e., as the monoidal DG category of endomorphisms of 
$\widehat{\fg}_{\kappa}\mod$ considered as a category with a
level $\kappa$ $G(K)$-action (this notion is defined
in \S \ref{s:critical}). By the results of \S \ref{s:strong},
this category may also be calculated as 
$\widehat{\fg}_{\kappa}\mod_{G(K),w}$, the category of weak
$G(K)$-coinvariants.

The category $\sH\sC_{G,\kappa}^{\on{aff}}$ plays a central
role in quantum local geometric Langlands, but is difficult
to study explicitly.
Theorem \ref{t:opers} amounts to a construction of a monoidal
functor:

\[
\IndCoh^!(\Op_{\ld{G}}) \to \sH\sC_{G,crit}^{\on{aff}}
\]

\noindent so the proof must provide some basic study of the right hand side.

The main technical work in this study is implicit in \S \ref{s:hc}.
The careful reader will find that the most technical results
in \S \ref{s:hc} are about categories
$\widehat{\fg}_{\kappa}\mod^{K,w}$ for $K \subset G(K)$ compact open,
and understanding these categories is an essential prerequisite to
understanding $\widehat{\fg}_{\kappa}\mod_{G(K),w} = 
\sH\sC_{G,\kappa}^{\on{aff}}$.

\subsection{Relation to older approaches}

It is roughly fair to say that this text is
an update of the appendices to \cite{fg2}, incorporating
modern homotopical techniques and working with
unbounded derived categories.

We remark that the extension to unbounded
derived categories is essential in applications
(see already \cite{dmod-aff-flag}), and the reader will observe that
most of the difficulties that come up in our setting
exactly have to do with 
the difference between bounded below derived categories
and unbounded ones.

Many of our constructions are also close in spirit to \cite{positselski},
although our perspective and emphasis are somewhat different.

\subsection{Leitfaden}

Here are the (essential) logical dependencies.

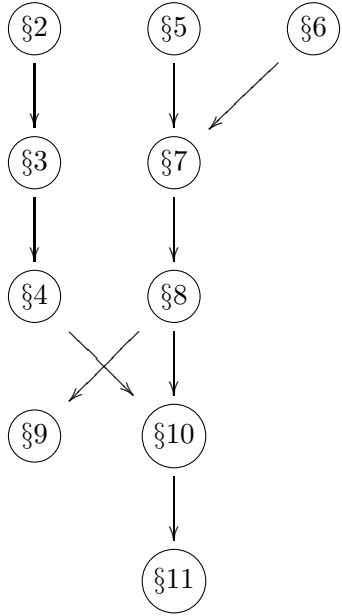
\[
\xymatrix{
\circled{\S \ref{s:tens}} \ar[d] & 
\circled{\S \ref{s:weak-pro}} \ar[d] & 
\circled{\S \ref{s:indcoh}} \ar[dl] \\
\circled{\S \ref{s:basics}} \ar[d] &
\circled{\S \ref{s:weak-ind}} \ar[d] \\
\circled{\S \ref{s:ren}} \ar[dr] &
\circled{\S \ref{s:strong}} \ar[dl] \ar[d] \\
\circled{\S \ref{s:sinf}} & \circled{\S \ref{s:hc}} \ar[d] &
\\
& \circled{\S \ref{s:critical}}
}
\]

Briefly, \S \ref{s:tens}-\ref{s:ren} develops the theory of
$\overset{\rightarrow}{\otimes}$-algebras and renormalization data.
The theory of weak group actions on categories is developed
in \S \ref{s:weak-pro} and \S \ref{s:weak-ind}, and is
related to strong group actions in \S \ref{s:strong}. We apply these
ideas to semi-infinite cohomology in \S \ref{s:sinf}.
The material we need on ind-coherent sheaves 
is developed in \S \ref{s:indcoh}. Finally, \S \ref{s:hc} introduces
Harish-Chandra data for group indschemes acting on 
$\overset{\rightarrow}{\otimes}$-algebras, and \S \ref{s:critical}
gives an application at the critical level. 

\subsection{Conventions}

We always work over the base field $k$ of characteristic
zero.

We use higher categorical language without mention: by \emph{category}, we
mean $\infty$-category in the sense of \cite{htt}, and similarly for monoidal
category and so on. We let $\Cat$ denote the
category of $(\infty-)$-categories, and
$\Gpd$ denote the category of $(\infty-)$groupoids.
We also refer to \S \ref{ss:convs-2} for some essential
notation used throughout the paper. 

Similarly, by \emph{scheme} we mean \emph{derived scheme over $k$} 
in the sense of \cite{grbook}, or \emph{spectral scheme over $k$}
in the sense of \cite{sag}. Similarly, by 
\emph{indscheme}, we mean what \cite{indschemes} calls
\emph{DG indscheme}. Algebras of all flavors are
assumed to be derived unless otherwise stated.
We emphasize that when we speak of \emph{DG} objects
or chain complexes of vector spaces or the like, we
really understand objects of suitable $\infty$-categories,
not explicit cochain models for them; we refer to 
\cite{grbook} \S 1 for an introduction to this way of thinking.

For $\sC$ a DG category, we let 
$\ul{\Hom}_{\sC}(\sF,\sG) \in \Vect$ denote the
$\Hom$-complex between objects, and we let 
$\Hom_{\sC}(\sF,\sG) \in \Gpd$ denote groupoid of
maps in $\sC$ regarded as an abstract category, i.e.,
forgetting the DG structure. We remind that
$\Omega^{\infty} \ul{\Hom}_{\sC}(\sF,\sG)  =
\Hom_{\sC}(\sF,\sG)$, where on the left hand side
we are regarding $\Vect$ as the $\infty$-category of
$k$-module spectra.

For a DG category $\sC$ with a $t$-structure, we let
$\tau^{\geq n}$ and $\tau^{\leq n}$ denote the truncation
functors; we use cohomological gradings throughout (as indicated
by the use of superscripts).

\subsection{Acknowledgements}

We thank Justin Campbell, Gurbir Dhillon, 
Dennis Gaitsgory, Nick Rozenblyum, and Harold Williams
for a number of essential discussions related to this text.

\section{Monoidal structures}\label{s:tens}

\subsection{}

In this section, we define monoidal structures
$\overset{!}{\otimes}$ and $\overset{\rightarrow}{\otimes}$ 
on $\Pro\Vect$. 
This material follows \cite{beilinson-top-alg}, the appendices
to \cite{fg2}, 
\cite{gaitsgory-kazhdan}, and 
\cite{positselski} Appendix D. 
The main difference with those sources is that
we work in the derived setting, which requires somewhat 
restructuring the usual definitions.

\subsection{Notation}\label{ss:convs-2}

Let $\Pro \Vect$ denote the pro-category of (the DG category) 
$\Vect$; we refer to \cite{htt} \S 7.1.6 and \cite{sag} \S A.8.1
for details on pro-categories.

\begin{rem}

There are cardinality issues to keep in mind when working
with a category such as $\Pro\Vect$.
Let us remind some relevant ideas from \cite{htt} \S 5. 
A category is \emph{accessible} if it satisfies a certain 
hypothesis involving cardinalities (and is idempotent complete); 
for instance, compactly generated
categories are accessible, where the cardinality condition is the
hypothesis that the subcategory of compact objects is
\emph{essentially small}. We also remind that 
\emph{presentable} accessible and cocomplete (i.e., admitting 
colimits). This hypothesis is designed so the ``naive" proof
of the adjoint functor theorem (involving potentially large limits)
goes through; in particular, presentable categories admit limits.

Now for an accessible category $\sC$ admitting finite limits, 
(e.g., $\sC$ is presentable), $\Pro(\sC)^{op}$ is defined
to be the category
of \emph{accessible} functors $\sC \to \Gpd$ preserving
finite limits. Note that $\Pro(\sC)$ is not presentable,
so the adjoint functor theorem and its relatives do not apply.

\end{rem}

By a \emph{DG category}, we mean a stable ($\infty$-)category
with a $\Vect^c$-module category structure where the
action is exact in each variable separately.
Here $\Vect^c \subset \Vect$ is the subcategory of 
compact objects 
(i.e., perfect objects, i.e., bounded complexes
with finite dimensional cohomologies).
By a DG functor,
we mean an exact functor compatible between the $\Vect^c$-module
structures. We let $\DGCat_{big}$ denote the 2-category of such.

We let $\DGCat \subset \DGCat_{big}$ denote the 2-category
of accessible DG categories under
accessible DG functors. Recall that
$\DGCat_{cont}$ denotes the 2-category of
cocomplete, presentable DG categories under continuous functors. 

For the set up of topological algebras, it would be more natural
to work with spectra and stable categories, but 
given our convention that we work over $k$, we stick 
to the language of DG categories.

We remind (c.f. \cite{higheralgebra} \S 4.8.1) 
that $\DGCat_{big}$ has a canonical symmetric monoidal
structure with unit $\Vect^c$. We denote
this monoidal structure by $\ol{\otimes}$.
If we worked in the spectral setting, functors
$\sC \ol{\otimes} \sD \to \sE$ would be the
same as functors $\sC \times \sD \to \sE$ exact
in each variable separately; in the DG setting,
they should be called \emph{bi-DG} functors.

Similarly, $\DGCat_{cont}$ has a symmetric
monoidal structure $\otimes$ such that
functors $\sC \otimes \sD \to \sE$ are the
same as bi-DG functors that commute 
with colimits in each variable separately.

For $\sF \in \sC$ and $\sG \in \sD$, we let
$\sF \boxtimes \sG$ denote the induced object of
$\sC \ol{\otimes} \sD$ or $\sC \otimes \sD$ as appropriate.

For $\sC$ a compactly generated DG category, we let
$\sC^c$ denote its subcategory of compact objects,
as in the case of $\Vect$ above.

\subsection{Review of topological tensor products}

Following the above references, we seek two
tensor products $\overset{!}{\otimes}$ and $\overset{\rightarrow}{\otimes}$
on $\Pro \Vect$. Ignoring homotopy coherences issues for the moment, we recall
the basic formulae characterizing these two tensor products concretely.

Roughly, if $V = \lim_i V_i, W = \lim_j W_j \in \Pro\Vect$
are filtered limits with $V_i, W_j \in \Vect$, then:

\[
V \overset{!}{\otimes} W = \underset{i,j}{\lim} \, V_i \otimes W_j.
\]

The (non-symmetric)
monoidal product $\overset{\rightarrow}{\otimes}$ 
is characterized by the fact that it is a bi-DG functor, and the functor:

\[
V \overset{\rightarrow}{\otimes} - : \Pro\Vect \to \Pro \Vect
\]

\noindent commutes with limits, while the functor:

\[
V \overset{\rightarrow}{\otimes} - : \Vect \to \Pro\Vect
\]

\noindent commutes with colimits.

Explicitly, if $W_j = \colim_k W_{j,k}$ with $W_{j,k} \in \Vect^c$,
we have:

\[
V \overset{\rightarrow}{\otimes} W = 
\underset{j}{\lim} \, \underset{k}{\colim} \, V \otimes W_{j,k}.
\]

\noindent (Clearly we should allow the indexing set for the
terms ``$k$" to depend on $j$.)

These two tensor products are connected as follows.
For $V_1,V_2,W_1,W_2$, there is a natural map:

\begin{equation}\label{eq:rar-!}
(V_1 \overset{!}{\otimes} V_2) 
\overset{\rightarrow}{\otimes}
(W_1 \overset{!}{\otimes} W_2) \to 
(V_1 \overset{\rightarrow}{\otimes} W_1) 
\overset{!}{\otimes}
(V_2 \overset{\rightarrow}{\otimes} W_2).
\end{equation}

\noindent In particular, there is a natural map:

\[
V \overset{\rightarrow}{\otimes} W \to V \overset{!}{\otimes} W.
\]

\subsection{Topological tensor products in the derived setting}

We now formally define the above structures and characterize
their categorical properties.

The tensor product $\overset{!}{\otimes}$ on $\Pro \Vect$ is
easy: as $\Ind$ of a monoidal category has a canonical
tensor product, so does $\Pro$.

\subsection{}

To construct $\overset{\rightarrow}{\otimes}$, first 
note that $\Pro(\Vect)^{op}$ is by definition the
category $\TwoHom(\Vect,\Gpd)$ of accessible functors
$\Vect \to \Gpd$. Any such functor factors canonically
as:

\[
\Vect \xar{F} \Vect \xar{\Oblv} \mathsf{Spectra} 
\xar{\Omega^{\infty}} \Gpd
\]

\noindent with $F$ a DG functor, i.e., 
$\Pro\Vect = \TwoHom_{\DGCat}(\Vect,\Vect)^{op}$.

\begin{notation}

For $V \in \Pro \Vect$, we let $F_V$ denote the
induced functor $\Vect \to \Vect$.
Clearly $F_V = \ul{\Hom}_{\Pro\Vect}(V,-)$.
Define $V \overset{\rightarrow}{\otimes} -:\Pro\Vect \to \Pro\Vect$
as the ``partially-defined left adjoint" to $F_V$,
i.e., for $W \in \Vect$ and $U \in \Pro\Vect$, 
we have functorial isomorphisms:

\begin{equation}\label{eq:rar-adj}
\ul{\Hom}_{\Pro\Vect}(V \overset{\rightarrow}{\otimes} W, U) \simeq
\ul{\Hom}_{\Vect}(W,F_V(U)).
\end{equation}

\noindent We extend this construction 
to general $W \in \Pro\Vect$ by
right Kan extension.

\end{notation}

We extend this construction to
a monoidal structure by reinterpreting it as
composition of functors in $\TwoHom_{\DGCat}(\Vect,\Vect)$. 
That is, we observe:

\[
F_V \circ F_W \simeq F_{V \overset{\rightarrow}{\otimes} W}.
\]

\noindent The left hand side extends to the evident
monoidal structure on $\TwoHom_{\DGCat}(\Vect,\Vect)^{op}$.

\subsection{Comparison of tensor products}\label{ss:tens-compat}

We now wish to give compatibilities
between $\overset{\rightarrow}{\otimes}$ and $\overset{!}{\otimes}$.
Roughly, we claim that these form a ``lax $\sE_2$" 
structure.\footnote{For what follows, it is important to think of the 
$\sE_2$ operad as $\sE_1^{\otimes 2}$ and not as the little discs operad:
the laxness evidently breaks the $SO(2)$-symmetry.}

Let $\Alg(\Cat)$ denote the category of monoidal categories and
lax monoidal functors,
which we consider as a symmetric monoidal category under products.
We claim that $(\Pro\Vect,\overset{\rightarrow}{\otimes})$
is a commutative algebra in this category with
operation $\overset{!}{\otimes}$. Note that this structure 
encodes the natural transformations \eqref{eq:rar-!}.
(In \S \ref{ss:alg-tensor}, we give some
simple consequences, and the reader may wish to skip ahead.)

To construct this compatibility, 
note that if we write $\Pro\Vect$ as $\TwoEnd_{\DGCat}(\Vect)^{op}$,
then $\overset{!}{\otimes}$ corresponds to \emph{Day convolution}.
Then this follows from formal facts about Day convolution.

With respect to the symmetric monoidal structure
$\ol{\otimes}$ on $\DGCat$, 
we have the internal $\Hom$ functor:

\[
\TwoHom_{\DGCat}(\sC,\sD) \in \DGCat_{big}
\]

\noindent which is the usual (DG) category of DG functors.
If $\sC$ is a monoidal DG category
and $\sD \in \Alg(\DGCat_{cont})$, then recall that
$\TwoHom_{\DGCat}(\sC,\sD)$ has the usual Day convolution monoidal structure.
It is characterized by the fact that:

\[
\TwoHom_{\Alg^{lax}(\DGCat)}
(\sE,\TwoHom_{\DGCat}(\sC,\sD)) =
\TwoHom_{\DGCat}(\sC \ol{\otimes} \sE, \sD)
\]

\noindent where by $\Alg^{lax}(\DGCat)$ we mean
monoidal DG categories under \emph{lax} monoidal functors. 

It is straightforward to see that Day convolution has the property
that the composition functor:

\[
\TwoHom_{\DGCat}(\sC, \sD) \ol{\otimes}
\TwoHom_{\DGCat}(\sD, \sE) \to 
\TwoHom_{\DGCat}(\sC, \sE)
\]

\noindent is lax monoidal (assuming $\sD,\sE \in \Alg(\DGCat_{cont},\otimes)$).
This immediately implies our claim about the monoidal structures
on $\Pro\Vect$.

\subsection{}

First, suppose that $\sC,\sD \in \DGCat_{cont}$ with $\sD$ compactly 
generated. Then observe that there is a canonical bi-DG functor:

\[
\Pro(\sC) \times \Pro(\sD) \xar{-\overset{\rightarrow}{\boxtimes}-} \Pro(\sC \otimes \sD)
\]

\noindent computed as follows. 

If $\sG \in \sD^c$, then the induced functor
$\xar{-\overset{\rightarrow}{\boxtimes}\sG}:\Pro(\sC) \to \Pro(\sC \otimes \sD)$ 
is the right Kan extension of the functor 
$-\boxtimes\sG:\sC \to \sC \otimes \sD \subset \Pro(\sC \otimes \sD)$.
In general, for $\sF \in \Pro(\sC)$, the functor 
$\overset{\rightarrow}{\boxtimes}-:\Pro(\sD) \to \Pro(\sC \otimes \sD)$
is computed by first left Kan extending the above functor from
$\sD^c$ to $\sD$, and then right Kan extending to $\Pro(\sD)$.

This operation is functorial in the sense that for
$F:\sC_1\to \sC_2 \in \DGCat_{cont}$, the diagram:

\[
\xymatrix{
\Pro(\sC_1) \times \Pro(\sD) \ar[r] \ar[d]^{-\overset{\rightarrow}{\boxtimes}-}
& \Pro(\sC_2) \times \Pro(\sD) \ar[d]^{-\overset{\rightarrow}{\boxtimes}-} \\
\Pro(\sC_1 \otimes \sD) \ar[r] & \Pro(\sC_2 \otimes \sD)
}
\]

\noindent canonically commutes. Indeed, this follows immediately
from:

\begin{lem}\label{l:pro-cts}

For $F: \sC \to \sD \in \DGCat_{cont}$, $\Pro(F):\Pro(\sC) \to \Pro(\sD)$
commutes with limits and colimits.

\end{lem}

\begin{proof}

$\Pro(F)$ tautologically commutes with limits. For the commutation
with colimits, note that $F$ admits an (accessible) right adjoint
$G$ by the adjoint functor theorem, so $\Pro(F)$ admits
the right adjoint $\Pro(G)$.

\end{proof}

\section{Modules and comodules}\label{s:basics}

\subsection{}

We let $\Alg^{\overset{\rightarrow}{\otimes}}$ denote the
category of (associative, unital) algebras
in $\Pro\Vect$ with respect to $\overset{\rightarrow}{\otimes}$. 
We refer to objects of $\Alg^{\overset{\rightarrow}{\otimes}}$
as \emph{$\overset{\rightarrow}{\otimes}$-algebras}. 
We remark that (lower categorical analogues of) 
such objects have been variously
referred to as \emph{topological algebras} or 
\emph{topological chiral algebras}
in the literature. 

In this section, we give basic definitions about 
modules over $\overset{\rightarrow}{\otimes}$-algebras.
Note that we are exclusively interested in \emph{discrete} modules,
i.e., modules in $\Vect$, not in $\Pro\Vect$, and our notation
will always take this for granted.

\begin{terminology}

We generally use the term \emph{discrete} to refer
to objects of $\Vect \subset \Pro\Vect$. For example,
we say a $\overset{\rightarrow}{\otimes}$-algebra
is \emph{discrete} if its underlying object
lies in $\Vect$ (in which case this structure 
is equivalent to a usual associative DG algebra structure).

This should not be confused with the usage of this phrase
in homotopy theory, where it is often used for an object
in the heart of a $t$-structure. In that setting, we prefer the 
term \emph{classical}, so e.g., a
\emph{classical} $\overset{\rightarrow}{\otimes}$-algebra
is one whose underlying object lies in $\Pro\Vect^{\heart}$.

\end{terminology}

\subsection{Comparison with comonads}\label{ss:comonad}

First, note that by construction, we have:

\[
\Alg^{\overset{\rightarrow}{\otimes}}
 \simeq \{\text{accessible DG comonads on }\Vect.
\}^{op}
\]

\noindent Here \emph{DG} indicates
that we have compatible comonad and DG functor structures;
in the stable setting, this would simply mean the underlying functor
of our comonad is exact.

Let $A$ be a $\overset{\rightarrow}{\otimes}$-algebra.
Define $A\mod_{top}$ as $A\mod(\Pro\Vect)$.
We define $A\mod_{naive}$ to be the ``naive" category of discrete $A$-modules:

\[
A\mod_{top} \underset{\Pro\Vect}{\times} \Vect.
\]

\noindent That is, an object of $A\mod_{naive}$ has an underlying vector
space $M \in \Vect$, an action map 
$A \overset{\rightarrow} M \to M \in \Pro\Vect$, and the usual
(higher) associativity data. (We use the notation ``naive" by
comparison with the renormalization setting introduced below.)

By \eqref{eq:rar-adj}, if $S \coloneqq F_A$ is the comonad corresponding to
$A$, we have a canonical equivalence:

\[
A\mod_{naive} \simeq S\comod 
\]

\noindent compatible with forgetful functors to $\Vect$.

As a consequence, $A\mod_{naive}$ is presentable and the forgetful
functor $\Oblv:A\mod_{naive} \to \Vect$ is continuous, conservative,
and, of course, comonadic.

\begin{rem}

To conclude: the language of $\overset{\rightarrow}{\otimes}$-algebras
is equivalent to the language of (DG) comonads on $\Vect$. 
Therefore, the wisdom in using this language may be reasonably questioned
by the reader; we use it here to connect to older work, and because
pro-vector spaces are typically nicer 
to describe than their corresponding comonads.

\end{rem}

\subsection{Tensor products}\label{ss:alg-tensor}

We now spell out what the material of \S \ref{ss:tens-compat}
means for $\overset{\rightarrow}{\otimes}$-algebras and 
their modules. (We remind that \S \ref{ss:tens-compat}
is a souped up version of \eqref{eq:rar-!}, which
may be more helpful to refer to.)

By \S \ref{ss:tens-compat}, $\Alg^{\overset{\rightarrow}{\otimes}}$
is symmetric monoidal with tensor product:

\[
A,B \mapsto A \overset{!}{\otimes} B.
\]

Similarly, we have the bi-DG functor:

\[
\begin{gathered}
A\mod_{top} \times B\mod_{top} \to 
A\overset{!}{\otimes} B \mod_{top} \\
(M,N) \mapsto M\overset{!}{\otimes} N.
\end{gathered}
\]

\noindent Clearly this induces a bi-DG functor:

\[
\begin{gathered}
A\mod_{naive} \times B\mod_{naive} \to 
A\overset{!}{\otimes} B \mod_{naive} \\
(M,N) \mapsto M \otimes N.
\end{gathered}
\]

\noindent This functor commutes with
colimits in each variable separately, so induces:

\begin{equation}\label{eq:tens-cat}
A\mod_{naive} \otimes B\mod_{naive} \to 
(A\overset{!}{\otimes} B)\mod_{naive}.
\end{equation}

To properly encode all higher categorical data,
note that we have upgraded $A\mapsto A\mod_{naive}$
to a contravariant lax symmetric monoidal functor
from $\overset{\rightarrow}{\otimes}$-algebras to $\DGCat_{cont}$.

\subsection{Forgetful functors}\label{ss:forget}

Like every functor in $\DGCat_{cont}$, $F$ is pro-representable,
i.e., there is a filtered projective system $i \mapsto \sF_i \in 
A\mod_{naive}$ such that:

\[
\colim_i \ul{\Hom}_{A\mod_{naive}}(\sF_i,-) = \Oblv.
\]

In fact, we claim that $\lim_i \Oblv(\sF_i) \in \Pro\Vect$ is the
pro-vector space underlying $A$.

Indeed, let $\Phi:\Vect \to A\mod_{naive}$ be the functor
right adjoint to the forgetful functor. 
Note that $\Oblv\Phi = F_A$ ($ \coloneqq \ul{\Hom}_{\Pro\Vect}(A,-)$), 
as is clear in the comonadic picture.

Then for any $V \in \Vect$, we obtain:

\[
\begin{gathered}
\ul{\Hom}_{\Pro\Vect}(\underset{i}{\lim} \, \Oblv(\sF_i), V) = 
\underset{i}{\colim} \, \ul{\Hom}_{\Vect} (\Oblv(\sF_i), V) = \\
\underset{i}{\colim} \, \ul{\Hom}_{\Vect} (\sF_i, \Phi(V)) =
\Oblv\Phi(V) = \ul{\Hom}_{\Pro\Vect}(A,-)
\end{gathered}
\]

\noindent as desired.

\subsection{$t$-structures}

Recall that $\Pro\Vect$ has a natural $t$-structure
with $(\Pro\Vect)^{\leq 0} = \Pro(\Vect^{\leq 0})$
and $(\Pro\Vect)^{\geq 0} = \Pro(\Vect^{\geq 0})$;
we omit the parentheses in the sequel as there can be 
no confusion. 

In the remainder of the section, we will be interested in
\emph{connective} $\overset{\rightarrow}{\otimes}$-algebras,
i.e., such algebras $A$ in $\Pro\Vect^{\leq 0}$.
Clearly this hypothesis is equivalent to the
comonad $F_A$ being left $t$-exact.

From this latter description, we 
see that $A\mod_{naive}$ carries a canonical 
$t$-structure such that $\Oblv:A\mod_{naive} \to \Vect$
is $t$-exact. Because $\Oblv$ commutes with colimits,
this $t$-structure is necessarily right complete.

\subsection{Convergence}\label{ss:conv}

In order to formulate Proposition \ref{p:alg-vs-cats}
below, we introduce the following terminology.

For $V \in \Pro\Vect$, the \emph{convergent completion}
of $V$ is:

\[
\lim_n \tau^{\geq -n}(V) \in \Pro\Vect.
\]

\noindent We say that $V$ is \emph{convergent} 
if the natural map $V \to \widehat{V}$
is an isomorphism.
Note that $V$ is convergent if and only if it lies
in $\Pro(\Vect^+) \subset \Pro(\Vect)$
(or equivalently: $F_V$ is left Kan extended
from $\Vect^+$). 

In particular, we obtain that
connective convergent pro-vector spaces are (contravariantly)
equivalent to left $t$-exact 
functors $\Vect^+ \to \Vect^+ \in \DGCat$.
Under this dictionary, connective
$\overset{\rightarrow}{\otimes}$-algebras are the same
as left $t$-exact (accessible) DG comonads on 
$\Vect^+$. 

\begin{rem}

If $A$ is a connective
$\overset{\rightarrow}{\otimes}$-algebra,
then its convergent completion $\widehat{A}$ is as well, and
$\widehat{A}\mod_{naive}^+ \isom A\mod_{naive}^+$.

\end{rem}

\subsection{Comparison with categorical data}

We have the following psychologically important result.

\begin{prop}\label{p:alg-vs-cats}

The functor:

\[
\begin{gathered}
\{\text{convergent, connective $\overset{\rightarrow}{\otimes}$-algebras}\} \to 
\DGCat_{/\Vect^+} \\
A \mapsto \big(\Oblv:A\mod_{naive}^+ \to \Vect^+\big)
\end{gathered}
\]

\noindent is fully-faithful. A DG category
$\sC$ with structural functor $F:\sC \to \Vect^+$ lies in the
essential image of this map if and only if:

\begin{itemize}

\item $F$ is conservative.

\item $\sC$ admits a (necessarily unique) $t$-structure for which $F$ is $t$-exact.

\item $\sC^{\geq 0}$ admits arbitrary colimits, and the functor
$F:\sC^{\geq 0} \to \Vect^{\geq 0}$ preserves such colimits.

\end{itemize}

Under this equivalence, $\sC$ is the bounded below derived
category of its heart $\sC^{\heart}$ with $F$
the derived functor of its restriction $\sC^{\heart}\to \Vect^{\heart}$
if and only if the corresponding $\overset{\rightarrow}{\otimes}$-algebra
$A$ is classical (i.e., lies in $\Pro\Vect^{\heart}$). 

\end{prop}

We first recall the following standard result about
simplicial objects, see e.g. \cite{higheralgebra} Remark 1.2.4.3.

\begin{lem}\label{l:cosimp-summand}

For a cosimplicial object $\sF^{\dot}$ in a
stable (e.g., DG) category $\sC$, let
$\Tot^{\leq n} \sF^{\dot}$ be the
limit over the subcategory $\bDelta_{\leq n} \subset \bDelta$
of totally ordered sets of cardinality $\leq n+1$.

Then for $n > 0$:

\[
\Ker(\Tot^{\leq n} \sF^{\dot} \to \Tot^{\leq n-1} \sF^{\dot})
\]

\noindent is isomorphic to a 
direct summand of $\sF^n[-n]$. 

\end{lem}

\begin{proof}[Proof of Proposition \ref{p:alg-vs-cats}]

First it is straightforward to see that $A\mod_{naive}^+$ actually
satisfies the above conditions.

Suppose $F:\sC \to \Vect^+$ with the above properties is given. 
Clearly the $t$-structure on $\sC$ is bounded from below (i.e., $\sC = \sC^+$),
compatible with filtered colimits, and right complete. 
Clearly the equivalence follows if we can show such $F$ is comonadic; the 
argument is well-known, but we reproduce it here for convenience.

First, we claim $F|_{\sC^{\geq 0}}:\sC^{\geq 0} \to \Vect^{\geq 0}$
commutes with arbitrary totalizations.
By Lemma \ref{l:cosimp-summand}, if $\sF^{\dot}$ is
a cosimplicial diagram in $\sC$ with
$\sF^i \in \sC^{\geq 0}$ for all $i$, then the totalization
exists and is calculated by:

\[
\tau^{\leq n} \Tot \sF^{\dot} = \tau^{\leq n} \Tot^{\leq n+1} \sF^{\dot}.
\]

\noindent Since $\Tot^{\leq n+1}$ is a finite limit, $t$-exactness
of $F$ implies the claim.

Now observe that $F$ admits a left $t$-exact (possibly non-continuous) right
adjoint $G$, as $F|_{\sC^{\geq 0}}$ admits a left exact right adjoint. 
Then for any $\sF \in \sC$, we have $\sF \in \sC^{\geq -N}$, for $N$ large enough, so $(GF)^n(\sF) \in \sC^{\geq -N}$ for any $n$, so the totalization 
$\Tot((GF)^{\dot +1}(\sF))$ exists and is preserved by 
the conservative functor $F$, implying comonadicity. 

It remains to show the compatibility with abelian categories. 
Suppose $\sA$ is a $k$-linear abelian category with
a $k$-linear functor $F^{\heart}:\sA \to \Vect^{\heart}$ that is exact,
continuous, conservative, and accessible. Then there is a
pro-object $\lim \, \sF_i \in \Pro(\sA)$ ($\sF_i \in \sA$) 
corepresenting $F^{\heart}$. It immediately follows that
this pro-object also corepresents the derived functor
$F (\coloneqq RF^{\heart}):D^+(\sA) \to \Vect^+$
(because the functor this pro-system defines maps injectives in 
$\sA^{\heart}$ into $\Vect^{\heart}$).
By \S \ref{ss:forget}, this implies that the 
corresponding $\overset{\rightarrow}{\otimes}$-
algebra has underlying object $\lim \, F(\sF_i) \in \Pro\Vect$.
Because $F^{\heart}$ is exact, $F$ is $t$-exact, 
so $F(\sF_i) \in \Vect^{\heart}$, 
implying $\lim \, F(\sF_i) \in \Pro\Vect^{\heart}$.

Conversely, suppose $A$ is classical.
Let $\Phi:\Vect \to A\mod_{naive}$ denote the 
(possibly discontinuous) right
adjoint to the forgetful functor. For $V \in \Vect^{\heart}$,\footnote{If 
we worked with a general commutative ring $k \in \mathsf{Ab}^{\heart}$,
$V$ should be an injective $k$-module.}
$\Oblv \Phi(V) = F_A(V) = \ul{\Hom}_{\Pro\Vect}(A,V) \in \Vect^{\heart}$,
so $\Phi(V) \in A\mod_{naive}^{\heart}$. Moreover,
$\Phi(V)$ is obviously injective in $A\mod_{naive}$
in the sense that for any $\sF \in A\mod_{naive}^{\geq 0}$,
$\ul{\Hom}(\sF,\Phi(V)) = \ul{\Hom}_{\Vect}(\Oblv(\sF),V) \in \Vect^{\leq 0}$. 
For $\sF \in A\mod_{naive}^{\heart}$, the map $\sF \to \Phi \Oblv(\sF)$
is a monomorphism in $A\mod_{naive}^{\heart}$ 
(as it splits after applying $\Oblv$), so such there are ``enough" 
injective objects, implying $A\mod_{naive}^+$ is the bounded
below derived category of its heart. Moreover,
this reasoning immediately shows that the forgetful functor
is the derived functor of its restriction to the hearts.

\end{proof}

\section{Renormalization}\label{s:ren}

\subsection{}

In our applications, the naive category $A\mod_{naive}$ is typically
not the one we want. For example, the forgetful functor
$\widehat{\fg}_{\kappa}\mod \to \Vect$ is not conservative,
so the above construction \emph{does not recover the correct category
$\widehat{\fg}_{\kappa}\mod$}, i.e., 
$U(\widehat{\fg}_{\kappa})\mod_{naive} \neq \widehat{\fg}_{\kappa}\mod$.

Following \cite{dmod-aff-flag}, a key role is 
played by \emph{renormalization} of derived categories. 
We refer to \emph{loc. cit}.,
\cite{km-indcoh}, and \cite{whit} for introductions to this notion in the
setting of Kac-Moody algebras. The basics of the theory of ind-coherent sheaves 
also play an instructional role: see \cite{indcoh} for an introduction.

In this section, we give a (somewhat abstract) introduction to this formalism.

\subsection{Renormalization data}

\begin{defin}

A \emph{renormalization datum} for a connective 
$\overset{\rightarrow}{\otimes}$-algebra $A$ is 
a DG category $A\mod_{ren} \in \DGCat_{cont}$,
equipped with a $t$-structure and an equivalence
$\rho:A\mod_{naive}^+ \isom A\mod_{ren}^+ \in \DGCat$, such that:

\begin{itemize}

\item $\rho$ is $t$-exact.

\item $A\mod_{ren}$ is compactly generated with compact generators
lying in $A\mod_{ren}^+$.

\item The $t$-structure on $A\mod_{ren}$ is compactly generated:
i.e., $\sG \in A\mod_{ren}^{\geq 0}$ 
if and only $\Hom_{A\mod_{ren}}(\sF,\sG) = 0$ for every
\emph{compact} $\sF \in A\mod_{ren}^{<0}$.

\end{itemize}

We will also say $A$ is a 
\emph{renormalized} $\overset{\rightarrow}{\otimes}$-algebra
to mean $A$ is a connective $\overset{\rightarrow}{\otimes}$-algebra 
equipped with a renormalization datum. 

\end{defin}

\begin{rem}

Once and for all, we emphasize: if $A$ is renormalized, it is in particular
connective.

\end{rem}

\begin{rem}\label{r:ren-cpts}

The subcategory $A\mod_{ren}^c$ of compact objects in $A\mod_{ren}$
embeds canonically into $A\mod_{naive}$ as 
$A\mod_{ren}^c \subset A\mod_{ren}^+ \simeq A\mod_{naive}^+$.
It is immediate to see that a renormalization datum is equivalent
to a choice of such a subcategory satisfying some conditions.

\end{rem}

\begin{rem}\label{r:ren-cat}

By Proposition \ref{p:alg-vs-cats}, the category
$\Alg_{conv,ren}^{\overset{\rightarrow}{\otimes}}$ of
convergent, renormalized
$\overset{\rightarrow}{\otimes}$-algebras 
are equivalent to
some categorical data:
$\sC \in \DGCat_{cont}$, a continuous functor
$F:\sC \to \Vect$, and a $t$-structure on $\sC$ such
that $F$ is $t$-exact and conservative on $\sC^+$, and
the $t$-structure on $\sC$ is generated
by eventually coconnective compact objects. 
(We remark that $F$ completely determines the $t$-structure
in this case.)
As we will show in Theorem \ref{t:tens-ren},
this equivalence canonically upgrades to a symmetric
monoidal one.

\end{rem}

\begin{rem}\label{r:oblv-pro}

Suppose $\sC$ is a compactly generated
DG category with a continuous functor $F:\sC \to \Vect$.
Then $F$ may be pro-represented by a pro-compact
object. Comparing with \S \ref{ss:forget}, we see
that this puts significant restrictions on 
which $\overset{\rightarrow}{\otimes}$-algebras $A$ admit
renormalization data. (For example, up to convergent
completion, $A \in \Pro\Vect$ 
must be expressible as a filtered limit of some
discrete $A$-modules that are \emph{almost compact}, i.e.,
whose truncations are compact in $A\mod_{naive}^{\geq -n}$
for all $n$.) 

\end{rem}

\subsection{Examples}

We now give some examples of renormalization data.

We begin with examples when $A$ is discrete, i.e., 
$A \in \Vect^{\leq 0} \subset \Pro\Vect^{\leq 0}$.

\begin{example}[Ind-coherent sheaves]\label{e:indcoh}

Let $A$ be a commutative, connective $k$-algebra
(almost) of finite type and let $S = \Spec(A)$. 
Recall that $\IndCoh(S)$ ($\coloneqq \Ind(\Coh(S))$)
equipped with the tautological embedding 
$\Coh(S) \into \QCoh(S)^+ = A\mod^+$
defines a renormalization datum for $A$.

\end{example}

\begin{example}\label{e:indcoh-assoc}

More generally, if $A$ is a left coherent\footnote{
Recall that an algebra is left coherent if it
is connective; the category $A\mod_{coh}$
defined below is actually a DG category, i.e.,
it is closed under cones; and $\tau^{\geq -n} A \in A\mod_{coh}$
for all $n$. For example, this is the case if
$A$ is left Noetherian.}
$k$-algebra, then define $A\mod_{coh} \subset A\mod$
as the subcategory of bounded complexes with finitely 
presented cohomologies. Then  
$A\mod_{ren} \coloneqq \Ind(A\mod_{coh})$ defines
a renormalization datum.

It is straightforward to show that this renormalization datum
is initial among all renormalization data for $A$.

\end{example}

\begin{example}[Quasi-coherent sheaves]\label{e:qcoh}

If $A$ is a connective associative $k$-algebra,
then $A\mod$ itself underlies a
renormalization datum if and only if $A$ is 
eventually coconnective, i.e., $A$ is also bounded below
as a complex of vector spaces.
Indeed, recall that for renormalization data, there
is an assumption that the category be compactly generated
by eventually coconnective objects, and $A\mod$ is compactly
generated by perfect ones.

\end{example}

We now give some examples involving honestly topological
algebras.

\begin{example}\label{e:indsch-indft}

Suppose $S = \colim_i S_i$ is an ind-affine
indscheme of ind-finite type. Then $\IndCoh(S)$
is naturally a renormalization for the pro-algebra
of functions on $S$. Indeed, this is a special case
of Example \ref{e:pro-algebras}.

\end{example}

\begin{example}[Pro-algebras]\label{e:pro-algebras}

Suppose $i \mapsto A_i \in \Alg(\Vect^{\leq 0})$ 
is a projective system of algebras. 
Let $A = \lim_i A_i \in \Pro(\Vect)$. Then
$A$ is a $\overset{!}{\otimes}$-algebra, and a posteriori
a $\overset{\rightarrow}{\otimes}$-algebra. 

Suppose that:

\begin{itemize}

\item $A_i$ is left coherent.

\item Each structural map $\vph_{ij}:A_i \to A_j$ is surjective
on $H^0$ with finitely generated kernel.

\end{itemize}

Let $A_i\mod_{ren}$ be as in Example \ref{e:indcoh-assoc}.
Note that our assumptions imply that 
restriction along $\vph_{ij}$ maps
$A_j\mod_{coh}$ to $A_i\mod_{coh}$. By ind-extension,
we obtain $t$-exact functors $A_j\mod_{ren} \to A_i\mod_{ren}$.

Then define:

\[
A\mod_{ren} \coloneqq \underset{i}{\colim} \, 
A_i\mod_{ren} \in \DGCat_{cont}.
\]

\noindent Here the structural functors
are the above functors. We claim that $A\mod_{ren}$
naturally defines a renormalization datum.

As noted above, these functors preserve compact objects,
so $A\mod_{ren}$ is compactly generated. Moreover,
by \cite{whit} Lemma 5.4.3 (1), 
there is a canonical $t$-structure on 
$A\mod_{ren}$ such that each functor
$\on{res}_i:A_i\mod_{ren} \to A\mod_{ren}$ is $t$-exact. 
This $t$-structure is tautologically compactly generated
and right complete.

Moreover, there is a canonical functor
$\Oblv:A\mod_{ren} \to \Vect \in \DGCat_{cont}$
pro-represented by the object:

\[
\underset{i,n} \lim \, \on{res}_i \tau^{\geq -n} A_i \in
\Pro(A\mod_{ren}). 
\]

\noindent It is straightforward to show that
each composition $\Oblv \circ \on{res}_i:A_i\mod_{ren} \to \Vect$
is the canonical forgetful functor on $A_i\mod_{ren}$.
It immediately follows that $\Oblv$ is 
$t$-exact.

Moreover, we claim that $\Oblv$ is 
conservative on bounded below objects. Indeed, 
in the above pro-system, all objects are connective
and all structural maps are surjective on $H^0$. As the
objects $\on{res}_i H^0(A_i)$ tautologically
generate $A\mod_{ren}^{\heart}$ under colimits,
this implies the claim. 

Finally, it suffices to note that 
at the level of bounded below derived categories, this
functor $\Oblv$ defines $A$ under the 
dictionary of Proposition \ref{p:alg-vs-cats}:
indeed, $\Oblv$ of this pro-generator is manifestly
the Postnikov completion of $A$ in $\Pro\Vect$.

\end{example}

\begin{rem}

In Example \ref{e:pro-algebras}, compact objects in
$A\mod_{ren}$ are closed under truncations.

\end{rem}

\begin{example}\label{e:lattice}

This example appears somewhat in the wrong
place: it uses some terminology
from \S \ref{ss:functors-ren}, and is really
motivated by Example \ref{e:lie}.

Suppose $A$ is a connective $\overset{\rightarrow}{\otimes}$-algebra
and that we are given a morphism $\vph:A_0 \to A$ of connective
$\overset{\rightarrow}{\otimes}$-algebras 
such that the forgetful functor $A\mod_{naive}^+ \to A_0\mod_{naive}^+$
is monadic.\footnote{In particular, this functor admits 
a left adjoint. So if $A$ is discrete and $A_0 = k$, this forces
$A$ to be eventually coconnective.}

Denote this monad by $T$. Now suppose moreover that
the composition:

\[
A_0\mod_{naive}^+ \xar{T} A_0\mod_{naive}^+ 
\overset{\rho}{\into} A_0\mod_{ren}
\]

\noindent renormalizes in the sense of \S \ref{ss:functors-ren}.
(E.g., this is automatic if $A_0\mod_{ren}$
is given by Example \ref{e:pro-algebras}.)

Then $T$ clearly induces a monad on $A_0\mod_{ren}$,
and $A\mod_{ren} \coloneqq T\mod(A_0\mod_{ren})$
obviously defines a renormalization datum for $A$.

\end{example}

\begin{example}[Tate Lie algebras]\label{e:lie}

Suppose $\fh \in \Pro\Vect^{\heart}$ is a 
Tate Lie algebra. By this, we mean that the dual Tate vector space
$\fh^{\vee} \in \Pro\Vect$ is given a coLie algebra
structure with respect to the $\overset{!}{\otimes}$
symmetric monoidal structure. 
Recall that in this case, $\fh$ necessarily admits
an open profinite dimensional subalgebra 
$\fh_0 \subset \fh$, where
these hypotheses force $\fh_0 = \lim_i \fh_i$
for $\fh_i$ ranging over the 
finite dimensional Lie algebra quotients of $\fh_0$.

(For example, we might have $\fh = \fg((t))$ 
for finite dimensional $\fg$; then $\fh_0$ may
be taken as $\fg[[t]]$ and $\fh_i = \fg[[t]]/t^i\fg[[t]]$.)

Then $A_0 = U(\fh_0) \coloneqq \lim_i U(\fh_i)$
satisfies the hypotheses of Example \ref{e:pro-algebras}
(c.f. Example \ref{e:functor-ren-trun} regarding
renormalization of the monad).
Note that each $U(\fh_i)\mod_{ren} = U(\fh_i)\mod$ here,
so objects restrictions of modules $U(\fh_i)$ give
compact generators of $U(\fh_0)\mod_{ren}$.

Moreover, $A = U(\fh)$ the completed enveloping
algebra of $\fh$, $A_0 \to A$ satisfies the hypotheses
of Example \ref{e:lattice} (say, by the PBW theorem).
In particular, we obtain $U(\fh)\mod_{ren}$.

Following Gaitsgory, we denote these DG categories
by $\fh_0\mod$ and $\fh\mod$, leaving renormalization
out of the notation. 

Note that the construction of $\fh\mod$ 
recovers the format of \cite{dmod-aff-flag} \S 23.
Indeed, unwinding the constructions, we find that
compact generators are given by inducing trivial modules from
$\fk_i$ to $\fh$ 
for $\fk_i \coloneqq \Ker(\fh_0 \to \fh_i)$.

\end{example}

\begin{example}

Renormalization data is given for the affine $\sW$-algebra
in \cite{whit}: the compact generators
are denoted $\sW_{\kappa}^n$ in \emph{loc. cit}. 
Outside of the Virasoro case, this example does not
fit into any of the above patterns.
(This is closely related to the fact that
the $\sW$-algebra chiral algebras are generally 
neither commutative nor chiral envelopes.)

\end{example}

\subsection{Construction of functors}\label{ss:functors-ren}

Let $A$ be a renormalized $\overset{\rightarrow}{\otimes}$-algebra.

Suppose $\sC \in \DGCat_{cont}$ and that we are given
a DG functor $F:A\mod_{ren}^+ \simeq A\mod_{naive}^+ \to \sC$. 

\begin{defin}

$F$ \emph{renormalizes} if it is left Kan extended
from $A\mod_{ren}^c$.

\end{defin}

For $F$ as above (not necessarily assumed to renormalize),
we define $F_{ren}:A\mod_{ren} \to \sC$ as the ind-extension
of $F|_{A\mod_{ren}^c}$. Note that $F_{ren}|_{A\mod_{ren}^+}$
is the left Kan extension of $F|_{A\mod_{ren}^c}$; 
therefore, $F$ renormalizes if and only if the natural map
$F_{ren}|_{A\mod_{ren}^+} \to F$ is an isomorphism.

Suppose now that $\sC$ admits a $t$-structure compatible
with filtered colimits, that $F$ is $t$-exact, 
and that $F|_{A\mod_{ren}^{\geq 0}}$
commutes with filtered colimits. 

\begin{warning}\label{w:f_ren}

It is not true in this generality that $F$ necessarily renormalizes: 
$F_{ren}$ may fail to be (left) $t$-exact.
(See Counterexample \ref{ce:indcoh-qcoh}.)

\end{warning}

However, we claim:

\begin{equation}\label{eq:ren-trun}
\tau^{\geq 0} F_{ren} = \tau^{\geq 0} F
\end{equation}

\noindent when restricted to $A\mod_{ren}^+$.

Indeed, for $\sF \in A\mod_{ren}^+$, write $\sF = \colim_i \sF_i$
with $\sF_i$ compact. Then: 

\[
\tau^{\geq 0} F_{ren}(\sF) = \tau^{\geq 0} \underset{i}{\colim} \, F(\sF_i) =
\underset{i}{\colim} \, F(\tau^{\geq 0} \sF_i) = 
F(\tau^{\geq 0} \underset{i}{\colim} \, \sF_i) = F(\tau^{\geq 0} \sF).
\]

There are two general settings in which 
$F$ does renormalize.

\begin{example}\label{e:functor-ren-left-separated}

If the $t$-structure on $\sC$ is left separated, 
then \eqref{eq:ren-trun} clearly implies that
$F$ renormalizes.

\end{example}

\begin{example}\label{e:functor-ren-trun}

Suppose merely that $F$ is left $t$-exact (or left $t$-exact up
to shift) and that 
compact objects of $A\mod_{ren}$ are closed under truncations. Then
we claim that $F$ renormalizes.
Indeed, then every $\sF \in A\mod_{ren}^{\geq 0}$
can be written as a filtered colimit $\sF = \colim_i \sF_i$
with $\sF_i$ compact and in $A\mod_{ren}^{\geq 0}$:
write $\sF$ as a filtered colimit of arbitrary 
compacts and then apply $\tau^{\geq 0}$.
Then we obtain: 

\[
F_{ren}(\sF) = \underset{i}{\colim} \, F(\sF_i) \isom F(\sF) 
\]

\noindent by assumption that $F$ commutes with
filtered colimits in $A\mod_{ren}^{\geq 0}$. 

\end{example}

\begin{example}[Forgetful functors]\label{e:oblv-ren}

The forgetful functor $\Oblv:A\mod_{naive}^+ \to \Vect$
renormalizes to give a functor $\Oblv_{ren}: A\mod_{ren} \to \Vect$
by Example \ref{e:functor-ren-left-separated}.
In what follows, we typically abbreviate the notation 
$\Oblv_{ren}$ to simply $\Oblv$. (Although we call this
functor \emph{forgetful}, it is not generally conservative.)

\end{example}

\begin{example}[Identity functor]\label{e:id-ren}

The embedding $A\mod_{naive}^+ \into A\mod_{naive}$ 
renormalizes to give a continuous functor 
$\id_{ren}:A\mod_{ren} \to A\mod_{naive} \in \DGCat_{cont}$, 
again by Example \ref{e:functor-ren-left-separated}.

\end{example}

\subsection{Morphisms}\label{ss:ren-morphism}

We have the following notion of compatibility between
algebra morphisms and
renormalization data.

\begin{defin}

A \emph{morphism} of renormalized $\overset{\rightarrow}{\otimes}$-algebras
is a map $f:A \to B$ of $\overset{\rightarrow}{\otimes}$ 
such that the ($t$-exact) functor 
$\Oblv:B\mod_{naive}^+ \to A\mod_{naive}^+ \subset A\mod_{ren}$
renormalizes to a functor $\Oblv = \Oblv_{ren}:B\mod_{ren} \to A\mod_{ren}$.

We let $\Alg_{ren}^{\overset{\rightarrow}{\otimes}}$ denote
the category of renormalized algebras and such morphisms.

\end{defin}

\begin{rem}

We emphasize that this is a property, not a structure, 
for the underlying map of $\overset{\rightarrow}{\otimes}$-algebras.

\end{rem}

\begin{example}

Example \ref{e:oblv-ren} says that the unit map 
$k \to A$ is a morphism of renormalized 
$\overset{\rightarrow}{\otimes}$-algebras.
More generally, this is true for any map from an eventually coconnective 
algebra with the ``trivial" renormalization from 
Example \ref{e:qcoh}.

\end{example}

\begin{counterexample}\label{ce:indcoh-qcoh}

Let $A$ be a (discrete) almost finite type,
eventually coconnective commutative
$k$-algebra with $S = \Spec(A)$ singular.

Take $A\mod_{ren_1} = \IndCoh(S)$ 
and $A\mod_{ren_2} = \QCoh(S)$, and
let us pedantically write $A_{ren_1}$, $A_{ren_2}$
for the corresponding renormalized algebras.
Then the identity map for $A$ defines a morphism
$A_{ren_1} \to A_{ren_2}$ of renormalized algebras, 
but not a morphism $A_{ren_2} \to A_{ren_1}$. 

\end{counterexample}

\subsection{Tensor products}\label{ss:alg-tensor-ren}

We now revisit the material of \S \ref{ss:alg-tensor}
in the presence of renormalizations.
So suppose $A$ and $B$ are renormalized 
$\overset{\rightarrow}{\otimes}$-algebras.

Then we claim that $A\mod_{ren} \otimes B\mod_{ren}$
defines a renormalization datum for $A \overset{!}{\otimes} B$.

More precisely, define:

\[
A \overset{!}{\otimes} B\mod_{ren}^c \subset 
A \overset{!}{\otimes} B\mod_{naive}
\]

\noindent as the DG subcategory Karoubi generated
by the essential image of the composition:

\[
A\mod_{ren}^c \times B\mod_{ren}^c \to 
A\mod_{naive} \otimes B\mod_{naive} \xar{\eqref{eq:tens-cat}}
A \overset{!}{\otimes} B\mod_{naive}.
\]

\noindent Now define $A \overset{!}{\otimes} B\mod_{ren}$
as $\Ind(A \overset{!}{\otimes} B\mod_{ren}^c)$.

\begin{thm}\label{t:tens-ren}

\begin{enumerate}

\item $A \overset{!}{\otimes} B\mod_{ren}$
is a renormalization datum for 
$A \overset{!}{\otimes} B$.\footnote{C.f. 
Remark \ref{r:ren-cpts}:
because 
$A \overset{!}{\otimes} B\mod_{ren}^c$ is tautologically
embedded into $A \overset{!}{\otimes} B\mod_{naive}^+$,
this is a property, not a structure.}

\item The natural functor:

\[
A\mod_{ren} \otimes B\mod_{ren} \to 
A \overset{!}{\otimes} B\mod_{ren}
\]

\noindent is an equivalence.

\end{enumerate}

\end{thm}

\begin{lem}\label{l:cpt-tstr}

Suppose $\sC, \sD_1, \sD_2 \in \DGCat_{cont}$ 
have $t$-structures compatible with filtered colimits
and $F:\sD_1 \to \sD_2 \in \DGCat_{cont}$ is a functor.

Recall that $\sC \otimes \sD_i$ admits a canonical
$t$-structure with $(\sC \otimes \sD_i)^{\leq 0}$
generated under colimits by objects $\sF \boxtimes \sG$ for
$\sF \in \sC^{\leq 0}$ and $\sG \in \sD_i^{\leq 0}$.

\begin{enumerate}

\item\label{i:tstr-1}

If $F$ is right $t$-exact, then so is 
$\id_{\sC} \otimes F: \sC \otimes \sD_1 \to \sC \otimes \sD_2$.

\item\label{i:tstr-2}

If the $t$-structure
on $\sC$ is compactly generated 
and $F$ is left $t$-exact,
then $\id_{\sC} \otimes F$ is left $t$-exact.

\item\label{i:tstr-3}

Under the assumptions of \eqref{i:tstr-2}, if
the $t$-structure on $\sC$ is right complete and
$F|_{\sD_1^{\geq 0}}$ is conservative,
then $\id_{\sC} \otimes F|_{(\sC \otimes \sD_1)^{\geq 0}}$ 
is conservative.

\end{enumerate}

\end{lem}

\begin{proof}

\eqref{i:tstr-1} is immediate. 
\eqref{i:tstr-2} is shown e.g. in \cite{whit} Lemma B.6.2,
but we recall the argument as it is used also for 
\eqref{i:tstr-3}.

Let $\sF \in \sC^{\leq 0}$ be compact.
Then $\sF$ defines a continuous functor 
$\bD \sF \coloneqq \ul{\Hom}_{\sC}(\sF,-): \sC \to \Vect$.
We can tensor to obtain:

\[
\bD\sF \otimes \id_{\sD_i}: \sC \otimes \sD_i \to \sD_i.
\]

\noindent As in the proof of \cite{whit} Lemma B.6.2,
if $\sF \in \sC^{\leq 0}$, then this functor
is left $t$-exact, and conversely, 
$\sG \in \sC \otimes \sD_i$ lies in cohomological degrees
$\geq 0$ if and only if 
$\bD\sF \otimes \id_{\sD_i}(\sG) \in \sD_i^{\geq 0}$
for each such $\sF$.
These facts immediately imply \eqref{i:tstr-2}.

Now for \eqref{i:tstr-3}, 
suppose $\sG \in (\sC \otimes \sD_1)^{\geq 0}$
with $(\id_{\sC} \otimes F)(\sG) = 0$.
Then for any $\sF$ as above, we claim:

\[
(\bD_{\sF} \otimes \id_{\sD_1})(\sG) = 0 \in \sD_1.
\]

Indeed, this object lies in degrees $\geq 0$, so it suffices
to show that $F$ applied to it is zero. Then:

\[
F(\bD\sF \otimes \id_{\sD_1})(\sG) =
(\bD\sF \otimes F)(\sG) =
\bD\sF (\id_{\sC} \otimes F)(\sG) = 0.
\]

Now right completeness of 
the (compactly generated) $t$-structure on 
$\sC$ is equivalent to $\sC$ being compactly generated
by objects $\sF$ of the above type (and their shifts),
so this implies that $\sG = 0$ as desired.

\end{proof}

\begin{proof}[Proof of Theorem \ref{t:tens-ren}]

Note that $A\mod_{ren} \otimes B\mod_{ren}$
admits a canonical $t$-structure, as in the
statement of Lemma \ref{l:cpt-tstr}.
This $t$-structure is obviously compactly
generated by objects bounded from below,
since this is true for each of the tensor factors.

By Lemma \ref{l:cpt-tstr}, the functor:

\[
\Oblv = \Oblv_A \otimes \Oblv_B:
A\mod_{ren} \otimes B\mod_{ren} \to 
\Vect
\]

\noindent is $t$-exact and conservative 
on $(A\mod_{ren} \otimes B\mod_{ren})^{\geq 0}$.\footnote{
Note that Lemma \ref{l:cpt-tstr} as formulated
should be applied to the functor:

\[
\id \otimes \Oblv:A\mod_{ren} \otimes B\mod_{ren} \to A\mod_{ren}.
\] 
}

By Remark \ref{r:ren-cat}, there is some convergent, connective 
$\overset{\rightarrow}{\otimes}$-algebra $C$ such that
$A\mod_{ren} \otimes B\mod_{ren}$ with its forgetful
functor defines a renormalization datum for $C$.

Since the forgetful functor 
$A\mod_{ren} \otimes B\mod_{ren} \to \Vect$ lifts
to $A \overset{!}{\otimes} B\mod_{naive}$
(by Example \ref{e:id-ren} and \S \ref{ss:alg-tensor}),
we have a canonical map $A \overset{!}{\otimes} B \to C$
of $\overset{\rightarrow}{\otimes}$-algebras.
To prove the theorem, it suffices to show that this map realizes
$C$ as the convergent completion of $A \overset{!}{\otimes} B$.

For this, suppose $i \mapsto \sF_i \in A\mod_{ren}^c$
and $j \mapsto \sG_j \in B\mod_{ren}^c$ 
pro-represent the forgetful functors. 
Clearly 
$\lim_{i,j} \sF_i \boxtimes \sG_j \in 
\Pro(A\mod_{ren} \otimes B\mod_{ren})$
pro-represents the forgetful functor to vector spaces.
As in \S \ref{ss:forget}, the object:

\[
\underset{i,j}{\lim} \, \Oblv(\sF_i \boxtimes \sG_j) 
\]

\noindent is canonically isomorphic to $C \in \Pro\Vect$.
But we can calculate this object as:\footnote{The second
equality is a general fact about maps out
of external products of two compact objects.} 

\[
\begin{gathered}
\underset{i,j}{\lim} \, \Oblv(\sF_i \boxtimes \sG_j) = 
\underset{i,j}{\lim} \, \underset{k,\ell}{\colim} \,  
\ul{\Hom}_{A\mod_{ren} \otimes B\mod_{ren}}
(\sF_k \boxtimes \sG_{\ell},\sF_i \boxtimes \sG_j) = \\
\underset{i,j}{\lim} \, \underset{k,\ell}{\colim} \, 
\Big(
\ul{\Hom}_{A\mod_{ren}}(\sF_k,\sF_i) \otimes 
\ul{\Hom}_{B\mod_{ren}}(\sG_{\ell},\sG_j)
\Big) = \\ 
\underset{i}{\lim} \, \Oblv(\sF_i) 
\overset{!}{\otimes} 
\underset{j}{\lim} \, \Oblv(\sG_j).
\end{gathered}
\]

This last term is the $\overset{!}{\otimes}$-tensor product
of the convergent completions of $A$ and $B$ respectively,
giving the claim.

\end{proof}

This construction obviously
equips $\Alg_{ren}^{\overset{\rightarrow}{\otimes}}$ with
a unique symmetric monoidal structure such that the
forgetful functor to $\Alg^{\overset{\rightarrow}{\otimes}}$
is symmetric monoidal. For this symmetric monoidal
structure, the functor:

\[
\begin{gathered}
\Alg_{ren}^{\overset{\rightarrow}{\otimes}} \to 
\DGCat_{cont} \\
A \mapsto A\mod_{ren}
\end{gathered}
\]

\noindent is symmetric monoidal by construction. 

\section{Weak actions of group schemes}\label{s:weak-pro}

\subsection{}

In this section, we begin a study of action of (suitable) group indschemes
$H$ on $\overset{\rightarrow}{\otimes}$-algebras and on categories.

We will explain, following Gaitsgory, that (under suitable hypotheses)
there is a naive notion of weak $H$-action on a DG category,
and a less naive notion, which we call \emph{genuine weak actions}.

The bulk of this section is devoted to developing the notion
of genuine actions when $H$ is a classical affine group scheme. 
With that said, this section begins with a general discussion of naive actions
on categories and $\overset{\rightarrow}{\otimes}$-algebras
in the case of general ind-affine group indschemes.

\subsection{Topological bialgebras}

A \emph{$\overset{\rightarrow}{\otimes}$-bialgebra} $B$ is a 
coalgebra $B$ in the symmetric monoidal category
$(\Alg^{\overset{\rightarrow}{\otimes}}, \overset{!}{\otimes})$.
In particular, such a $B$ is 
equipped with an $\overset{\rightarrow}{\otimes}$-algebra
structure and is equipped with a coproduct
$\Delta:B \to B \overset{!}{\otimes} B$ that is a
morphism of $\overset{\rightarrow}{\otimes}$-algebras.

There is a natural notion of \emph{coaction} of such a $B$ on 
a $\overset{\rightarrow}{\otimes}$-algebra $A$.
Here we have a coaction map $\coact:A \to B \overset{!}{\otimes} A$,
which is a map of $\overset{\rightarrow}{\otimes}$-algebras (and
satisfies higher compatibilities with $\Delta$ and so on).

\begin{variant}

A \emph{$\overset{!}{\otimes}$-bialgebra} 
is a bialgebra in the symmetric monoidal category
$(\Pro\Vect,\overset{!}{\otimes})$. Any such object
has an underlying $\overset{\rightarrow}{\otimes}$-bialgebra structure.

Note that in the $\overset{!}{\otimes}$-setting, \emph{commutative} 
and \emph{cocommutative}
$\overset{!}{\otimes}$-bialgebra structures have evident meaning,
while in the $\overset{\rightarrow}{\otimes}$-setting,
only \emph{cocommutative} $\overset{\rightarrow}{\otimes}$-bialgebra 
structures make sense.

\end{variant}

\subsection{}

In the above setting, note that $B\mod_{naive} \in \DGCat_{cont}$ inherits
a canonical monoidal DG structure. For example, the monoidal 
operation is given by:

\[
B\mod_{naive} \otimes B\mod_{naive} \xar{\text{\S \ref{ss:alg-tensor}}}
B\overset{!}{\otimes} B \mod_{naive} \xar{\Delta_*}
B\mod_{naive}
\]

\noindent where $\Delta_*$ is restriction of module structures along the map
$\Delta$.

Similarly, if $B$ coacts on $A$, then $B\mod_{naive}$ acts on $A\mod_{naive}$.

\subsection{}\label{ss:coact-ren}

Now suppose that $B$ is given a renormalization datum.
Recall from \S \ref{ss:alg-tensor-ren} that
$\Alg_{ren}^{\overset{\rightarrow}{\otimes}}$ is a symmetric
monoidal category. 

Therefore, it makes sense to say that
a bialgebra structure on $B$ is \emph{compatible} with the
renormalization datum on $B$: this means that the counit and
comultiplication maps are morphisms of renormalized 
$\overset{\rightarrow}{\otimes}$-algebras. Similarly,
for $A \in \Alg_{ren}^{\overset{\rightarrow}{\otimes}}$,
we may speak of a coaction of $B$ on $A$
being compatible with the given renormalization data:
this means the coaction data makes $A$ a comodule for $B$ in
the symmetric monoidal category $\Alg_{ren}^{\overset{\rightarrow}{\otimes}}$.

In such cases, $B\mod_{ren}$ inherits a canonical monoidal structure
and $B\mod_{ren}$ acts on $A\mod_{ren}$.

\subsection{Group setting}\label{ss:gp-setting}

Now suppose that $H$ is an ind-affine group indscheme.
We suppose $H$ is \emph{reasonable} in the
sense of \cite{hitchin} (or \S \ref{ss:indsch} below): that is,
$H = \colim H_i$ for $H_i \subset H$ 
eventually coconnective quasi-compact quasi-separated 
subschemes\footnote{For emphasis: the $H_i$ may not
necessarily be group subschemes.}
with all maps $H_i \to H_j$ almost finitely presented. 

Then $B = \Fun(H) \coloneqq \lim_i \Gamma(H_i,\sO_{H_i}) \in \Pro\Vect$
is a commutative $\overset{!}{\otimes}$-bialgebra,
and in particular inherits a $\overset{\rightarrow}{\otimes}$-bialgebra
structure.

We say that $H$ \emph{naively acts} on 
$A \in \Alg^{\overset{\rightarrow}{\otimes}}$ if $B$ coacts on $A$.
We let $\Alg^{\overset{\rightarrow}{\otimes},H\actson}$ denote
the category of $\overset{\rightarrow}{\otimes}$-algebras with
naive $H$-actions (i.e., the category of $B$-comodules
in $\Alg^{\overset{\rightarrow}{\otimes}}$).

\subsection{Naive group actions on categories}\label{ss:naive-act}

Assume in the above notation that
each of the (commutative) algebras $\Gamma(H_i,\sO_{H_i})$
are coherent, as in Example \ref{e:indcoh-assoc}.
Then $B$ admits a canonical renormalization as in \emph{loc. cit}.
We define $\IndCoh^*(H) \coloneqq B\mod_{ren}$.

\begin{rem}

In \S \ref{s:indcoh}, we will define $\IndCoh^*$ in much
greater generality.
However, this elementary definition
coincides in the present setting.

\end{rem}

\begin{rem}

The notation is taken from \cite{dmod} (see also
\cite{km-indcoh}), to which we refer for an
explanation. The main purpose of this notation
is to remind us that to avoid the pitfalls inherent in 
working with $\IndCoh$
in the infinite type setting.

\end{rem}

\begin{example}

Suppose that $H$ is the loop group $G(K)$ for $G$ 
an affine algebraic group.\footnote{In particular, $G$ is classical and finite type.} 
By \cite{indschemes},
there is a canonical equivalence $\QCoh(G(K)) \simeq \IndCoh^*(G(K))$
defined as such. But we note that this equivalence
uses the compact open subgroup $G(O)$ in an essential
way: the functor
$\Oblv:\IndCoh^*(G(K)) \to \Vect$, which tautologically exists in the
above definition, corresponds to the composition: 

\[
\QCoh(G(K)) \xar{\pi_*} \QCoh(\Gr_G) \overset{-\otimes \omega_{\Gr_G}}{\simeq} \IndCoh(\Gr_G) \xar{\Gamma^{\IndCoh}(\Gr_G,-)} \Vect.
\] 

\noindent Here $\pi:G(K) \to \Gr_G = G(K)/G(O)$ is the projection,
and $\IndCoh$ is defined in the standard sense 
on $\Gr_G$ because it is of ind-finite type;
the rest of the notation is standard in the subject, and
the functor of tensoring with the dualizing sheaf
is an equivalence by a theorem of \cite{indschemes} (and
formal smoothness of $\Gr_G$).

\end{example}

\begin{defin}

A \emph{naive weak action} of $H$ on $\sC \in \DGCat_{cont}$ is an
$\IndCoh^*(H)$-module\footnote{Here we are considering
$\IndCoh^*(H)$ as an algebra object in $\DGCat_{cont}$,
so e.g. the action functor is $\IndCoh^*(H) \otimes \sC \to \sC$.
In particular, the induced action $\IndCoh^*(H) \times \sC \to \sC$
commutes with colimits in each variable separately.}
structure for $\sC$.

\end{defin}

\begin{rem}

The antipode for $H$ induces a canonical equivalence
between left and right modules for $\IndCoh^*(H)$, so we often
ignore the distinction going forward.

\end{rem}

\begin{rem}

We sometimes omit ``weak": the distinction between naive
and genuine actions in this section only occurs for weak
group actions, not for strong group actions.

\end{rem}

\begin{example}

$H$ has a canonical naive action on $\IndCoh^*(H)$.

\end{example}

\begin{example}

$H$ has a canonical naive action on $\Vect$. Indeed,
$H$ naively acts on $k$ with coaction given by the unit map,
and this is compatible with renormalization.

\end{example}

\begin{example}

For any indscheme $X$ of ind-finite type with an $H$
action, $H$ naively acts on $\IndCoh(X)$.

\end{example}

\begin{example}

If $H$ acts on a Tate Lie algebra $\fk$, then
$H$ naively weakly acts on $\fk\mod$, 
(defined as in Example \ref{e:lie}).
In particular, $H$ weakly acts on $\fh\mod$.

\end{example}

\subsection{}\label{ss:h-acts-ren}

We let $\Alg_{ren}^{\overset{\rightarrow}{\otimes},H\actson}$ denote
the category of renormalized 
$\overset{\rightarrow}{\otimes}$-algebras with
naive $H$-actions compatible with the renormalization, 
i.e., the category of $B$-comodules
in $\Alg^{\overset{\rightarrow}{\otimes}}$ for $B = \Gamma(H,\sO_H)$
equipped with the renormalization datum $\IndCoh^*(H)$.

Note that for $A \in \Alg_{ren}^{\overset{\rightarrow}{\otimes},H\actson}$,
$H$ acts naively on $A\mod_{ren}$ (c.f. \S \ref{ss:coact-ren}).

\subsection{}

We let $H\mod_{weak,naive}$ 
denote the 2-category of categories
with a naive weak action of $H$, i.e., $\IndCoh^*(H)\mod(\DGCat_{cont})$.

For $\sC \in \IndCoh^*(H)\mod$, we define the \emph{naive weak invariants}
and \emph{coinvariants} as:

\[
\sC^{H,w,naive} \coloneqq 
\TwoHom_{H\mod_{weak,naive}}(\Vect,\sC), \hspace{1cm}
\sC_{H,w,naive} \coloneqq \Vect \underset{\IndCoh^*(H)}{\otimes} \sC.
\]

\subsection{Genuine actions}\label{ss:gen-pro-start}

In the remainder of this section and in \S \ref{s:weak-ind},
we study a more robust variant of the above
notion, under somewhat more restrictive hypotheses. In this
section, we focus on the case where 
$H$ is profinite dimensional, which
contains the main phenomena.

\subsection{Finite dimensional reminder}

We first remind the reader of the following foundational result,
which will play a key role.

Let $H$ be an affine algebraic group. In this case, we remove
the label ``naive" from the notation, e.g.,
$H\mod_{weak} = H\mod_{weak,naive}$ \textemdash{} the naivet\'e
is only in the infinite type setting.

\begin{thm}[Gaitsgory, \cite{shvcat}]\label{t:weak}

For $H$ an affine algebraic group, the functor:

\[
\begin{gathered}
H\mod_{weak} \to \Rep(H)\mod = \Rep(H)\mod(\DGCat_{cont}) \\
\sC \mapsto \sC^{H,w}
\end{gathered}
\]

\noindent is an equivalence. (Here the functor
exists because $\Rep(H) \coloneqq \QCoh(\bB H)$ 
is tautologically isomorphic to $\TwoHom_{H\mod_{weak}}(\Vect,\Vect)$
as a monoidal category.)

\end{thm}

\subsection{Profinite dimensional setting}

In the remainder of this section, we suppose $H$ is a classical 
affine group scheme. 

\subsection{}

Let $B = \Fun(H) \in \Vect^{\heart}$ as before. 
Because $H$ can be written as a limit of smooth schemes
under smooth morphisms, the tautological 
functor 
$B\mod_{ren} = \IndCoh^*(H) \to \QCoh(H) = B\mod_{naive}$ is an equivalence.

\subsection{}

We begin with a remark in the naive setting.

Note that the Beck-Chevalley conditions apply for the cosimplicial diagram
defining $\sC^{H,w,naive}$. Therefore, $\Oblv:\sC^{H,w,naive} \to \sC$
admits a continuous right adjoint $\Av_*^{w,naive}$, and $\Oblv$
is comonadic. The comonad on $\sC$ is given by convolution with
the coalgebra $\sO_H$ in the monoidal category $\QCoh(H)$.

In particular, for $\sC = \Vect$, we obtain
that $\Rep(H)_{naive} \coloneqq \Vect^{H,w,naive} (= \QCoh(\bB H))$
is canonically equivalent to the category of $B$-comodules (with $B$ as
above).

\subsection{}

We now define $\Rep(H) = \Rep(H)_{ren}$ as 
$\Ind(\Rep(H)^c))$ for
$\Rep(H)^c \subset \Rep(H)_{naive}$ the full
subcategory generated by finite dimensional representations,
i.e., objects whose image in $\Vect$ is compact.\footnote{Note that
this example fits into the formalism of \S \ref{s:ren}.
Indeed, $B$ is the union of its finite dimensional sub-coalgebras,
so $B^{\vee} \in \Pro\Vect$ is a profinite dimensional algebra,
in particular, an $\overset{\rightarrow}{\otimes}$-algebra.
Its modules are tautologically the same as $B$-comodules.
This definition of $\Rep(H)$ is then obtained by applying
Example \ref{e:pro-algebras}.} Since $\Rep(H)^c$ is closed
under tensor products in $\Rep(H)_{naive}$, 
$\Rep(H)$ is a rigid symmetric monoidal DG category.

Note that $\Rep(H)$ carries a canonical $t$-structure
for which the forgetful functor to $\Vect$ is $t$-exact.
We have $\Rep(H)^+ \simeq \Rep(H)_{naive}^+$.

\subsection{}\label{ss:pro-gen}

The following definition plays a key role.

\begin{defin}

The category $H\mod_{weak}$ of categories \emph{with a 
genuine}\footnote{The terminology is borrowed
from equivariant homotopy theory. In that
context, for finite $H$,
one extends the \emph{naive} notion of $H$-action
on a spectrum in such a way that the trivial representation
(and more generally, any permutation representation) becomes
compact. This is somewhat analogous to the 
present context, where we \emph{renormalize} 
$H\mod_{weak,naive}$ so that the trivial representation
$\Vect$ becomes (completely) compact.

Although the subtleties in our context
only occur for group schemes (which are
analogous to profinite groups) and
group indschemes (which are analogous to 
locally compact totally disconnected groups), 
we still find this analogy to be somewhat evocative.} 
\emph{weak $H$-action} is 
$\Rep(H)\mod = \Rep(H)\mod(\DGCat_{cont})$.

\end{defin}

\begin{construction}

Note that $\Rep(H)_{naive} = \TwoHom_{H\mod_{weak,naive}}(\Vect,\Vect)$.
In particular, $\Vect$ admits commuting actions of $\Rep(H)_{naive}$
and $\QCoh(H)$. In particular, since $\Rep(H)_{ren} \to \Rep(H)_{naive}$
is symmetric monoidal, $\Vect$ is a bimodule for $\Rep(H)_{ren}$
and $\QCoh(H)$, and therefore tensoring defines a functor:

\[
H\mod_{weak} \to H\mod_{weak,naive}. 
\]

\end{construction}

\begin{notation}\label{n:h-gen}
	
Following the case of finite dimensional $H$, we think of the
underlying object of $\DGCat_{cont}$ as the weak $H$-invariants
of a DG category acted on by $H$.

To accommodate this, suppose we are given an object
of $H\mod_{weak}$. By definition, this means that
we are given an object $\sD \in \Rep(H)\mod$.
We use the notation $\sC^{H,w}$ in place of $\sD$,
where we let $\sC$ denote the underlying object 
of $H\mod_{weak,naive}$. We then abusively write $\sC \in H\mod_{weak}$
to summarize the situation.

Roughly, the reader should think $\sC \in H\mod_{weak}$ means
that $\sC \in H\mod_{weak,naive}$, and we are given a
``correction" $\sC^{H,w}$ to $\sC^{H,w,naive}$. 

We emphasize that this 
``forgetful functor" $H\mod_{weak} \to \DGCat_{cont}$ (factoring through
$H\mod_{weak,naive}$) is \emph{not} conservative.
	
\end{notation}

\begin{rem}\label{r:naive-to-gen}

Because $\Rep(H)$ is rigid monoidal, $\Vect$ is dualizable
over $\Rep(H)$. Therefore, the functor 
$H\mod_{weak} \to H\mod_{weak,naive}$ admits left and
right adjoints. It is immediate to see that they are
computed as strong and weak invariants respectively,
with $\Rep(H)$ acting through $\Rep(H)_{naive}$.

In particular, for $\sC \in H\mod_{weak}$, there is
a canonical functor:

\[
\sC^{H,w} \to \sC^{H,w,naive}.
\]

It is not difficult to see that each of these functors 
$H\mod_{weak,naive} \to H\mod_{weak}$
are fully-faithful. Indeed, it is well-known that
it suffices to verify this for either functor, and for the
right adjoint this is the content of \cite{locsys} Proposition 3.5.1
(which is proved by a standard Beck-Chevalley argument).

\end{rem}

\begin{example}\label{e:vect-gen}

We have a canonical object $\Vect \in H\mod_{weak}$
with $\Vect^{H,w} = \Rep(H)$.
Clearly $\TwoHom_{H\mod_{weak}}(\Vect,\sC) = \sC^{H,w}$.

\end{example}

\begin{example}

By Theorem \ref{t:weak}, genuine and naive actions coincide
in the finite dimensional case. It is straightforward to show
that if $H = \prod_{i=1}^{\infty} \bG_a$, then
$\Rep(H)$ is not equivalent to $\Rep(H)_{naive}$, so the
two notions do not coincide in this case.

\end{example}

\begin{rem}

The relationship between $H\mod_{weak}$ and 
$H\mod_{weak,naive}$ is somewhat analogous to
the relationship between $\IndCoh$ and $\QCoh$,
though it occurs a categorical level higher.
Namely, there is a non-conservative functor
$H\mod_{weak} \to H\mod_{weak,naive}$ analogous
to the functor $\Psi:\IndCoh(S) \to \QCoh(S)$
for an eventually coconnective Noetherian scheme $S$,
and in both cases, there are fully-faithful left and right
adjoints.

\end{rem}

\subsection{}\label{ss:gp-inv-limit}

The key advantage of genuine $H$-actions is that the theory
completely reduces to the finite dimensional setting, as we now discuss.

Indeed, recall that $H$ is a limit $\lim_i H_i$
of affine algebraic groups under smooth\footnote{Of course we are using
characteristic zero in an essential way.} surjective maps.

Let $K_i \subset H$ denote the kernel of the map $H \to H_i$.
Note that there is a canonical functor 
$H\mod_{weak} \to H_i\mod_{weak}$,
sending $\sC$ to:

\[
\sC^{K_i,w} \coloneqq \sC^{H,w} \underset{\Rep(H_i)}{\otimes} \Vect.
\]

\noindent That is, we apply the restriction along the tensor functor
$\Rep(H_i) \to \Rep(H)$ and the inverse to Theorem \ref{t:weak}.

\begin{prop}

The induced functor:

\[
H\mod_{weak} \to \underset{i}{\lim} \, H_i\mod_{weak}
\]

\noindent is an equivalence.

\end{prop}

This follows immediately from the next lemma.

\begin{lem}\label{l:reph-co/lim}

The morphism:

\[
\underset{i}{\colim} \, \Rep(H_i) \to \Rep(H) \in \ComAlg(\DGCat_{cont})
\]

\noindent is an equivalence.

\end{lem}

\begin{proof}

This is a special case of Example \ref{e:pro-algebras}.

\end{proof}

\begin{cor}

For any $\sC \in H\mod_{weak}$, the functor:

\[
\underset{i}{\colim} \, \sC^{K_i,w} \to \sC \in \DGCat_{cont}
\]

\noindent is an equivalence. Moreover, each of the structural
functors in this colimit admits a continuous right adjoint.

\end{cor}

\begin{cor}

The functor $\Oblv:\sC^{H,w} \to \sC$ admits a continuous
right adjoint $\Av_*^w$.

\end{cor}

\subsection{Functoriality}\label{ss:gp-sch-funct}

Suppose $f:H_1 \to H_2$ is a morphism of classical affine group schemes.

We claim that there are induced adjoint functors:

\[
\ind^w:H_1\mod_{weak} \rightleftarrows H_2\mod_{weak}: \Res
\]

\noindent with the \emph{weak induction} functor
$\ind^w$ \emph{also} canonically isomorphic to
the right adjoint to $\Res$.

Indeed, unwinding the definitions, $H_i\mod_{weak} \simeq \Rep(H_i)\mod$,
and we take $\ind^w$ to correspond to restriction of module categories along
the symmetric monoidal functor $\Rep(H_2) \to \Rep(H_1)$. This
functor obviously admits a left adjoint $\Rep(H_1)\otimes_{\Rep(H_2)}-$,
which is defined to be $\Res$. Then $\Res$ is both
left and right adjoint because $\Rep(H_1)$ is self-dual
as a $\Rep(H_2)$-module category by general properties of
rigid monoidal DG categories, c.f. \cite{dgcat}.

In particular, taking $H \to \Spec(k)$, we see 
$\Res:\DGCat_{cont} \to H\mod_{weak}$ sends $\Vect$ to 
itself with the trivial $H$-action. The equality of
left and right adjoints here should be interpreted as an 
``invariants = coinvariants" statement for genuine $H$-actions.
We remark that the corresponding statement is false
in the setting of naive weak actions.

\subsection{Genuine actions via canonical renormalization}\label{ss:can-renorm}

The following is a typical construction of genuine weak $H$-actions.

Suppose $H$ acts naively on $\sC$.
Suppose moreover that $\sC$ is equipped with a 
$t$-structure such that $\Oblv\Av_*^{w,naive}:\sC \to \sC$
is $t$-exact. Then $\sC^{H,w,naive}$ has a (unique) $t$-structure 
such that $\Oblv:\sC^{H,w,naive} \to \sC$ is $t$-exact
(c.f. the proof of Proposition \ref{p:gen-t-str} \eqref{i:gen-t-1} below).
Note that the functor $\Av_*^{w,naive}$ is also $t$-exact in this case.

In what follows, we let $\sC^{H,w,c} \subset \sC^{H,w,naive}$ denote the
(non-cocomplete) DG subcategory of objects
$\sF \in \sC^{H,w,naive}$ with $\Oblv(\sF)$
compact in $\sC$.

\begin{defin}

In the above setting, we say
the naive action of $H$ on $\sC$ 
\emph{canonically renormalizes
(compatibly with the $t$-structure)}
if: 

\begin{enumerate}

\item $\sC$ and
its $t$-structure are compactly generated.

\item Compact objects in $\sC$ are
bounded (i.e., eventually connective and coconnective).

\item The essential image of the functor 
$\Oblv:\sC^{H,w,c} \cap \sC^{H,w,naive,\leq 0} \to \sC^c \cap \sC^{\leq 0}$ Karoubi generates.
(Here $\sC^c \subset \sC$ is the subcategory of
compact objects.)

\end{enumerate}

\end{defin}

\begin{rem}\label{r:renorm}

Note that under the above hypotheses,
the functor $\sC^{H,w,c} \to \sC^c$ Karoubi generates.

\end{rem}

The following result summarizes the main 
features of this setting.

\begin{prop}\label{p:gen-t-str} 

Suppose $H$ acts naively on $\sC$, $\sC$ is equipped with 
a $t$-structure compatible with the $H$-action, 
and suppose the $H$-action canonically renormalizes.

Define $\sC^{H,w}$ as $\Ind(\sC^{H,w,c})$;
as $\sC^{H,w,c} \subset \sC^{H,w,naive}$ is
a $\Rep(H)^c$-submodule category, $\sC^{H,w}$
has a canonical $\Rep(H)$-module structure.

Let $\psi$ denote the canonical functor:

\[
\psi:\sC^{H,w} \to \sC^{H,w,naive} \in \DGCat_{cont}
\]

\noindent ind-extending the embedding $\sC^{H,w,c} \into \sC^{H,w,naive}$.
Note that $\psi$ is a morphism of $\Rep(H)$-module categories.

We use a standard abuse of notation in letting
$\Oblv:\sC^{H,w} \to \sC$ denote the composition
$\sC^{H,w} \xar{\psi} \sC^{H,w,naive} \xar{\Oblv} \sC$.

\begin{enumerate}

\item\label{i:gen-t-1} $\sC^{H,w,naive}$ admits a 
unique $t$-structure such that
such that the (conservative) forgetful functor to $\sC$ is $t$-exact.

\item\label{i:gen-t-2} $\sC^{H,w}$ has a 
unique compactly generated $t$-structure 
such that the forgetful functor to $\sC$ is $t$-exact
and conservative on eventually coconnective subcategories.

\item\label{i:gen-t-2.5} For $V \in \Rep(H)^{\heart}$, the action
functors $V \star -:\sC^{H,w} \to \sC^{H,w}$ and 
$V \star -:\sC^{H,w,naive} \to \sC^{H,w,naive}$ are $t$-exact.

\item\label{i:gen-t-3} The functor $\psi$
is $t$-exact and an equivalence on eventually 
coconnective subcategories:

\[
\psi:\sC^{H,w,+} \isom \sC^{H,w,naive,+}.
\]

\item\label{i:gen-t-3.5} 

The composition:\footnote{In fact, each functor here
is an equivalence.}

\[
\sC^{H,w} \underset{\Rep(H)}{\otimes} \Vect \to 
\sC^{H,w,naive} \underset{\Rep(H)}{\otimes} \Vect \to
\sC
\]

\noindent is an equivalence. In particular,
$\sC^{H,w} \in \Rep(H)\mod$ induces 
a canonical genuine $H$-action on 
$\sC$. 

\item\label{i:gen-t-4} Let $\Av_*^w:\sC \to \sC^{H,w}$
denote the right adjoint to $\Oblv$. Then the induced natural transformation:

\[
\psi \circ \Av_*^w \to \Av_*^{w,naive}
\]

\noindent (of functors $\sC \to \sC^{H,w,naive}$) 
is an isomorphism.

\end{enumerate}

\end{prop}

\begin{proof}

\eqref{i:gen-t-1} follows by noting 
that $\sC^{H,w,naive}$ is the totalization 
$\Tot \big(
\sC \otimes \QCoh(H)^{\otimes \dot}\big)$
and all of the structural maps in the underlying semi-cosimplicial 
diagram are $t$-exact.\footnote{Note that this argument 
is general for naive $H$-actions and compatible
$t$-structures. I.e., it is not specific to 
canonical renormalization.}

In \eqref{i:gen-t-2}, the uniqueness is clear: the $t$-
structure must have 
$\sC^{H,w,\leq 0}$ generated under colimits
by $\sC^{H,w,c} \cap \sC^{H,w,naive,\leq 0}$. 
As is standard, this does define a $t$-structure, and the
forgetful functor to $\sC$ is clearly right $t$-exact.
We will complete the proof of \eqref{i:gen-t-2}
later in the argument;
but now, we will verify that \eqref{i:gen-t-2.5}
holds for this $t$-structure (without relying on
any as yet unproved parts of \eqref{i:gen-t-2}).

For $V \in \Rep(H)^{\heart}$, we
first show that the functor
$V \star - :\sC^{H,w,naive} \to \sC^{H,w,naive}$
is $t$-exact. Here it suffices to verify this
after applying $\Oblv$. 
But $\Oblv \circ (V \star -) = V \otimes \Oblv(-)$,
which is clearly $t$-exact. 

To see $V \star -:\sC^{H,w} \to \sC^{H,w}$ is $t$-exact, 
note that we can assume $V$ is finite dimensional
(because the $t$-structure on $\sC^{H,w}$, being
compactly generated, is compatible with filtered 
colimits). By the naive case,
this functor preserves $\sC^{H,w,c} \cap \sC^{H,w,naive,\leq 0}$,
so we obtain right $t$-exactness in the genuine setting. 
Now $V \star -$ is right adjoint
to the left $t$-exact functor $V^{\vee} \star -$, giving 
the left $t$-exactness.

We now make an auxiliary observation.
Suppose $\sG \in \sC^{H,w,c} \subset \sC^{H,w,naive}$.
By assumption, $\sG$ lies in cohomological degrees $\geq -N$ for
$N \gg 0$. We claim that $\sG$ is actually compact as an
object of $\sC^{H,w,naive,\geq -N}$. Indeed,
because $\sG$ is eventually connective, the functor:

\[
\Hom_{\sC^{H,w,naive}}(\sG,-):\sC^{H,w,naive,\geq -N} \to \Gpd
\]

\noindent factors through the subcategory of $M$-truncated groupoids
for some $M \gg 0$ (depending on $N$ and $\sG$). 
Moreover, we have:

\[
\Hom_{\sC^{H,w,naive}}(\sG,-) \isom \Tot 
\Hom_{\sC \otimes \QCoh(H)^{\otimes \dot}}
(\Oblv(\sG)\boxtimes \sO_H^{\boxtimes \dot}
,-) 
\]

\noindent where each of these functors factors through $M$-truncated
groupoids. Therefore, we have $\Tot \isom \Tot^{\leq M+1}$
here, so commuting finite limits with filtered colimits in $\Gpd$
gives the claim about $\sG$.

Using this observation, we will now show 
\eqref{i:gen-t-4}. 
Note that the canonical natural transformation:

\begin{equation}\label{eq:av-conv}
\Fun(H) \star - \to \Av_*^{w,naive}\Oblv \in 
\TwoHom(\sC^{H,w,naive},\sC^{H,w,naive})
\end{equation}

\noindent is an isomorphism, where here
$\star$ denotes the action of $\Rep(H)$ on $\sC^{H,w,naive}$
and $\Fun(H)$ is the regular representation in 
$\Rep(H)^{\heart} \subset \Rep(H)$. 
Indeed, the identification\footnote{This follows
from identifying both sides with 
$\Fun(H)\mod(\sC)$ using Barr-Beck and the Beck-Chevalley
formalism.}
$\sC = \sC^{H,w,naive} \otimes_{\Rep(H)} \Vect$
and the Beck-Chevalley formalism imply this. 

We now similarly claim that there is a canonical isomorphism:

\[
\Fun(H) \star - \isom \Av_*^w\Oblv: \sC^{H,w} \to \sC^{H,w}.
\]

\noindent Both functors commute with colimits, so it suffices
to verify that for every $\sF \in \sC^{H,w,c}$,
the natural map:

\[
\Fun(H) \star \sF \to \Av_*^w\Oblv(\sF)
\]

\noindent is a natural isomorphism. Let $\sG \in \sC^{H,w,c}$ be
given. Write $\Fun(H)$ as a filtered colimit 
$\colim_i V_i$ where $V_i \in \Rep(H)^{\heart}$
are finite dimensional. We claim that:

\[
\underset{i}{\colim} \, \ul{\Hom}_{\sC^{H,w,c}}(\sG, V_i \star \sF) \isom 
\ul{\Hom}_{\sC^{H,w,naive}}(\sG, \underset{i}{\colim} \, V_i \star \sF).
\]

\noindent Indeed, because $\sF$ and $\sG$ are 
eventually coconnective, by \eqref{i:gen-t-2.5}
(in the naive case), we have
$V_i \star \sF,\sF,\sG  \in \sC^{H,w,naive,\geq -N}$
for $N \gg 0$ (and for all $i$).
Then the fact that  
$\sG$ is compact in 
$\sC^{H,w,naive,\geq -N-r}$ for all $r \geq 0$
gives the claim.

Therefore, we have:

\[
\begin{gathered}
\ul{\Hom}_{\sC^{H,w}}(\sG,\Fun(H) \star \sF) =
\ul{\Hom}_{\sC^{H,w}}(\sG,\underset{i}{\colim} \, V_i \star \sF) =
\underset{i}{\colim} \, \ul{\Hom}_{\sC^{H,w}}(\sG, V_i \star \sF) = \\
\underset{i}{\colim} \, \ul{\Hom}_{\sC^{H,w,c}}(\sG, V_i \star \sF) = 
\ul{\Hom}_{\sC^{H,w,naive}}(\sG, \underset{i}{\colim} \, V_i \star \sF) =
\ul{\Hom}_{\sC^{H,w,naive}}(\sG, \Av_*^{w,naive}(\sF)) = \\
\ul{\Hom}_{\sC}(\Oblv(\sG), \Oblv(\sF)).
\end{gathered}
\]

To complete the argument, note that
both functors $\Av_*^{w,naive}$ and $\psi\Av_*^w$
commute with colimits, so it suffices to show that our
natural transformation is an equivalence when evaluated
on compact objects in $\sC$. 
Moreover, because the naive $H$-action on
$\sC$ canonically renormalizes, it suffices
to check this on compact objects of the form $\Oblv(\sF)$
for $\sF \in \sC^{H,w,c}$. 
But then $\Av_*^{w,naive}\Oblv$ and
$\psi\Av_*^w\Oblv$ are each canonically given by the
action of $\Fun(H)$, and $\psi$ is $\Rep(H)$-linear,
giving the claim.

Returning to \eqref{i:gen-t-2},
we claiom that for $\sF \in \sC^{H,w,\geq 0}$ we have
$\Oblv(\sF) \in \sC^{\geq 0}$.
Note that $\Av_*^w\Oblv(\sF) = \Fun(H) \star \sF$ as before,
and since $\Fun(H)$ is in degree $0$, 
$\Fun(H) \star \sF$ is also in degrees $\geq 0$.
Therefore, for 
$\sG \in \sC^{H,w,c} \cap \sC^{H,w,naive,<0}$, we have:

\[
\Hom_{\sC}(\Oblv\sG,\Oblv\sF) =  
\Hom_{\sC^{H,w}}(\sG,\Av_*^w\Oblv\sF) = 0 \in \Gpd
\]

\noindent By hypothesis, $\sC^{<0}$ is generated
under colimits by such objects $\Oblv \sG$, giving the claim.

To conclude \eqref{i:gen-t-2}, we need to show that 
$\Oblv$ is conservative on eventually coconnective objects.
Suppose $\sF \in \sC^{H,w,\geq 0}$ with $\Oblv(\sF) = 0$.
Then we have: 

\[
\sF[1] = \Coker(\sF \to \Av_*^w\Oblv(\sF)) = \Coker(k \xar{1 \mapsto 1} \Fun(H))
\star \sF.
\]

\noindent By \eqref{i:gen-t-2.5}, the right hand
side is in $\sC^{H,w,\geq 0}$, so $\sF[1] \in \sC^{H,w,\geq 0}$,
so $\sF \in \sC^{H,w,\geq 1}$. Iterating this, we obtain
$\sF = 0$ as desired. 

Then \eqref{i:gen-t-3} follows from \eqref{i:gen-t-2} and
\eqref{i:gen-t-4} by observing that these results imply that the
forgetful functor $\sC^{H,w,+} \to \sC^+$ is comonadic
with comonad given by the action of $\Fun(H) \in \QCoh(H)$.

It remains to show \eqref{i:gen-t-3.5}. 
By the Beck-Chevalley formalism,
$\sC^{H,w} \xar{\sF \mapsto \sF \boxtimes_{\Rep(H)} k}
 \sC^{H,w} \otimes_{\Rep(H)} \Vect$
admits a conservative right adjoint, and the
corresponding monad on $\sC^{H,w}$ is 
the action of $\Fun(H) \in \Rep(H)$. 
Using our canonical identification of
that convolution with $\Av_*^w\Oblv$,
we obtain that the functor from \eqref{i:gen-t-3.5}
is fully-faithful. By Remark \ref{r:renorm},
it is also essentially surjective.
 
\end{proof}

\subsection{Canonical renormalization for $\IndCoh$}

We also have the following variant.

\begin{lem}\label{l:gen-indscheme}

Suppose that $X$ is an indscheme 
locally almost of finite type 
acted on by $H$. Then the naive $H$-action on
$\IndCoh(X)$ canonically renormalizes (relative
to its canonical $t$-structure).

\end{lem}

\begin{proof}

We use the following construction.
Suppose $\sF \in \IndCoh(X)^{H,w,naive,\heart}$
and $\sG \subset \Oblv(\sF)$ is a subobject.
Define a subobject
$\widetilde{\sG} \subset \sF \in \IndCoh(X)^{H,w,naive,\heart}$
as the fiber product (in the \emph{abelian} category
$\IndCoh(X)^{H,w,naive,\heart}$):

\[
\xymatrix{
\widetilde{\sG} = \sF \underset{\Av_*^w\Oblv(\sF)}{\times} \Av_*^w(\sG) 
\ar[r]\ar[d] & \sF \ar[d]^{\coact} \\
\Av_*^w(\sG) \ar[r] & \Av_*^w\Oblv(\sF).
}
\]

\noindent Observe that 
$\Oblv(\widetilde{\sG}) \subset \sG \subset \Oblv(\sF)$
(using the counit of the adjunction).\footnote{In 
fact, $\widetilde{\sG}$ is maximal
among subobjects of $\sF$ with this property.}

Now if $\sG$ is coherent, then $\Oblv(\widetilde{\sG})$ is
as well (because we are in a Noetherian setup). 
Moreover, the map 
$\on{SubObj}(\Oblv(\sF)) \xar{\sG \mapsto \widetilde{\sG}} \on{SubObj}(\sF)$
commutes with filtered colimits (because we are taking
fiber products in a Grothendieck
abelian category).
Therefore, we have: 

\[
\sF = \underset{\sG \subset \Oblv(\sF)\text{ coherent}}{\colim} \widetilde{\sG}.
\]

\noindent Applying $\Oblv$, we see that
$\Oblv(\sF)$ is a filtered colimit of objects
coming from $\IndCoh(X)^{H,w,c}$. Since such objects
$\Oblv(\sF)$ generate $\IndCoh(X)^{\leq 0}$ under
colimits (since $\Av_*^w$ is $t$-exact and conservative), 
this gives the claim.

\end{proof}

\subsection{Varying the group}

We record the following result for later use.

\begin{lem}\label{l:gps-canon-renorm}

Suppose $f:H_1 \to H_2$ is a morphism of classical
affine group schemes. Suppose $\sC \in \DGCat_{cont}$
is equipped with a $t$-structure 
and a compatible naive action of $H_2$ 
that renormalizes. Then:

\begin{enumerate}

\item\label{i:gps-1} The induced naive $H_1$-action also
renormalizes. 

\item\label{i:gps-2} The category:

\[
\Rep(H_1) \underset{\Rep(H_2)}{\otimes} \sC^{H_2,w}
\]

\noindent is compactly generated, and the
natural functor:

\begin{equation}\label{eq:gps-functor}
\Rep(H_1) \underset{\Rep(H_2)}{\otimes} \sC^{H_2,w} \to 
\sC^{H_1,w,naive}
\end{equation}

\noindent maps compact objects to 
objects in $\sC^{H_1,w,c}$ (with this subcategory
defined as in Proposition \ref{p:gen-t-str}).

\item\label{i:gps-3}
 
The functor:

\[
\big(\Rep(H_1) \underset{\Rep(H_2)}{\otimes} \sC^{H_2,w}\big)^c \to  
\sC^{H_1,w,c}
\]

\noindent induced by \eqref{eq:gps-functor}
is fully-faithful, so induces a fully-faithful functor:

\begin{equation}\label{eq:gps-functor-2}
\Rep(H_1) \underset{\Rep(H_2)}{\otimes} \sC^{H_2,w} \to 
\sC^{H_1,w}.
\end{equation}

\item\label{i:gps-4} If $\sC = \IndCoh(X)$ for $X$ 
locally almost of finite type and acted on by $H_2$
(as in Lemma \ref{l:gen-indscheme}), then 
\eqref{eq:gps-functor-2} is an equivalence.

\item\label{i:gps-5} If $H_1$ and $H_2$ are algebraic
groups, then \eqref{eq:gps-functor-2} is an equivalence.

\end{enumerate}

\end{lem}

\begin{proof}

\eqref{i:gps-1} is immediate from the definitions.

In \eqref{i:gps-2}, note that all the categories
appearing in the construction computing the
tensor product 
$\Rep(H_1) \otimes_{\Rep(H_2)} \sC^{H_2,w}$
are compactly generated, and each of the
functors in the underlying semisimplicial
diagram preserve compact objects. 
Therefore, this tensor product is compactly
generated by objects of the form
$V \boxtimes_{\Rep(H_2)} \sF$ for
$V \in \Rep(H_1)^c$ and $\sF \in \sC^{H_2,w,c}$.
Moreover, the functor \eqref{eq:gps-functor} sends
this object to $V \star \Oblv(\sF)$,
where $\Oblv$ denotes the functor 
$\sC^{H_2,w,naive} \to \sC^{H_1,w,naive}$
and $V$ denotes the action of $\Rep(H_1)$
on $\sC^{H_1,w,naive}$. Clearly this
object lies in $\sC^{H_1,w,c}$, giving the claim.

For \eqref{i:gps-3},
let $\on{Ind}_{H_1}^{H_2}:\Rep(H_1) \to \Rep(H_2)$
denote the (continuous) right adjoint to the
restriction functor. 
Because $\on{Ind}_{H_1}^{H_2}$
is a morphism of $\Rep(H_2)$-module categories
(by rigid monoidality of $\Rep(H_2)$),
we see that 
$\on{Ind}_{H_1}^{H_2}  \underset{\Rep(H_2)}{\otimes}
\id_{\sC^{H_2,w}}$ is right adjoint to the
functor $(\sF \in \sC^{H_2,w}) \mapsto 
k \boxtimes_{\Rep(H_2)} \sF$ (for $k$ the trivial
representation).

Now let $\sF,\sG \in \sC^{H_2,w}$ and let
$V,W \in \Rep(H_1)$ be given 
with $W \in \Rep(H_1)^c$. 
By the above, we have:

\[
\begin{gathered}
\ul{\Hom}_{\Rep(H_1) \underset{\Rep(H_2)}{\otimes}
\sC^{H_2,w}}
(W \underset{\Rep(H_2)}{\boxtimes} \sG,
 V \underset{\Rep(H_2)}{\boxtimes} \sF) = \\
 \ul{\Hom}_{\Rep(H_1) \underset{\Rep(H_2)}{\otimes}
\sC^{H_2,w}}
(k \underset{\Rep(H_2)}{\boxtimes} \sG,
 (W^{\vee} \otimes V) \underset{\Rep(H_2)}{\boxtimes} \sF) = \\
\ul{\Hom}_{\sC^{H_2,w}}(\sG,
\on{Ind}_{H_1}^{H_2}(W^{\vee} \otimes V) \star \sF). 
\end{gathered}
\]

\noindent Now if we assume 
$\sF,\sG \in \sC^{H_2,w,+}$ and 
$V \in \Rep(H_1)^+$, then by
Proposition \ref{p:gen-t-str} the above 
$\ul{\Hom}$ maps
isomorphically (via the functor $\psi$ from
\emph{loc. cit}.) onto:

\begin{equation}\label{eq:h1-h2-hom}
\ul{\Hom}_{\sC^{H_2,w,naive}}(\sG,
\on{Ind}_{H_1}^{H_2}(W^{\vee} \otimes V) \star \sF).
\end{equation}

Before calculating this
term further, let $\Oblv^{H_2 \to H_1}:\sC^{H_2,w,naive} \to
\sC^{H_1,w,naive}$ be the restriction functor
and let
$\Av_*^{w,naive,H_1 \to H_2}:\sC^{H_1,w,naive} \to 
\sC^{H_2,w,naive}$ denote its right adjoint. 
Note that although this latter functor
may not commute with colimits,\footnote{For 
example, if $H_2$ is a point
and $H_1 = \prod_{i=1}^{\infty} \bG_a$.} 
its restriction to $\sC^{H_2,w,naive,\geq 0}$
commutes with filtered colimits.\footnote{Indeed, 
compact generation of
$\sC^{H_2,w}$ implies $\sC^{H_2,w,\geq 0} \isom
\sC^{H_2,w,naive,\geq 0}$ is compactly generated
by $\tau^{\geq 0}(\sC^{H_2,w,c})$; these
clearly map into $\tau^{\geq 0}(\sC^{H_1,w,c})$
under $\Oblv$. So $\Oblv:\sC^{H_2,w,\geq 0} \to 
\sC^{H_1,w,\geq 0}$ preserves compacts, so
its right adjoint preserves filtered colimits.}
We claim that for any $U \in \Rep(H_1)^+$,
the natural map:

\[
\Ind_{H_1}^{H_2}(U) \star \sF \to 
\Av_*^{w,naive,H_1 \to H_2}\big(U \star 
\Oblv^{H_2\to H_1}
(\sF)\big)
\]

\noindent is an isomorphism. 
First, if $U$ is the regular representation,
this follows from the identification
$\Av_*^{w,naive}\Oblv = \Fun(H) \star -$ from
\eqref{eq:av-conv}. If $U = \Fun(H) \otimes Q$
for $Q \in \Vect^+$, then the claim follows
from commutation the fact that
$\Av_*^{w,naive,H_1 \to H_2}$ commutes with
colimits bounded uniformly from below.
Finally, general $U \in \Rep(H_1)^+$ follows
using the cobar resolution for $U$, using
the $t$-structure to justify commuting 
the totalization with various functors.\footnote{More
precisely, any truncation $\tau^{\leq N}$ applied
to this totalization coincides with a suitable
finite totalization.}

Applying this to $U = W^{\vee} \otimes V$
from above, we calculate \eqref{eq:h1-h2-hom}
as:

\[
\begin{gathered}
\ul{\Hom}_{\sC^{H_2,w,naive}}(\sG,
\on{Ind}_{H_1}^{H_2}(W^{\vee} \otimes V) \star \sF) = \\
\ul{\Hom}_{\sC^{H_2,w,naive}}(\sG,
\Av_*^{w,naive,H_1 \to H_2}\big(W^{\vee} \otimes V 
\star \Oblv^{H_2\to H_1}
(\sF)) = \\
\ul{\Hom}_{\sC^{H_1,w,naive}}(\Oblv^{H_2 \to H_1}(\sG),
(W^{\vee} \otimes V)
\star \Oblv^{H_2\to H_1}(\sF)) = \\
\ul{\Hom}_{\sC^{H_2,w,naive}}(W \star 
\Oblv^{H_2 \to H_1}(\sG),
V \star \Oblv^{H_2\to H_1}(\sF)). 
\end{gathered}
\]

\noindent As the left hand side of 
\eqref{eq:gps-functor} is compactly generated
by objects of the form $V \boxtimes_{\Rep(H_2)} \sF$
for $\sF \in \sC^{H_2,w,c}$ and $V \in \Rep(H_2)^c$,
this gives fully-faithfulness of
\eqref{eq:gps-functor} when restricted
to compact objects. 

Next, \eqref{i:gps-5} follows from Lemma \ref{l:recognition}
(applied to $H = H_1$ and $\sD$ the essential image
of \eqref{eq:gps-functor-2}).

Finally, we show \eqref{i:gps-4}. 
It suffices to show that any object
$\sF \in \IndCoh(X)^{H_1,w,c}$ lies in the essential
image of \eqref{eq:gps-functor-2}. Moreover,
we can assume $\sF$ lies in the heart of the
$t$-structure. 

In the course of reductions, we use the
following (simple) observation repeatedly:
if $f:Y \to X$ is an equivariant
map of locally almost of finite type indschemes 
acted on by $H_2$, and $\sF$ is of the form 
$f_*^{\IndCoh}(\sG)$ for some $\sG \in \IndCoh(Y)^{H_2,w}$
such that $\sG$ lies in the essential image of:

\[
\Rep(H_1) \underset{\Rep(H_2)}{\otimes} \IndCoh(Y)^{H_2,w} \to 
\IndCoh(Y)^{H_1,w}
\]

\noindent then $\sF$ lies in the essential image of
the functor:

\[
\Rep(H_1) \underset{\Rep(H_2)}{\otimes} \IndCoh(X)^{H_2,w} \to 
\IndCoh(X)^{H_1,w}.
\]

Because $\IndCoh(X)^{\heart} = \IndCoh(X^{cl})^{\heart}$
and similarly for equivariant categories, we may
assume (by the above) that $X$ is classical. 
Then $X$ is a colimit
$X = \colim X_i$ under closed 
embeddings of finite type classical 
schemes acted on by $H_2$. Therefore, applying the
above reduction technique again, 
we may assume $X$ is a classical scheme of finite type.

Now observe that there exists a map $H_2 \to H_2^{\prime}$
of group schemes with $H_2^{\prime}$ an affine algebraic group
and such that $H_2$ acts on $X$ through $H_2^{\prime}$.
We claim that there is a commutative diagram:

\[
\xymatrix{
H_1 \ar[r] \ar[d] & H_1^{\prime} \ar[d] \\
H_2 \ar[r] & H_2^{\prime}
}
\]

\noindent with 
$H_1^{\prime}$ again an affine algebraic
group and such that the $H_1$-equivariant structure
on $\sF$ comes from an $H_1^{\prime}$-equivariant
structure. Indeed, we can write 
$H_1 = \lim_i H_{1,i}$ where each $H_{1,i}$ an affine
algebraic group over $H_2^{\prime}$ (so in particular,
it acts on $X$). 
Then the $H_1$-equivariant structure on $\sF$ is
encoded by the coaction map:

\[
\Oblv(\sF) \to \Fun(H_1) \star \Oblv(\sF) \in \IndCoh(X)^{\heart}.
\]

\noindent The right hand side is
$\colim_i \Fun(H_{1,i}) \star \Oblv(\sF)$, and because
$\sF$ is coherent, the above map factors through
$\Fun(H_{1,i}) \star \Oblv(\sF)$ for some $i$. 
Since we are in a 
1-categorical context here, it suffices to take
$H_1^{\prime} = H_{1,i}$ for such $i$. 

We then have a commutative diagram:

\[
\xymatrix{
\Rep(H_1^{\prime}) \underset{\Rep(H_2^{\prime})}{\otimes} \IndCoh(X)^{H_2^{\prime},w} \ar[rr] \ar[d] & & 
\IndCoh(X)^{H_1^{\prime},w} \ar[d] \\ 
\Rep(H_1) \underset{\Rep(H_2)}{\otimes} \IndCoh(X)^{H_2,w} \ar[rr] & & 
\IndCoh(X)^{H_1,w}.
}
\]

\noindent We have lifted $\sF$ along the right vertical
arrow. By \eqref{i:gps-5}, the top arrow is an equivalence.
Therefore, $\sF$ lies in the essential image of the
bottom functor, giving the result. 

\end{proof}

In the course of the above, we appealed to the
following result.

\begin{lem}\label{l:recognition}

Let $H$ be an affine algebraic group acting weakly\footnote{As
$H$ is finite type, this simply means $\sC$ is a $\QCoh(H)$-module
category. 

Similarly, when we refer to weak invariants in
this lemma, this is what we would call \emph{naive} weak invariants
in an infinite type setting. I.e., we are forming these
weak invariants without any canonical renormalization or
any such.} 
on $\sC \in \DGCat_{cont}$. 
Suppose $\sD \subset \sC^{H,w}$ be a DG subcategory such that: 

\begin{itemize}

\item $\sD$ is closed under colimits and the $\Rep(H)$-action.

\item $\sD$ is compactly generated and the inclusion 
$\sD \into \sC^{H,w}$ preserves compact objects.

\item The composite $\sD \into \sC^{H,w} \xar{\Oblv} \sC$ generates
$\sC$ under colimits. 

\end{itemize}

Then $\sD = \sC^{H,w}$.

\end{lem}

\begin{proof}

We will repeatedly use the fact shown in the course of the proof of
Proposition \ref{p:gen-t-str} that
$\Av_*^w\Oblv = \Fun(H) \star -$ (as endofunctors on $\sC^{H,w}$).

First, note that for $\sG \in \sC$, $\Av_*^w(\sG) \in \sD$.
Indeed, since the functor $\Oblv|_{\sD}:\sD \to \sC$
generates under colimits, it suffices to see that
$\Oblv\Av_*^w:\sC^{H,w} \to \sC^{H,w}$ maps $\sD$ into itself.
But this functor is given by the action of the
regular representation in $\Rep(H)$, so by 
assumption preserves $\sD$. 

Now let $k \in \Rep(H)$ denote the trivial representation.
By Lemma \ref{l:cosimp-summand}, we have:

\[
k \isom \tau^{\leq n}
\Tot^{\leq n+1} (\Av_*^w\Oblv)^{\dot +1}(k)
\]

\noindent for all $n \geq 0$. Because $H$ is finite type
and we are in characteristic $0$, $\Rep(H)$ has
finite cohomological dimension. Therefore, for $n \gg 0$,
the boundary map:

\[
\tau^{>n}
\Tot^{\leq n+1} (\Av_*^w\Oblv)^{\dot +1}(k) \to k[1] \in \Rep(H)
\]

\noindent is nullhomotopic. Therefore, $k$ is a direct summand
of $\Tot^{\leq n+1} (\Av_*^w\Oblv)^{\dot +1}(k)$ for $n \gg 0$.

Now for $\sF \in \sC^{H,w}$, we have $\sF = k \star \sF$
(for $\star$ denoting the action of $\Rep(H)$), which
implies that $\sF$ is a direct summand of:

\[
\Tot^{\leq n+1} (\Av_*^w\Oblv)^{\dot +1}(\sF).
\]

\noindent By the above, this object lies in $\sD$,
so $\sF$ does as well.

\end{proof}

\section{Ind-coherent sheaves on some infinite dimensional spaces}\label{s:indcoh}

\subsection{}

This section, wedged as it is between \S \ref{s:weak-pro} and
\S \ref{s:weak-ind}, is an extended digression. 

To orient the reader, we provide a somewhat extended introduction.

\subsection{}

First, let us explain the role this material plays in 
\S \ref{s:weak-ind}. 

In \emph{loc. cit}., a certain monoidal category denoted
$\IndCoh_{ren}^*(K\backslash H/K)$ plays a key role,
where $H$ a Tate group indscheme and $K \subset H$ is a compact open
subgroup (see \emph{loc. cit}. for the terminology). 

The definition of this category is not hard: 
$\IndCoh_{ren}^*(K\backslash H/K)$ is compactly generated,
and compact objects are objects of $\IndCoh(H/K)^{K,w,naive}$ that map
to $\Coh(H/K) \subset \IndCoh(K\backslash H/K)$. However, this
description breaks symmetry, so the monoidal structure is not
so evident. This problem becomes compounded when we 
try to compare these categories for different
compact open subgroups $K$, or different groups $H$, and so on. 

Therefore, the ultimate goal of this section is to introduce
a class of prestacks we call \emph{renormalizable}, which includes
prestacks of the form $K\backslash H/K$, and for which there
is a robust theory of ind-coherent sheaves (denoted
$\IndCoh_{ren}^*$). As the above description indicates, 
the most important application of this
material in \S \ref{s:weak-ind} is to resolve some homotopy coherence
issues.

\subsection{}

In the above example, we could avoid breaking the symmetry by 
regarding $\IndCoh_{ren}^*(K\backslash H/K)$ as
$K \times K$-equivariant ind-coherent sheaves on $H$. However,
$H$ is of ind-infinite type, so is outside the usual framework
of \cite{indcoh} and \cite{grbook}. Therefore, we first develop
$\IndCoh^*$ on schemes (possibly of infinite type, but qcqs and eventually
coconnective) and \emph{reasonable} indschemes (see \S \ref{ss:indsch}
for the definition).

\subsection{Definition for schemes}

Let $\presup{>-\infty}{\Sch}_{qcqs}$ denote
the category of quasi-compact quasi-separated
eventually coconnective schemes. 

For $S \in \presup{>-\infty}{\Sch}_{qcqs}$,
we define $\Coh(S) \subset \QCoh(S)$ as the full
subcategory of objects $\sF \in \QCoh(S)^+$ 
such that $\sF \in \QCoh(S)^{\geq -N}$ implies
$\sF$ is compact in $\QCoh(S)^{\geq -N}$. 
We define $\IndCoh^*(S)$ as
$\Ind(\Coh(S))$, and we define
a $t$-structure on $\IndCoh^*(S)$ by
taking connective objects to be generated
under colimits by $\Coh(S) \cap \QCoh(S)^{\leq 0}$.
Note that there is a canonical continuous functor
$\Psi = \Psi_S:\IndCoh^*(S) \to \QCoh(S)$ 
ind-extending the embedding $\Coh(S) \into \QCoh(S)$.

\begin{lem}\label{l:psi}

Under the above hypotheses, the functor 
$\Psi:\IndCoh^*(S) \to \QCoh(S)$ is
$t$-exact and an equivalence on eventually coconnective
subcategories.

\end{lem}

\begin{proof}

Clearly $\Psi$ is right $t$-exact.

Let $\sG \in \Perf(S) \subset \Coh(S)$ 
be given. Then the functors:

\[
\ul{\Hom}_{\QCoh(S)}(\sG,\Psi(-)), 
\ul{\Hom}_{\IndCoh^*(S)}(\sG,-):\IndCoh^*(S) \to \Vect
\]

\noindent are canonically isomorphism 
as both commute with
colimits and the restriction of the two to
$\Coh(S)$ are clearly equal.

Therefore, for $\sF \in \IndCoh^*(S)^{>0}$
and $\sG \in \Perf(S) \cap \QCoh(S)^{\leq 0}$,
we have:

\[
\Hom_{\QCoh(S)}(\sG,\Psi(\sF)) = 
\Hom_{\IndCoh^*(S)}(\sG,\sF) = 0.
\]

\noindent As $\QCoh(S)^{\leq 0}$ is generated
under colimits by such $\sG$,
this implies $\Psi(\sF) \in \IndCoh^*(S)^{>0}$,
giving the $t$-exactness.

Now take $\sF \in \Coh(S)$. We claim that the natural
transformation:

\[
\Hom_{\IndCoh^*(S)}(\sF,\tau^{\geq 0}(-)) \to 
\Hom_{\QCoh(S)}(\sF,\Psi(\tau^{\geq 0}(-))
\]

\noindent of functors $\IndCoh^*(S) \to \Gpd$
is an isomorphism. 
Indeed, recall that
there exists $\sF^{\prime} \in \Perf(S)$
and a map $\sF^{\prime} \to \sF$ inducing
an isomorphism on $\tau^{\geq 0}$; therefore,
we may assume $\sF \in \Perf(S)$, and
the result follows from the
(evident) identity:

\[
\Hom_{\QCoh(S)}(\sF,\Psi(-)) \isom 
\Hom_{\IndCoh^*(S)}(\sF,-)
\]

\noindent for such perfect $\sF$. 

We immediately obtain that for any
$\sF \in \IndCoh^*(S)$, the natural transformation:

\[
\Hom_{\IndCoh^*(S)}(\sF,\tau^{\geq 0}(-)) \to 
\Hom_{\QCoh(S)}(\sF,\Psi(\tau^{\geq 0}(-))
\]

\noindent is an isomorphism. Clearly this
is equivalent to fully-faithfulness
of $\Psi|_{\IndCoh^*(S)^+}$.

As $\QCoh(S)^{\geq 0}$ is generated under
colimits by $\tau^{\geq 0}(\Perf(S))$
and this category is in the essential
image of $\Psi$ by $t$-exactness,
we obtain that $\Psi$ induces
an equivalence on coconnective objects
as desired.

\end{proof}

\begin{rem}

The notation $\IndCoh^*$ is parallel to similar notation
from \cite{dmod} and is used to emphasize the differences
between the Noetherian and non-Noetherian situations.
Note that one can dualize to obtain a theory $\IndCoh^!$
parallel to $D^!$ from \emph{loc. cit}. Because
$\IndCoh$ is canonically self-dual on indschemes locally almost
of finite type, we 
do not include superscripts when working with such objects
(since our theory manifestly 
recovers that of \cite{grbook} in this case).

\end{rem}

\subsection{}

For $f:S \to T$ in $\presup{>-\infty}{\Sch}_{qcqs}$,
define $f_*^{\IndCoh}:\IndCoh^*(S) \to \IndCoh^*(T) \in \DGCat_{cont}$
to be the unique left $t$-exact (continuous DG) 
fitting into a commutative diagram:

\[
\xymatrix{
\IndCoh^*(S) \ar[r]^{f_*^{\IndCoh}} \ar[d]^{\Psi} & \IndCoh^*(T) \ar[d]^{\Psi} \\
\QCoh(S) \ar[r]^{f_*} & \QCoh(T).
}
\]

As in \cite{indcoh} Proposition 3.2.4, this construction canonically
upgrades to a functor $\presup{>-\infty}{\Sch}_{qcqs} \to \DGCat_{cont}$.

\begin{rem}

For $f$ affine, $f_*^{\IndCoh}$ is $t$-exact.

\end{rem}

\begin{notation}

For $f$ the projection map $S \to \Spec(k)$, we use the notation 
$\Gamma^{\IndCoh}(S,-):\IndCoh^*(S) \to \Vect$ for the corresponding pushforward
functor (and similarly for the more general $S$ considered later in this section).

\end{notation}

\subsection{}\label{ss:flat}

Next, we discuss behavior with respect to flat morphisms.

\begin{lem}

Suppose $f:S \to T \in \presup{>-\infty}{\Sch}_{qcqs}$ is flat.
Then $f_*^{\IndCoh}$ admits a left adjoint.

\end{lem}

\begin{proof}

In this case, the adjoint functors
$f^*:\QCoh(T) \rightleftarrows \QCoh(S):f_*$
preserve the subcategories $\QCoh(-)^{\geq -n}$ for all integers $n$,
and in particular induce an adjunction.
It immediately follows that $f^*$ maps $\Coh(T) \subset \QCoh(T)$
to $\Coh(S)$, and that the ind-extension of this functor is
the sought-after left adjoint.

\end{proof}

For such flat $f$, we denote this left adjoint by $f^{*,\IndCoh}$.

\begin{lem}\label{l:flat-basechange}

For a Cartesian diagram:

\[
\xymatrix{
S^{\prime} \ar[d]^{\psi} \ar[r]^{\vph} &
T^{\prime} \ar[d]^g \\
S \ar[r]^f & T
}
\]

\noindent with $g$ flat and $S,T,T^{\prime}$
in $\presup{>-\infty}{\Sch}_{qcqs}$ 
(so $S^{\prime} \in \presup{>-\infty}{\Sch}_{qcqs}$ as well), 
the natural transformation:

\[
g^{*,\IndCoh} f_*^{\IndCoh} \to
\vph_*^{\IndCoh} \psi^{*,\IndCoh}
\]

\noindent is an isomorphism.

\end{lem}

\begin{proof}

Clear from the corresponding statement for 
$\QCoh$.

\end{proof}

Now by \cite{grbook} Theorem V.1.3.2.2, the functor $\IndCoh^*$
extends canonically to a functor:

\[
\IndCoh^*:
\on{Corr}(\presup{>-\infty}{\Sch}_{qcqs})_{all;flat}
\to \DGCat_{cont}.
\]

\noindent Here we are using the notation from
\emph{loc. cit}. We remind that the source
category is the correspondence category
for $\presup{>-\infty}{\Sch}_{qcqs}$
in which morphisms from one eventually
coconnective scheme $S$ to another $T$
are diagrams:

\[
\xymatrix{
& H \ar[dl]_{\alpha} \ar[dr]^{\beta} & \\
S & & T
}
\]

\noindent with $\alpha$ flat;
the functor $\IndCoh^*$ 
attaches to such a correspondence
the functor $\beta_*^{\IndCoh}\alpha^{*,\IndCoh}$.
(We have omitted the ``admissible" morphism data
from \emph{loc. cit}.; one may take only
isomorphisms for our purposes, i.e., only work with
a 1-category of correspondences.)

\begin{rem}

The above material 
extends if we replace flatness by finite $\Tor$-
dimension.
However, we do not need this extension and therefore do 
not emphasize it.

\end{rem}

\subsection{}

We have the following basic result.

\begin{lem}\label{l:zar-desc}

$\IndCoh^*$ satisfies Zariski descent on 
$\presup{>-\infty}{\Sch}_{qcqs}$.

\end{lem}

\begin{proof}

The argument from \cite{indcoh} Proposition 4.2.1 applies in this
setting.

\end{proof}

More generally, we have:

\begin{prop}\label{p:flat-desc-sch}

$\IndCoh^*$ satisfies flat descent on 
$\presup{>-\infty}\Sch_{qcqs}$ (for upper-$*$ functors).

\end{prop}

\begin{proof}

Let $f:T \to S$ be a faithfully flat map 
in $\presup{>-\infty}\Sch_{qcqs}$. 
By definition, we need to show that:

\[
\IndCoh^*(S) \to \Tot_{semi}(\IndCoh^*(T^{\times_S \dot+1}))
\]

\noindent is an isomorphism, where $\Tot_{semi}$ indicates
the limit over the semisimplicial category $\bDelta_{inj}$
(which we use because only the semisimplicial part of
the Cech nerve has flat structural maps).

Next, observe that by construction, $\IndCoh^*(S)$ is
naturally a $\QCoh(S)$-module category (in $\DGCat_{cont}$),
and similarly for $T$. Moreover, 
$f^{*,\IndCoh}:\IndCoh^*(S) \to \IndCoh^*(T)$ is $\QCoh(S)$-linear,
where $\IndCoh^*(T)$ is a $\QCoh(S)$-module category via
$f^*:\QCoh(S) \to \QCoh(T)$. 

Therefore, we obtain a functor:

\begin{equation}\label{eq:flat-tens}
f^{*,\IndCoh,enh}:\IndCoh^*(S) \underset{\QCoh(S)}{\otimes} \QCoh(T) \to \IndCoh^*(T).
\end{equation}

\noindent The same argument as in \cite{indcoh} Proposition 4.4.2
shows that this functor is fully-faithful. Note that the essential
image of this functor is the subcategory generated
under colimits and the $\QCoh(T)$-action 
by the essential image of $f^{*,\IndCoh}$.

Next, observe that the above constructions are suitably functorial 
and therefore induce a fully-faithful functor:

\begin{equation}\label{eq:semisimp}
\Tot_{semi}(\IndCoh^*(S) \underset{\QCoh(S)}{\otimes} \QCoh(T^{\times_S \dot+1})) \to 
\Tot_{semi}(\IndCoh^*(T^{\times_S \dot+1})).
\end{equation}

\noindent Below, we will show that this functor is actually an equivalence.

Assuming this, let us deduce the descent claim. 
As $\IndCoh^*(S)$ is compactly generated,
hence dualizable, and $\QCoh(S)$ is rigid monoidal, we obtain that
$\IndCoh^*(S)$ is dualizable as a $\QCoh(S)$-module category. Therefore:

\[
\IndCoh^*(S) \underset{\QCoh(S)}{\otimes}
\Tot_{semi}(\QCoh(T^{\times_S \dot+1})) \isom 
\Tot_{semi}(\IndCoh^*(S) \underset{\QCoh(S)}{\otimes} \QCoh(T^{\times_S \dot+1})).
\]

\noindent The left hand side is then $\IndCoh^*(S) \otimes_{\QCoh(S)} \QCoh(S)$
by flat descent for $\QCoh$ (see \cite{sag} Corollary D.6.3.3).

We now show that \eqref{eq:semisimp} is an equivalence. Suppose
we are given an object of the right hand side. In particular,
we are given $\sF \in \IndCoh^*(T)$ with an isomorphism
$\alpha:p_1^{*,\IndCoh}(\sF) \isom p_2^{*,\IndCoh}(\sF)$ for
$p_i:T \times_S T \to T$ the projections. 

We will show that the map $\sF \to 
p_{1,*}^{\IndCoh}p_2^{*,\IndCoh}(\sF)$ 
adjoint to $\alpha$ realizes $\sF$
as a summand.
Assuming this claim, we obtain that $\sF$ is a direct
summand of 
$f^{*,\IndCoh}f_*^{\IndCoh}(\sF) = 
p_{1,*}^{\IndCoh}p_2^{*,\IndCoh}(\sF)$,
in particular, a summand of an object lying in the
essential image of \eqref{eq:flat-tens}. This implies that $\sF$
is in the essential image of \eqref{eq:flat-tens},
our original object lies in the essential image of 
\eqref{eq:semisimp}, completing the argument.

Let $\Delta:T \to T\times_S T$
denote the diagonal map and let 
$\Delta^{*,\IndCoh}:\IndCoh^*(T \times_S T)
\to \Pro(\IndCoh^*(T))$ denote the ``partially-defined"
left adjoint to $\Delta_*^{\IndCoh}$,
noting that $\Delta^{*,\IndCoh}$ is defined on
$p_i^{*,\IndCoh}(\sF)$. Then observe that
$\Delta^{*,\IndCoh}(\alpha) = \id_{\sF}$.
Indeed, the standard argument in a simplicial
setting applies in our setting: applying
the partially-defined $*$-restriction
along the diagonal $T \to T \times_S T \times_S T$ 
to the cocycle relation here gives the claim.

Therefore, the diagram:

\[
\xymatrix{
p_1^{*,\IndCoh}(\sF) 
\ar[rr]^{\overset{\alpha}{\simeq}} \ar[dr] & &
p_2^{*,\IndCoh}(\sF) \ar[dl] \\
& \Delta_*^{\IndCoh}(\sF)
}
\]

\noindent commutes, where the diagonal arrows are the
obvious ones induced by adjunction and the
observation $p_i\Delta = \id$.
By adjunction, this means that the composition map:

\[
\sF \to 
p_{1,*}^{\IndCoh}(\sF) p_2^{*,\IndCoh}(\sF) \to 
p_{1,*}^{\IndCoh}\Delta_*^{\IndCoh}(\sF) = \sF 
\]

\noindent is the identity for $\sF$. But this
composition is clearly the map under consideration.

\end{proof}

\subsection{Indschemes}\label{ss:indsch}

We now extend the above to the setting of indschemes.

Let $\PreStk_{conv}$ denote the category of \emph{convergent}
prestacks. Recall that these are by definition 
accessible functors $\presup{>-\infty}{\AffSch}^{op} \to \Gpd$; the
natural functor $\PreStk \to \PreStk_{conv}$ admits fully-faithful
left and right adjoints\footnote{We remind that $\PreStk_{conv}$
is typically (e.g., in \cite{grbook}) regarded as a full subcategory
of $\PreStk$ via this right adjoint. This is because under this
embedding, $\PreStk_{conv}$
then contains many subcategories of $\PreStk$ of interest, e.g.,
$\Sch$ and $\IndSch$. 

We mostly ignore this embedding in what follows and only
consider the projection $\PreStk \to \PreStk_{conv}$, but it may help guide
the reader to keep this in mind.}
given by Kan extensions.
We recall that the composition $\Sch \into \PreStk \to \PreStk_{conv}$
is still fully-faithful; we regard $\Sch$ as a subcategory of
convergent prestacks via this functor.

\begin{defin}

A \emph{reasonable indscheme} is an object $S \in \PreStk_{conv}$
that can be written as a filtered colimit $\colim_i S_i$
in $\PreStk_{conv}$
of quasi-compact quasi-separated eventually coconnective schemes 
$S_i$ under almost finitely presented closed embeddings.

\end{defin}

Let $\IndSch_{reas} \subset \PreStk_{conv}$ denote
the subcategory of reasonable indschemes.

\begin{rem}\label{r:afp}

By \cite{sag} Corollary 5.2.2.2, a closed embedding
$T_1 \into T_2 \in \Sch_{qcqs}$ 
is almost finitely presented if and only if for every $n$,
$\tau^{\geq -n} i_*(\sO_{T_1}) \in \QCoh(T_2)^{\geq -n}$
is compact (in this category). In particular, if $T_1$
is eventually coconnective, this is equivalent to 
$i_*(\sO_{T_1})$ lying in $\Coh(T_2) \subset \QCoh(T_2)$.

\end{rem}

\subsection{}

We let $\Sch_{reas}$ denote 
$\Sch_{qcqs} \cap \IndSch_{reas} \subset \PreStk_{conv}$.
We refer to objects of this category as \emph{reasonable} schemes.
Note that any any of the following
classes of quasi-compact quasi-separated schemes is reasonable:

\begin{itemize}
	
\item Eventually coconnective.

\item Locally coherent.

\item Locally \emph{eventually coherent}\footnote{This condition
for $S \in \Sch_{qcqs}$ means that the 
Postnikov maps $\tau^{\geq -n-1}S \to \tau^{\geq -n} S$
are almost finitely presented for $n \gg 0$.
}
in the sense of \cite{indcoh} \S 2.

\end{itemize}

\subsection{}

We define:

\[
\begin{gathered}
\IndCoh^*:\IndSch_{reas} \to \DGCat_{cont} \\
S \mapsto \IndCoh^*(S) \\
(f: S \to T) \mapsto (f_*^{\IndCoh}:\IndCoh^*(S) \to \IndCoh^*(T)).
\end{gathered}
\]

\noindent by left Kan extension.

\begin{rem}

This definition of $\IndCoh^*$ evidently extends to
all convergent prestacks (as the relevant left Kan extension). However,
this definition does not recover the category we are after for 
the class of \emph{(weakly) renormalizable prestacks} introduced below.
Therefore, we do not consider this total left Kan extension here.

\end{rem}

\subsection{}

Explicitly, for  
$S = \colim_i S_i$ with $S_i$ eventually coconnective
and quasi-compact quasi-separated schemes 
and structural maps being almost finitely presented 
closed embeddings, we have
$\IndCoh^*(S) = \colim_i \IndCoh^*(S_i) \in \DGCat_{cont}$.

As each of the structure functors in this colimit is $t$-exact
(since pushforward for affine morphisms is), $\IndCoh^*(S)$
inherits a canonical $t$-structure
(see e.g. \cite{whit} Lemma 5.4.3 (1)). This $t$-structure
is characterized by the fact that each pushforward functor
$\IndCoh^*(S_i) \to \IndCoh^*(S)$ is $t$-exact.

In addition, by Remark \ref{r:afp}, each of these 
functors preserves compact objects. In particular,
$\IndCoh^*(S)$ is compactly generated, and so is $\IndCoh^*(S)^{\leq 0}$. 

\begin{defin}\label{d:coh}

$\Coh(S) \subset \IndCoh^*(S)$ is the subcategory
of compact objects. We refer to such objects as \emph{coherent}.

\end{defin}

We record the following characterization of coherent
sheaves for future use. 

\begin{lem}\label{l:coh-t-str}

For $S$ a reasonable indscheme, 
$\sF \in \IndCoh^*(S)$ is coherent if and only if
$\sF \in \IndCoh^*(S)^+$ and for all $N \gg 0$ with
$\sF \in \IndCoh^*(S)^{\geq -N}$, $\sF$ is compact
in the category $\IndCoh^*(S)^{\geq -N}$.

\end{lem}

\begin{proof}

\step\label{st:acpt-coh-sch}

First, we remark that this result 
is immediate from the definitions
and Lemma \ref{l:psi} when
$S \in \presup{>-\infty}{\Sch}_{qcqs}$.

\step\label{st:acpt}

For convenience, we introduce the following terminology.
Suppose $\sC \in \DGCat_{cont}$ is equipped with a 
right separated $t$-structure
compatible with filtered colimits. 
We say $\sF \in \sC$ is \emph{almost compact} if
$\sF \in \sC^+$ and 
for all $N\gg 0$ with
$\sF \in \sC^{\geq -N}$, $\sF$ is compact
in $\sC^{\geq -N}$. 

In this terminology, our goal is to show that
compactness is equivalent to almost compactness
in $\IndCoh^*(S)$. One direction is evident: 
compactness implies almost compactness
as compact objects in $\IndCoh^*(S)$ are eventually connective.

\step\label{st:acpt-pres}

Suppose $\sC,\sD\in \DGCat_{cont}$ are equipped with
right separated $t$-structures compatible with filtered colimits.
Let $F:\sC \to \sD$ be a $t$-exact functor
admitting a continuous right adjoint $G$ and
such that $F|_{\sC^+}$ conservative. We claim that
for $\sF \in \sC^+$ with $F(\sF)$ almost compact,
$\sF$ is itself almost compact. 

First, note that the $t$-structures are automatically
right complete. Therefore, almost compactness of $F(\sF)$
implies that it is eventually connective.
As $F|_{\sC^+}$ is conservative and $t$-exact, this implies
$\sF$ is also eventually connective.

By Lemma \ref{l:cosimp-summand}, $F|_{\sC^+}$ is comonadic
(c.f. the proof of Proposition \ref{p:alg-vs-cats}).
Therefore, for an integer $N$ and 
$\sG \in \sC^{\geq -N}$, we have:

\begin{equation}\label{eq:acpt-comonad}
\Hom_{\sC}(\sF,\sG) = 
\Tot \, \Hom_{\sD}(F(\sF),F(GF)^{\dot}(\sG)) \in \Gpd.
\end{equation}

\noindent There exists an integer $M$ (depending on
$N$ and $\sF$) such that each term in this
totalization is an $M$-truncated groupoid
(as $F(\sF)$ is eventually connective).
Therefore, the above totalization commutes with a finite
totalization. As $\Hom_{\sD}(F(\sF),-)$
and $F$ and $G$ all commute with filtered colimits
in $\sD^{\geq -N}$, and as filtered colimits commute
with finite limits in $\Gpd$, this implies that the left hand
side of \eqref{eq:acpt-comonad} commutes with filtered
colimits in the variable $\sG (\in \sC^{\geq -N})$ as desired.

\step\label{st:cohs-pushed-fwd}

Let $S = \colim_i S_i$ as in the definition of reasonable
indscheme. Let $\alpha_i:S_i \to S$ denote
the structural morphisms. 

By definition, we have $\IndCoh^*(S) = \colim_i \IndCoh^*(S_i)$ 
(under pushforwards), with the colimit being taken
in $\DGCat_{cont}$. 

We will also need the following variant.
Let $\Cat_{pres}$
denote the category of presentable 
categories and functors commuting with colimits. 
For any integer $n$, we claim:

\begin{equation}\label{eq:coconn-colim}
\IndCoh^*(S)^{\geq -n} = \underset{i}{\colim} \, 
\IndCoh^*(S_i)^{\geq -n} \in 
\Cat_{pres}
\end{equation}

\noindent with the colimit being taken in $\Cat_{pres}$.
Indeed, we have $\IndCoh^*(S) = \lim_i \IndCoh^*(S_i)$
under right adjoints, where this limit may be
formed in any of $\DGCat_{cont}$, $\Cat_{pres}$, and $\Cat$.
As these right adjoints are left $t$-exact, we find
that 
$\IndCoh^*(S)^{\geq -n} = \lim_i \IndCoh^*(S_i)^{\geq -n}$.
The functors in this limit also admit left adjoints,
so the limit coincides with the colimit in $\Cat_{pres}$.

\step 

We now conclude the argument.
Suppose $\sF \in \IndCoh^*(S)$ is almost compact.
By assumption, $\sF \in \IndCoh^*(S)^{\geq -n}$ for
some $n \in \bZ$.

Write $\sF = \colim_j \sF_j$ for 
$\sF_j \in \Coh(S)$. We obtain 
$\sF = \colim_j \tau^{\geq -n}(\sF_j)$. 
By almost compactness of $\sF$, there exists an index
$j$ such that $\sF$ is a summand of $\tau^{\geq -n}(\sF_j)$.

As $\Coh(S) = \colim_i \Coh(S_i) \in \Cat$ 
(c.f. \cite{higheralgebra} Lemmas 7.3.5.10-13),
there exists an index
$i$ and some $\widetilde{\sF}_j \in \Coh(S_i)$
such that $\alpha_{i,*}^{\IndCoh}(\widetilde{\sF}_j) = \sF_j$.
Moreover, by \eqref{eq:coconn-colim} and 
\cite{higheralgebra} Lemma 7.3.5.10, after possibly increasing
the index $i$ there exists
$\widetilde{\sF} \in \IndCoh^*(S_i)^{\geq -n}$ (a summand
of $\tau^{\geq -n}(\sF_j)$)
with: 

\[
\alpha_{i,*}^{\IndCoh}(\widetilde{\sF}) = \sF
\]

\noindent (as summands of $\sF_j$).

By Step \ref{st:acpt-pres}, $\widetilde{\sF}$ is almost
compact in $\IndCoh^*(S_i)$. As in Step \ref{st:acpt-coh-sch},
this means $\widetilde{\sF} \in \Coh(S_i)$. As
$\alpha_{i,*}^{\IndCoh}$ admits a continuous right adjoint,
we obtain the result.

\end{proof}

\subsection{}

The following technical result is convenient for comparing
different possible presentations of a reasonable
indscheme. 

\begin{prop}\label{p:reas-presentations}

Let $S = \colim_{i \in \sI} S_i^1 = \colim_{j \in \sJ} S_j^2$ be two
expressions of $S$ as a reasonable indscheme, i.e.,
these colimits are filtered,
$S_i^1,S_j^2 \in \presup{>-\infty}{\Sch}_{qcqs}$
and the structure maps in each of these colimits
are almost finitely presented. Let $\alpha_i:S_i^1 \to S$
and $\beta_j:S_j^2 \to S$ denote the structure maps.

Then for any choice of indices $i \in \sI$ and $j \in \sJ$
such that the map $\alpha_i^1:S_i^1 \to S$ factors 
as $S_i^1 \xar{\iota} S_j^2 \xar{\beta_j} S$,
the map $\iota$ is almost of finite presentation.

\end{prop}

\begin{proof}

By Remark \ref{r:afp}, it suffices to show
that $\iota_*(\sO_{S_i^1}) \in \Coh(S_j^2)$.
Clearly this object lies in $\IndCoh^*(S_j^2)^+$.
By Lemma \ref{l:coh-t-str} and Step \ref{st:acpt-pres}
from its proof, it therefore suffices to show that
$\beta_{j,*}^{\IndCoh}\iota_*(\sO_{S_i^1}) \in \Coh(S)$,
but this is clear as 
$\beta_{j,*}^{\IndCoh}\iota_* = \alpha_{i,*}^{\IndCoh}$.

\end{proof}

\subsection{Reasonable schemes}

Observe that we have a functor:

\[
\QCoh(-):\Sch_{qcqs} \to \DGCat_{cont}
\]

\noindent encoding pushforward of quasi-coherent sheaves.
By left Kan extension,
there is a canonical natural transformation:

\[
\Psi:\IndCoh^*(-)|_{\Sch_{reas}} \to \QCoh(-)|_{\Sch_{reas}}
\]

\noindent of functors $\Sch_{reas} \to \DGCat_{cont}$.

\begin{lem}\label{l:sch-reas}

For every $S \in \Sch_{reas}$,
$\Psi$ induces an equivalence 
$\IndCoh^*(S)^+ \isom \QCoh(S)^+$.
Moreover, $\Psi$ identifies $\Coh(S) \subset
\IndCoh^*(S)^+$ (as defined in Definition \ref{d:coh})
with the subcategory of cohomologically bounded
objects in $\QCoh(S)$ that are compact in 
$\QCoh(S)^{\geq -N}$ for all $N \gg 0$.

\end{lem}

\begin{proof}

In what follows, for $\sC$
a DG category with a $t$-structure and $n \geq 0$,
we let $\sC^{[-n,0]}$ denote the subcategory $\sC$ of
objects in cohomological degrees $[-n,0]$.

Because $S$ is reasonable, we have $S = \colim S_i$
a filtered colimit under almost finitely presented closed
embeddings. 
As in Step \ref{st:cohs-pushed-fwd} from the proof of
Lemma \ref{l:coh-t-str},
we have
$\IndCoh^*(S)^{[-n,0]} = \colim_i \IndCoh^*(S_i)^{[-n,0]}$
with the colimit being taken in $\Cat_{pres}$
(the category of presentable categories and functors
commuting with colimits). 

Recall that e.g. 
$\QCoh(S)^{[-n,0]} = \QCoh(\tau^{\geq -n} S)^{[-n,0]}$.
Therefore, we obtain:

\[
\begin{gathered}
\IndCoh^*(S)^{[-n,0]} = \underset{i}{\colim} \, 
\IndCoh^*(S_i)^{[-n,0]}
\overset{\Psi}{\isom}  
\underset{i}{\colim} \, \QCoh^*(S_i)^{[-n,0]} = \\
\underset{i}{\colim} \, \QCoh^*(\tau^{\geq -n} S_i)^{[-n,0]} = 
\underset{i}{\colim} \, \IndCoh^*(\tau^{\geq -n} S_i)^{[-n,0]} =
\IndCoh^*(\tau^{\geq -n} S)^{[-n,0]}
\end{gathered}
\]

\noindent with all colimits being taken in $\Cat_{pres}$,
and where we have used that 
$\tau^{\geq -n} S = \colim_i \tau^{\geq -n} S_i$.

By right completeness of the $t$-structures on $\IndCoh^*(S)$
and $\QCoh(S)$, we obtain the claims.

\end{proof}

\subsection{Proper morphisms}

We now discuss proper morphisms.
We refer to \cite{sag} Part II for an extensive
discussion of such morphisms in derived algebraic geometry.
However, we take a more restrictive definition than \emph{loc. cit}.
(to simplify terminology): we say 
$f:S \to T \in \Sch_{qcqs}$ is \emph{proper} if it is proper
in the sense of \cite{sag} \emph{and} almost of finite presentation
(only finite type is required in \emph{loc. cit}.). 

\begin{lem}\label{l:upper-!-sch}

Suppose $f:S \to T \in \Sch_{reas}$ is proper.
Then $f_*^{\IndCoh}$ admits a continuous right adjoint
$f^!$.

\end{lem}

\begin{proof}

This is immediate
from Lemma \ref{l:sch-reas} and \cite{sag} Theorem 5.6.0.2.

\end{proof}

\begin{lem}\label{l:proper}

Suppose we are given a Cartesian diagram
of schemes:

\[
\xymatrix{
S^{\prime} \ar[d]^{\psi} \ar[r]^{\vph} &
T^{\prime} \ar[d]^g \\
S \ar[r]^f & T
}
\]

\noindent with all terms lying in $\Sch_{reas}$\footnote{This
is \emph{not} automatic for $S^{\prime} = S \times_T T^{\prime}$
even if $S$, $T$, and $T^{\prime}$ lie in $\Sch_{reas}$.}
and $f$ proper.

\begin{enumerate}

\item\label{i:proper-2} 

The natural map:

\[
\psi_*^{\IndCoh}\vph^! \to f^! g_*^{\IndCoh}
\]

\noindent is an isomorphism.

\item\label{i:proper-3} 

Suppose that $g$ is flat.
Then the natural\footnote{We remark that the existence
of this map depends on \eqref{i:proper-2}.} map:

\[
\psi^{*,\IndCoh} f^! \to \vph^! g^{*,\IndCoh}
\]

\noindent is an isomorphism.

\end{enumerate}

\end{lem}

\begin{proof}

The same argument 
from \cite{indcoh} Proposition 3.4.2 applies for
\eqref{i:proper-2}.
Similarly, again using Lemma \ref{l:sch-reas}, the same argument as 
in \cite{indcoh} Proposition 7.1.6 applies\footnote{There
is one small modification to make. The argument
in \cite{indcoh} uses Proposition 3.6.11 from \emph{loc. cit}.,
which in turn uses Lemma 3.6.13 from \emph{loc. cit}. 
The argument from \emph{loc. cit}.
for this lemma does not literally work: \cite{indcoh} reduces
to the classical case by means which are not available to us here.
The difference in the argument is not significant, but we provide
the details below.

The lemma in question (in our setup) asserts that
for flat $f:S \in \Sch_{reas}$ and $\sF \in \QCoh(S)$
flat, the functor $\sF \otimes - \IndCoh^*(S) \to \IndCoh^*(S)$
is $t$-exact. (Here we are using the natural action of
$\QCoh(S)$ on $\IndCoh^*(S)$ obtained by ind-extension from
the action of $\Perf(S)$ on $\Coh(S)$.)

In our setup, Zariski descent (Lemma \ref{l:zar-desc} below, which
is independent of Lemma \ref{l:proper}),
reduces to considering the case where $S$ is affine.
Then the result is immediate from Lazard's theorem in the
derived setup: see \cite{higheralgebra} Theorem 7.2.2.15.}
for \eqref{i:proper-3}.

\end{proof}

\subsection{Flatness}\label{ss:flat-prestks}

We say a morphism $f:T_1 \to T_2 \in \PreStk_{conv}$ is \emph{flat}
if for any eventually coconnective affine $S \in \presup{>-\infty}{\AffSch}$
and any map $S \to T_2$, the fiber product $T_1 \times_{T_2} S$
lies in $\presup{>-\infty}{\Sch}_{qcqs}$ and its structure
map to $S$ is flat. Similarly, we say $f$ is a \emph{flat cover}
if it is flat and $T_1 \times_{T_2} S \to S$ is faithfully flat.

\begin{rem}

As our definition requires flat morphisms 
to be schematic (in the relevant sense for convergent prestacks)
and to be quasi-compact quasi-separated, it is much more
stringent than usual notions of flatness. We hope the
reader will forgive this abuse, which
we find unburdens the terminology and notation to some degree.

\end{rem}

Clearly flat morphisms (resp. covers) are closed under compositions
and base-change. 

\begin{rem}\label{r:flat-reas}

If $f:S \to T$ is flat and $T$ is a reasonable
indscheme, then $S$ is a reasonable indscheme as well.

\end{rem}

\subsection{}

We will now study a variety of base-change results 
for flat morphisms.

\begin{lem}\label{l:flat-indsch}

Let $f:S \to T \in \IndSch_{reas}$ be a flat morphism.

\begin{enumerate}

\item\label{i:flat-1} $f_*^{\IndCoh}$ admits a left adjoint
$f^{*,\IndCoh}$.

\item\label{i:flat-2} 
Suppose $T = \colim_i T_i$ as in the definition
of reasonable indscheme. 
For any index $i$, let 
$S_i \coloneqq S \times_T T_i$ and denote the relevant
structural maps as:

\[
\xymatrix{
S_i \ar[d]^{\beta_i} \ar[r]^{f_i} & T_i \ar[d]^{\alpha_i} \\
S \ar[r]^f & T.
}
\]

\noindent Let $\alpha_i^!$ and $\beta_i^!$ denote
the (continuous) right adjoints to $\alpha_{i,*}^{\IndCoh}$
and $\beta_{i,*}^{\IndCoh}$. 

Then the natural map:

\[
f_{i,*}^{\IndCoh}\beta_i^! \to \alpha_i^! f_*^{\IndCoh}
\]

\noindent is an isomorphism.

\item\label{i:flat-3}

The natural map:

\[
f_i^{*,\IndCoh} \alpha_i^! \to \beta_i^! f^{*,\IndCoh}
\]

\noindent is an isomorphism.

\item\label{i:flat-4}

Given a Cartesian diagram in $\IndSch_{reas}$:

\[
\xymatrix{
S^{\prime} \ar[d]^{\psi} \ar[r]^{\vph} &
T^{\prime} \ar[d]^g \\
S \ar[r]^f & T
}
\]

\noindent (with $f$ flat), the natural morphism:

\begin{equation}\label{eq:flat-basechange}
f^{*,\IndCoh}g_*^{\IndCoh} \to \psi_*^{\IndCoh}\vph^{*,\IndCoh}
\end{equation}

\noindent is an isomorphism.

\end{enumerate}

\end{lem}

\begin{proof}

By Lemma \ref{l:proper} \eqref{i:proper-2}-\eqref{i:proper-3}, 
the adjoint functors:

\begin{equation}\label{eq:fi-adj}
f_i^{*,\IndCoh}:\IndCoh^*(T_i)\rightleftarrows 
\IndCoh^*(S_i) :f_{i,*}^{\IndCoh}
\end{equation}

\noindent are (canonically) compatible with the structure functors
in the limits $\IndCoh^*(S) = \lim_i \IndCoh^*(S_i)$
and $\IndCoh^*(T) = \lim_i \IndCoh^*(T_i)$
(the limits being under upper-! functors),
and therefore induce an adjunction
$(f^{*,\IndCoh},f_*^{\IndCoh})$ satisfying 
\eqref{i:flat-2} and \eqref{i:flat-3}.

For \eqref{i:flat-4}, we are immediately reduced
to the case where $T^{\prime} \in 
\presup{>-\infty}\Sch_{qcqs}$ (as $\IndCoh^*(T^{\prime})$
is necessarily generated under colimits by objects pushed
forward from eventually coconnective schemes).
Then $g$ factors as $T^{\prime} \xar{\ol{g}} T_i \xar{\alpha_i} T$.
By Lemma \ref{l:flat-basechange}, we are reduced to the
case where $T^{\prime} = T_i$ and $g = \alpha_i$.

In this case, the claim follows from the fact that
the adjoint functors \eqref{eq:fi-adj} are
(canonically) compatible with the structural functors
in the colimits $\IndCoh^*(S) = \colim_i \IndCoh^*(S_i)$
and $\IndCoh^*(T) = \colim_i \IndCoh^*(T_i)$
(in $\DGCat_{cont}$, under lower-* functors) 
by Lemma \ref{l:flat-basechange}.

\end{proof}

\begin{cor}\label{c:flat-t-exact}

Let $f:S \to T \in \IndSch_{reas}$ be flat. Then
$f^{*,\IndCoh}: \IndCoh^*(T) \to \IndCoh^*(S)$ is $t$-exact.

\end{cor}

\begin{proof}

We use the notation from Lemma \ref{l:flat-indsch} \eqref{i:flat-2}.
Then $\IndCoh^*(T)^{\leq 0}$ is generated
under colimits by objects of the form $\alpha_{i,*}^{\IndCoh}(\sF)$
for $\sF \in \IndCoh^*(T_i)^{\leq 0}$. We then have:

\[
f^{*,\IndCoh}(\alpha_{i,*}^{\IndCoh}(\sF)) = 
\beta_{i,*}^{\IndCoh} f_i^{*,\IndCoh}(\sF) 
\]

\noindent by Lemma \ref{l:flat-indsch} \eqref{i:flat-4}
and this is in $\IndCoh^*(S)^{\leq 0}$ because $f_i$
is a flat map of schemes. 
Therefore, $f^{*,\IndCoh}$ is right $t$-exact.

For left $t$-exactness, suppose $\sF \in \IndCoh^*(T)^{\geq 0}$.
To see $f^{*,\IndCoh}(\sF) \in \IndCoh^*(S)^{\geq 0}$, it is equivalent
to show that $\beta_{i}^! f^{*,\IndCoh}(\sF) \in \IndCoh^*(S_i)^{\geq 0}$
for all $i$. But by Lemma \ref{l:proper} \eqref{i:proper-3},
we have:

\[
\beta_{i}^! f^{*,\IndCoh}(\sF) = 
f_i^{*,\IndCoh}\alpha_i^!(\sF).
\]

\noindent Then $\alpha_i^!(\sF) \in \IndCoh^*(T_i)^{\geq 0}$ 
by assumption on $\sF$, so the same is true after applying
$f_i^{*,\IndCoh}$ to it.

\end{proof}

\subsection{}

We say that a morphism
$f:S \to T$ of reasonable indschemes is \emph{ind-proper} 
if for some (equivalently,\footnote{C.f. 
Proposition \ref{p:reas-presentations}.}
any) presentations $S = \colim_{i \in \sI} S_i$ and
$T = \colim_{j \in \sJ} T_j$ as in the definitions of reasonable
indschemes, and any index $i \in \sI$ there exists
$j \in \sJ$ such that
$S_i \to S \to T$ factors through a proper morphism 
$S_i \to T_j$.

\begin{lem}\label{l:ind-proper-upper-!}

Let $f:S \to T$ be an ind-proper morphism of reasonable
indschemes.
Then $f_*^{\IndCoh}$ admits a continuous right
adjoint $f^!$.

\end{lem}

\begin{proof}

As $\IndCoh^* = \Ind(\Coh)$ here,
the continuity of the right adjoint $f^!$ is equivalent
to $f_*^{\IndCoh}$ preserving compacts. 
Then we are immediately reduced to the case
where 
$S,T \in \presup{>-\infty}{\Sch}_{qcqs}$, which is
covered by Lemma \ref{l:upper-!-sch}.

\end{proof}

We have the following (somewhat partial) generalization
of Lemma \ref{l:proper}.

\begin{lem}\label{l:indproper-flat}

Suppose we are given a Cartesian diagram
of reasonable indschemes:

\[
\xymatrix{
S^{\prime} \ar[d]^{\psi} \ar[r]^{\vph} &
T^{\prime} \ar[d]^g \\
S \ar[r]^f & T
}
\]

\noindent with $f$ ind-proper and $g$ flat.

\begin{enumerate}

\item\label{i:indproper-2} 

The natural map:

\[
\psi_*^{\IndCoh}\vph^! \to f^! g_*^{\IndCoh}
\]

\noindent is an isomorphism.

\item\label{i:indproper-3} 

The natural map:

\[
\psi^{*,\IndCoh} f^! \to \vph^! g^{*,\IndCoh}
\]

\noindent is an isomorphism.

\end{enumerate}

\end{lem}

\begin{proof}

Writing $S = \colim_j S_j$ as in the definition
of reasonable indscheme, both statements immediately
reduce to the case where $S \in \presup{>-\infty}\Sch_{qcqs}$. 

Now take $T = \colim_i T_i$ as in the definition
of reasonable indscheme. By assumption on $f$,
the map $f$ factors as $S \xar{\ol{f}} T_j \xar{\alpha_j} T$
for $\alpha_j:T_j \to T$ the structure map and 
$\ol{f}$ proper. By Lemma 
\ref{l:proper} \eqref{i:proper-2} and \eqref{i:proper-3},
both statements reduce to the case where $S = T_j$
and $f = \alpha_j$ (Here the
results follow from 
Lemma \ref{l:flat-indsch} \eqref{i:flat-2}-\eqref{i:flat-3}.

\end{proof}

\subsection{}

We record the following basic result for later use.

\begin{lem}\label{l:coh-local}

For a flat cover $f:S \to T \in \IndSch_{reas}$,
$\sF \in \IndCoh^*(T)$ lies in $\Coh(T)$ if and only
$f^{*,\IndCoh}(\sF)$ lies in $\Coh(S)$.

\end{lem}

\begin{proof}

First, note that $f^{*,\IndCoh}$ is conservative. 
Indeed, if $T \in \presup{>-\infty}\Sch_{qcqs}$,
this follows from 
Proposition \ref{p:flat-desc-sch},
and the general case results from this using 
Lemma \ref{l:proper} \eqref{i:proper-3}.
(See also Theorem \ref{t:flat-desc}).

Therefore, if $f^{*,\IndCoh}(\sF) \in \Coh(S)$,
we find in particular that 
$\sF \in \IndCoh^*(T)^+$. Now we obtain the result from 
Lemma \ref{l:coh-t-str} and Step \ref{st:acpt-pres} from
the proof of \emph{loc. cit}.

\end{proof}

\subsection{}

We remark that because we can $*$-pullback along flat maps,
the functor $\IndCoh^*:\IndSch_{reas} \to \DGCat_{cont}$
extends (again by \cite{grbook} Theorem V.1.3.2.2) 
to a functor:

\[
\on{Corr}(\IndSch_{reas})_{all;flat}
\to \DGCat_{cont}.
\]

\subsection{Relationship to $D$-modules}\label{ss:indcoh-dmod}

Recall the functor $D^*:\IndSch_{reas} \to \DGCat_{cont}$
constructed in \cite{dmod}.
We will construct a canonical natural transformation:

\[
\IndCoh^* \to D^*
\]

\noindent that can be thought of as inducing an ind-coherent
sheaf to a $D$-module.

Each of these functors is by definition left Kan extended
from $\presup{>-\infty}\Sch_{qcqs}$, so it suffices
to define the natural transformation for the 
restrictions of these functors here.
Moreover, each of these functors is a Zariski sheaf,
so it suffices to define the natural transformation
on $\presup{>-\infty}\AffSch$. 
This in turn is equivalent to specifying
a compatible sequence of natural transformations
on each $\presup{\geq -n}{\AffSch} = 
\Pro(\presup{\geq -n}{\AffSch}_{ft})$.

By definition, $D^*|_{\presup{\geq -n}{\AffSch}}$ 
is right Kan extended from $\presup{\geq -n}{\AffSch}_{ft}$. 
Therefore, we need to specify the natural
transformation on $\presup{\geq -n}{\AffSch}_{ft}$
compatibly over all $n$. Here we define our
natural transformation as the (``right\footnote{As opposed
to left.}") 
$D$-module induction functor
$\ind:\IndCoh \to D$ constructed in \cite{grbook}.

\begin{rem}\label{r:indcoh-dmod-lax}

By construction, this natural transformation upgrades
to a natural transformation of lax symmetric monoidal
functors; see \S \ref{ss:indcoh-lax} below.

\end{rem}

\subsection{Weakly renormalizable prestacks}

We now introduce a convenient class of prestacks.\footnote{It 
is formally convenient to
not have to sheafify in forming quotients such as $X/H$, 
so we prefer to work with prestacks.}

\begin{defin}

A convergent prestack $S \in \PreStk_{conv}$ is 
\emph{weakly renormalizable} if there exists
a flat covering map $T \to S$ with
$T \in \IndSch_{reas}$.

We let $\PreStk_{w.ren} \subset \PreStk_{conv}$
denote the subcategory of weakly renormalizable prestacks.

\end{defin}

\subsection{Morphisms}

We now introduce some classes of morphisms
between weakly renormalizable prestacks.

\begin{defin}

\begin{enumerate}

\item A morphism $f:S_1 \to S_2 \in \PreStk_{w.ren}$
is \emph{reasonable indschematic} if for any 
flat morphism $T \to S_2$ with 
$T \in \IndSch_{reas}$, the fiber product 
$T_1 \times_{T_2} S \in \PreStk_{conv}$
is a reasonable indscheme.

\item A morphism $S_1 \to S_2 \in \PreStk_{w.ren}$ 
is \emph{locally flat} if for any $T \in \IndSch_{reas}$ 
and any flat morphism $T \to S_1$,
the composition
$T \to S_1 \to S_2$ is flat.

\end{enumerate}

\end{defin}

\begin{example}\label{e:reas-indsch}

Any morphism from a reasonable indschemes to
a weakly renormalizable prestack is reasonable indschematic.
In particular, any morphism between reasonable indschemes
is reasonable indschematic. 

\end{example}

\begin{rem}\label{r:loc-flat}

It is straightforward to show that to check $S_1 \to S_2$
is locally flat, it suffices to check the condition
from the definition for some flat cover $T$ of $S_1$. 
In particular, a morphism of reasonable indschemes
is locally flat if and only if it is flat.

\end{rem}

\begin{example}

If $H$ is a classical affine group scheme, $\bB H \to \Spec(k)$
is locally flat.

\end{example}

The following is immediate:

\begin{lem}\label{l:reas-mor-flat}

Reasonable indschematic and locally flat 
morphisms are closed under compositions.
Any base-change of a reasonable indschematic morphism
by a locally flat morphism is again reasonable indschematic,
and any base-change of a locally flat morphism by
a reasonable indschematic morphism is again locally flat.

\end{lem}

\subsection{Set-theoretic remarks}

In what follows, we will not explicitly address certain
set-theoretic issues. More precisely, we will want to 
form limits e.g. over all reasonable indschemes
flat over a given weakly renormalizable prestack. 
This indexing category is not essentially small, so there
are set-theoretic issues. 

To address these, fix a regular cardinal $\kappa$
and replace ``flat" everywhere by ``flat and locally
$\kappa$-presented." One should understand
weakly renormalizable prestacks in this sense
(i.e., these are prestacks admitting a locally $\kappa$-presented
flat cover by a reasonable indscheme), and so on.

As will follow from Theorem \ref{t:flat-desc}, 
all of our constructions
are invariant under extension of $\kappa$, i.e.,
if $S$ is a weakly renormalizable prestack relative
to $\kappa$ and $\kappa^{\prime} \geq \kappa$ is another
regular cardinal, then the categories $\IndCoh^*$
defined using $\kappa$ and $\kappa^{\prime}$ coincide.

Again, since the cutoff $\kappa$ plays such a minor role, 
in order to simplify the exposition we do not
mention it again.

\subsection{$\IndCoh^*$ on weakly renormalizable prestacks}\label{ss:indcoh-wren}

Define $\PreStk_{w.ren,loc.flat}$ as the 1-full subcategory 
of $\PreStk_{w.ren}$ where we only allow locally flat morphisms,
and define $\IndSch_{reas,flat}$ similarly.

\begin{defin}

$\IndCoh^*:\PreStk_{w.ren,loc.flat}^{op} \to \DGCat_{cont}$
is the right Kan extension of the functor
$\IndSch_{reas,flat}^{op} \to \DGCat_{cont}$
(which sends $S$ to $\IndCoh^*(S)$ and sends flat 
$f:T_1 \to T_2$ to $f^{*,\IndCoh}$).

\end{defin}

By \cite{grbook} Theorem V.2.6.1.5,\footnote{We remark that the hypotheses
from \emph{loc. cit}. are trivially verified in this setting,
c.f. Lemma \ref{l:reas-mor-flat}.}
the above construction upgrades canonically to a functor:

\[
\IndCoh^*:\on{Corr}(\PreStk_{w.ren})_{reas.indsch;loc.flat} \to \DGCat_{cont}.
\]

\noindent Here ``$reas.indsch$" is shorthand for
\emph{reasonable indschematic} and ``$loc.flat$" is shorthand
for \emph{locally flat}. Therefore, the
notation indicates that for a reasonable indschematic morphism
$f:S \to T$ between weakly renormalizable prestacks, 
we have a pushforward functor
$f_*^{\IndCoh}:\IndCoh^*(S) \to \IndCoh^*(T)$;
for $f$ locally flat, we have a functor $f^{*,\IndCoh}$;
and the two satisfy base-change. 

\begin{rem}

Using the 2-category of correspondences as in \cite{grbook},
one can further encode that $f^{*,\IndCoh}$ is left adjoint
to $f_*^{\IndCoh}$ for flat $f$.

\end{rem}

\begin{rem}

By \cite{indcoh} Proposition 11.4.3, for $S$ weakly
renormalizable and locally almost of finite type,
$\IndCoh^*(S)$ is canonically isomorphic to the usual
category $\IndCoh((S)$.

\end{rem}

\subsection{}

The following result justifies the definition of $\IndCoh^*$
for weakly renormalizable prestacks.

\begin{thm}\label{t:flat-desc}

$\IndCoh^*$ satisfies flat descent on 
$\PreStk_{w.ren}$.

\end{thm}

\begin{proof}

Let $S \in \PreStk_{w.ren}$ be given and
let $f:T \to S$ be a flat cover. 
We need to show that:

\[
\IndCoh^*(S) \to \Tot_{semi}(\IndCoh^*(T^{\times_S \dot+1}))
\]

\noindent is an isomorphism.

We proceed in increasing generality.

\step First, suppose $S \in \IndSch_{reas}$. 

Let $S = \colim_i S_i$ as in the definition of reasonable
indscheme. We then have:

\[
\begin{gathered}
\IndCoh^*(S) = \underset{i,\text{upper-}!}{\lim} \, \IndCoh^*(S_i) =
\underset{i,\text{upper-}!}{\lim} \, 
\underset{\text{upper-}*}{\Tot_{semi}} \, 
\IndCoh^*(S_i \times_S T^{{\times}_S \dot+1}) 
\overset{Prop. \ref{p:flat-desc-sch}}{=} \\
\underset{\text{upper-}*}{\Tot_{semi}} 
\underset{i,\text{upper-}!}{\lim} \, 
\IndCoh^*(S_i \times_S T^{{\times}_S \dot+1}) = 
\underset{\text{upper-}*}{\Tot_{semi}} \, 
\IndCoh^*(T^{{\times}_S \dot+1})
\end{gathered}
\]

\noindent as desired, where we have used 
Lemma \ref{l:proper} \eqref{i:proper-3} to commute
the limits. 

\step Next, suppose $S$ is a general weakly renormalizable  
prestack and $T$ is a reasonable indscheme.

We denote the functor under consideration by:

\[
F:\IndCoh^*(S) \coloneqq
\underset{U \to S \text{ flat}}{\underset{U \in \IndSch_{reas}}{\lim}}
\IndCoh^*(U)
\to \Tot_{semi}(\IndCoh^*(T^{\times_S \dot+1}))
\]

\noindent We will show $F$ is an equivalence by explicitly
constructing an inverse functor $G$.

Namely, we have a functor (induced by
$*$-pullback):

\[
\Tot_{semi}(\IndCoh^*(T^{\times_S \dot+1})) \to 
\Tot_{semi}
\underset{U \to S \text{ flat}}{\underset{U \in \IndSch_{reas}}{\lim}} 
\IndCoh^*(U \times_S T^{\times_S \dot+1})
\]

\noindent Exchanging the order of limits on the
right hand side and noting that
$U \times_S T^{\times_S \dot+1}$ is the Cech nerve
of the flat cover $U\times_S T \to U \in \IndSch_{reas}$,
the previous step implies that the right hand side
is canonically isomorphic to:

\[
\underset{U \to S \text{ flat}}{\underset{U \in \IndSch_{reas}}{\lim}} 
\IndCoh^*(U) \eqqcolon \IndCoh^*(S).
\]

\noindent Therefore, we obtain our functor 
$G:\Tot_{semi}(\IndCoh^*(T^{\times_S \dot+1})) \to 
\IndCoh^*(S)$. 

To verify that $G$ and $F$ are inverses,
it suffices to show $GF \simeq \id$ and
$FG \simeq \id$. We construct such 
isomorphisms by straightforward means below.

First, note that for $U \in \IndSch_{reas}$ equipped with a
flat map to $S$, we have a projection morphism of augmented
simplicial prestacks:

\[
U \underset{S}{\times} T^{\times_S \dot +1} \to 
T^{\times_S \dot+1}.
\] 

\noindent This is functorial in $U$, so passing
to the limits and using the augmentation
to obtain the horizontal arrows, we
get the commutative diagram:

\[
\xymatrix{
\IndCoh^*(S) \ar[r]^F \ar@{=}[d] &
\Tot_{semi}\IndCoh^*(T^{\times_S \dot+1}) \ar[d] \\
\underset{U \to S \text{ flat}}{\underset{U \in \IndSch_{reas}}{\lim}} \IndCoh^*(U) \ar[r]^(.4){\simeq} & 
\Tot_{semi}
\underset{U \to S \text{ flat}}{\underset{U \in \IndSch_{reas}}{\lim}} 
\IndCoh^*(U \times_S T^{\times_S \dot+1}) 
}
\]

\noindent By definition, $G$ is the composition of
the right horizontal arrow and the inverse to the
bottom arrow. The commutativity of this diagram
therefore gives $GF \simeq \id$.

To construct an isomorphism $FG \simeq \id$, it suffices
to do so after further composition with the functor:

\[
\eta:
\Tot_{semi}\IndCoh^*(T^{\times_S \dot+1})
\isom 
\Tot_{semi} \Tot_{semi} \IndCoh^*(T^{\times_S \dot+\dot+2}).
\]

\noindent Note that $\eta$ is an isomorphism because 
$T$ is a reasonable indscheme. 

The target of $\eta$ is the
double totalization of the bi-semi-cosimplicial
object obtained from $\IndCoh^*(T^{\times_S \dot+1})$
by restricting along the join (alias: concatenation) map
$\on{join}:\bDelta_{inj} \times \bDelta_{inj} \to \bDelta_{inj}$.
Moreover, by construction, the functor $\eta FG$ 
is the natural map in such
a situation (from the limit of a functor
to the limit of its restriction to another category).

Let $p_1:\bDelta_{inj} \times \bDelta_{inj} \to \bDelta_{inj}$
be the first projection. There is an evident natural
transformation $p_1 \to \on{join}$ inducing a commutative
diagram:

\[
\xymatrix{
\Tot_{semi}\IndCoh^*(T^{\times_S \dot+1}) \ar[d]^{\simeq}
\ar[dr]
 & \\
\Tot_{semi} \Tot_{semi} \IndCoh^*(T^{\times_S \dot+1}) 
\ar[r] &
\Tot_{semi} \Tot_{semi} \IndCoh^*(T^{\times_S \dot+\dot+2}).
}
\]

\noindent The diagonal arrow is $\eta FG$ by the above
discussion, while the left and bottom arrows compose
to give $\eta$. This gives the claim. 

\step Finally, we treat the general case in which
$S$ and $T$ are both weakly renormalizable prestacks.

By assumption on $S$, there exists $S^{\prime} \in \IndSch_{reas}$
and $S^{\prime} \to S$ a flat cover. We then obtain
a commutative diagram:

\[
\xymatrix{
\IndCoh^*(S) \ar[r] \ar[d] &
\Tot_{semi} \IndCoh^*(T^{\times_S \dot+1}) \ar[d]
\\
\Tot_{semi} \IndCoh^*(S^{\prime,\times_S \dot+1}) \ar[r] & 
\Tot_{semi} \Tot_{semi} 
\IndCoh^*(S^{\prime,\times_S \dot+1} \times_S T^{\times_S \dot +1})
}
\]

\noindent The left, bottom and right arrows are isomorphisms
by the previous step, so the top arrow is as well.

\end{proof}

\begin{cor}\label{c:indcoh-naive}

Let $S = T/H$ for $H$ a classical affine group scheme 
acting on $T \in \IndSch_{reas}$. 

Then the functor:

\[
\IndCoh^*(S) \to \IndCoh^*(T)^{H,w,naive}
\]

\noindent is an equivalence.

\end{cor}

\begin{proof}

Clear from the Theorem \ref{t:flat-desc} 
(e.g., using Proposition \ref{p:placid-strict}
and Proposition \ref{p:strict-basics} \eqref{i:strict-2}
to convert $\IndCoh^*$ on the relevant products to
tensor products).

\end{proof}

\subsection{$t$-structures}

Let $S \in \PreStk_{w.ren}$ be given. Then 
$\IndCoh^*(S)$ has a unique $t$-structure
such that for every $U \in \IndSch_{reas}$ and flat $U \to S$,
the pullback functor $\IndCoh^*(S) \to \IndCoh^*(U)$ is
$t$-exact.

Indeed, by definition, we have:

\[
\IndCoh^*(S) = 
\underset{U \to S \text{ flat}}{\underset{U \in \IndSch_{reas}}{\lim}} 
\IndCoh^*(U)
\]

\noindent and all of the structural functors are $t$-exact
by Corollary \ref{c:flat-t-exact}.

\subsection{Coherence}

Next, for $S \in \PreStk_{w.ren}$, we define
$\Coh(S) \subset \IndCoh^*(S)$ as:

\[
\Coh(S) = 
\underset{U \to S \text{ flat}}{\underset{U \in \IndSch_{reas}}{\lim}} 
\Coh(U).
\]

\noindent Clearly $*$-pullback along locally flat 
maps preserve $\Coh$.

\subsection{Renormalizable prestacks}

\begin{defin}

$S \in \PreStk_{conv}$ is \emph{renormalizable} if there
exists a flat cover $f:T \to S$ with 
$T \in \IndSch_{laft}$ an indscheme
locally almost of finite type.

For $S$ renormalizable, we let $\IndCoh_{ren}^*(S)$
denote $\Ind(\Coh(S))$.

\end{defin}

\begin{rem}

One might prefer a definition in greater generality 
(e.g., without finiteness hypotheses on $T$). However,
this definition suffices for our applications,
and this finiteness hypothesis simplifies the theory
(essentially by Lemma \ref{l:ren-coh-trun} below).

\end{rem}

\begin{example}\label{e:glob-quot}

By Lemma \ref{l:coh-local} and Theorem \ref{t:flat-desc},
$\sF \in \IndCoh^*(S)$ is coherent if and only if its
$*$-pullback to some flat cover is so.
In particular, $S = T/H$ for $T \in \IndSch_{laft}$ and $H$ a 
classical affine group scheme acting on $T$, then
$S$ is renormalizable with 
$\IndCoh_{ren}^*(S) = \IndCoh^*(T)^{H,w}$.

\end{example}

\begin{lem}\label{l:ren-coh-trun}

For $S$ a renormalizable prestack, coherent objects in 
$\IndCoh^*(S)$ are closed under truncations.

\end{lem}

\begin{proof}

By Theorem \ref{t:flat-desc} and the definition, 
this reduces to the case of
indschemes locally almost of finite type where it is clear.

\end{proof}

\begin{prop}\label{p:ren-t-str}

Let $S$ be a renormalizable prestack. Define a 
$t$-structure on 
$\IndCoh_{ren}^*(S)$ by taking connective objects to be generated
under colimits by $\Coh(S) \cap \IndCoh^*(S)^{\leq 0}$.

Then then canonical functor $\IndCoh_{ren}^*(S) \to \IndCoh^*(S)$
is $t$-exact and induces an equivalence on eventually coconnective
subcatgories. 

\end{prop}

\begin{proof}

Lemma \ref{l:ren-coh-trun} implies that 
$\Coh(S) \into \IndCoh_{ren}^*(S)$ is closed under
truncations for this $t$-structure. 
This clearly implies that $\IndCoh_{ren}^*(S) \to \IndCoh^*(S)$
is $t$-exact. 

Next, observe that if $\sF \in \Coh(S)$, then
$\sF$ is compact in $\IndCoh^*(S)^{\geq -N}$ for all $N \gg 0$.
Indeed, this follows from Step \ref{st:acpt-pres} from
the proof of Lemma \ref{l:coh-t-str}.
Combined with the fact that compact objects in $\IndCoh_{ren}^*(S)$
are closed under truncations, this 
implies that $\IndCoh_{ren}^*(S)^+ \to \IndCoh^*(S)^+$
is fully-faithful.

Finally, an argument as in Lemma \ref{l:gen-indscheme} 
shows that the functor is essentially surjective. 

\end{proof}

\subsection{}  

Let $f:S_1 \to S_2 \in \PreStk_{ren}$ 
be a reasonable indschematic morphism. 
Then $f_*^{\IndCoh}:\IndCoh^*(S_1) \to \IndCoh^*(S_2)$ is
left $t$-exact. Therefore, by 
Lemma \ref{l:ren-coh-trun} and Proposition \ref{p:ren-t-str}, 
there exists a unique left $t$-exact functor 
$f_*^{\IndCoh_{ren}}:\IndCoh_{ren}^*(S_1) \to \IndCoh_{ren}^*(S_2)$
fitting into a commutative diagram:

\[
\xymatrix{
\IndCoh_{ren}^*(S_1) \ar[r]^{f_*^{\IndCoh_{ren}}} \ar[d] & \IndCoh_{ren}^*(S_2)
\ar[d] \\
\IndCoh^*(S_1) \ar[r]^{f_*^{\IndCoh}}  & \IndCoh^*(S_2)
}
\]

\noindent (with vertical arrows the canonical functors).

Similarly, if $f$ is locally flat, then $f^{*,\IndCoh}$ is $t$-exact,
so there is a unique functor 
$f^{*,\IndCoh_{ren}}:\IndCoh_{ren}^*(S_2) \to \IndCoh_{ren}^*(S_1)$
fitting into an analogous diagram to the above.

Observe that renoralizable prestacks are closed under
fiber squares with one leg locally flat and the other
leg indschematic locally almost of finite type.
As in \cite{indcoh} Proposition 3.2.4, there is a unique functor:

\[
\IndCoh_{ren}^*:
\on{Corr}(\PreStk_{ren})_{indsch.lafp;loc.flat} \to \DGCat_{cont}
\]

\noindent equipped with a natural transformation to 
the functor
$\IndCoh^*:\on{Corr}(\PreStk_{ren})_{indsch.lafp;loc.flat} \to \DGCat_{cont}$
(obtained by restriction from the functor in \S \ref{ss:indcoh-wren})
that on every $S \in \PreStk_{ren}$ evaluates to the
canonical functor $\IndCoh_{ren}^*(S) \to \IndCoh^*(S)$.
Here \emph{indsch.lafp} is shorthand for \emph{indschematic
locally almost of finite presentation}.\footnote{We remark that 
such a morphism is in particular reasonable indschematic.}

\subsection{Symmetric monoidal structures}\label{ss:indcoh-lax}

Let $\sC \subset \PreStk_{conv}$ denote any one
of the subcategories:

\[
\presup{>-\infty}{\Sch}_{qcqs} \subset 
\Sch_{reas} \subset 
\IndSch_{reas} \subset 
\PreStk_{w.ren}
\]

\noindent or take $\sC = \PreStk_{ren}$.

\noindent It is direct from the various definitions 
that $\sC$ is closed under pairwise Cartesian products in $\PreStk_{conv}$.
In particular, the various correspondence categories
we have considered admit symmetric monoidal structures
with monoidal product given objectwise by Cartesian product:
see \cite{grbook} \S V.3.

We will upgrade the various versions of $\IndCoh^*$ considered
so far to have lax symmetric monoidal structures.
This construction is straightforward following \cite{grbook};
we indicate the logic below.

\subsection{}

First, observe that the functor:

\[
(\Psi:\IndCoh^* \to \QCoh):\presup{>-\infty}{\Sch}_{qcqs} \to 
\TwoHom(\Delta^1,\DGCat_{cont}) 
\] 

\noindent admits a unique lax symmetric monoidal structure upgrading
the standard symmetric monoidal structure on  
$\QCoh:\presup{>-\infty}{\Sch}_{qcqs} \to \DGCat_{cont}$.
Indeed, this follows by the same method as in \cite{indcoh} Proposition 3.2.4.

In particular, for $S,T \in \presup{>-\infty}{\Sch}_{qcqs}$,
there is an exterior product functor:

\[
-\boxtimes-:\IndCoh^*(S) \otimes \IndCoh^*(T) \to \IndCoh^*(S \times T).
\]

\noindent This functor is uniquely characterized by
the fact that it maps $\Coh(S) \times \Coh(T) \to \Coh(S \times T)$
and is compatible with $\Psi$ and exterior product of coherent sheaves.

\begin{rem}

We will discuss when the above functor is an equivalence
in \S \ref{ss:strict}.

\end{rem}

\subsection{}

The above lax symmetric monoidal structure canonically extends
to one on the functor 
$\IndCoh^*:\on{Corr}(\presup{>-\infty}{\Sch}_{qcqs})_{all;flat}
\to \DGCat_{cont}$
using \cite{grbook} Theorem V.1.3.2.2 (and the definition of
the symmetric monoidal structure on correspondences
from \cite{grbook} \S V.3.2.1).

\subsection{}

Similar logic applies for reasonable indschemes:
by Kan extension, $\IndCoh^*:\IndSch_{reas} \to \DGCat_{cont}$
has a canonical lax symmetric monoidal structure,
and this extends canonically to a lax symmetric monoidal
structure on the functor 
$\on{Corr}(\IndSch_{reas})_{all;flat} \to \DGCat_{cont}$.

\subsection{} 

Next, we apply \cite{grbook} Proposition V.3.3.2.4 to obtain
a lax monoidal structure on 
$\IndCoh^*:\on{Corr}(\PreStk_{w.ren})_{reas.indsch;loc.flat} \to \DGCat_{cont}$.

Finally, by a similar argument as for eventually coconnective
schemes (i.e., using $t$-structures), the functor
$\IndCoh_{ren}^*:
\on{Corr}(\PreStk_{ren})_{indsch.lafp;loc.flat} \to \DGCat_{cont}$
admits a canonical lax symmetric monoidal structure
(characterized by compatibility with the one on
$\IndCoh^*$ and the natural transformation
$\IndCoh_{ren}^* \to \IndCoh^*$).

\subsection{Strictness}\label{ss:strict}

We now study when our lax symmetric monoidal
functors behave as honest symmetric monoidal functors.

\begin{defin}

A weakly renormalizable prestack $S$ is \emph{strict}
if for every $T \in \presup{>-\infty}{\Sch}_{qcqs}$,
the functor:

\begin{equation}\label{eq:boxtimes}
-\boxtimes-:\IndCoh^*(S) \otimes \IndCoh^*(T) \to \IndCoh^*(S \times T)
\end{equation}

\noindent is an equivalence.

\end{defin}

Before giving examples, we record some basic properties
about this notion.

\begin{prop}\label{p:strict-basics}

\begin{enumerate}

\item\label{i:strict-1} 

Suppose $S \in \PreStk_{w.ren}$ is strict.
Then for every $T \in \IndSch_{reas}$, the 
natural functor:

\[
-\boxtimes-:\IndCoh^*(S) \otimes \IndCoh^*(T) \to \IndCoh^*(S \times T)
\]

\noindent is an equivalence.

\item\label{i:strict-2}

Suppose $S \in \IndSch_{reas}$ is strict.
Then for every $T \in \PreStk_{w.ren}$, the 
natural functor:

\[
-\boxtimes-:\IndCoh^*(S) \otimes \IndCoh^*(T) \to 
\IndCoh^*(S \times T)
\]

\noindent is an equivalence.

\item\label{i:strict-3} 

Suppose $S \in \IndSch_{reas}$ is a filtered colimit
under almost finitely presented closed embeddings
$S = \colim_i S_i$ with $S_i \in \presup{>-\infty}{\Sch}_{qcqs}$
strict. Then $S$ is strict.

\item\label{i:strict-4} 

Suppose $S \in \PreStk_{w.ren}$ admits a flat cover
by a strict reasonable indscheme. Then $S$ is strict.

\end{enumerate} 

\end{prop}

\begin{proof}

\eqref{i:strict-1} (resp. \eqref{i:strict-3}) 
is immediate from the presentation
of $\IndCoh^*(T)$ (resp. $\IndCoh^*(S)$) 
as a colimit (using that $T$, resp. $S$, is assumed to
be in $\IndSch_{reas}$).
Then \eqref{i:strict-2} is similarly formal, noting
that we can commute the relevant tensor product and limit
because $\IndCoh^*(S)$ is compactly generated (hence dualizable)
for $S \in \IndSch_{reas}$. 
The same applies in \eqref{i:strict-4}: we have to 
commute a limit with a tensor product against $\IndCoh^*(T)$,
which we can do because the relevant $T$ here is assumed to be
in $\presup{>-\infty}{\Sch}_{qcqs}$ so $\IndCoh^*(T)$
is compactly generated.

\end{proof}

\begin{warning}

In some contexts, people speak use a more general
notion of indscheme in which transition maps are not
required to be closed embeddings, and use the
term \emph{strict indscheme} to refer to the (more standard)
notion of indscheme we have used. This terminology has no 
relationship to the above notion of strict indscheme(/prestack);
we hope our use of this terminology in the present
context does not create confusion.

\end{warning}

\subsection{}

We now give some examples of strictness.
Note that Proposition \ref{p:strict-basics} \eqref{i:strict-3}
and \eqref{i:strict-4} reduce us to constructing
examples of strict $S$ with $S \in \presup{>-\infty}{\Sch}_{qcqs}$.

\begin{lem}\label{l:tens-ff}

For every $S,T \in \presup{>-\infty}{\Sch}_{qcqs}$,
the functor \eqref{eq:boxtimes} is fully-faithful.

\end{lem}

\begin{proof}

The same argument as in \cite{indcoh} Proposition 4.6.2 
applies. 

\end{proof}

\begin{prop}\label{p:ft-strict}

If $S \in \presup{>-\infty}{\Sch}_{qcqs}$ is almost finite type
(over $\Spec(k)$\footnote{As in \cite{indcoh}, it is important
here that $k$ have characteristic $0$ or be a perfect field of
characteristic $p$.}), then $S$ is strict.

\end{prop}

\begin{proof}

Let $T \in \presup{>-\infty}{\Sch}_{qcqs}$ be given.
By Lemma \ref{l:tens-ff}, we need to show \eqref{eq:boxtimes} is essentially surjective.
By Zariski descent, we reduce to the case where $S$ is separated.

We have a standard convolution functor:

\begin{equation}\label{eq:conv-indcoh}
\begin{gathered}
-\star-:
\IndCoh(S \times S) \otimes \IndCoh^*(S \times T) \to 
\IndCoh^*(S \times T) \\
\big(\sK \in \IndCoh(S \times S), \sF \in \IndCoh^*(S \times T)\big) 
\mapsto 
p_{23,*}^{\IndCoh}(\id_S \times \Delta_S 
\times \id_T)^!(\sK \boxtimes \sF)
\end{gathered}
\end{equation}

\noindent (where e.g., $\Delta_S:S \to S \times S$ is the
diagonal and $p_{23}:S \times S \times T \to S \times T$
is the projection onto the last two coordinates).

By Lemma \ref{l:proper} \eqref{i:proper-2}, we have:

\[
\Delta_{S,*}^{\IndCoh}(\omega_S)\star \sF = \sF
\]

\noindent for any $\sF \in \IndCoh^*(S \times T)$.
In particular, \eqref{eq:conv-indcoh} is essentially surjective.

Now note the composition:

\begin{equation}\label{eq:tens-conv}
\begin{gathered}
\IndCoh(S) \otimes \IndCoh(S) \otimes \IndCoh^*(S \times T) 
\to \\
\IndCoh(S \times S) \otimes \IndCoh^*(S \times T) \xar{\eqref{eq:conv-indcoh}} 
\IndCoh^*(S \times T)
\end{gathered}
\end{equation}

\noindent factors through the subcategory 
$\IndCoh(S) \otimes \IndCoh^*(T)$.
Now the first functor in \eqref{eq:tens-conv}
is an equivalence by \cite{indcoh} Proposition 4.6.2,
so we are done.

\end{proof}

Before proceeding, we need the following auxiliary result.

\begin{lem}\label{l:placid-pres}

Suppose $S \in \presup{>-\infty}{\Sch}_{qcqs}$ is
written as a filtered inverse limit
$S = \lim_i S_i$ under flat affine structure maps
with $S_i \in \presup{>-\infty}{\Sch}_{qcqs}$.
Then $*$-pullback induces an equivalence:

\[
\underset{i,\text{upper-}*}{\colim} \, 
\IndCoh^*(S_i) \isom \IndCoh^*(S) \in \DGCat_{cont}.
\]

\end{lem}

\begin{proof}

First, note that the functor:

\begin{equation}\label{eq:qcoh-colim}
\underset{i}{\colim} \, \QCoh(S_i) \to \QCoh(S) \in \DGCat_{cont}
\end{equation}

\noindent is an equivalence. Indeed, both sides
are monadic over $\QCoh(S_{i_0})$ for any fixed index
$i_0$, and the induced map on monads is an isomorphism 
(here we are not using the flatness assumption).

Therefore, for any index $i_0$ and any
$\sF \in \Perf(S_{i_0})$, $\sG \in \QCoh(S_{i_0})$,
the natural map:

\begin{equation}\label{eq:colim-perf-hom}
\underset{i \to i_0} {\colim} \, 
\Hom_{\QCoh(S_i)}(\alpha_{i_0,i}^*(\sF),\alpha_{i_0,i}^*(\sG)) \to 
\Hom_{\QCoh(S)}(\alpha_{i_0}^*(\sF),\alpha_{i_0}^*(\sG)) \in \Gpd
\end{equation}

\noindent is an isomorphism; here $\alpha_{i_0,i}:S_i \to S_{i_0}$
and $\alpha_{i_0}:S \to S_{i_0}$ denote the structural maps.

We now claim that \eqref{eq:colim-perf-hom} is an isomorphism  
$\sG \in \QCoh(S_{i_0})^+$ and $\sF \in \Coh(S_{i_0})$.
Indeed, if $\sF,\sG \in \QCoh(S_{i_0})^{\geq -N}$,
we choose $\sF^{\prime} \in \Perf(S_{i_0})$ equipped
with an isomorphism $\tau^{\geq -N} \sF^{\prime} \isom \sF$.
By flatness, we have:

\[
\begin{gathered}
\alpha_{i_0,i}^*(\sG) \in \QCoh(S_i)^{\geq -N} \\
\tau^{\geq -N}\alpha_{i_0,i}^*(\sF^{\prime}) \isom  
\alpha_{i_0,i}^*(\sF)
\end{gathered}
\] 

\noindent and similarly for $\alpha_{i_0}$. Therefore,
the sides of \eqref{eq:colim-perf-hom} are unchanged
under replacing $\sF$ by $\sF^{\prime}$, so we are reduced
to that case.

In particular, \eqref{eq:colim-perf-hom} is an isomorphism
when $\sF,\sG \in \Coh(S_{i_0})$.
Unwinding the above logic, it follows that the natural functor:

\begin{equation}\label{eq:indcoh-placid-desc}
\underset{i}{\colim} \, 
\IndCoh^*(S_i) \isom \IndCoh^*(S) \in \DGCat_{cont}
\end{equation}

\noindent is fully-faithful. To show that this functor
is an equivalence, 
it suffices to show any $\sF \in \Coh(S)$ lies in the essential
image.

Suppose $\sF$ lies in cohomological
degrees $\geq -N$, and choose $\sF^{\prime} \in \Perf(S)$
with $\tau^{\geq -N}(\sF^{\prime}) \isom \sF$.
Observe that we also have:

\[
\tau^{\geq -N}\sF^{\prime} \isom \sF \in \IndCoh^*(S)
\]

\noindent \emph{where the 
truncation is for the $t$-structure on $\IndCoh^*$}
(and we are using the embeddings 
$\Perf(S) \subset \Coh(S) \subset\IndCoh^*(S)$).
Indeed, both sides are bounded from below, so it suffices
to check this after applying the $t$-exact functor $\Psi$;
then the corresponding isomorphism was a defining
property of $\sF^{\prime}$.

By \eqref{eq:qcoh-colim}, 
there exists an index $i_0$ and some 
$\sF_{i_0}^{\prime} \in \Perf(S_{i_0})$ 
equipped with an 
isomorphism $\alpha_{i_0}^*(\sF_{i_0}^{\prime}) \isom 
\sF^{\prime}$. We then obtain:

\[
\sF = 
\tau^{\geq -N}(\sF^{\prime}) =
\tau^{\geq -N}\alpha_{i_0}^{*,\IndCoh}(\sF_{i_0}^{\prime}) =
\alpha_{i_0}^{*,\IndCoh}\tau^{\geq -N}(\sF_{i_0}^{\prime})
 \in \IndCoh^*(S)
\]

\noindent (where all truncations are for $t$-structures
on $\IndCoh^*$). This shows $\sF$ is in the 
essential image of \eqref{eq:indcoh-placid-desc} as desired.

\end{proof}

Next, we show:

\begin{prop}\label{p:placid-strict}

Suppose $S \in \presup{>-\infty}{\Sch}_{qcqs}$
can be written\footnote{This is an analogue (better suited
for $\IndCoh$) of the notion of \emph{placid} 
scheme introduced in \cite{dmod}.}
 as a filtered limit $S = \lim_i S_i$ 
 under flat affine structure maps with each $S_i$ locally
 almost of finite type. Then $S$ is strict.

\end{prop}

\begin{proof}

Let $T \in \presup{>-\infty}{\Sch}_{qcqs}$ be given.
We have a commutative diagram in $\DGCat_{cont}$:

\[
\xymatrix{
\IndCoh^*(S) \otimes \IndCoh^*(T) \ar[r] & 
\IndCoh^*(S \times T) \\
\underset{i}{\colim} \, \IndCoh^*(S_i) \otimes \IndCoh^*(T)
\ar[r] \ar[u] &
\underset{i}{\colim} \, \IndCoh^*(S_i \times T). \ar[u]
}
\]

\noindent By Lemma \ref{l:placid-pres} (applied to
$S$ and $S \times T$), the vertical arrows are isomorphisms.
By Proposition \ref{p:ft-strict}, the bottom arrow is an
isomorphism. Therefore, the top arrow is an isomorphism,
as desired.

\end{proof}

\section{Weak actions of group indschemes}\label{s:weak-ind} 

\subsection{} In this section, we study weak group actions
on categories for \emph{Tate group indschemes}. This
is a convenient class of group indschemes containing
loop groups and their relatives. 

For a \emph{compact open} subgroup $K \subset H$,
we define a weak Hecke category $\sH_{H,K}^w$, which is a certain
monoidal DG category. 
We then define $H\mod_{weak}$ in such a way that we have an equivalence:

\begin{equation}\label{eq:weak-hecke-equiv}
(-)^{K,w}:H\mod_{weak} \isom \sH_{H,K}^w\mod.
\end{equation}

\noindent The main homotopical difficulty is to give a definition
that is manifestly independent of $K$, which we do by a sort
of brute force argument.

The main new feature in the setting of (\emph{polarizable}) 
group indschemes, as opposed
to group schemes, is the presence of the
\emph{modular character}; see \S \ref{ss:pol-mod}. 
We especially draw the reader's attention to Proposition \ref{p:chi}.
The reader who is not worried about homotopical issues may essentially skip to 
\S \ref{ss:pol-mod}, taking \eqref{eq:weak-hecke-equiv} as something
like a definition.

\subsection{Tate group indschemes}

Let $H$ be a group indscheme.

\begin{defin}

A \emph{compact open subgroup} of $H$ is
a classical affine group scheme $K$ with a closed
embedding $K \into H$ that is a homomorphism,
and such that $H/K$ is an indscheme 
locally almost of finite type (equivalently, presentation). 

A \emph{Tate group indscheme} is a reasonable group indscheme admitting
a compact open subgroup.

\end{defin}

\begin{example}\label{e:loop-gp}

For $G$ an affine algebraic group, one can take 
$H = G(K)$ and $H_0 = G(O)$. More generally, a compact open
subgroup is a closed group subscheme containing 
$\Ker(G(O) \to G(O/t^N))$ for $N \gg 0$. 

\end{example}

\begin{example}\label{e:fml-loop-gp}

For $G$ an affine algebraic group and $H_0 \subset G(K)$ a compact open
subgroup, one can take $H = G(K)_{H_0}^{\wedge}$ 
to be the formal completion of $G(K)$ along $H_0$. 

\end{example}

\begin{rem}

If $H$ is an ind-affine 
Tate group indscheme, then $H$ satisfies 
the hypotheses of \S \ref{ss:naive-act}, and the
definition of $\IndCoh^*(H)$ from \emph{loc. cit}. clearly
coincides with the construction from \S \ref{s:indcoh}
in this case. We remark that the notion of naive
action of $H$ from \emph{loc. cit}. extends
to the possibly non-ind-affine setting considered here:
this means an $\IndCoh^*(H)$-module category. 

\end{rem}

Throughout this section, $H$ denotes a Tate group indscheme.
Our main objective in this section is to develop a theory
of genuine $H$-actions on categories.

\begin{rem}\label{r:tate-good-properties}

We remark that $H$ being a Tate group indscheme implies in particular
that $H$ is reasonable (as an indscheme).
Moreover, by Proposition \ref{p:placid-strict} and
Proposition \ref{p:strict-basics} \eqref{i:strict-3},
$H$ is strict (in the sense of \S \ref{ss:strict}).

\end{rem}

\subsection{Topological conventions}\label{ss:conv-tate-gp-indsch}

The reader may safely ignore this discussion at first pass.

Throughout this section, we impose a convention (implicit above)
that all quotients are understood as prestack
quotients (i.e., geometric realizations of the bar construction)
sheafified for the Zariski topology. In particular,
$\bB(-)$ indicates the Zariski sheafified classifying space.

Therefore, in the definition of compact open subgroup above,
the condition that $H/K$ be an indscheme is a 
funny condition which is satisfied if e.g. the \emph{\'etale}
sheafification of this quotient is an indscheme and
the projection from $H$ to this quotient is a $K$-torsor
locally trivial in the Zariski topology. This is the case
when $K$ is prounipotent or when $H = G(K)$ and
$K = G(O)$ (as is well-known, see e.g. \cite{hitchin}
Theorem 4.5.1).

This convention can be relaxed to the \'etale topology
at the cost of replacing the fundamental
role of quasi-compact quasi-separated
schemes in \S \ref{s:indcoh} by quasi-compact quasi-separated
algebraic spaces. That be done without serious
modification (using \cite{sag} Proposition 9.6.1.1
as a staring point), but we content ourselves with the above
restrictions since they suffice for our applications.

(We remark that changing \emph{Zariski} to \emph{\'etale} would
mean that for any $K \subset H$ with a prounipotent tail,
whenever $K$ ``should" be regarded as a compact open subgroup,
it is.) 

\subsection{}

The following terminology will be convenient in what follows.

\begin{defin}

A \emph{DG 2-category} is a category enriched
over $\DGCat_{cont}$. We let $\TwoDGCat$ denote the category
of DG 2-categories (for our purposes, it is sufficient
to view $\TwoDGCat$ as a 1-category).

\end{defin} 

\subsection{Hecke categories}

Let $H$ be a group indscheme and let $K$ be a group subscheme.

First, recall\footnote{If $H$ is classical, which is our
main case of interest (see Remark \ref{r:mea-culpa}), what follows is complete.
If $H$ is derived, \cite{grbook} \S V.3.4 covers the homotopy
coherence issues.} that the (Zariski
sheafified) quotient 
$K \backslash H/K = \bB K \times_{\bB H} \bB K$ is an algebra in
$\on{Corr}(\PreStk_{conv})_{all;all}$ with unit
and multiplication maps defined by the correspondences:

\[
\xymatrix{
& \bB K \ar[dl] \ar[dr] & & & K \backslash H \overset{K}{\times} 
H/K \ar[dl] \ar[dr] \\
\Spec(k) & & K \backslash H/K & 
K \backslash H/K \times K \backslash H/K 
& & K \backslash H/K.
}
\]

Now observe that by our assumptions on $H$ and $K$, each
of the left arrows in the above diagrams are locally flat
while each of the right arrows are locally
almost of finite presentation. 
Therefore, $K \backslash H/K$ is canonically
an algebra in 
$\on{Corr}(\PreStk_{ren})_{indsch.lafp;loc.flat}$.

\begin{defin}

The \emph{(genuine) weak Hecke category}
$\sH_{H,K}^w \in \Alg(\DGCat_{cont})$ is 
$\IndCoh(H/K)^{K,w} = \IndCoh_{ren}^*(K\backslash H/K)$, 
with monoidal
structure induced from the above algebra (under correspondences)
structure on $K\backslash H/K$ and
by applying the symmetric monoidal
functor:

\[
\IndCoh_{ren}^*:\on{Corr}(\PreStk_{ren})_{indsch.lafp;loc.flat} \to 
\DGCat_{cont}
\]

\noindent constructed in \S \ref{s:indcoh}.

\end{defin}

\begin{rem}

We do not need the whole theory of \S \ref{s:indcoh}
to construct the above monoidal structure on
$\sH_{H,K}^w$. Indeed, this follows
from Example \ref{e:hecke-direct-constr} below.

When we introduce $H\mod_{weak}$ below, we will
have $H\mod_{weak} \simeq \sH_{H,K}^w\mod$,
so this elementary construction gives a quick construction
of $H\mod_{weak}$. 

However, it is not clearly independent
of the choice of compact open subgroup $K$, and this
leads to difficulties with the functoriality in $H$.
For our purposes, the most important use of
the theory of \S \ref{s:indcoh} is to deal with 
these functoriality problems.

\end{rem}

\subsection{Genuine actions}\label{ss:gen-tate}

Let $\mathsf{TateGp}$ denote the category of Tate group indschemes.

Our first goal will be to show the following result.
In what follows, we regard $\mathsf{TateGp}$ as a symmetric monoidal category
via products, and similarly for $\TwoDGCat$.

\begin{prop-const}\label{pc:weak}

There is a canonical lax symmetric monoidal functor:

\[
\begin{gathered}
\mathsf{TateGp}^{op} \to \TwoDGCat \\
H \mapsto H\mod_{weak}
\end{gathered}
\]

\noindent with the following properties.

\begin{enumerate}

\item\label{i:pc-1} For any $H \in \mathsf{TateGp}^{op}$ and
any $K \subset H$ compact open, there is a canonical
equivalence:

\[
H\mod_{weak} \xar{(-)^{K,w}} \sH_{H,K}^w\mod \coloneqq 
\sH_{H,K}^w\mod(\DGCat_{cont}).
\]

\item\label{i:pc-2} For any morphism $f:H_1 \to H_2 \in \mathsf{TateGp}$
and any $K_i \subset H_i$ compact open subgroups
with $f(K_1) \subset K_2$, 
the functor:

\[
H_2\mod_{weak} \to H_1\mod_{weak}
\]

\noindent canonically fits into a commutative diagram:

\[
\xymatrix{
H_2\mod_{weak} \ar[r] \ar[d]^{\simeq} & H_1\mod_{weak} \ar[d]^{\simeq} \\
\sH_{H_2,K_2}^w\mod \ar[r] & \sH_{H_1,K_1}^w\mod
}
\]

\noindent where the bottom arrow is constructed using
the Hecke-bimodule $\IndCoh_{ren}^*(K_1 \backslash H_2/K_2)$.

\item\label{i:pc-3} For $I$ a finite set and
$\{H_i \in \mathsf{TateGp}\}_{i \in I}$ equipped with compact open
subgroups $K_i \subset H_i$, the functor:

\[
\prod_{i \in I} (H_i\mod_{weak}) \to 
(\prod_{i \in I} H_i\mod_{weak}) 
\]

\noindent coming from the lax symmetric monoidal structure corresponds under the
equivalences from \eqref{i:pc-1} to the functor:

\[
\begin{gathered}
\prod_{i \in I} (\sH_{H_i,K_i}^w\mod) \to  
\sH_{\prod_{i \in I} H_i, \prod_{i \in I} K_i}^w\mod \simeq 
(\underset{i}{\otimes} \, \sH_{H_i,K_i}^w)\mod \\
\{\sC_i \in \sH_{H_i,K_i}^w\mod\}_{i \in I} \mapsto
\underset{i}{\otimes} \, \sC_i.
\end{gathered}
\] 

\end{enumerate}

\end{prop-const}

\begin{rem}\label{r:mea-culpa}

Let $\mathsf{TateGp}^{cl} \subset \mathsf{TateGp}$ denote 
the subcategory of Tate group indschemes that are classical as prestacks.\footnote{As
in \cite{indschemes} Theorem 9.3.4, any formally smooth Tate group indscheme
that is weakly $\aleph_0$ in the sense of \emph{loc. cit}. is
automatically classical. In particular, this applies for a loop group,
or for its formal completion at any compact open subgroup scheme.}
Note that $\mathsf{TateGp}^{cl}$ is actually a $(1,1)$-category.

To simplify the exposition, we actually 
only give the restriction of the functor from 
Proposition-Construction \ref{pc:weak} to $\mathsf{TateGp}^{cl}$;
this suffices for our applications.
The requisite homotopy coherence needed to provide all of
Proposition-Construction \ref{pc:weak} can be given using 
\cite{grbook} \S V.3.4. But the argument (at least in the form
the author has in mind) is more tedious than merits inclusion (without
real applications) here.

\end{rem}

We will give the construction after some preliminary remarks.

\subsection{Morita theory}\label{ss:morita-start}

We will need a review of Morita categories in a higher categorical context.

\subsection{}\label{ss:morita-colims}

First, let $\sC$ be a symmetric monoidal category with colimits 
and whose monoidal product commutes with colimits in each variable.

Let $\sC\mod$ denote the 2-category of $\sC$-module categories $\sM$
with colimits and for which for every $\sF \in \sC$ and $\sG \in \sM$,
the action functors $\sF \star -:\sM \to \sM$ and
$-\star \sG :\sC \to \sM$ 
admit right adjoints\footnote{In the language of \cite{higheralgebra}
\S 4.2.1, $\sM$ is cotensored and enriched over $\sC$.}
 (so in particular, the action functor 
commutes with colimits in each variable). 

In this case, let $\Mor(\sC)$ be the 2-category defined as the
full subcategory of $\sC\mod$ consisting of objects of the
form $A\mod$ for $A \in \Alg(\sC)$. As in \cite{higheralgebra}
Remark 4.8.4.9, this recovers the standard 
Morita 2-category if $\sC$ is a $(1,1)$-category. 

In particular, we remind from \emph{loc. cit}. that if we have
$\sM_i \in \Mor(\sC)$ for $i = 1,2$ and we choose
$A_i \in \Alg(\sC)$ such that $\sM_i = A_i\mod$, then
the category of morphisms $F:\sM_1 \to \sM_2$ is canonically 
equivalent to the data of an $(A_2,A_1)$-bimodule.

\subsection{}

In what follows, we will use the (fully-faithful) Segal functor 
$\on{Seq}_{\dot}:\TwoCat \to \TwoHom(\bDelta^{op},\Cat)$;
we refer to \cite{grbook} Appendix A.1 for details.

\subsection{}

For $\sC$ a symmetric monoidal category that may not have
colimits, there is no hope for defining its Morita category.\footnote{Indeed,
composition in a Morita category involves tensor products
of bimodules, and this means certain geometric realizations
need to exist.} However, we can still define
the associated simplicial category $\on{Seq}_{\dot}(\Mor(\sC))$
as follows.

First, if $\sC$ is essentially small, embed $\sC$ into
its Yoneda category $\Yo(\sC) \coloneqq \TwoHom(\sC^{op},\Gpd)$.
Note that $\Yo(\sC)$ is equipped with a symmetric monoidal
structure commuting with colimits in each variable.
In particular, $\Mor(\Yo(\sC))$ and $\on{Seq}_{\dot}(\Mor(\sC))$
are defined.

Now for $[n] \in \bDelta^{op}$, define
$\on{Seq}_n(\sC)$ to be the full subcategory of 
$\on{Seq}_n(\Yo(\sC))$ whose objects
are sequences $\sM_0 \to \ldots \to \sM_n$
where each $\sM_i$ is of the form $A_i\mod$ for
$A_i \in \Alg(\sC) \subset \Alg(\Yo(\sC))$ and
each morphism $\sM_i \to \sM_{i+1}$ admits
a right adjoint in the 2-category $\Yo(\sC)\mod$.\footnote{By the
definition in \S \ref{ss:morita-colims}, this means that
the underlying functor admits a right adjoint that 
commutes with all colimits and
is $\Yo(\sC)$-linear.} It is straightforward to see
that this latter condition is equivalent to
$\sM_i \to \sM_{i+1}$ corresponding to an
$(A_{i+1},A_i)$-bimodule that lies in 
$\sC \subset \Yo(\sC)$.

\subsection{}

In general, we define $\on{Seq}_{\dot}(\Mor(\sC)):\bDelta^{op} \to \Cat$ 
e.g. by extending the universe so that $\sC$ is essentially small.
It is straightforward to see that we do actually obtain a simplicial
category\footnote{By definition, a category has essentially small
$\Hom$s. In principle, universe extension risks breaking this property,
and the content of the claim is that this does not happen here.}
in this way. 

Moreover, if $\sC$ does admit colimits and its monoidal
product preserves such colimits, then the simplicial
category just defined canonically coincides with the same denoted
simplicial category defined by applying the Segal construction
to the category $\Mor(\sC)$ from \S \ref{ss:morita-colims}.
Therefore, we are justified in not distinguishing the two notationally.

\begin{notation}\label{n:morita-alg-notation}

In what follows, for $\sC$ as above and $A \in \Alg(\sC)$, we 
let $[A]$ denote the induced
object in $\on{Seq}_{0}(\Mor(\sC))$.

\end{notation}

\subsection{}\label{ss:morita-finish} 

Finally, we conclude by remarking that $\Mor(-)$
is by construction functorial for (left) lax 
symmetric monoidal functors.

\subsection{Construction of the functor}

We now return to the setting of \S \ref{ss:gen-tate}.

\begin{proof}[Proof of Proposition-Construction \ref{pc:weak}]

To simplify the exposition, we ignore the symmetric
monoidal structures until the last step.

\step\label{st:tate-verdier}

First, define the category\footnote{The subscript
is an abbreviation of \emph{framed}.}
$\mathsf{TateGp}_{fr} \subset \TwoHom(\Delta^1,\mathsf{TateGp})$ 
as the subcategory consisting of maps $K \to H$ that are 
the embedding of a 
compact open subgroup in a Tate group indscheme. Note that the
forgetful functor $\Oblv_{fr}:\mathsf{TateGp}_{fr} \xar{(K \subset H) \mapsto H} \mathsf{TateGp}$
is $1$-fully-faithful (i.e., the induced maps on $\Hom$s are fully-faithful
morphisms of groupoids). 

We claim that $\Oblv_{fr}$
is a Verdier localization functor. By definition,
this means that
for every $\sC \in \Cat$, the functor:

\[
\TwoHom(\mathsf{TateGp},\sC) \to 
\TwoHom(\mathsf{TateGp}_{fr},\sC)
\]

\noindent  of restriction along
$\Oblv_{fr}$ is fully-faithful functor
with essential image consisting
of functors sending 
$\Oblv_{fr}$-local\footnote{This phrase
refers to morphisms in $\mathsf{TateGp}_{fr}$
that map to isomorphisms under $\Oblv_{fr}$.}
morphisms to isomorphisms.

Indeed, note
that the left adjoint to $\Oblv_{fr}$ is
pro-representable: 
it maps $H \in \mathsf{TateGp}$ to the pro-object
$``\lim"_{K \subset H \text{compact open}} (K \to H)$,\footnote{
The key point in verifying this formula is that
for any $f:H_1 \to H_2 \in \mathsf{TateGp}$ and
$K_2 \subset H_2$ a compact open subgroup,
the category of compact open subgroups $K_1 \subset H_1$
mapping into $K_2$ is non-empty and filtered.}
where $K \to H$ is regarded as an object
in $\mathsf{TateGp}_{fr}$ and the quotation
marks emphasize that this
(filtered) limit takes place in the relevant
pro-category. 
This pro-valued left adjoint is clearly
fully-faithful,
and this is well-known to imply the 
Verdier localization property. 

\step\label{st:tate-2}

Next, we construct a \emph{(right) lax} morphism:

\[
\on{Seq}_{\dot}(\mathsf{TateGp}_{fr}^{op}) \to
\on{Seq}_{\dot}(\Mor(\on{Corr}(\PreStk_{ren})_{indsch.lafp;loc.flat}))
\]

\noindent of simplicial categories sending $K \subset H \in \mathsf{TateGp}_{fr}$
(regarded as a $0$-simplex of the left hand side)
to $[K\backslash H/K]$
(regarded as a $0$-simplex of the right hand side as in 
Notation \ref{n:morita-alg-notation}). Here by \emph{lax}, we mean that
our morphism is merely a lax natural transformation, i.e., a functor
over $\bDelta^{op}$ between the corresponding coCartesian Grothendieck
fibrations. 

This step is where we use Remark \ref{r:mea-culpa} to ignore
homotopy coherence issues. That is,
we treat $\mathsf{TateGp}$ as a $(1,1)$-category (e.g,. by actually restricting
to $\mathsf{TateGp}^{cl}$).

Our lax functor assigns to every $[n] \in \bDelta$ and
$[n]$-shaped diagram $(K_0\subset H_0) \to \ldots \to (K_n \subset H_n)$
in $\mathsf{TateGp}_{fr}$ the $[n]$-simplex
of $\on{Seq}_{\dot}(\Mor(\on{Corr}(\PreStk_{ren})_{indsch.lafp;loc.flat}))$
defined by:

\[
K_0\backslash H_0/K_0 \actson 
K_0\backslash H_1/K_1 \actedon
K_1\backslash H_1/K_1 \actson \ldots 
K_{n-1}\backslash H_n/K_n \actedon
K_n \backslash H_n/K_n. 
\]

\noindent The notation indicates that we consider 
$K_i\backslash H_i/K_i$ as an algebra in $\on{Corr}(\PreStk_{ren})_{indsch.lafp;loc.flat}$
and $K_i\backslash H_{i+1}/K_i$ as a bimodule
in this same correspondence category
for $K_i\backslash H_i/K_i$ and $K_{i+1}\backslash H_{i+1}/K_{i+1}$.

We now need to specify where morphisms (in the relevant Grothendieck
construction) are sent. To simplify the notation, we spell out
the construction only for morphisms lying over 
the \emph{active} morphism 
$\alpha:[1] \xar{0 \mapsto 0, 1 \mapsto 2} [2] \in \bDelta$.
(The reader will readily see that this simplification really is only cosmetic.)

Then a $2$-simplex of the left hand side above corresponds to a datum
$(K_0 \subset H_0) \to (K_1 \subset H_1) \to (K_2 \subset H_2) \in 
\mathsf{TateGp}$, and a map to a $1$-simplex $(\widetilde{K}_0 \subset \widetilde{H}_0)$ 
is equivalent to maps 
and $f_i:H_{\alpha(i)} \to \widetilde{H}_i$ ($i =0,1$) sending
$K_{\alpha(i)}$ to $\widetilde{K}_{i}$. 

The relevant map in the right hand side is induced by the augmented
simplicial diagram:

\[
\begin{gathered}
K_0\backslash H_1 /K_1 \times 
K_1 \backslash H_1/K_1 \times
K_1 \backslash H_1/ K_2 \rightrightarrows
K_0\backslash H_1 /K_1 \times K_1 \backslash H_2/ K_2 \to 
\widetilde{K}_0 \backslash \widetilde{H}_1 \widetilde{K}_1 \\
\in
\on{Corr}(\PreStk_{ren})_{indsch.lafp;loc.flat}.
\end{gathered}
\]

\noindent Here all arrows 
are morphisms in the correspondence category (i.e.,
they represent correspondences),
the underlying simplicial diagram is the bar construction
for the relative tensor product
of these left and right $K_1\backslash H_1/K_1$-modules,
and the augmentation is given by the correspondence:

\[
\xymatrix{
& K_0 \backslash H_1 \overset{K_1}{\times} H_2 / K_2
 \ar[dl] \ar[dr] & \\
K_0\backslash H_1 /K_1 \times K_1 \backslash H_2/ K_2 & & 
\widetilde{K}_0 \backslash \widetilde{H}_1 / \widetilde{K}_1.
}
\]

\noindent Here the left arrow is obvious, and the right
arrow is the composition:

\[
K_0 \backslash H_1 \overset{K_1}{\times} H_2 / K_2 \to 
K_0 \backslash H_2 \overset{K_2}{\times} H_2 / K_2 \xar{\text{mult}.}
K_0 \backslash H_2/K_2 \xar{f_1} 
\widetilde{K}_0 \backslash \widetilde{H}_1 / \widetilde{K}_1.
\]

Unwinding the constructions, this was exactly the sort
of datum we needed to specify. 

\step

Applying the lax symmetric monoidal functor 
$\IndCoh_{ren}^*:\on{Corr}(\PreStk_{ren})_{indsch.lafp;loc.flat} \to 
\DGCat_{cont}$ and the functoriality from 
\S \ref{ss:morita-finish}, we obtain a lax functor:

\[
\on{Seq}_{\dot}(\mathsf{TateGp}_{fr}^{op}) \to 
\on{Seq}_{\dot}(\Mor(\DGCat_{cont})).
\]

\noindent As these are each Segal categories for actual 2-categories,
this is the same\footnote{By definition of lax functor; c.f. \cite{grbook} Appendix A.1.3.}
as a lax functor $\mathsf{TateGp}_{fr}^{op} \to \Mor(\DGCat_{cont})$.
The latter 2-category is by 
construction contained in $\DGCat_{cont}\mod$, which is
itself contained in $\TwoDGCat$.

Therefore, we obtain a lax functor of 2-categories:

\[
\mathsf{TateGp}_{fr}^{op} \to \TwoDGCat.
\]

\noindent We claim that this lax functor is an 
actual functor.

Suppose we are given
$(K_0 \subset H_0) \to (K_1 \subset H_1) \to (K_2 \subset H_2) \in 
\mathsf{TateGp}$. We obtain a diagram that commutes up to a natural 
transformation:

\[
\xymatrix{
\sH_{H_0,K_0}^w\mod \ar[dr] \ar[rr] & & \sH_{H_2,K_2}^w\mod 
\\
& \sH_{H_1,K_1}^w\mod \ar[ur] \ar@{=>}[u] &
}
\]

\noindent corresponding to the map 
of Hecke bimodules:

\[
\IndCoh_{ren}^*(K_0\backslash H_1/K_1)
\underset{\sH_{H_1,K_1}^w}{\otimes}
\IndCoh_{ren}^*(K_1\backslash H_2/K_2) \to 
\IndCoh_{ren}^*(K_0\backslash H_2/K_2).
\]

\noindent that we need to show is an isomorphism.

To verify this, note that we have a canonical
monoidal functor $\Rep(K_1) \to \sH_{H_1,K_1}^w$.
By Lemma \ref{l:gps-canon-renorm} \eqref{i:gps-4},
we have an equivalence: 

\[
\Rep(K_0) \underset{\Rep(K_1)}{\otimes}
\sH_{H_1,K_1}^w \isom 
\IndCoh_{ren}^*(K_0\backslash H_1/K_1)
\]

\noindent of $\sH_{H_1,K_1}^w$-module categories.

\noindent Therefore, we can calculate:

\[
\begin{gathered}
\IndCoh_{ren}^*(K_0\backslash H_1/K_1)
\underset{\sH_{H_1,K_1}^w}{\otimes}
\IndCoh_{ren}^*(K_1\backslash H_2/K_2) 
= \\
\Rep(K_0) \underset{\Rep(K_1)}{\otimes}
\sH_{H_1,K_1}^w \underset{\sH_{H_1,K_1}^w}{\otimes}
\IndCoh_{ren}^*(K_1\backslash H_2/K_2) = \\
\Rep(K_0) \underset{\Rep(K_1)}{\otimes} 
\IndCoh_{ren}^*(K_1\backslash H_2/K_2) \isom
\IndCoh_{ren}^*(K_0\backslash H_2/K_2)
\end{gathered}
\]

\noindent as desired.

\step Next, suppose that we are given
an $\Oblv_{fr}$-local morphism, or equivalently,
$H \in \mathsf{TateGp}$ with an embedding of compact open 
subgroups $K_1 \subset K_2 \subset H$.

Then the functor:

\[
\mathsf{TateGp}_{fr}^{op} \to \TwoDGCat
\]

\noindent sends this datum to the functor:

\[
\sH_{H,K_2}^w\mod \to \sH_{H,K_1}^w\mod
\]

\noindent defined by the bimodule
$\IndCoh_{ren}^*(K_1\backslash H/K_2)$.
By Step \ref{st:tate-verdier}, to obtain the
functor from Proposition-Construction \ref{pc:weak},
it suffices to show the above is an equivalence. 
As $K_2$ has a cofinal sequence of normal compact open
subgroups, it suffices to treat the case where $K_1$ is
normal in $K_2$.

Note that the composition: 

\begin{equation}\label{eq:hecke-comp}
\sH_{H,K_2}^w\mod \to \sH_{H,K_1}^w\mod \to \DGCat_{cont} 
\end{equation}

\noindent sends $\sD \in \sH_{H,K_2}^w\mod$ to:

\[
\Rep(K_1) \underset{\Rep(K_2)}{\otimes} \sD
\]

\noindent by the isomorphism:

\[
\sH_{H,K_2}^w \underset{\Rep(K_2)}{\otimes} 
\Rep(K_1)
 \isom 
\IndCoh_{ren}^*(K_2\backslash H/K_1)
\]

\noindent of $\sH_{H,K_2}^w$-module categories
(obtained as in the previous step from 
Lemma \ref{l:gps-canon-renorm} \eqref{i:gps-4}).
By normality of $K_1$ in $K_2$,
we may further identify:

\[
\Rep(K_1) \underset{\Rep(K_2)}{\otimes} \sD =
\Vect \underset{\Rep(K_2/K_1)}{\otimes} \sD.
\]

\noindent By Theorem \ref{t:weak}, this implies
that \eqref{eq:hecke-comp} is conservative
and commutes with limits, so is monadic.

Therefore, it suffices to show that the functor:

\begin{equation}\label{eq:hecke-comp-2}
\sH_{H,K_1} \to 
\TwoEnd_{\sH_{H,K_2}^w\mod}(\IndCoh_{ren}^*(K_2\backslash H/K_1))
\end{equation}

\noindent is an equivalence (as the right hand side
is the monad defined by \eqref{eq:hecke-comp}).
This follows by similar logic \textemdash{} the right hand side
is:

\[
\begin{gathered}
\TwoHom_{\sH_{H,K_2}^w\mod}(\sH_{H,K_2}^w \underset{\Rep(K_2)}{\otimes} 
\Rep(K_1),
\IndCoh_{ren}^*(K_2\backslash H/K_1)) = \\
\TwoHom_{\Rep(K_2)\mod}(\Rep(K_1),
\IndCoh_{ren}^*(K_2\backslash H/K_1)) = \\
\TwoHom_{\Rep(K_2/K_1)\mod}(\Vect,
\IndCoh_{ren}^*(K_2\backslash H/K_1)) = \\
\TwoHom_{\Rep(K_2/K_1)\mod}(\Vect,
\IndCoh_{ren}^*(K_1\backslash H/K_1)^{K_2/K_1,w}) 
\overset{Thm. \ref{t:weak}}{=} \\
\TwoHom_{\QCoh(K_2/K_1)\mod}(\QCoh(K_2/K_1),
\IndCoh_{ren}^*(K_1\backslash H/K_1)) =  
\sH_{H,K_1}^w
\end{gathered}
\]

\noindent as desired (it is immediate to check that this
identification is compatible with the functor
\eqref{eq:hecke-comp-2}).

\step 

Finally, we briefly remark that all of the 
above immediately upgrades to the
(lax) symmetric monoidal setting. 

In detail, note that $\mathsf{TateGp}_{fr}$ is symmetric monoidal under products.
and the functor $\mathsf{TateGp}_{fr} \to \mathsf{TateGp}$ 
is a symmetric monoidal Verdier localization.

The functor from Step \ref{st:tate-2} upgrades
to a functor of simplicial symmetric monoidal categories,
noting that by the construction of 
\S \ref{ss:morita-start}-\ref{ss:morita-finish}, 
$\Mor(\sC)$ is canonically a simplicial
symmetric monoidal category. 

This implies that the functor
$\mathsf{TateGp}_{fr}^{op} \to \DGCat_{cont}\mod$
is naturally a (left) lax symmetric monoidal 
functor, since $\IndCoh^*$ is.\footnote{To make the implicit explicit:
we are using the tensor product on
$\DGCat_{cont}\mod$ obtained by viewing
$\DGCat_{cont}$ as a commutative algebra in the
symmetric monoidal category denoted
$\sC\!\on{at}_{\infty}(\sK)^{\otimes}$ in \cite{higheralgebra}
Proposition 4.8.1.14 (so $\DGCat_{cont}\mod$
are the modules in this category).}
Finally, the forgetful functor 
$\DGCat_{cont}\mod \to \TwoDGCat$ is by construction
lax symmetric monoidal, giving the result.

\end{proof}

\subsection{Forgetful functors}\label{ss:tate-forgetful}

As in \S \ref{ss:pro-gen}, for $H$ a Tate group indscheme,
there is a canonical non-conservative functor:

\[
\begin{gathered}
\Oblv_{gen}:H\mod_{weak} \to \DGCat_{cont} \\
\sC \mapsto \underset{K \subset H \text{ compact open}}{\colim} \sC^{K,w}.
\end{gathered}
\]

\noindent We denote the colimit appearing on the right hand side
also by $\sC$ in a similar abuse of notation as in the profinite dimensional
setting. We remark that this forgetful functor 
$\Oblv_{gen}$ upgrades to a functor
to $H\mod_{weak,naive}$,. 
which we also denote by $\Oblv_{gen}$. 

As in the profinite dimensional setting, where there is
not confusion we often omit $\Oblv_{gen}$ from the notation,
i.e., we often speak of genuine $H$-actions on $\sC \in \DGCat_{cont}$ 
by which we mean that we are given an object of $H\mod_{weak}$
that maps to $\sC$ under $\Oblv_{gen}$.

\begin{lem}\label{l:tate-oblv-co/lims}

The above forgetful functor commutes with limits and colimits.

\end{lem}

\begin{proof}

First note that for every $K \subset H$ compact open, 
the functor $H\mod_{weak} \xar{\sC \mapsto \sC^{K,w}} \DGCat_{cont}$ 
commutes with limits
and colimits. Indeed, it may be calculated as the composition:

\[
H\mod_{weak} \isom \sH_{H,K}^w \xar{\Oblv} \DGCat_{cont} 
\]

\noindent and the latter functor commutes with limits and colimits
(as this is always the case for a category of modules over an algebra).

It immediately follows that our forgetful functor commutes
with colimits. Commutation with limits follows by noting
that each structural functor in the colimit admits a continuous
right adjoint given by $*$-averaging, 
so we may also calculate it as the functor:

\[
\sC \mapsto \underset{K \subset H \text{ compact open}}{\lim} \sC^{K,w}
\]

\noindent (where the structural functors in the limit are these right
adjoints).

\end{proof}

\subsection{Invariants and coinvariants}\label{ss:tate-inv-coinv}

Let $H$ be a Tate group indscheme. 

Define $\triv:\DGCat_{cont} \to H\mod_{weak}$ as the 
restriction functor along the homomorphism $H \to \Spec(k)$
(regarding the target as the trivial group).

\begin{rem}

By Proposition-Construction \ref{pc:weak} \eqref{i:pc-2},
for $K \subset H$ compact open we have:

\[
\triv(\Vect)^{K,w} = \Rep(K).
\]

\end{rem}
 
We define the functor of 
(genuine, weak) 
invariants:

\[
\begin{gathered}
H\mod_{weak} \to \DGCat_{cont} \\
\sC \mapsto \sC^{H,w}
\end{gathered}
\]

\noindent to be the right
adjoint to $\triv$, and we define (genuine, weak) 
the coinvariants functor $\sC \mapsto \sC_{H,w}$ to be
the left adjoint. These may be computed explicitly after a 
choice of compact open subgroup $K$ as:

\begin{equation}\label{eq:inv-coinv-hecke}
\begin{gathered}
\sC^{H,w} \simeq \TwoHom_{\sH_{H,K}^w\mod}(\Rep(K),\sC^{K,w}) \\
\sC_{H,w} \simeq \Rep(K) \underset{\sH_{H,K}^w}{\otimes} \sC^{K,w}.
\end{gathered}
\end{equation}

\begin{rem}

The comparison between invariants and coinvariants
is more subtle in the group indscheme setting.
than in the group scheme setting.\footnote{And
these functors behave less well. For example,
they may fail to be conservative (c.f. \cite{shvcat} Theorem 2.5.4).}

\end{rem}

\subsection{Rigid monoidal categories}\label{ss:rigid}

Before proceeding, we review some constructions with rigid monoidal DG categories,
following \cite{dgcat} \S 6. We refer to \emph{loc. cit}. for the relevant notion
of rigid monoidal DG category; we remind that this is a property for some
$\sA \in \Alg(\DGCat_{cont})$ to satisfy.\footnote{If $\sA$ is compactly generated 
and rigid, then the subcategory $\sA^c \subset \sA$ is closed under the monoidal
operation and rigid according to the more standard notion of rigid monoidal (in terms of
existence of duals).}
 
We will construct a canonical morphism $\vph_{\sA}:\sA \to \sA$ of monoidal categories
that plays a key role. 

Let $\sA^{\vee}$ be the dual of $\sA$ as an object of $\DGCat_{cont}$.
Note that $\sA^{\vee}$ is canonically $\sA$-bimodule in $\DGCat_{cont}$\
(as it is the dual of the $\sA$-bimodule
$\sA$). Therefore, we obtain a monoidal functor:

\begin{equation}\label{eq:vph-constr-1}
\sA \to \TwoEnd_{\rightmod \sA}(\sA^{\vee}) \in \Alg(\DGCat_{cont}).
\end{equation}

\noindent (The notation indicates endomorphisms as a right $\sA$-module in $\DGCat_{cont}$.) 
On the other hand, by definition of rigidity, the functor:

\[
\begin{gathered}
\sA \to \sA^{\vee} \\
\sF \mapsto (\sG \mapsto \ul{\Hom}_{\sA}(\e_{\sA},\sF \star \sG)
\end{gathered}
\]

\noindent is an equivalence of right $\sA$-module categories (here $\e_{\sA}$ is the
unit object). Therefore, the right hand side of \eqref{eq:vph-constr-1} identifies
canonically with $\sA$, and we obtain the desired morphism $\vph_{\sA}$.

\begin{rem}

By construction, there are natural isomorphisms:

\[
\ul{\Hom}_{\sA}(\e_{\sA},\sF \star \sG) \isom 
\ul{\Hom}_{\sA}(\e_{\sA},\sG \star \vph_{\sA}(\sF))
\]

\noindent for $\sF,\sG \in \sA$.

\end{rem}

For $\sM \in \sA\mod \coloneqq \sA\mod(\DGCat_{cont})$, we let
$\vph_{\sA,*}(\sM) \in \sA\mod$ denote the restriction of $\sM$ along
$\vph_{\sA}$, and we let $\vph_{\sA}^*$ denote the inverse to the equivalence 
$\vph_{\sA,*}$.

By \cite{dgcat} Corollary 6.3.3, 
for $\sM,\sN \in \sA\mod$ with $\sM$ dualizable in $\DGCat_{cont}$, 
there is a canonical equivalence:

\begin{equation}\label{eq:vph}
\TwoHom_{\sA\mod}(\sM,\sN) \isom \sM^{\vee} \underset{\sA}{\otimes} \vph_{\sA,*}(\sN)
\end{equation}

\noindent functorial in $\sM$ and $\sN$.

\subsection{Polarizations and the modular character}\label{ss:pol-mod}

We now introduce the following terminology.

\begin{defin}

A \emph{polarization} of $H \in \mathsf{TateGp}$ 
is a compact open subgroup $K \subset H$
such that $H/K$ is ind-proper. If a polarization exists,
we say that $H$ is \emph{polarizable}. 

\end{defin}

\begin{example}

The loop group of a reductive group 
is polarizable. 

\end{example}

\begin{example}\label{e:hc}

If $H$ is formal in ind-directions, i.e., its reduced
locus $H^{red} \subset H$ is a compact open subgroup,
then $H$ is polarizable (and equipped with the canonical
polarization $H^{red}$).
In this case, we say $H$ is a
group indscheme $H$ \emph{of Harish-Chandra type}.

\end{example}

We have the following result, which is evident from the
definitions (and preservation of coherent objects
under flat pullbacks and proper pushforwards):

\begin{lem}\label{l:pol-rigid}

For $K$ a polarization of $H$, the genuine Hecke category
$\sH_{H,K}^w$ is rigid monoidal (in the sense of \cite{dgcat} \S 6).

\end{lem}

\begin{cor}\label{c:mod-char}

If $H$ is polarizable, the coinvariants
functor $H\mod_{weak} \to \DGCat_{cont}$ is corepresentable.

\end{cor}

\begin{proof}

Let $K$ be a polarization of $H$. 
Then we have the equivalence $H\mod_{weak} \simeq \sH_{H,K}^w\mod$, and the
right hand side is rigid monoidal. Using the notation of \S \ref{ss:rigid}, we obtain
functorial identifications:

\[
\begin{gathered}
\sC_{H,w} \overset{\eqref{eq:inv-coinv-hecke}}{=}
\Rep(K) \underset{\sH_{H,K}^w}{\otimes} \sC^{K,w} \overset{\eqref{eq:vph}}{=} 
\TwoHom_{\sH_{H,K}^w\mod}(\Rep(K),\vph_{\sH_{H,K}^w}^*(\sC^{K,w})) = \\
\TwoHom_{\sH_{H,K}^w\mod}(\vph_{\sH_{H,K}^w,*}(\Rep(K)),\sC^{K,w})
\end{gathered}
\]

\noindent where we have repeatedly used that $\vph_{\sH_{H,K}^w,*}$ and
$\vph_{\sH_{H,K}^w}^*$ are mutually inverse equivalences.
Applying the equivalence $\sH_{H,K}^w\mod \simeq H\mod_{weak}$ now
gives the result.

\end{proof}

\begin{defin}

For $H$ a polarizable Tate group indscheme, the \emph{modular character}
$\chi_{Tate,H} = \chi_{Tate} \in H\mod_{weak}$ is the object corepresenting the functor
of coinvariants. 

\end{defin}

\subsection{}\label{ss:hmod-sym-mon}

Next, observe that because the functor of Proposition-Construction \ref{pc:weak}
is lax symmetric monoidal, $H\mod_{weak}$ is naturally symmetric monoidal
with unit $\triv(\Vect)$. We denote the tensor product by 
$\otimes$; explicitly, for $\sC_1,\sC_2 \in H\mod_{weak}$, we have
an object $\sC_1 \otimes \sC_2 \in (H \times H)\mod_{weak}$ from the
lax symmetric monoidal functoriality, and then we restrict along
the diagonal map.

\begin{lem}

If $H$ is a polarizable Tate group indscheme, the functor:

\[
H\mod_{weak} \to H\mod_{weak}
\]

\noindent given by tensoring with the modular character 
$\chi_{Tate}$ is an equivalence.

\end{lem}

\begin{proof}

Let $K \subset H$ be a polarization.
By the proof of Corollary \ref{c:mod-char} 
(and Proposition-Construction \ref{pc:weak} \eqref{i:pc-3}),
we obtain a commutative diagram:

\begin{equation}\label{eq:chi}
\vcenter{\xymatrix{
H\mod_{weak} \ar[r]^{-\otimes \chi_{Tate}} \ar[d]_{{(-)^{K,w}}} & 
H\mod_{weak} \ar[d]^{(-)^{K,w}} \\
\sH_{H,K}^w\mod \ar[r]^{\vph_{\sH_{H,K},*}} & \sH_{H,K}\mod.
}}
\end{equation}

\noindent Each of the vertical arrows and the bottom arrow are equivalences,
so the top arrow is as well. 

\end{proof}

Let $H$ be a polarizable Tate group indscheme.
By the lemma, there is a canonical object 
$\chi_{-Tate} \in H\mod_{weak}$ inverse to $\chi_{Tate}$ under
tensor product, i.e., we have:

\[
\chi_{-Tate} \otimes \chi_{Tate} = \triv(\Vect) \in H\mod_{weak}.
\]

\begin{prop}\label{p:chi}

For $H$ as above and $\sC \in H\mod_{weak}$, there is a canonical isomorphism:

\[
\sC_{H,w} \simeq 
(\sC \otimes \chi_{-Tate})^{H,w}
\]

\noindent functorial in $\sC$.

\end{prop}

\begin{proof}

We have:

\[
(\sC \otimes \chi_{-Tate})^{H,w} =
\TwoHom_{H\mod_{weak}}(\Vect,\sC \otimes \chi_{-Tate}) =
\TwoHom_{H\mod_{weak}}(\chi_{Tate},\sC) = \sC_{H,w}
\]

\noindent where the last equality was the definition
of $\chi_{Tate}$.

\end{proof}

\subsection{}

The following result gives a 
somewhat non-canonical description of $\chi_{Tate}$.

\begin{prop}\label{p:chi-restr}

Let $H$ be a polarizable Tate group indscheme.
Then for any compact open subgroup $K \subset H$,
there exists a canonical isomorphism:

\[
\Oblv_K^H(\chi_{Tate,H}) \simeq \Vect \in K\mod_{weak}
\]

\noindent for $\Oblv_K^H:H\mod_{weak} \to K\mod_{weak}$
the restriction functor and for $\Vect \in K\mod_{weak}$
regarded with the trivial action.

\end{prop}

\begin{proof}

For, note that $\Oblv_K^H$ admits a left adjoint
$\ind_K^{H,w}$ which also calculates its right adjoint.
Indeed, under the equivalences:

\[
\begin{gathered} 
H\mod_{weak} \overset{(-)^{K,w}}{\simeq} \sH_{H,K}^w\mod \\
K\mod_{weak} \overset{(-)^{K,w}}{\simeq}  \Rep(K)\mod 
\end{gathered}
\]

\noindent $\Oblv_K^H$ corresponds to restriction
along the monoidal functor $\Rep(K) \to \sH_{H,K}^w$.
This immediately gives the existence of the left adjoint,
and the fact that it also calculates the right adjoint
follows from the fact that $\sH_{H,K}^w$ is canonically self-dual
as a $\Rep(K)$-module category (which is the case because
$\sH_{H,K}$ is canonically self-dual as a 
DG category via Serre duality, 
and $\Rep(K)$ is rigid symmetric monoidal).

Now for $\sC \in K\mod_{weak}$, we obtain:

\[
\begin{gathered}
\TwoHom_{K\mod_{weak}}(\Oblv_K^H(\chi_{Tate,H}),\sC) =
\TwoHom_{H\mod_{weak}}(\chi_{Tate,H},\ind_K^{H,w}(\sC)) = \\
\ind_K^{H,w}(\sC)_{H,w} = \sC_{K,w} = \sC^{K,w} 
\end{gathered}
\]

\noindent functorially in $\sC$, giving the claim.

\end{proof}

\begin{warning}

Suppose $K_1 \subset K_2 \subset H$. Then
we obtain isomorphisms:

\[
\alpha_i:\Oblv_{K_i}^H(\chi_{Tate,H}) \simeq \Vect \in K_i\mod_{weak}, \hspace{.25cm} i = 1,2
\]

\noindent However, $\Oblv_{K_1}^{K_2}(\alpha_2) \neq \alpha_1$.
Rather, one can check that the two isomorphisms differ by tensoring
with $\det(\fk_2/\fk_1)[\dim(\fk_2/\fk_1)] \in \Rep(K_1) = 
\TwoEnd_{K_1\mod_{weak}}(\Vect)$.\footnote{This factor
arises because the proof of Proposition \ref{p:chi-restr}
(necessarily) used \emph{Serre} duality on $\IndCoh(H/K)$ to obtain
the canonical self-duality for 
$\sH_{H,K}^w$.}

\end{warning}

\section{Strong actions}\label{s:strong}

\subsection{}

In this section, we relate weak actions for a Tate group
indschemes $H$ to strong actions of $H$, as defined
in \cite{beraldo-*/!}.

\subsection{}

Let us spell out our goals more precisely.
Let $D^*(H) \in \Alg(\DGCat_{cont})$ 
be the monoidal DG category defined (with the
same notation) in \cite{dmod}. 
Let $H\mod \coloneqq D^*(H)\mod$ be the (2-)category of
categories with a \emph{strong} $H$-action.

In this section, we will construct a restriction functor:

\[
\Oblv = \Oblv^{str\to w}: H\mod \to H\mod_{weak} 
\]

\noindent compatible with forgetful functors
to $\DGCat_{cont}$ (where 
for $H\mod_{weak}$, we are considering the forgetful functor
$\Oblv_{gen}$ of \S \ref{ss:tate-forgetful}).

\subsection{}\label{ss:oblv-str-adjs}

Moreover, we will show that 
$\Oblv:H\mod \to H\mod_{weak}$ admits a left and right
adjoints that are morphisms of $\DGCat_{cont}$-module
categories, and with the following property.

For $\sC \in H\mod_{weak}$, define:

\[
\sC^{\exp(\fh),w} \coloneqq
\underset{K \subset H \text{ compact open}}{\colim} 
\sC^{H_K^{\wedge},w} \in \DGCat_{cont}
\]

\noindent under the obvious structural functors.
Here $H_K^{\wedge}$ is the formal completion 
of $H$ along $K$, which is necessarily a Tate group indscheme.
That is, we consider the restriction of
$\sC$ along the forgetful functor $H\mod_{weak} \to H_K^{\wedge}$,
apply the invariants construction of 
\S \ref{ss:tate-inv-coinv} for $H_K^{\wedge}$,
and pass to the colimit. 

Similarly, define:

\[
\sC_{\exp(\fh),w} \coloneqq
\underset{K \subset H \text{ compact open}}{\lim} 
\sC_{H_K^{\wedge},w}.
\]

As we will see, each of the structural functors
in the limit (resp. colimit) 
defining $\sC^{\exp(\fh),w}$ (resp. $\sC_{\exp(\fh),w}$)
admits a left adjoint (resp. continuous right adjoint),
so these two expressions can be expressed as limits or
colimits in $\DGCat_{cont}$.

Then we will see that the composition of
our right (resp. left) adjoint
$H\mod_{weak} \to H\mod$ with the forgetful
functor $H\mod \to \DGCat_{cont}$ sends
$\sC \in H\mod_{weak}$ to $\sC^{\exp(\fh),w}$
(resp. $\sC_{\exp(\fh),w}$).

In other words, we will show that $H$ acts strongly
on $\sC^{\exp(\fh),w}$ and
$\sC_{\exp(\fh),w}$, and that these categories satisfy
the evident universal properties with respect to these actions
and $\Oblv^{str\to w}$.

\begin{rem}\label{r:inv-coinv-w/str}

If $H$ is polarizable, then it is straightforward to deduce
from Proposition \ref{p:chi} that the above functors
$(-)^{\exp(\fh),w}, (-)_{\exp(\fh),w}:H\mod_{weak} \to \DGCat_{cont}$ 
are
equivalent up to certain twists.
We will formulate this statement precisely in 
Proposition \ref{p:inv-coinv-w/str}, where 
we will also show that this
isomorphism is strongly $H$-equivariant in a canonical way.

\end{rem}

\subsection{Strategy}\label{ss:strong-strategy}

To orient the reader in what follows, we give a brief
overview of the approach.

To give a functor $H\mod \to H\mod_{weak}$ the commutes with
colimits and is a morphism of $\DGCat_{cont}$-module categories
is equivalent to specifying an object of $H\mod_{weak}$
with a right $D^*(H)$-module structure.

In a suitable sense, this object is $D^*(H)$ considered
as weakly acted on via the left action of $H$, 
and with the evident commuting strong action of $H$ on the
right.

Implementing this strategy turns out the be somewhat involved.
It is not so difficult to define $D^*(H)$ as an
object of $H\mod_{weak}$: this is done is 
\S \ref{ss:str-ker-constr}. However, the commuting $D^*(H)$-action
takes some work, and will be given in \S \ref{ss:str-functor}.

\subsection{Warmup}

First, we discuss the case where $H$ is a classical
affine group scheme. While do not rely on this special case
in the general construction, it is illustrative of the
main ideas. 

Let $H = \lim_i H_i$ be a cofiltered 
limit of affine algebraic groups under smooth surjective
homomorphisms. As above, to construct our functor:

\[
H\mod \coloneqq D^*(H)\mod \to H\mod_{weak} = \Rep(H)\mod
\]

\noindent is suffices to construct a
$(\Rep(H),D^*(H))$-bimodule in $\DGCat_{cont}$.

This bimodule is $\fh\mod \in \DGCat_{cont}$ (c.f. Example 
\ref{e:lie}). 
We have:

\[
\fh\mod = \underset{i}{\lim} \, \fh_i\mod
\]

\noindent where each structural 
functor $\fh_i\mod \to \fh_j\mod$ takes the 
Lie algebra invariants with respect to 
$\Ker(\fh_i \onto \fh_j)$.
As is standard, $H_i$ acts strongly on $\fh_i\mod$,
and the above structural functors are equivariant
in the suitable homotopy coherent sense
for the $H_i$-action on $\fh_j\mod$ induced
by the homomorphism $H_i \to H_j$. 
Therefore, we obtain an action:

\[
D^*(H) \coloneqq \underset{i}{\lim} \, D(H_i) \actson
\underset{i}{\lim} \, \fh_i\mod = \fh\mod.
\]

Now for any map of indices $i\to j$, 
there is an action of $\Rep(H_j)$ on $\fh_i\mod$
commuting with the strong $H_j$-action:
it is given by restricting an $H_j$-representation
to $H_i$ and then tensoring with the Lie algebra representation.
Again, this is suitably homotopy coherent,
so we obtain an action:

\[
D^*(H) \otimes \Rep(H_j) \actson \fh\mod.
\]

\noindent Finally, these actions are suitably compatible
with varying $j$, so we obtain:

\[
D^*(H) \otimes \Rep(H) = 
D^*(H) \otimes \underset{j}{\colim} \, \Rep(H_j) \actson 
\fh\mod
\] 

\noindent as desired.

\subsection{A remark on naive coinvariants}

Before proceeding, it is convenient to record the following 
technical result. The reader may safely skip this material
and refer back to it as needed.

Let $S$ be a reasonable indscheme and let $\sP_K \to S$ be a $K$-
torsor for $K$ a classical affine group scheme.
By Corollary \ref{c:indcoh-naive}, naive weak $K$-invariants
in $\IndCoh^*(\sP_K)$ are given by $\IndCoh^*(S)$.
Moreover, the naive $K$-action on $\IndCoh^*(\sP_K)$
clearly canonically renormalizes, and the corresponding
category of genuine $K$-invariants is $\IndCoh^*(S)$
by Lemma \ref{l:coh-local}. This leaves the case
of naive coinvariants.

\begin{lem}\label{l:torsor-coinv}

In the above setting, the 
$\IndCoh$-pushforward functor $\IndCoh^*(\sP_K) \to \IndCoh^*(S)$
induces an equivalence:

\begin{equation}\label{eq:torsor-coinv}
\IndCoh^*(\sP_K)_{K,w,naive} \to \IndCoh^*(S). 
\end{equation}

\end{lem}

\begin{proof}

\step 

First, note that we are reduced to the case where $S$ is a
quasi-compact quasi-separated eventually coconnective scheme. 
Indeed, if $S = \colim_i S_i$
with $S_i \in \presup{>-\infty}{\Sch}_{qcqs}$ and structural
maps almost finitely presented, then: 

\[
\begin{gathered} 
\IndCoh^*(S) = \underset{i}{\colim} \, \IndCoh^*(S_i) \in \DGCat_{cont} \\
\IndCoh^*(\sP_K) = \underset{i}{\colim} \, \IndCoh^*(\sP_K \underset{S}{\times} S_i) \in \DGCat_{cont}.
\end{gathered}
\]

\noindent This clearly gives the reduction. In the remainder of the
argument, we therefore assume $S \in \presup{>-\infty}{\Sch}_{qcqs}$.

\step 

Next, suppose $\sP_K \to S$ is trivial, i.e., $\sP_K \isom K \times S$
$K$-equivariantly. By Lemma \ref{l:placid-pres} (applied to $K$), 
we have:

\[
\IndCoh^*(K) \otimes \IndCoh^*(S) \isom \IndCoh^*(\sP_K).
\]

\noindent Observe that $\Perf(K) \isom \Coh(K)$ because $K$ is a limit
of smooth schemes under flat affine morphisms. Therefore, the
above coincides with $\QCoh(K) \otimes \IndCoh^*(S)$.
This clearly gives the result in this case.

\step 

Next, we show the result when $\sP_K$ is
Zariski-locally trivial. For this, we first
establish some general facts
about $\IndCoh^*$.
 
Suppose $j:U \into S$ is a quasi-compact open subscheme. Then the natural functor:

\[
\IndCoh^*(S) \underset{\QCoh(S)}{\otimes} \QCoh(U) 
\to \IndCoh^*(U) 
\]

\noindent is an equivalence. Indeed, 
$\IndCoh^*(U)$ is the essential image of the
functor 
$j_*^{\IndCoh}j^{*,\IndCoh}:\IndCoh^*(S) \to \IndCoh^*(S)$,
while 
$\IndCoh^*(S) \otimes_{\QCoh(S)} \QCoh(U) $
is the essential image of: 

\[
\begin{gathered}
\id_{\IndCoh^*(S)} \underset{\QCoh(S)}{\otimes} j_*j^*:
\IndCoh^*(S) = 
\IndCoh^*(S) \underset{\QCoh(S)}{\otimes} \QCoh(S) 
\to \\
\IndCoh^*(S) \underset{\QCoh(S)}{\otimes} \QCoh(S) 
= \IndCoh^*(S).
\end{gathered}
\]

\noindent These endofunctors of $\IndCoh^*(S)$
coincide, giving the result.

As a consequence, suppose
$U_1,U_2 \subset S$ are quasi-compact
opens covering $S$; then we claim that the map:

\[
\IndCoh^*(U_1) 
\underset{\IndCoh^*(U_1 \cap U_2)}{\coprod}
\IndCoh^*(U_2) \to \IndCoh^*(S) \in \DGCat_{cont}
\]

\noindent is an equivalence (this pushout
being formed in $\DGCat_{cont}$).
Indeed, it is well-known\footnote{This identity
is implicit in the proof of \cite{qcoh}
Proposition 2.3.6. One can find this
statement explicitly in \cite{shvcat}
by combining Theorem 2.1.1 and
Proposition 6.2.7 from \emph{loc. cit}.}
that we have:

\[
\QCoh(U_1) 
\underset{\QCoh(U_1 \cap U_2)}{\coprod}
\QCoh(U_2) \to \QCoh(S) \in \DGCat_{cont}
\]

\noindent and therefore in $\QCoh(S)\mod$.
Tensoring $\IndCoh^*(S)$ over $\QCoh(S)$
preserves this colimit, giving the claim
from the above.

Now for any $U_1,U_2 \subset S$ as above,
we obtain:

\[
\IndCoh^*(\sP_K \underset{S}{\times} U_1)_{K,w,naive}
\underset{\IndCoh^*(
\sP_K \underset{S}{\times} U_1 \cap U_2)_{K,w,naive}}{\coprod}
\IndCoh^*(\sP_K \underset{S}{\times} U_2)_{K,w,naive}
\isom 
\IndCoh^*(\sP_K)_{K_w}
\]

\noindent by applying the above to the base-changed
Zariski cover of $\sP_K$ and by commuting
geometric realizations with pushouts.
We now clearly obtain the claim by induction on the
number of opens required to trivialize $\sP_K$.

\step Next, we show the result
for $K$ prounipotent.

By the previous step, it suffices to note
that any $K$-torsor on an affine scheme $T$ is trivial.
This is standard:
prounipotent $K$ has a lower central series 
$K = K^1 \trianglerighteq K^2 \trianglerighteq \ldots $ 
where all subquotients are (possibly infinite) products of copies of $\bG_a$. 
For such products, the claim follows from 
vanishing of higher (flat) cohomology of $T$ 
with coefficients in its structure sheaf.
By induction, any $K/K^n$-torsor on an affine scheme is 
trivial, and then we deduce the same for $K$ 
using countability of this filtration
and surjectivity of
$\pi_0(\Hom(T,K/K^{n+1})) \to 
\pi_0(\Hom(T,K/K^n))$.

\step Finally, we show the result in 
general.

Let $K \to K^{red}$ be the
proreductive\footnote{Here we use
\emph{proreductive} as shorthand for
\emph{pro-(algebraic group with reductive connected components)}.}
quotient of $K^{red}$, and let $K^u$ be the
kernel of this homomorphism, i.e., the prounipotent
radical of $K$.

Because representations of $K^{red}$ are semisimple,
for any $\sC$ with a naive weak $K^{red}$ action,
the functor $\sC_{K^{red},w,naive} \to 
\sC^{K^{red},w,naive}$ is an equivalence.
Indeed, the argument from \cite{shvcat} \S 7.2
applies just as well in the proreductive
case as in the reductive one.

We then obtain:

\[
\IndCoh^*(\sP_K)_{K,w,naive} =
(\IndCoh^*(\sP_K)_{K^u,w,naive})_{K^{red},w,naive} \isom
\IndCoh^*(\sP_{K^{red}})^{K^{red},w,naive}
\]

\noindent for $\sP_{K^{red}} \to S$ the induced
$K^{red}$-torsor (appealing to the previous step
here). We now obtain the result
by Corollary \ref{c:indcoh-naive}.

\end{proof}

\subsection{Induction}

We begin with the following general lemma.

\begin{lem}\label{l:induction}

\begin{enumerate}

Let $f:H_1 \to H_2$ be a morphism in $\mathsf{TateGp}$.

\item 

The forgetful functor:

\[
H_2\mod_{weak} \to H_1\mod_{weak}
\]

\noindent admits a left adjoint 
$\ind^w = \ind_{H_1}^{H_2,w}:H_1\mod_{weak} \to H_2\mod_{weak}$.

\item\label{i:induction-comm} Suppose there exists $K \subset H_1$ compact open
such that $f$ realizes $K$ as a compact open subgroup
of $H_2$ as well.\footnote{Using Lemma \ref{l:torsor-coinv},
one can show that the conclusion holds more generally 
if there exist $K_i \subset H_i$ compact
open subgroups such that $f$ maps
$K_1$ into $K_2$ via a closed embedding.}

Then the diagram:

\[
\xymatrix{
H_1\mod_{weak} \ar[rr]^{\Oblv_{gen}} \ar[d]^{\ind_{H_1}^{H_2,w}}
&& H_1\mod_{weak,naive} \ar[d]^{\ind_{H_1}^{H_2,w,naive}}
\\
H_2\mod_{weak} \ar[rr]^{\Oblv_{gen}} && H_2\mod_{weak,naive}
}
\]

\noindent commutes (where a priori it only 
commutes up to a natural transformation). Here
the functor on the right is tensoring over
$\IndCoh^*(H_1)$ with $\IndCoh^*(H_2)$.

\end{enumerate}

\end{lem}

\begin{proof}

As $H_i\mod_{weak} \to K_i\mod_{weak}$ are (by construction)
monadic functors, the first claim easily reduces to the setting
of \S \ref{ss:gp-sch-funct}. 

Similarly, such considerations formally reduce the
second claim to the case where $H_1 = K$. We denote $H_2$
simply by $H$ in this case. So we wish to show the diagram:

\[
\xymatrix{
K\mod_{weak} \ar[r]^{\Oblv_{gen}} \ar[d]^{\ind_{K}^{H,w}}
& K\mod_{weak,naive} \ar[d]^{\ind_{H_1}^{H_2,w,naive}}
\\
H\mod_{weak} \ar[r]^{\Oblv_{gen}} & H\mod_{weak,naive}
}
\]

\noindent commutes.
Each of the functors involved commutes
with colimits and is $\DGCat_{cont}$-linear,
so it suffices to check that the diagram commutes
when evaluated on the trivial representation 
$\Vect \in K\mod_{weak}$ (since this object generates
by definition). In this case, the claim is that the
natural map:

\[
\IndCoh^*(H) \underset{\IndCoh^*(K)}{\otimes} \Vect \isom
\IndCoh(H/K). 
\]

\noindent This follows from Lemma \ref{l:torsor-coinv}.

\end{proof}

\subsection{}\label{ss:ind-dmod}

Now for $H$ a Tate group indscheme and $K \subset H$
compact open, let $H_K^{\wedge}$ denote the formal
completion of $H$ along $K$. We can form
$\ind_{H_K^{\wedge}}^{H,w}(\Vect) \in H\mod_{weak}$
(where $\Vect \in H_K^{\wedge}\mod_{weak}$ is our standard
trivial object).
By Lemma \ref{l:induction}, we have:

\[
\Oblv_{gen}(\ind_{H_K^{\wedge}}(\Vect)) = 
\IndCoh^*(H) \underset{\IndCoh^*(H_K^{\wedge})}{\otimes} \Vect. 
\]

\noindent This tensor product evidently maps to 
$D(H/K)$, and we claim that the induced functor is an equivalence.
Indeed, we have a commutative diagram:

\[
\xymatrix{
\IndCoh(H/K) \ar@{=}[r] & \IndCoh^*(H) \underset{\IndCoh^*(K)}{\otimes} \Vect
\ar[rr] \ar[dr]
&&
\IndCoh^*(H) \underset{\IndCoh^*(H_K^{\wedge})}{\otimes} \Vect
\ar[dl] \\
& & D(H/K).
}
\]

\noindent The diagonal arrows admit continuous,
monadic right adjoints and the induced natural
transformation on monads is an isomorphism, giving the claim.

In the above setting, we use our standard abuse
of notation in letting 
$D(H/K) \in H\mod_{weak}$ denote the object
$\ind_{H_K^{\wedge}}^{H,w}(\Vect)$. 

\subsection{}\label{ss:str-ker-constr}

Note that this object is manifestly covariant
in $K$, so we can form:

\[
D^*(H) \coloneqq \underset{K}{\lim} \, D(H/K) \in H\mod_{weak}.
\]

\noindent Note that under $\Oblv_{gen}$ and the equivalence
of \S \ref{ss:ind-dmod}, these structural
functors map to de Rham pushforward functors.

By definition (and Lemma \ref{l:tate-oblv-co/lims}),
this object maps under
$\Oblv_{gen}$ to  
the category $D^*(H) \in \DGCat_{cont}$ defined
in \cite{dmod}, justifying the notation.

\begin{rem}\label{r:d*-adjs}

Each of the structural functors in the above diagram
admits a left adjoint in the 2-category $H\mod_{weak}$:
indeed, these functors are given by $D$-module 
$*$-pullback along the smooth maps
$H/K_1 \to H/K_2$.\footnote{This discussion is a bit informal,
since it really applies after applying $\Oblv_{gen}$.
But e.g., it easily follows from
Lemma \ref{l:dmod-canon} below that the left adjoints exist
in the genuine setting as well.}
Therefore, this limit is also a colimit (in $H\mod_{weak}$)
under those left adjoints.

\end{rem}

\subsection{}\label{ss:d*-coreps}

By Remark \ref{r:d*-adjs}, 
$D^*(H) \in H\mod_{weak}$ corepresents the functor:

\[
\sC \mapsto 
\underset{K \subset H \text{ compact open}}{\lim} 
\sC^{H_K^{\wedge},w} =
\underset{K \subset H \text{ compact open}}{\colim} 
\sC^{H_K^{\wedge},w} = \sC^{\exp(\fh),w}
\]

\noindent where the structural functors in the colimit
are the evident forgetful functors, and the structural
functors in the limit are their right adjoints.

\subsection{}

Below, we will construct an action of 
$D^*(H) \in \Alg(\DGCat_{cont})$ on this object 
$D^*(H) \in H\mod_{weak}$
encoding the right action of $H$ on itself. As in 
\S \ref{ss:strong-strategy}, this would suffice to construct
a functor of the desired type. 
By the discussion of \S \ref{ss:d*-coreps}, the formula
from \S \ref{ss:oblv-str-adjs}
for the right adjoint would be immediate, and the
formula for the left adjoint would follow dually.

Therefore, we will give this construction below 
following a sequence of digressions.

\subsection{Some generalities on $D$-modules}

We make the above construction somewhat more explicit.
The reader may safely skip this material and refer
back to it as needed.

First, it is convenient to extend the generality of the
above construction. 
Let $X$ be an indscheme locally almost of finite type,
and suppose $X$ is acted on by $H$.
Then we there is a canonical object
$D(X) \in H\mod_{weak}$ attached to $X$ (and mapping to the
category of $D$-modules on $X$ under $\Oblv_{gen}$).

We sketch the construction. A variant of
Proposition-Construction \ref{pc:weak} attaches
an object $\IndCoh(X) \in H\mod_{weak}$ to $X$
such that for any congruence subgroup $K$,
$\IndCoh(X)^{K,w} = \IndCoh_{ren}^*(X/K)$ as
an $\sH_{H,K}^w$-module. 

Now let $X_{\dot}^{inf} \in \IndSch_{laft}^{\bDelta^{op}}$ 
be the infinitesimal groupoid of $X$, i.e., the simplicial 
indscheme locally almost of finite type obtained as the
Cech nerve of $X \to X_{dR}$. By functoriality, 
this diagram is a simplicial diagram of indschemes (locally
almost of finite)
acted on by $H$.
Therefore, by the above construction, 
we obtain a simplicial diagram 
$\IndCoh(X_{\dot}^{inf}) \in H\mod_{weak}$. We define
$D(X) \in H\mod_{weak}$ as its colimit.
By Lemma \ref{l:tate-oblv-co/lims}
and \cite{grbook} Proposition III.3.3.3.3(b), this object
indeed maps to the usual
category of $D$-modules $D(X) \in \DGCat_{cont}$ under the 
forgetful functor $H\mod_{weak} \to \DGCat_{cont}$.

Now that for any choice of 
compact open subgroup $K \subset H$,
we have $D(X)^{K,w} \in \sH_{H,K}^w\mod$.
By construction $D(X)^{K,w}$ is compactly generated with
compact objects induced from 
$\IndCoh_{ren}^*(X/K) = \IndCoh(X)^{K,w}$.

\begin{lem}

For $K \subset H$ compact open, the object $D(H/K)$
defined in \S \ref{ss:ind-dmod} coincides with the
object we have just constructed.

\end{lem}

\begin{proof}

This is essentially a slight refinement of the argument
from \S \ref{ss:ind-dmod}.

Let $\sC_1 \in H\mod_{weak}$ denote the object from
\S \ref{ss:ind-dmod} and let $\sC_2 \in H\mod_{weak}$ 
denote the object just constructed (i.e., from the construction
defined for any indscheme locally almost of finite type). 
There is a natural map 
$\sC_1 = \ind_{H_K^{\wedge}}^{H,w}(\Vect) \to \sC_2$,
and we claim it is an isomorphism.
It suffices to check this after applying weak $K$-invariants.

By construction, we have:

\[
\sC_1^{K,w} = \Rep(K) \underset{\sH_{H_K^{\wedge},K}^w}{\otimes}
\sH_{H,K}^w.
\]

\noindent This gives rise to a functor:

\[
\IndCoh_{ren}^*(K\backslash H/K) = 
\Rep(K) \underset{\Rep(K)}{\otimes}
\sH_{H,K}^w \to 
\Rep(K) \underset{\sH_{H_K^{\wedge},K}^w}{\otimes}
\sH_{H,K}^w = \sC_1^{K,w}.
\]

\noindent This functor admits a right adjoint that is
continuous and conservative, so monadic. The further
composition with $\sC_1^{K,w} \to \sC_2^{K,w}$ behaves
similarly (by construction), and the induced maps
on monads is an isomorphism, giving the claim.

\end{proof}

We now show the following result for $X$ any indscheme
locally almost of finite type.

\begin{lem}\label{l:dmod-canon}

$D(X)^{K,w}$ admits a (unique) compactly generated
$t$-structure for which the forgetful 
functor $D(X)^{K,w} \to D(X)^{K,w,naive}$ is $t$-exact
and induces an equivalence 
$D(X)^{K,w,+} \isom D(X)^{K,w,naive,+}$.
An object in $D(X)^{K,w}$ is compact if and only if
it is eventually coconnective and its image in
$D(X)$ is compact.

\end{lem}

In other words, as a category with a genuine 
$K$-action, $D(X)$ is given by the canonical
renormalization construction from \S \ref{ss:can-renorm}.

\begin{proof}[Proof of Lemma \ref{l:dmod-canon}]

Define the $t$-structure on $D(X)^{K,w}$ by 
taking connective objects to be generated by objects induced
from $\IndCoh(X)^{K,w,\leq 0}$. 
To see $D(X)^{K,w} \to D(X)^{K,w,naive}$ is $t$-exact, 
it is equivalent to see that 
the further forgetful functor:

\[
\Oblv:D(X)^{K,w} \to D(X)
\]

\noindent is $t$-exact. Clearly this functor is right $t$-exact.
Now for $\sF \in D(X)^{K,w,\geq 0}$, the underlying
object of $\IndCoh(X)^{K,w}$ is coconnective by
design, so the same is true for the underlying
object of $\IndCoh(X)$. Therefore, $\Oblv(\sF) \in D(X)$
maps to a coconnective object of $\IndCoh(X)$; 
this is equivalent to $\Oblv(\sF)$ being coconnective,
as desired.

Because the functor $D(X)^{K,w,+} \to D(X)^+$ is $t$-exact
and these $t$-structures are right complete, this
functor is comonadic. The forgetful functor 
$D(X)^{K,w} \to D(X)^{K,w,naive}$ induces an equivalence
on the corresponding comonads on $D(X)$, and
the latter category maps comonadically to $D(X)$.
This implies $D(X)^{K,w,+} \isom D(X)^{K,w,naive,+}$.

The last part is proved similarly to 
Lemma \ref{l:gps-canon-renorm} \eqref{i:gps-4}.
We need to show that
if $\sF \in D(X)^{K,w,+} = D(X)^{K,w,naive,+}$ 
has $\Oblv(\sF) \in D(X)$
compact, then $\sF$ is compact in $D(X)^{K,w}$. 
We are clearly reduced to the case where $X$ is classical.
In this case, $X$ is a colimit under closed embeddings
of finite type schemes acted on
by $K$, so we are further reduced to the case where $X$
is a finite type scheme. Moreover, we can assume
$\sF$ lies in the heart of the $t$-structure,
since it is bounded and each of its cohomology groups
satisfy the same hypothesis.

Now there exists $K^{\prime} \lhd K$ compact open 
(i.e., $K/K^{\prime}$ is
an affine algebraic group) with the action 
of $K$ on $X$ factoring through 
$K/K^{\prime}$. Further, as in the
proof of Lemma \ref{l:gps-canon-renorm} \eqref{i:gps-4},
the hypothesis on $\sF$ implies that there 
is a compact open subgroup $K^{\prime\prime} \subset K^{\prime}$
also normal in $K$ such that
$\sF$ lies in the essential image of the functor:

\[
D(X)^{K/K^{\prime\prime},w,naive,+} \to D(X)^{K,w,naive,+}.
\]

\noindent Now the result follows from Lemma \ref{l:recognition}
(applied to $K/K^{\prime\prime}$).

\end{proof}

\subsection{Naive Hecke actions}\label{ss:naive-hecke}

Suppose that $H$ is a Tate group indscheme and
$K \subset H$ is compact open. Suppose $H$ acts naively
on $\sC$, i.e., $\sC$ is a module for $\IndCoh^*(H)$.
Then we claim there is an induced action of  
the monoidal category:\footnote{We emphasize that the middle term
uses \emph{non-renormalized} $\IndCoh^*$.}

\[
\sH_{H,K}^{w,naive} \coloneqq \IndCoh^*(K\backslash H/K) = \IndCoh^*(H/K)^{K,w,naive}
\]

\noindent on $\sC^{K,w,naive}$.

Indeed, by Lemma \ref{l:torsor-coinv},
the $\IndCoh(H/K) \in H\mod_{weak,naive}$ corepresents
the functor of naive $K$-invariants, so we obtain:

\[
\sH_{H,K}^{w,naive} = \TwoEnd_{H\mod_{weak,naive}}(\IndCoh(H/K))
\actson \TwoHom_{H\mod_{weak,naive}}(\IndCoh(H/K),\sC).
\]

The following result is a formal consequence of 
Remark \ref{r:naive-to-gen}.

\begin{lem}\label{l:naive-hecke-ff}

The functor:

\[
H\mod_{weak,naive} \to \sH_{H,K}^{w,naive}\mod
\]

\noindent constructed above is fully-faithful.

\end{lem}

\begin{proof}

First, suppose $\sC$ is equipped with a naive
action of the compact open subgroup $K$.
Then the natural functor:

\[
\Vect \underset{\Rep_{naive}(K)}{\otimes} \sC^{K,w,naive} \to
\sC
\]

\noindent is an equivalence. Indeed, this is the
content of Remark \ref{r:naive-to-gen} (and is shown
in \cite{locsys} Proposition 3.5.1).

In particular, we obtain:

\[
\Vect \underset{\Rep_{naive}(K)}{\otimes} \sH_{H,K}^{K,w,naive}
\simeq \IndCoh(H/K).
\]

\noindent Clearly this is an equivalence of $\sH_{H,K}^{wm,naive}$-
module categories.

Now for any $\sC$ with a naive weak action of $H$, we calculate:

\[
\begin{gathered}
\IndCoh^*(H/K) \underset{\sH_{H,K}^{w,naive}}{\otimes} \sC^{K,w} = 
\Vect \underset{\Rep_{naive}(K)}{\otimes} \sH_{H,K}^{K,w,naive}
\underset{\sH_{H,K}^{w,naive}}{\otimes} \sC^{K,w} = \\
\Vect \underset{\Rep_{naive}(K)}{\otimes} \sC^{K,w}
\isom \sC.
\end{gathered}
\]

\noindent This map is obviously the counit for the
evident adjunction 
$H\mod_{weak,naive} \leftrightarrows \sH_{H,K}^{w,naive}\mod$,
so we obtain the claim.

\end{proof}
 
\subsection{Canonical renormalization}

We now wish to give an analogue of the construction from 
\S \ref{ss:can-renorm} in the Tate setting.

\subsection{}

We begin with some general results about renormalizing
monoidal structures and module structures.

\subsection{}

We will need the following general constructions in what follows.
The reader may safely skip this material and refer back
to it as necessary.

\begin{lem}\label{l:monoidal-t-str}

Suppose we are given:

\begin{itemize}

\item $(\sA,\star) \in \Alg(\DGCat_{cont})$ a monoidal DG category.

\item $\sA_{ren} \in \DGCat_{cont}$ a compactly generated
DG category with $\sA_{ren}^c$ its subcategory of compact objects.

\item $t$-structures on $\sA$ and $\sA_{ren}$ compatible with
filtered colimits.

\item A $t$-exact functor $\Psi:\sA_{ren} \to \sA$ commuting
with colimits and inducing an equivalence $\sA_{ren}^+ \isom \sA^+$
on eventually coconnnective subcategories.

\end{itemize}

Suppose in addition that the following properties are satisfied:

\begin{enumerate}

\item The unit object $\e \in \sA$ lies in $\sA^+$.

\item $\sA_{ren}^c$ is contained in $\sA_{ren}^+$.

\item\label{i:monoidal-3} For every $\sF \in \sA_{ren}^c$, 
the functors
$\Psi(\sF) \star - :\sA \to \sA$
and $- \star \Psi(\sF):\sA \to \sA$
are left $t$-exact up to shift.

\item\label{i:monoidal-4} 
For every $\sF \in \sA_{ren}^c$, the continuous functors
$\sA_{ren} \to \sA_{ren}$ defined by ind-extension
of:

\[
\begin{gathered}
\sA_{ren}^c \xar{\Psi(\sF) \star \Psi(-)} 
\sA^+ \simeq \sA_{ren}^+ \subset \sA_{ren} \\
\sA_{ren}^c \xar{\Psi(-) \star \Psi(\sF)} 
\sA^+ \simeq \sA_{ren}^+ \subset \sA_{ren}
\end{gathered}
\]

\noindent are left $t$-exact up to shift.

\end{enumerate}

Then $\sA_{ren}$ admits a unique monoidal structure such that:

\begin{itemize}

\item The functor $\Psi$ admits a monoidal structure.

\item For every $\sF \in \sA_{ren}^c$, the functors
$\sF \star - :\sA_{ren} \to \sA_{ren}$ and
$ - \star \sF: \sA_{ren} \to \sA_{ren}$
preserve $\sA_{ren}^+$.

\end{itemize}

\end{lem}

\begin{proof}

\step 

We begin with a general observation constructions.

Let us denote by 
$\TwoEnd_{\DGCat_{cont}}^{>-\infty}(\sA_{ren}) \subset
\TwoEnd_{\DGCat_{cont}}(\sA_{ren})$ 
the subcategory 
of functors $F:\sA_{ren} \to \sA_{ren}$
that are left $t$-exact up to shift.
Then the restriction functor:

\begin{equation}\label{eq:restr-ff}
\TwoEnd_{\DGCat_{cont}}^{>-\infty}(\sA_{ren}) \to 
\TwoEnd_{\DGCat}(\sA_{ren}^+) 
\end{equation}

\noindent is fully-faithful. For this, define
$\TwoEnd_{\DGCat}^{LKE}(\sA_{ren}^+) \subset
\TwoEnd_{\DGCat}(\sA_{ren}^+)$ to be the
subcategory of functors left Kan extended
from their restrictions to $\sA_{ren}^c$.
Then \eqref{eq:restr-ff} clearly maps through
this subcategory. Now the restriction
functor $\TwoEnd_{\DGCat}^{LKE}(\sA_{ren}^+) \to 
\TwoHom_{\DGCat}(\sA_{ren}^c,\sA_{ren}^+)$
is fully-faithful, and so is its composition
with \eqref{eq:restr-ff}, so \eqref{eq:restr-ff} is
fully-faithful.

We remark that the essential image of \eqref{eq:restr-ff}
consists of those DG functors $F:\sA_{ren}^+ \to \sA_{ren}^+$
that are left Kan extended from $\sA_{ren}^c$ and
such that the resulting 
ind-extended functor $\sA_{ren} \to \sA_{ren}$
is left $t$-exact up to shift.

Finally, we remark that \eqref{eq:restr-ff} is manifestly 
a monoidal DG functor (between non-cocomplete
DG categories).

\step 

Next, we define an auxiliary category.

Let $\sB^c \subset \sA$ be the full 
subcategory Karoubi generated by objects of the
form $\Psi(\sF_1) \star \ldots \star \Psi(\sF_n)$
for $\sF_1,\ldots,\sF_n \in \sA_{ren}^c$. 
(We allow $n = 0$, i.e., $\e$ is one of our generators
of $\sB$.)

Clearly $\sB^c$ is an essentially small monoidal DG category;
let $\sB \coloneqq \Ind(\sB^c) \in \Alg(\DGCat_{cont})$.

Note that $\sB^c \subset \sA^+$ by assumption.
Define a continuous DG functor $\zeta:\sB \to \sA_{ren}$
by ind-extension from:

\[
\sB^c \subset \sA^+ \simeq \sA_{ren}^+ \subset \sA_{ren}.
\]

We remark that $\zeta$ is a colocalization functor,
i.e., it admits a fully-faithful left adjoint.
Namely, this left adjoint is the ind-extension of the 
fully-faithful functor $\Psi:\sA_{ren}^c \to \sB^c \subset \sA^+$.

\step 

We now construct a $\sB$-bimodule structure
on $\sA_{ren}$ in $\DGCat_{cont}$.

Let e.g. $\sB^{mon\mathendash op}$ denote $\sB$ with
its monoidal structure reversed. So we wish to 
construct a continuous monoidal DG functor
$\sB \otimes \sB^{mon\mathendash op} \to 
\TwoEnd_{\DGCat_{cont}}(\sA_{ren})$. This is equivalent
to giving a monoidal DG functor:

\[
(\sB \otimes \sB^{mon\mathendash op})^c = 
\sB^c \ol{\otimes} \sB^{c,mon\mathendash op} \to 
\TwoEnd_{\DGCat_{cont}}(\sA_{ren}).
\]

\noindent (We remind that $\ol{\otimes}$ indicates
the tensor product on the category of small DG categories.)

Note that $\sA^+ \simeq \sA_{ren}^+$ is
an $\sB^c$-bimodule (in $\DGCat$) by our assumption
\eqref{i:monoidal-3}.
Therefore, we obtain a monoidal functor:

\[
\sB^c \ol{\otimes} \sB^{c,mon\mathendash op} \to 
\TwoEnd_{\DGCat}(\sA_{ren}^+). 
\]

\noindent Moreover, by assumption \eqref{i:monoidal-4}, 
this functor maps
into the essential image of \eqref{eq:restr-ff}. Therefore,
it lifts canonically to a monoidal functor:

\[
\sB^c \ol{\otimes} \sB^{c,mon\mathendash op} \to 
\TwoEnd_{\DGCat_{cont}}^{>-\infty}(\sA_{ren}) \subset
\TwoEnd_{\DGCat_{cont}}(\sA_{ren})
\]

\noindent as desired.

\step 

Next, observe that our functor $\zeta$ from 
above is a morphism of $\sB$-bimodule categories 
(in $\DGCat_{cont}$)
Indeed, this results from the fact that the 
embedding $\sB^c \into \sA^+$ is a morphism
of $\sB^c$-bimodule categories (in $\DGCat$).

In particular, $\Ker(\zeta)$
is a two-sided monoidal ideal in $\sB$. 
As $\zeta$ was a colocalization DG functor,
this means that $\sA_{ren}$ admits a unique monoidal
structure such that $\zeta$ is monoidal.
This monoidal structure clearly has the desired properties.

\end{proof}

\begin{example}\label{e:monoidal-cpts-trun}

Note that the assumption \eqref{i:monoidal-4} is automatic
given the other assumptions (notably, \eqref{i:monoidal-3})
if compact objects in $\sA_{ren}^c$ are closed under
truncations.

\end{example}

\begin{example}\label{e:hecke-direct-constr}

By Example \ref{e:monoidal-cpts-trun},
Lemma \ref{l:monoidal-t-str} 
applies for $\sA_{ren} = \sH_{H,K}^w \to 
\sH_{H,K}^{w,naive} = \sA$.
In particular, it may be used to directly construct 
the monoidal structure on $\sH_{H,K}^w$ from that
of $\sH_{H,K}^{w,naive}$.

\end{example}

We will also need a variant of the above construction for module
categories.

\begin{lem}\label{l:module-t-str}

In the setting of Lemma \ref{l:monoidal-t-str}, suppose
we are additionally given: 

\begin{itemize}

\item $\sM \in \DGCat_{cont}$ a module category
(in $\DGCat_{cont}$) for $\sA$. 

\item $\sM_{ren} \in \DGCat_{cont}$ a compactly generated DG category.

\item $t$-structures on $\sM$ and $\sM_{ren}$ compatible with
filtered colimits and such that $\sM_{ren}^c$ (the subcategory
of compact objects) is contained in $\sM_{ren}^+$.

\item A $t$-exact functor $\psi:\sM_{ren} \to \sM$ commuting
with colimits and inducing an equivalence $\sM_{ren}^+ \isom \sM^+$.

\end{itemize}

Suppose that:

\begin{enumerate}

\item For every $\sF \in \sA_{ren}^c$, the functor
$\Psi(\sF) \star - : \sM \to \sM$ preserves $\sM^+$.

\item\label{i:module-2} For every $\sF \in \sA_{ren}^c$, the continuous
functor $\sM_{ren} \to \sM_{ren}$ defined
by ind-extension from:

\[
\sM_{ren}^c \xar{\Psi(\sF) \star \psi(-)} 
\sM^+ \simeq \sM_{ren}^+ \subset \sM_{ren}
\]

\noindent is left $t$-exact up to shift.

\item For every $\sG \in \sM_{ren}^c$,
the functor $-\star \psi(\sG):\sA \to \sM$ maps
$\sA^+$ to $\sM^+$.

\item\label{i:module-4} For every $\sG \in \sM_{ren}^c$, the continuous
functor $\sA_{ren} \to \sM_{ren}$ defined
by ind-extension from:

\[
\sA_{ren}^c \xar{\Psi(-) \star \psi(\sG)} 
\sM^+ \simeq \sM_{ren}^+ \subset \sM_{ren}
\]

\noindent is left $t$-exact up to shift.

\end{enumerate} 

Then there is a unique action of $\sA_{ren}$ on $\sM_{ren}$
such that:

\begin{itemize}

\item The functor $\psi:\sM_{ren} \to \sM$ is a morphism
of $\sA_{ren}$-module categories, where $\sA_{ren}$ acts
on $\sM$ by restriction along $\Psi:\sA_{ren} \to \sA$.

\item For every $\sF \in \sA_{ren}^c$, the functor
$\sF \star - :\sM_{ren} \to \sM_{ren}$ 
preserves $\sM_{ren}^+$.

\end{itemize}

\end{lem}

\begin{proof}

As in the proof of Lemma \ref{l:monoidal-t-str}, 
the restriction functor:

\begin{equation}\label{eq:restr-ff-2}
\TwoEnd_{\DGCat_{cont}}^{>-\infty}(\sM_{ren}) \to 
\TwoEnd_{\DGCat}(\sM_{ren}^+) 
\end{equation}

\noindent is fully-faithful, using
similar notation as in that argument.

We use the notation from the proof of Lemma \ref{l:monoidal-t-str}
freely below.
By assumption, the (non-cocomplete) monoidal DG category
$\sB^c$ is equipped with a monoidal DG functor to 
the right hand side of \eqref{eq:restr-ff-2} and
maps into the essential image of that functor by assumption,
so we obtain an induced action of $\sB$ on $\sM_{ren}$.
We again denote this action using the notation $\star$.

As in the proof of Lemma \ref{l:monoidal-t-str},
the monoidal functor $\zeta$ admits a fully-faithful
left adjoint $\xi:\sA_{ren} \into \sB$. 

We will use the following
observation. Let $\sG \in \sM_{ren}^c$. By construction,
the functor $\sA_{ren} \xar{\xi(-)\star \sG} \sM_{ren}$
is ind-extended from the composition:
 
\[
\sA_{ren}^c \to \sA^+ \xar{- \star \psi(\sG)} \sM^+ \simeq \sM_{ren}^+ \subset \sM_{ren}.
\]

\noindent Therefore, our assumptions imply that
this functor is left $t$-exact up to shift.

Now observe that $\xi$ is automatically left lax monoidal.
Therefore, it suffices to show:

\begin{enumerate}

\item\label{i:mod-unit} For $\sG \in \sM_{ren}$,
the natural map:

\[
\xi(\e_{\sA}) \star \sG \to \e_{\sB} \star \sG = \sG
\]

\noindent is an isomorphism.

\item\label{i:mod-assoc} For $\sF_1,\sF_2 \in \sA_{ren}$ and $\sG \in \sM_{ren}$,
the natural map:

\[
\xi(\sF_1) \star \xi(\sF_2) \star \sG \to 
\xi(\sF_1 \star \sF_2) \star \sG
\]

\noindent is an isomorphism.
 
\end{enumerate} 

As $\xi$ is a left adjoint, each of the
functors appearing above commutes with colimits in each
variable. Therefore, we may assume 
$\sG \in \sM_{ren}^c \subset \sM_{ren}^+$
in each of the above cases, and $\sF_1,\sF_2 \in \sA_{ren}^c$
in \eqref{i:mod-assoc}.

For \eqref{i:mod-unit}, note that 
$\xi(\e_{\sA}) \star \sG \in \sM_{ren}^+$ by the
observation above, so as the
same is true for $\sG$, it suffices to check that the map
is an isomorphism after applying $\psi$; this is clear.

For \eqref{i:mod-assoc}, the functors
$\xi(\sF_i) \star -:\sM_{ren} \to \sM_{ren}$ are
preserve $\sM_{ren}^+$ by construction, so again
the two terms we are comparing lie in $\sM_{ren}^+$
so it suffices to (trivially) observe that the 
relevant map becomes an isomorphism after applying $\psi$.
 
\end{proof}

\begin{example}\label{e:module-cpts-trun}

As in Example \ref{e:monoidal-cpts-trun}, assumption
\eqref{i:module-2} (resp. \eqref{i:module-4})
is automatic if compact objects in $\sM_{ren}$
(resp. $\sA_{ren}$) are closed under truncations.

\end{example}

\subsection{}\label{ss:can-renorm-tate}

Suppose $H$ is a Tate group indscheme and suppose 
$\sC \in \DGCat_{cont}$ is acted on naively by $H$. 
Suppose in addition that
$\sC$ is equipped with a $t$-structure.

\begin{defin}

The naive action of $H$ on $\sC$ \emph{canonically renormalizes}
(relative to the $t$-structure) if:

\begin{itemize}

\item For every compact open subgroup $K \subset H$, the induced
naive action of $K$ on $\sC$ canonically renormalizes
(in the sense of \S \ref{ss:can-renorm}).

\item For every compact open subgroup $K \subset H$,
the data:

\[
\begin{gathered}
\Psi:\sA_{ren} = \sH_{H,K}^w \to \sA = \sH_{H,K}^{w,naive} \\
\psi:\sM_{ren} = \sC^{K,w} \to \sM = \sC^{K,w,naive}
\end{gathered}
\]

\noindent satisfy the hypotheses of Lemma \ref{l:module-t-str}.
(Here $\sC^{K,w}$ is defined as in \S \ref{ss:can-renorm}.)

\item For every pair $K_1 \subset K_2 \subset H$ 
of embedded compact open subgroups of $H$, the morphism:

\[
\sC^{K_2,w} \underset{\Rep(K_2)}{\otimes} \Rep(K_1) \to
\sC^{K_1,w}
\]

\noindent of Lemma \ref{l:gps-canon-renorm} is an equivalence.

\end{itemize}

\end{defin}

\begin{rem}

Technically there is some room for confusion: if $H$
is itself a classical affine group scheme, then this
condition is a bit more stringent than the one
from \S \ref{ss:can-renorm}. The author hopes that this
will not cause any confusion.

\end{rem}

\begin{prop}\label{p:canon-renorm-tate}

Suppose $\sC$ is equipped with a $t$-structure and a 
naive action of $H$ that canonically renormalize. 

Define the category 
$\mathsf{Gen}(\sC)$ to consist of objects 
$\sD \in H\mod_{weak}$ equipped
with an isomorphism $\Oblv_{gen}(\sD) \simeq \sC \in
H\mod_{weak,naive}$ (c.f. \S \ref{ss:tate-forgetful})
and with the property that for any compact open subgroup
$K$, $\sD^{K,w}$ is compactly generated and
the induced functor:

\[
\sD^{K,w} \to \sD^{K,w,naive} = \sC^{K,w,naive}
\]

\noindent is fully-faithful on compact objects and
induces an isomorphism $\sD^{K,w,c} \isom \sC^{K,w,c}$
(the right hand side being defined by \S \ref{ss:can-renorm}).

Then the category $\mathsf{Gen}(\sC)$ is contractible,
i.e., equivalent to $\ast \in \Gpd \subset \Cat$.

\end{prop}

\begin{proof}

Fix a compact open subgroup $K \subset H$.
Define $\mathsf{Gen}_K(\sC)$ 
to consist of $\sD \in H\mod_{weak}$ equipped with
an isomorphism $\Oblv_{gen}(\sD) \simeq \sC$
and satisfying the similar property as 
for $\mathsf{Gen}(\sC)$, but \emph{only for $K$}
(not for \emph{all} compact open subgroups).
Clearly $\mathsf{Gen}(\sC) \subset \mathsf{Gen}_K(\sC)$
is a full subcategory. 
Therefore,
it suffices to show that $\mathsf{Gen}_K(\sC)$ is
contractible and that $\mathsf{Gen}(\sC)$ is non-empty.

We need to check that $\mathsf{Gen}(\sC)$ is non-empty.
Let $\sC^{K,w}$ be defined as by canonical renormalization
for $K$. By Lemma \ref{l:module-t-str} (and by definition
of canonical renormalization for $H$),
there is a canonical $\sH_{H,K}^w$-module structure
on $\sC^{K,w}$. Let $\iota_K(\sC) \in H\mod_{weak}$
be the corresponding object with $\iota_K(\sC)^{K,w} = 
\sC^{K,w}$ (as $\sH_{H,K}^w$-modules). There is an evident
isomorphism $\Oblv_{gen}\iota_K(\sC) = \sC \in H\mod_{weak,naive}$.
This construction clearly 
defines an object of $\mathsf{Gen}_K(\sC)$, and contractibility
of the category is immediate from Lemmas \ref{l:module-t-str} 
and \ref{l:naive-hecke-ff}.

We claim moreover that the above construction 
defines an object of $\mathsf{Gen}(\sC)$.

For any $K^{\prime} \subset K$,
we have:

\[
\iota_K(\sC)^{K^{\prime},w} = 
\iota_K(\sC)^{K,w} \underset{\Rep(K)}{\otimes} \Rep(K^{\prime}) =
\sC^{K,w} \underset{\Rep(K)}{\otimes} \Rep(K^{\prime}) \to
\sC^{K^{\prime},w}
\]

\noindent and this functor satisfies the conclusions
of Lemma \ref{l:module-t-str} with respect to the
Hecke categories relative to $K^{\prime}$.
This gives an isomorphism 
$\iota_K(\sC) \simeq \iota_{K^{\prime}}(\sC) \in H\mod_{weak}$.

As the intersection of compact open subgroups is again
compact open, we see that for any (possibly not
nested) $K,K^{\prime} \subset H$
compact open subgroups, there exists an isomorphism 
$\iota_K(\sC) \iota_K(\sC) \simeq \iota_{K^{\prime}}(\sC)$.
This implies that:

\[
\iota_K(\sC) \in \bigcap_{K^{\prime}} \mathsf{Gen}_{K^{\prime}}(\sC) 
\eqqcolon \mathsf{Gen}(\sC)
\]

\noindent as desired.

\end{proof}

\begin{notation}

In the above setting, we follow our standard abuses
of notation in letting $\sC \in H\mod_{weak}$ denote
the canonical object constructed via Proposition 
\ref{p:canon-renorm-tate} (namely, the object defined
by the forgetful functor 
$\ast = \mathsf{Gen}(\sC) \to H\mod_{weak}$).

\end{notation}

We also need the following variant.

\begin{prop}\label{p:canon-renorm-a}

Suppose\footnote{The notation is potentially 
misleading: this category $\sA$ behaves more like the category
$\sA_{ren}$ from Lemma \ref{l:monoidal-t-str}.} 
$\sA \in \Alg(\DGCat_{cont})$ is
equipped with a $t$-structure
such that $\id:\sA \to \sA$ satisfies the hypotheses
for the functor ``$\Psi$" from Lemma 
\ref{l:monoidal-t-str}.\footnote{This is just a convenient way
to say $\sA$ is compactly generated with the action of
compact objects being given by functors that are left $t$-exact
up to shift and with unit being eventually coconnective.}

Suppose $\sC \in \DGCat_{cont}$ is equipped with a 
$t$-structure and an $\IndCoh^*(H) \otimes \sA$-module
structure such that: 

\begin{itemize}

\item The underlying naive $H$-action
canonically renormalizes. 

\item For any $K \subset H$ compact open, 
the evident $\sA$-action on $\sC^{K,w,naive}$
satisfies the hypotheses of Lemma \ref{l:module-t-str}
relative to 
$\psi:\sC^{K,w} \to \sC^{K,w,naive} \in \DGCat_{cont}$
(and $\id:\sA \to \sA$).

\end{itemize}

Then:

\begin{enumerate}

\item\label{i:a-1} For any compact open subgroup $K \subset H$,
the morphism:

\[
\sH_{H,K}^w \otimes \sA \to \sH_{H,K}^{w,naive} \otimes 
\sA
\]

\noindent satisfies the hypotheses of Lemma \ref{l:monoidal-t-str},
where both sides are equipped with the natural tensor
product $t$-structures.

Moreover, the corresponding action of 
$\sH_{H,K}^{w,naive} \otimes \sA$ on $\sC^{K,w,naive}$
satisfies the hypotheses of Lemma \ref{l:module-t-str}
(relative to the above functor and $\sC^{K,w} \to \sC^{K,w,naive}$).

\item\label{i:a-2}

Define the category 
$\mathsf{Gen}_{\sA}(\sC)$ to consist of objects\footnote{Here
$\sA$-modules in $H\mod_{weak}$ are defined
because $H\mod_{weak}$ is tensored over $\DGCat_{cont}$.} 
$\sD \in \sA\mod(H\mod_{weak})$ equipped
with an isomorphism $\Oblv_{gen}(\sD) \simeq \sC \in
\sA\mod(H\mod_{weak,naive})$ (c.f. \S \ref{ss:tate-forgetful})
and with the property that on forgetting the
$\sA$-action, $\sD$ defines an object
of $\mathsf{Gen}(\sC)$ (as in the notation of Proposition 
\ref{p:canon-renorm-tate}). 

Then the category $\mathsf{Gen}_{\sA}(\sC)$ is contractible,
i.e., equivalent to $\ast$.

\end{enumerate}

\end{prop}

\begin{proof}

\eqref{i:a-1} is immediate 
from Lemma \ref{l:cpt-tstr} \eqref{i:tstr-2}.
Then \eqref{i:a-2} follows by the exact same argument
as in Proposition \ref{p:canon-renorm-tate}.

\end{proof}

\subsection{Preliminary remarks about $D$-modules}

Let $H$ be a Tate group indscheme and fix
a compact open subgroup $K_0$.

Then $K_0$ induces a $t$-structure on $D^*(H)$.
Indeed, we have:

\[
D^*(H) = 
\underset{K \subset K_0 \subset H \text{ compact open}}{\colim}
D^*(H/K)
\]

\noindent under $*$-pullback functors (which are defined
as each pullback here is smooth). 
As this colimit is filtered and these 
functors are all $t$-exact up to shift
(being smooth pullbacks), we obtain the claim.
Explicitly, this $t$-structure is normalized by the
fact that for each projection 
$\pi_K:H \to H/K$, the functor 
$\pi_K^{*,dR}[-\dim(K_0/K)]:D(H/K) \to D^*(H)$
is $t$-exact.

\begin{rem}

Note that the $t$-structures attached to different
compact open subgroups differ by shifts by 
a locally constant function on $H$, namely, their
relative dimension. For our present 
purposes, such differences are irrelevant,
so we do not emphasize the choice of $K_0$ in what follows.

\end{rem}

\subsection{}

We will use the following basic observation.

\begin{lem}\label{l:dmod-t-str}

The $t$-structure just constructed on $D^*(H)$ satisfies
the following properties.

\begin{itemize}

\item The $t$-structure is right complete.

\item Any compact $\sF \in D^*(H)$ 
is eventually coconnective.

\item Compact objects are closed under truncations. 

\item For any $\sF \in D^*(H)$ compact, the monoidal
operations $\sF \star -: D^*(H) \to D^*(H)$
and $-\star \sF: D^*(H) \to D^*(H)$ are left $t$-exact up 
to shift.

\end{itemize}

\end{lem}

\begin{proof}

The first three claims are evident from the construction.
For the last one, note that there is a compact
open subgroup $K \subset H$ and a coherent $D$-module
$\sF_0 \in D(H/K)$ such that $\sF = \pi_K^{*,dR}(\sF_0)$.
Then for $\sG \in D^*(H)$, we have:

\[
\sF \star \sG = \sF_0 \overset{K}{\star} \Av_*^K(\sG)
\]

\noindent where $\star$ is convolution on $D^*(H)$,
$\Av_*^K$ indicates (strong) $K$-averaging on the left,
and $-\overset{K}{\star}-:D(H/K) \otimes D(K/H) \to D^*(H)$
is the relative convolution. The functor $\Av_*^K$ is
left $t$-exact up to shift: 
it is right adjoint to a functor that is $t$-exact up to
shift. Then the claim is evident from the fact that
$\sF_0$ has support some finite type scheme and from standard
cohomological estimates.

\end{proof}

\begin{rem}\label{r:d*-unit}

The upshot is that $D^*(H)$ \emph{almost} satisfies
the hypotheses of the monoidal DG category
$\sA$ from Proposition \ref{p:canon-renorm-a}: its unit
object is not eventually coconnective, but $D^*(H)$ 
otherwise satisfies the evident non-unital analogue.

\end{rem}

\subsection{Main construction}\label{ss:str-functor}

We are now equipped to give the main construction.

To avoid confusion, we let $D^*(H)^{gen} \in H\mod_{weak}$
denote the object constructed in \S \ref{ss:str-ker-constr},
and we use $D^*(H)$ to indicate the underlying DG category
$\Oblv_{gen}(D^*(H)^{gen})$. 

First, note that there is a canonical monoidal
functor $\IndCoh^*(H) \to D^*(H)$; indeed,
this functor is constructed in \S \ref{ss:indcoh-dmod} with
the monoidal structure coming from Remark \ref{r:indcoh-dmod-lax}.

In particular, $D^*(H)$ is canonically a
$(\IndCoh^*(H),D^*(H))$-bimodule. 
The left $\IndCoh^*(H)$-module (i.e., naive weak $H$-module)
structure here is by construction to one arising from
realizing $D^*(H)$ as $\Oblv_{gen}(D^*(H)^{gen})$.

Next, observe that the left action action 
of $\IndCoh^*(H)$ on $D^*(H)$ canonically renormalizes
in the sense of \S \ref{ss:can-renorm-tate}, and
the corresponding object 
(via Proposition \ref{p:canon-renorm-tate}) of $H\mod_{weak}$ is
$D^*(H)^{gen}$.
Indeed, this is a routine verification by 
Lemma \ref{l:dmod-canon} and 
Examples \ref{e:monoidal-cpts-trun} and \ref{e:module-cpts-trun}.
The last axiom for canonical renormalization (on
varying the compact open subgroups) reduces
to Lemma \ref{l:gps-canon-renorm} \eqref{i:gps-4}.

Therefore, by Proposition \ref{p:canon-renorm-a}
and Lemma \ref{l:dmod-t-str} (c.f. Remark \ref{r:d*-unit}) 
we obtain an a priori non-unital\footnote{It is clear
that our discussion goes through in a non-unital setting,
but this also follows directly from the unital case by
freely adjoining a unit, i.e., applying 
Proposition \ref{p:canon-renorm-a} with 
$\sA = D^*(H) \times \Vect$.}
action of $D^*(H)$ on $D^*(H)^{gen} \in H\mod_{weak}$.

By \cite{higheralgebra} Proposition 5.4.3.16, it
is a property (not a structure) for $D^*(H)$ to act
unitally on $D^*(H)^{gen}$. We verify this explicitly
as follows.

Let $K_0 \subset H$ be a fixed compact open subgroup, which
we also use to normalize the $t$-structure on $D^*(H)$.
As $H\mod_{weak} \simeq \sH_{H,K_0}^w\mod$, we need
to verify that the induced (right) $D^*(H)$-action on
$D^*(H)^{K_0,w}$ is unital. 

Note that this 
is tautologically the case for $D^*(H)^{K_0,w,naive}$.
As $D^*(H)^{K_0,w,+} \isom D^*(H)^{K_0,w,naive,+}$ by construction,
it suffices to show that the unit object
$\delta_1 \in D^*(H)$ acts by a left $t$-exact functor
on $D^*(H)^{K_0,w}$. 

Recall that $\delta_1 = \colim_K \delta_K$ where
the colimit runs over compact open subgroups and the term
$\delta_K$ indicates the $\delta$ $D$-module on $H$ supported 
on $K$.\footnote{Note that this object of $D^*(H)$ 
is in cohomological degree $-\dim(K_0/K)$ if $K \subset K_0$.}
Each $\delta_K$ is compact, so 
(by the construction of Lemma \ref{l:module-t-str})
acts on $D^*(H)^{K_0,w}$ by a functor that
is left $t$-exact up to shift. In fact, these
functors are left $t$-exact as is: the induced functor
$D^*(H)^{K_0,w,naive}$ is $\Oblv\Av_*^K$, which is 
left $t$-exact.\footnote{Here, of course,
the averaging is taken on the right, i.e.,
it does not interact with the weak $K_0$-invariants.}

By the above description of $\delta_1$, it also acts
by a left $t$-exact functor, so our earlier remarks we are done.

This completes the construction of a (right) $D^*(H)$-action
on $D^*(H)^{gen} \in H\mod_{weak}$, and therefore
(as in \S \ref{ss:strong-strategy}), induces
a functor:

\[
H\mod = D^*(H)\mod \to H\mod_{weak}.
\]

\subsection{Invariants vs. coinvariants}

We now complete the promise from Remark \ref{r:inv-coinv-w/str},
comparing weak invariants and coinvariants for $\exp(\fh)$
in the polarizable case.

Recall the category $D^!(H) \coloneqq 
D^*(H)^{\vee} = \TwoHom_{\DGCat_{cont}}(D^*(H),\Vect)$
from \cite{dmod}. Clearly $D^!(H)$ is canonically a 
$(D^*(H),D^*(H))$-bimodule. 
As $H$ is placid (in the sense
of \emph{loc. cit}.), any left (resp. right) 
$H$-invariant dimension theory on $H$
defines an equivalence $D^*(H) \isom D^!(H)$ of
left (resp. right) $D^*(H)$-modules.

Note that invariant dimension theories do exist on $H$:
any choice of congruence subgroup defines one
(see \cite{dmod} Construction 6.12.6). In particular,
$D^!(H)$ is invertible as a bimodule. 

\begin{prop}\label{p:inv-coinv-w/str}

Let $H$ be a polarizable Tate group indscheme. Then there
is a canonical isomorphism of functors:

\[
D^!(H) \underset{D^*(H)}{\otimes} (-)_{\exp(\fh),w} \simeq 
(- \otimes \chi_{-Tate})^{\exp(\fh),w}:
H\mod_{weak} \to H\mod.
\]

\end{prop}

\begin{proof}

In what follows, we use the symmetric monoidal structure
$-\otimes -$ on $H\mod_{weak}$ from \S \ref{ss:hmod-sym-mon}.

Let $\sC \in H\mod_{weak}$. We claim:

\[
\begin{gathered}
\sC_{\exp(\fh),w} = \Big(\sC \otimes 
\Oblv^{str\to w}\big(D^*(H)\big)\Big)_{H,w} \\
\sC^{\exp(\fh),w} = \Big(\sC \otimes 
\Oblv^{str\to w}\big(D^!(H)\big)\Big)^{H,w}
\end{gathered}
\]

\noindent as objects of\footnote{Here we are using the fact that
$\Oblv^{str\to w}(D^!(H)),\Oblv^{str\to w}(D^*(H)) \in 
D^*(H)\mod(H\mod_{weak})$, where this structure
arises from the bimodule structures
on $D^!(H)$ and $D^*(H)$.}
$H\mod$ functorially in $\sC$. Indeed, the first
identity is immediate, and the second identity
follows similarly the fact that $(-)^{H,w}$ is 
$\DGCat_{cont}$-linear for polarizable $H$.

Then the claim is straightforward:

\[
\begin{gathered} 
\sC_{\exp(\fh),w} = \Big(\sC \otimes 
\Oblv^{str\to w}\big(D^*(H)\big)\Big)_{H,w} 
\overset{Prop. \ref{p:chi}}{=} \\
\Big(\sC \otimes \chi_{-Tate} \otimes 
\Oblv^{str\to w}\big(D^*(H)\big)\Big)^{H,w} = \\
D^!(H)^{\otimes -1} \underset{D^*(H)}{\otimes} 
\Big(\sC \otimes \chi_{-Tate} \otimes 
\Oblv^{str\to w}\big(D^!(H)\big)\Big)^{H,w}  = \\
D^!(H)^{\otimes -1} \underset{D^*(H)}{\otimes} 
(\sC \otimes \chi_{-Tate})^{\exp(\fh),w}
\end{gathered} 
\]

\noindent for $D^!(H)^{\otimes -1} \coloneqq 
\TwoHom_{D^*(H)\mod}(D^!(H),D^*(H))$ the $D^*(H)$-bimodule
inverse to $D^!(H)$. This clearly gives the identity.

\end{proof}

\begin{example}\label{e:indcoh*-coreps}

Suppose $\sC = \IndCoh^*(H) \in H\mod_{weak}$
(i.e., the evident object that corepresents $\Oblv_{gen}$).
Then $\sC_{\exp(\fh),w} \simeq D^*(H) \in H\mod$. 
If one takes\footnote{Note that
$\IndCoh^!(H)$ is also the internal $\Hom$ 
in the symmetric monoidal category $H\mod_{weak}$ from
$\IndCoh^*(H)$ to the trivial object $\Vect$.}
$\IndCoh^!(H) \coloneqq \IndCoh^*(H) \otimes \chi_{-Tate}
\in H\mod_{weak}$, then Proposition \ref{p:inv-coinv-w/str}
says $\IndCoh^!(H)^{\exp(\fh),w} = D^!(H)$, as expected.

\end{example}

\section{Semi-infinite cohomology}\label{s:sinf}

\subsection{Construction of semi-infinite cohomology}

Let $H$ be a Tate group indscheme.

\begin{defin}

The \emph{absolute semi-infinite cohomology} functor:

\[
C^{\sinf}(\fh,-):\Vect_{\exp(\fh),w} \to \Vect 
\in H\mod
\]

\noindent is the counit map corresponding to the
adjunction constructed in \S \ref{s:strong}.

\end{defin}

The goal for this section is to show that for $H$
formally smooth, this functor
identifies (in a suitable sense) with the classical 
functor of semi-infinite cohomology
for Tate Lie algebras.

\begin{rem}\label{r:sinf-strong-eq}

As in indicated in the notation above, 
$C^{\sinf}(\fh,-)$ is strongly $H$-equivariant.
This is a non-obvious (if widely anticipated) 
property from the traditional construction of 
semi-infinite cohomology via Clifford algebras.

\end{rem}

\begin{rem}

The above functor is defined (and is strongly $H$-equivariant)
for any Tate group indscheme $H$. However, for the purposes
of relating this functor to classical constructions, we
may assume $H$ is polarizable; indeed, replacing $H$
by its formal completion along any compact open
subgroup manifestly does not change
$C^{\sinf}(\fh,-)$ as a morphism in $\DGCat_{cont}$.
Therefore, in the analysis of this section, $H$ is frequently
taken to be polarizable.

\end{rem}

\begin{rem}

The above construction (hence our comparison theorem) 
only applies for those
Tate Lie algebras $\fh \in \Pro\Vect^{\heart}$ 
arising as the Lie algebra of some formally smooth
Tate group indscheme. Equivalently, there must
exist $\fk \subset \fh$ a compact open Lie subalgebra
arising that arises as the Lie algebra of some affine
group scheme. Certainly this is the case whenever $\fh$ has
a pro-nilpotent compact open subalgebra, which covers
all examples of interest.

We anticipate (as indicated in the notation) 
that there is a theory of weak actions for the
``formal group" $\exp(\fh)$ for a general Tate Lie algebra
$\fh$. The argument given below for the comparison theorem
should then apply as is in that setup. However, as 
the applications we have in mind do not require such
a theory, we do not develop one in this text.

\end{rem}

\begin{rem}

We thank Gurbir Dhillon for insisting that we include this material
in the present text and for helpful discussions related to it.

\end{rem}

\subsection{}

Throughout this section, $H$ denotes a Tate group indscheme.

We maintain the conventions of \S \ref{ss:conv-tate-gp-indsch}:
all quotients (including classifying stacks) are 
Zariski sheafified.

\subsection{Central extensions}

We begin by discussing some general constructions relating
to central extensions.

\subsection{}\label{ss:bg_m-std}

We begin with the following construction.

First, note that there is a canonical action of
$\bB \bG_m$ on $\Vect$, i.e., an action 
of $\QCoh(\bB \bG_m)$ equipped with the convolution monoidal 
structure on $\Vect$. Indeed, this monoidal category is canonically
equivalent to $\bZ$-graded vector spaces with degree-wise.
Our monoidal functor $ \QCoh(\bB \bG_m) \to \Vect$
takes the $(-1)$-st degree component.\footnote{The sign
here makes normalizations for later constructions more convenient:
see Remark \ref{r:central-fmla}.}

Extend this construction to an action of $\bB \bG_m \times \bZ$ by
having the generator $1 \in \bZ$ act on $\Vect$
as $(-)[1]:\Vect \isom \Vect$.

Now for any $H$ a Tate group indscheme equipped with a homomorphism:

\[
(\vareps,\delta):H \to \bB \bG_m \times \bZ
\]

\noindent of groups,
we obtain an action of $H$ on $\Vect$
by restriction along the monoidal pushforward functor:\footnote{This
functor is defined by the formalism of \S \ref{s:indcoh}
because $H$ and $\bB \bG_m \times \bZ$ are weakly renormalizable
prestacks and this morphism is reasonable indschematic.}

\[
\IndCoh^*(H) \to \IndCoh^*(\bB \bG_m \times \bZ) =
\QCoh(\bB \bG_m \times \bZ).
\]

\begin{rem}\label{r:central-fmla}

Under the above construction, any $h \in H(k)$ defines
a skyscraper sheaf in $\IndCoh^*(H)$, so by extension, 
an automorphism of $\Vect$. By construction, this
automorphism is $\vareps(h) \otimes - [\delta(h)]$,
where $\vareps(h)$ is the $k$-line defined by $h$
and $\vareps$.

\end{rem}

\begin{prop}\label{p:central-extns}

The above construction gives an equivalence of groupoids:

\begin{equation}\label{eq:central-extn}
\TwoHom_{\mathsf{Gp}}(H,\bB \bG_m \times \bZ) \to 
\TwoHom_{\Alg(\DGCat_{cont})}(\IndCoh^*(H),\Vect).
\end{equation}

\end{prop}

\begin{proof}

\step First, it is convenient to dualize (in the
sense of \cite{dgcat}). For $S$ a reasonable
indscheme, we let $\IndCoh^!(S) \in \DGCat_{cont}$ denote
the dual to $\IndCoh^*(S)$ (which exists
because $\IndCoh^*(S)$ is compactly generated). 
Note that this construction is covariant in $S$; we
denote pullback along a map $f:S \to T$ by
$f^!:\IndCoh^!(T) \to \IndCoh^!(S)$.\footnote{This construction
should not be confused with the one studied in 
the $\IndCoh^*$-setting of \S \ref{s:indcoh}
for a proper (or ind-proper) morphism. Because 
we only use this construction in the proof of
the present proposition, we abuse notation by using the
same notation to mean different things (in somewhat 
different contexts).} In particular, there
is a canonical object $\omega_S \in \IndCoh^!(S)$,
the $!$-pullback of $k \in \Vect = \IndCoh^!(\Spec(k))$ 
along the structure map $S \to \Spec(k)$.

There is a standard natural transformation 
$\Upsilon_S: \QCoh(-) \to \IndCoh^!(-)$
such that for any $S \in \IndSch_{reas}$, 
$\Upsilon_S(\sO_S) = \omega_S$.
Indeed, for 
$S \in \presup{>-\infty}\Sch_{qcqs}$, $\Upsilon_S$
is by definition dual to $\Psi_S:\IndCoh^*(S) \to \QCoh(S)$
(using the standard self-duality of $\QCoh(-)$ on qcqs schemes). 
In general, the construction
is obtained by right Kan extension from this one.

Note that $\Upsilon_S$ is fully-faithful
for any $S \in \IndSch_{reas}$. Indeed, by construction,
this reduces to $S \in \presup{>-\infty}\Sch_{qcqs}$,
and in that case it follows
because the dual functor $\Psi_S$ 
admits a fully-faithful left adjoint.

\step Now suppose that $S$ is \emph{strict}
in the sense of \S \ref{ss:strict}.
By functoriality, $\IndCoh^!(S)$ is
symmetric monoidal with tensor product 
$-\overset{!}{\otimes} -$ with unit object
$\omega_S$ (c.f. \cite{indcoh} \S 5.6).

We remark that the
(symmetric) monoidal
category $(\IndCoh^!(S),\overset{!}{\otimes})$ 
acts canonically on $\IndCoh^*(S)$ by duality.

In what follows, we say $\sL \in \IndCoh^!(S)$ is \emph{invertible}
if $\sL \overset{!} {\otimes} -:\IndCoh^!(S) \to \IndCoh^!(S)$
is an equivalence. 

\step Suppose that $S \in \presup{>-\infty}{\Sch}_{qcqs}$.
We suppose that $S$ is strict and 
$\sL \in \IndCoh^!(S)$ is invertible.

We claim that:

\begin{enumerate}

\item $\sL$ lies in the essential image of $\Upsilon_S$.

\item The object $\ol{\sL} \in \QCoh(S)$ with 
$\Upsilon_S(\ol{\sL}) = \sL$ (which is well-defined by the above)
is invertible in $\QCoh(S)$ and
therefore corresponds to a $\bZ$-graded\footnote{We
normalize this identification 
by having the suspension $(-)[1]$ correspond
to increasing the grading by $1$.}
line bundle
on $S$ (the $\bZ$-grading being locally constant).

\end{enumerate}

First, note that $\sL \in \IndCoh^!(S)$ is compact:
this follows from invertibility and compactness
of $\omega_S$ (we emphasize that $S$ is eventually coconnective).
We let $\bD \sL \in \Coh(S)$ denote the corresponding object
under the equivalence:

\[
(\IndCoh^!(S)^c)^{op} \simeq \IndCoh^*(S)^c = \Coh(S).
\]

For this, let $s:\Spec(K) \to S$ be a map with $K$ a field. 
Clearly $s^!(\sL)$ is invertible in 
$\IndCoh^!(\Spec(K)) = K\mod$. In particular, it corresponds to
a graded line. We then obtain:

\[
\begin{gathered}
s^!(\sL) = 
\langle K,s^!(\sL)\rangle = 
\langle s_*^{\IndCoh}(K),\sL\rangle =
\Hom_{\IndCoh^*(S)}(\bD\sL,s_*^{\IndCoh}(K)) =
s^*(\bD\sL)^{\vee}
\end{gathered}
\]

\noindent where the brackets indicate pairings between
evident dual categories, the upper-* fiber is
the usual quasi-coherent fiber of
$\bD \sL \in \Coh(S) \subset \QCoh(S)$,
and the dual indicates the dual
as a $K$-vector space. 

In particular, the $*$-fibers of $\bD \sL$ at field-valued 
points are concentrated in some single cohomological
degree and 1-dimensional there.
Now it follows from Nakayama's lemma that $\bD \sL$ is
perfect and a graded line bundle. 

Applying duality again,
this translates to saying that $\sL \in \Upsilon_S(\Perf(S))$,
and moreover, is $\Upsilon_S$ of a graded line bundle.

\step Next, we observe that the above immediately
generalizes to the case where $S \in \IndSch_{reas}$
can be expressed as a filtered colimit of strict schemes 
$S_i \in \presup{>-\infty}{\Sch}_{qcqs}$ under
almost finitely presented closed embeddings. 
Indeed, this case immediately reduces to the schematic
case considered above.

In particular, this applies for $S = H$ our Tate group indscheme,
as observed in Remark \ref{r:tate-good-properties}.

\step We now complete the argument.

By duality and strictness of $H$, 
a monoidal functor $\IndCoh^*(H) \to \Vect$
is equivalent to a comonoidal functor $\Vect \to \IndCoh^!(H)$,
i.e., an object $\sL \in \IndCoh^!(H)$ with isomorphisms
$m^!(\sL) \isom \sL \boxtimes \sL$, $e^!(\sL) = k$ equipped with
higher homotopical compatibilities (for $m:H \times H \to H$ the
multiplication and $e:\Spec(k) \to H$ the unit).

Observe that $\sL$ is invertible in $\IndCoh^!(H)$:
its inverse is the pullback of $\sL$ along the inversion
map $H \isom H$. 

Therefore, $\sL$ defines a $\bZ$-graded line bundle
on $H$, or equivalently, a map $H \to \bB \bG_m \times \bZ$.
The comonoidal structure above is equivalent to making
the map into a map of group prestacks. 
It is immediate to verify that this equivalence is the inverse to 
the functor \eqref{eq:central-extn}.

\end{proof}

\subsection{The Tate canonical extension}

We now construct a canonical central extension:

\[
1 \to \bG_m \to H_{Tate} \to H \to 1 
\]

\noindent of any polarizable Tate group indscheme $H$.

Take $\chi_{Tate} \in H\mod_{weak}$. Recall that
$H\mod_{weak}$ is naturally symmetric monoidal and
that $\chi_{Tate}$ is invertible for this monoidal
structure. 

Recall from Proposition \ref{p:chi-restr} 
that $\Oblv_{gen}(\chi_{Tate}) \in \DGCat_{cont}$ 
is a trivial gerbe, i.e., 
this DG category is \emph{non-canonically} isomorphic to $\Vect$
(the identification depends on a choice of
compact open subgroup $K$ of $H$). 
In particular, we have a \emph{canonical} isomorphism 
$\TwoEnd_{\DGCat_{cont}}(\Oblv_{gen}(\chi_{Tate})) = \Vect$.

As $H$ acts naively on $\Oblv_{gen}(\chi_{Tate})$,
this defines a canonical homomorphism 
$\IndCoh^*(H) \to \Vect$.
By Proposition \ref{p:central-extns},
we obtain a homomorphism map:

\[
(\vareps_{Tate},\delta_{Tate}):H \to \bB \bG_m \times \bZ.
\]

\noindent By definition, $H_{Tate}$ is the central
extension defined by the homomorphism $\vareps_{Tate}$.

\begin{rem}\label{r:vect-chi}

Define an object\footnote{By analogy with usual
representations, $\Vect_{\chi_{Tate}}$ is the 
character defined by the ``1-dimensional" representation
$\chi_{Tate}$.}
$\Vect_{\chi_{Tate}} \in H\mod_{weak}$ as:

\[
\Vect_{\chi_{Tate}} \coloneqq 
\chi_{Tate} \otimes \triv(\Oblv_{gen}(\chi_{Tate}^{\otimes -1})).
\]

\noindent Note that $\Vect_{\chi_{Tate}}$ maps canonically under 
$\Oblv_{gen}$ to $\Vect$, and the induced naive $H$-action
is the one constructed above. 
By Proposition \ref{p:chi-restr}, any
choice of compact open subgroup $K \subset H$
induces an isomorphism  
$\chi_{Tate} \simeq \Vect_{\chi_{Tate}} \in H\mod_{weak}$.

\end{rem}

\subsection{}

We now discuss basic properties of $H_{Tate}$.

\begin{prop-const}\label{pc:tate-cpt-open}

For any compact open subgroup $K$ of $H$, there is
a canonical splitting of $H_{Tate}$ over $K$.

\end{prop-const}

\begin{proof}

Immediate from Proposition \ref{p:chi-restr}.

\end{proof}

\begin{cor}

$H_{Tate}$ is a Tate group indscheme.

\end{cor}

\begin{warning}

In general, for $K_1 \subset K_2 \subset H$ compact open
subgroups, the canonical splitting for $K_2$ may not
restrict to the canonical splitting for $K_1$
(although this is automatic if $K_1$ is pro-unipotent).

\end{warning}

\begin{prop-const}\label{pc:chi-tate}

Suppose $H$ has the property that 
$\delta_{Tate}$ is identically $0$. 

Let $\Res:H\mod_{weak} \to H_{Tate}\mod_{weak}$
denote the functor of restriction along $H_{Tate} \to H$,
and let $\chi_{Tate} \in H\mod_{weak}$ denote the
modular character for $H$.

Then $\Res(\chi_{Tate})$ is canonically trivialized,
i.e., there is a canonical isomorphism:

\[
\Res(\chi_{Tate}) \simeq \triv \Oblv_{gen}(\Res(\chi_{Tate})
\in H_{Tate}\mod_{weak}.
\]

\end{prop-const}

\begin{proof}

In the notation of Remark \ref{r:vect-chi}, it
suffices to construct an isomorphism 
$\Res(\Vect_{\chi_{Tate}}) \simeq \Vect \in H_{Tate}\mod_{weak}$.

Note that by standard cohomological estimates,
the underlying naive 
action of $H$ on $\Oblv_{gen}(\Vect_{\chi_{Tate}})$
canonically renormalizes in the sense of 
Proposition \ref{p:canon-renorm-tate},
and that $\Vect_{\chi_{Tate}}$
is obtained by this canonical renormalization
procedure. The same applies for $H_{Tate}$ in place of $H$.
Therefore, it suffices to give the construction
on underlying naive categories.

But here the result follows from Proposition \ref{p:central-extns}
and the evident trivialization of the composite homomorphism:

\[
H_{Tate} \to H \to \bB \bG_m \times \bZ.
\]

\end{proof}

\begin{rem}

By Theorem \ref{t:weak} for of $\bG_m$, it is easy to
see that $\Res(\chi_{Tate}) \in H_{Tate}\mod_{weak}$
is the modular character for $H_{Tate}$. Therefore,
the modular character of $H_{Tate}$ is canonically trivialized.

\end{rem}

\subsection{}

We now wish to formulate in a precise way the following idea:
\textit{for $\sC \in H\mod_{weak}$, 
$(\sC \otimes \Vect_{\chi_{Tate}})^{H,w}$ is the subcategory
of $\sC^{H_{Tate},w}$ of objects on which
$\bG_m$ acts by homotheties.}

There is a somewhat more satisfying formulation in the 
naive setting than the genuine one,
so we separate the two cases.

\subsection{}

Let $k(1) \in \Rep(\bG_m)$ denote the standard representation,
and for $n\in \bZ$, let $k(n)$ denote its $n$th tensor power.
We let $\Vect_{(n)} \subset \Rep(\bG_m)$ denote the 
image of $\Vect \xar{k \mapsto k(n)} \Rep(\bG_m)$,
i.e., the category of graded vector spaces of pure degree $n$.

Suppose $\sC \in H\mod_{weak}$ and restrict
$\sC$ to $H_{Tate}\mod_{weak}$; we now omit $\Res$ from
the notation.
As the central $\bG_m \subset H_{Tate}$ acts trivially on
$\sC$, there is a forgetful functor:

\[
\sC^{H_{Tate},w} \to \sC^{\bG_m,w} = 
\Rep(\bG_m) \otimes \Oblv_{gen}(\sC).
\]

\noindent We remark that this functor factors through
$\sC^{H_{Tate},w,naive}$, and that the corresponding
functor 
$\sC^{H_{Tate},w,naive} \to \Rep(\bG_m) \otimes \Oblv_{gen}(\sC)$ is 
conservative.

For $n \in \bZ$, define 
$\sC_{(n)}^{H_{Tate},w,naive} \subset \sC^{H_{Tate},w,naive}$
as the full subcategory: 

\[
\sC^{H_{Tate},w,naive} 
\underset{\Rep(\bG_m) \otimes \Oblv_{gen}(\sC)}{\times}
\Vect_{(n)} \otimes \Oblv_{gen}(\sC).
\]

In words: this is the full subcategory of 
$H_{Tate}$-equivariant objects where the central $\bG_m$ acts
by the $n$th power of its canonical character, this notion
being defined because $\bG_m$ acts trivially on $\sC$.

Note that: 

\begin{equation}\label{eq:central-gm}
\sC^{H,w,naive} \isom \sC_{(0)}^{H_{Tate},w}
\end{equation}

\noindent by semi-simplicity of $\Rep(\bG_m)$.

\begin{prop}\label{p:chi-tate-inv-naive}

Let $H$ be a Tate group indscheme with $\delta_{Tate}$ 
identically $0$.

Then for any $\sC \in H\mod_{weak}$, the canonical functor:

\[
(\sC \otimes \Vect_{\chi_{Tate}})^{H,w,naive} \to 
(\sC \otimes \Vect_{\chi_{Tate}})^{H_{Tate},w,naive} 
\overset{Prop.-Const. \ref{pc:chi-tate}}{\simeq}
\sC^{H_{Tate},w,naive} 
\]

\noindent is fully-faithful with essential image
$\sC_{(1)}^{H_{Tate},w,naive}$.

\end{prop}

\begin{proof}

Consider the isomorphism 
$\Vect \isom \Vect_{\chi_{Tate}} \in H_{Tate}\mod_{weak}$
from Proposition-Construction \ref{pc:chi-tate}
restricted to $\bG_m$. 
Note that $\Vect_{\chi_{Tate}}|_{\bG_m}$ is canonically
isomorphic to $\Vect$ with its trivial action (as it was
obtained by restriction from $H$).
Therefore, this isomorphism is an equivalence:

\[
\Vect \isom \Vect_{\chi_{Tate}}|_{\bG_m} = \Vect \in \bG_m\mod_{weak}.
\]

Therefore, this isomorphism amounts to specifying an invertible
object of $\TwoEnd_{\bG_m\mod_{weak}}(\Vect) = \Rep(\bG_m)$.
It follows from the construction that this object
is $k(1)$.

We now obtain the result from \eqref{eq:central-gm}.

\end{proof}

\subsection{}\label{ss:tate-tw-inv}

We now explain how to adapt the above to the setting of
genuine actions.

Note that by Theorem \ref{t:weak} (for $\bG_m$),
$\Res:H\mod_{weak} \to H_{Tate}\mod_{weak}$
admits a right (and left) $\DGCat_{cont}$-linear adjoint
commuting with colimits. We abuse notation
somewhat in denoting this functor by $(-)^{\bG_m,w}$.

For any $n \in \bZ$, there is an adjunction map:

\[
\Vect_{\chi_{Tate}}^{\otimes n} \to (\Vect_{\chi_{Tate}}^{\otimes n})^{\bG_m,w} 
\overset{Prop.-Const. \ref{pc:chi-tate}}{\simeq}
\Vect^{\bG_m,w} \in H\mod_{weak}.
\]

\noindent The induced map:

\[
\underset{n \in \bZ}{\oplus}
\Vect_{\chi_{Tate}}^{\otimes n} 
\Vect^{\bG_m,w} \in H\mod_{weak}
\]

\noindent is an isomorphism.\footnote{The 
$\oplus$ denotes the coproduct in $H\mod_{weak}$.
Note that (even infinite) coproducts in $H\mod_{weak}$
coincide with products as the
same is true in $\DGCat_{cont}$.} 

On tensoring, for any $\sC \in H\mod_{weak}$, we obtain an
isomorphism:

\[
\underset{n \in \bZ}{\oplus}
(\sC \otimes \Vect_{\chi_{Tate}}^{\otimes n}) \isom  
\Res(\sC)^{\bG_m,w} \in H\mod_{weak}.
\]

\noindent Passing to invariants, we obtain:

\[
\underset{n \in \bZ}{\oplus}
(\sC \otimes \Vect_{\chi_{Tate}}^{\otimes n})^{H,w} \isom  
\sC^{H_{Tate},w}.
\]

We record these observations as the following 
analogue of Proposition \ref{p:chi-tate-inv-naive}.

\begin{prop}\label{p:chi-tate-inv-gen}

For $\sC \in H\mod_{weak}$, $\sC^{H_{Tate},w}$ is canonically
$\bZ$-graded with $(\sC \otimes \Vect_{\chi_{Tate}})^{H,w}$
as its degree $1$ component.

\end{prop}

\subsection{}

We remark briefly on another interpretation of the
above results. We allow ourselves to be slightly imprecise here
in speaking about $\bB \bG_m$-actions on categories
on equal footing with genuine actions of Tate group indschemes,
although this is not formally allowed in the theory developed
in \S \ref{s:weak-ind} (though one could suitably extend the
theory without difficulty).

By fiat, genuine $\bB \bG_m$-actions on $\sC \in \DGCat_{cont}$
are the same as naive ones, i.e., $\IndCoh^*(\bB \bG_m)$-actions.
As $\IndCoh^*(\bB \bG_m) = \QCoh(\bZ)$ with convolution on the
left corresponding to tensor products on the right,
such a datum is equivalent to a $\bZ$-grading 
$\sC = \oplus_{n \in \bZ} \sC_{(n)}$. Here
$\sC_{(0)} = \sC^{\bB \bG_m,w}$.

Note that the $\bB \bG_m$ on $\Vect$ constructed in 
\S \ref{ss:bg_m-std} has $\Vect = \Vect_{(-1)}$, i.e., it
is $\Vect$ graded in pure degree $-1$. For simplicity,
we denote this object by $\Vect_{(-1)} \in \bB \bG_m\mod_{weak}$.

We have a fiber sequence of groups:

\[
H_{Tate} \to H \xar{\vareps_{Tate}} \bB \bG_m
\]

\noindent as $\delta_{Tate}$ is assumed to be $0$.
By design, the pullback of $\Vect_{(-1)}$ along
$\vareps_{Tate}$ is $\Vect_{\chi_{Tate}}$.

Therefore, for $\sC$ with a genuine (or naive) action of $H$,
$\bB \bG_m$ acts on $\sC^{H_{Tate},w}$, i.e., we obtain
a grading $\sC^{H_{Tate},w} = 
\oplus_{n \in \bZ} \sC_{(n)}^{H_{Tate},w}$.
We then have:

\[
(\sC \otimes \Vect_{\chi_{Tate}})^{H,w} = 
\Big(\big(\oplus_{n \in \bZ} \sC_{(n)}^{H_{Tate},w}\big) \otimes 
\Vect_{(-1)} \Big)^{\bB \bG_m,w} =
\sC_{(1)}^{H_{Tate},w}
\]

\noindent as desired.

\subsection{Representations of Tate group indschemes}

The following result describes the major structures of
$\Rep(H)$.

\begin{prop}\label{p:rep-tate}

Let $H$ be polarizable.

\begin{enumerate}

\item\label{i:rep-1} $\Rep(H)$ is compactly generated.

\item\label{i:rep-2} Suppose $H$ is a classical indscheme.
There is a unique $t$-structure
on $\Rep(H)$ such that for any compact open
subgroup $K \subset H$, the (conservative) forgetful functor
$\Rep(H) \to \Rep(K)$ is $t$-exact.

\item\label{i:rep-3} Suppose that $H$ is of 
Harish-Chandra type (c.f. Example \ref{e:hc}) 
and formally smooth. Then $\Rep(H)^+$ is the bounded below
derived category of $\Rep(H)^{\heart}$.

\end{enumerate}

\end{prop}

\begin{proof}

Let $K \subset H$ be a polarization.

For any $\sC \in H\mod_{weak}$, 
the forgetful functor $\sC^{H,w} \to \sC^{K,w}$ is conservative
and admits a continuous left adjoint 
$\Av_!^w:\sC^{K,w} \to \sC^{H,w}$. Indeed, the former property
is true for any compact open subgroup while the latter
is true by ind-properness of $H/K$.
Therefore, if $\sC^{K,w}$ is compactly generated,
then $\sC^{H,w}$ is compactly generated. 
Applying this for $\sC = \Vect$ gives \eqref{i:rep-1}.

Next, in the setting of \eqref{i:rep-2}, 
observe that it suffices to show 
$\Oblv\Av_!^w:\Rep(K) \to \Rep(K)$ 
is right $t$-exact. Indeed, we are reduced to showing
this by the monadicity of $\Oblv$ shown above. (We remark
that $t$-exactness of the restriction functor to some
compact open subgroup clearly implies the same for any compact
open subgroup.)

Because $H$ is classical,
the same is true of $H/K$, i.e., we can write 
$H/K = \colim_i S_i$ a filtered
colimit of classical proper $k$-schemes.

Suppose $V \in \Rep(K)^{\heart}$
is finite-dimensional. It suffices to show 
that for such $V$, $\Oblv\Av_!^w(V) \in \Rep(K)^{\leq 0}$,
or equivalently, that the underlying vector space of this
$K$-representation is in $\Vect^{\leq 0}$.
Let $\sE_V \in \QCoh(H/K)$
be the corresponding (naively $H$-equivariant) vector bundle.
Then:

\[
\Oblv\Av_!^w(V) = 
\Gamma^{\IndCoh}(H/K,\sE_V \otimes \omega_{H/K}) = 
\underset{i}{\colim} \, 
\Gamma^{\IndCoh}(S_i,\sE_V|_{S_i} \otimes \omega_{S_i}) =
\Gamma(S_i,\sE_V^{\vee}|_{S_i})^{\vee}.
\]

\noindent We have
$\Gamma(S_i,\sE_V^{\vee}|_{S_i}) \in \Vect^{\geq 0}$
as $S_i$ is classical, so we obtain the claim by dualizing.

Finally, in the setting of \eqref{i:rep-3}, we suppose
$K$ is chosen so $H$ is formally complete along it;
note that the above argument shows that $\Oblv\Av_!^w$
is $t$-exact by formal smoothness of $H/K$.

In the case $H = K$, the fact that
$\Rep(K)^+$ is the bounded below derived category of its
heart is standard. Any object $\Av_*^w(V)$ for 
$V \in \Vect^{\heart}$ is injective in $\Rep(K)^{\heart}$.
Moreover, any object of $\Rep(K)^{\heart}$ admits
an injective resolution by such objects. Finally, 
for $W \in \Rep(K)^{\heart}$ and $V$ as above,
$\ul{\Hom}_{\Rep(K)}(W,\Av_*^w(V)) = \ul{\Hom}_{\Vect}(W,V)$
is concentrated in cohomological degree $0$. These
observations imply the claim. 

In general, the argument follows by 
Lemma \ref{l:monad-der-ab} below
(or see a variant of this argument
in \cite{whit} Lemma A.18.1).

\end{proof}

We used the following result above.

\begin{lem}\label{l:monad-der-ab}

Suppose $\sC,\sD \in \DGCat$ are equipped with $t$-structures.
compatible with filtered colimits.
Let $G:\sC \to \sD \in \DGCat_{cont}$ 
be a conservative, $t$-exact functor with a $t$-exact left
adjoint $F$. 

Suppose $\sD^+$ is the bounded below derived category of
$\sD^{\heart}$. Then $\sC^+$ is the bounded below derived category
of $\sC^{\heart}$.

\end{lem}

\begin{proof}

Let $I \in \sC^{\heart}$ be an injective object. We need to show
that for any $\sF \in \sC^{\heart}$, 
$\ul{\Hom}_{\sC}(\sF,I) \in \Vect^{\heart}$.
Clearly this complex is in degrees $\geq 0$. 
We will show by induction on $i>0$ that $H^i\ul{\Hom}_{\sC}(\sF,I)  = 
\Ext_{\sC}^i(\sF,I)$ vanishes for all $\sF$.
For $i = 1$, we have  
$\Ext_{\sC}^1(\sF,I) = \Ext_{\sC^{\heart}}^1(\sF,I) = 0$, giving the
base case. Suppose the result is true for $i \geq 1$, and we will deduce
it for $i+1$.

First, note that the counit map 
$FG(\sF) \to \sF \in \sC^{\heart}$ is an epimorphism.
Indeed, we can check this after applying the conservative, $t$-exact
functor $G$, and then the map splits.

Let $\sF_0$ be the kernel of this counit. We obtain an exact sequence:

\[
\Ext_{\sC}^i(\sF_0,I) \to \Ext_{\sC}^{i+1}(\sF,I) \to 
\Ext_{\sC}^{i+1}(FG(\sF),I) = \Ext_{\sD}^{i+1}(G(\sF),G(I)).
\]

\noindent The first term vanishes by induction. The last term vanishes
because $G:\sC^{\heart} \to \sD^{\heart}$ admits a $t$-exact left
adjoint so preserves injectives, and by assumption on $\sD$.
This gives the claim.

\end{proof}

Combining Propositions \ref{p:rep-tate}, 
\ref{p:chi-tate-inv-naive} and \ref{p:chi-tate-inv-gen},
we obtain:

\begin{cor}

For polarizable $H$, the categories
$\Rep_{\pm Tate}(H)$ are compactly generated.
If $H$ is classical, there is a unique compactly generated
$t$-structure on $\Rep_{\pm Tate}(H)$ for which the
forgetful functor to $\Rep(K)$ is $t$-exact
for any compact open subgroup $K \subset H$
(using Proposition \ref{p:chi-restr}).
The category $\Rep_{Tate}(H)^+$ (resp. $\Rep_{-Tate}(H)^+$) 
maps isomorphically onto
the subcategory of $\Rep(H_{Tate})^+$ consisting of
objects on which the central $\bG_m$ acts by (direct sums of
shifts of) its standard representation (resp. the inverse
to the standard representation).
If $H$ is additionally of Harish-Chandra type,
then $\Rep_{\pm Tate}(H)^+$ is the bounded below derived
category of $\Rep_{\pm Tate}(H)^{\heart}$.

\end{cor} 

\subsection{Passage to Lie algebras}\label{ss:sinf-lie}

Let $H$ be a Tate group indscheme of 
Harish-Chandra type. We assume $H$ is polarizable
in what follows (although most of the discussion generalizes 
to the non-polarizable case by replacing $H$ with
its formal completion along some compact open subgroup).

Define $\fh\mod \coloneqq \Vect^{\exp(\fh),w}$.
Similarly, define $\fh_{Tate}\mod$
(resp. $\fh_{-Tate}\mod$) as
$(\Vect_{\chi_{Tate}})^{\exp(\fh),w}$
(resp. $(\Vect_{\chi_{-Tate}})^{\exp(\fh),w}$).\footnote{The
notation is potentially confusing. There is a central
extension $\fh_{Tate}$ around (the Lie algebra of $H_{Tate}$),
and we are in effect considering modules over it on which
the central element $1 \in k \subset \fh_{Tate}$ acts by
the identity (or minus the identity), of course in a suitable
derived categorical sense. 

This is a somewhat standard abuse, and we hope that it does
not cause confusion. To be clear: we will never consider
all modules over the Tate Lie algebra $\fh_{Tate}$.}

\begin{defin}

For $K \subset H$ a fixed compact open subgroup,
the \emph{relative semi-infinite cohomology} functor:

\[
C^{\sinf}(\fh,\fk;-):\fh_{-Tate}\mod \to \Vect \in H\mod
\]

\noindent corresponds to 
$C^{\sinf}(\fh,-):\Vect_{\exp(\fh),w} \to \Vect$ under the
equivalence:

\[
\begin{gathered}
\fh_{-Tate}\mod \coloneqq
(\Vect_{\chi_{-Tate}})^{\exp(\fh),w} 
\overset{Prop. \ref{p:inv-coinv-w/str}}{\simeq} \\
D^!(H) \underset{D^*(H)}{\otimes} 
(\Vect_{\chi_{-Tate}} \otimes \chi_{Tate})_{\exp(\fh),w} \simeq \\
\Vect_{\exp(\fh),w}
\end{gathered}
\]

\noindent using the choice of $K$ both to identify
$\Vect_{\chi_{Tate}}$ with $\chi_{Tate}$ via
Proposition \ref{p:chi-restr} and to
identify $D^!(H)$ with $D^*(H)$ via
\cite{dmod} Construction 6.12.6.

\end{defin}

Note that $H$ acts strongly on $\fh\mod$ 
by the construction of \S \ref{s:strong}.
For $K \subset H$ compact open, 
we have $\fh\mod^K = \Rep(H_K^{\wedge})$ by construction.
In particular, we have $\fh\mod = \colim_K \fh\mod^K$.
Moreover, for $H$ formally smooth,
$\fh\mod^K$ has a canonical $t$-structure by 
Proposition \ref{p:rep-tate}.

\subsection{}

We now suppose that $H$ is formally smooth.
In this case, its Lie algebra $\fh$ is naturally
a Tate Lie algebra in the sense of Example \ref{e:lie},
and we have two possibly conflicting definitions
of $\fh\mod$. However, we claim that they do not in fact
conflict. 

Below, we understand $\fh\mod$ in the sense defined
immediately above, i.e., as $\Vect^{\exp(\fh),w}$.

\begin{lem}\label{l:hmod}

\begin{enumerate}

\item For each pair $K_1 \subset K_2 \subset H$
of compact open subgroups, the (conservative) restriction
functor:

\[
\fh\mod^{K_2} \to \fh\mod^{K_1}
\]

\noindent is $t$-exact. In particular, the colimit
$\fh\mod$ over all such compact open subgroups
admits a canonical $t$-structure.

\item 

The forgetful functor 
$\fh\mod \coloneqq \Vect^{\exp(\fh),w} \to \Vect$
is $t$-exact and conservative on eventually coconnective
subcategories. The corresponding 
$\overset{\rightarrow}{\otimes}$-algebra (as defined
by Proposition \ref{p:alg-vs-cats}) is the
completed universal enveloping algebra of the
Tate Lie algebra $\fh$. Moreover, the compact
generators of $\fh\mod$ correspond to the
renormalization datum specified in 
Example \ref{e:lie}.

\end{enumerate}

\end{lem}

\begin{proof}

The $t$-exactness of the various restriction functors
is clear from Proposition \ref{p:rep-tate}.

Moreover, for $K_1 \subset K_2 \subset H$ compact
open subgroups and for $V \in \fh\mod^{K_1,\geq 0}$, we claim
that the adjunction map $\Oblv \Av_*^{K_1 \to K_2}(V) \to V$
induces a monomorphism in $\fh\mod^{K_1,\heart}$
upon applying $H^0$. Indeed, we can test this after applying
the (conservative, $t$-exact) 
forgetful functor to $\fk_2\mod^K_1$, where it is evident.

It follows that for any $V \in \fh\mod^{\geq 0}$ and $K$
a congruence subgroup, the adjunction
map $\Oblv \Av_*^K(V) \to V$ gives a monomorphism on $H^0$.
As $V = \colim_K \Oblv \Av_*^K(V)$, this implies 
that $\Oblv:\fh\mod \to \Vect$ is conservative on 
eventually coconnective subcategories.

Now define an object: 

\[
\sP \coloneqq \underset{K}{\lim} \, 
\ind_{\fk}^{\fh}(k) \in \Pro(\fh\mod^{\heart})
\subset \Pro(\fh\mod^+)
\]

\noindent where the notation is understood as follows.
First, the limit is formed in the pro-category,
and is indexed by compact open subgroups $K \subset H$.
Then $k \in \fk\mod$ denotes the trivial representation
and $\ind_{\fk}^{\fh}:\fk\mod \to \fh\mod^K$
is the left adjoint to the forgetful functor.

Then $\sP$ pro-corepresents
the forgetful functor $\fh\mod^+ \to \Vect^+$.
Moreover, under the forgetful functor, $\sP$
maps an object of 
$\Pro(\Vect^{\heart})$. By Proposition \ref{p:alg-vs-cats},
$\fh\mod^+$ is the bounded below derived category of its
heart, and this heart is the category of discrete
modules for $\Oblv(\sP)$ with respect to its natural 
$\overset{\rightarrow}{\otimes}$-algebra structure.

One can identify $\Oblv(\sP)$ with its 
$\overset{\rightarrow}{\otimes}$-algebra structure
as follows. Let $\fh^{disc} \in \LieAlg(\Vect)$
denote the Lie algebra obtained by forgetting the topology
on $\fh$.\footnote{In other words, we pass to the inverse
limit of the pro-vector space underlying $\fh$ and
then apply $H^0$ if for some pathological reason there
are higher cohomology groups.}
We have a canonical map $\fh^{disc}\to \fh$ of Tate Lie algebras,
giving rise to a forgetful functor $\fh\mod \to \fh^{disc}\mod$.
By \cite{grbook}, the (non-topological) algebra
attached to $\fh^{disc}$ is the usual enveloping
algebra $U(\fh^{disc})$. Moreover, the natural map:

\[
\ind_{\fk^{disc}}^{\fh^{disc}}(k) \to \ind_{\fk}^{\fh}(k) \in \Vect
\]

\noindent is an isomorphism. This immediately implies the claim.

Finally, it is immediate from the
constructions to identify the compact generators.

\end{proof}

\begin{cor}\label{c:h-mod-der-cat}

Under the above hypotheses, 
$\fh\mod^+$ is the bounded below derived category
of $\fh\mod^{\heart}$.

\end{cor}

\begin{proof}

Immediate from Lemma \ref{l:hmod} and 
Proposition \ref{p:alg-vs-cats}.

\end{proof}

\subsection{Classical semi-infinite cohomology}

Let $H$ be a formally smooth polarizable Tate group indscheme.

We let $\fh_{Tate_{std}}$ denote the central extension 
$0 \to k \to \fh_{Tate_{std}} \to \fh \to 0$ of
Tate Lie algebras constructed e.g. in 
\cite{hitchin} \S 7.13. We abuse notation in
letting $\fh_{Tate_{std}}\mod$ 
denote not the category of representations as is, but
the analogue where we impose the requirement that
the central $1 \in k$ act by the identity. We remind
that $\fh_{Tate_{std}}$ is canonically split over
any Lie subalgebra $\fk_0 \subset \fh$ that is a lattice 
(in the usual sense of Tate vector spaces).

By Lemma 19.8.1 from \cite{fg2} and 
Corollary \ref{c:h-mod-der-cat} above, 
for $K \subset H$ a compact open subgroup, we have DG a
functor:

\[
C_{std,0}^{\sinf}(\fh,\fk;-):\fh_{-Tate_{std}}\mod^{+} \to 
\Vect
\] 

\noindent of \emph{standard semi-infinite cohomology}
(defined in terms of Clifford algebras and spin representations)
whose restriction to $\fh_{-Tate_{std}}\mod^{\geq -n}$
commutes with filtered colimits for any $n$.
We also let:

\[
C_{std}^{\sinf}(\fh,\fk;-):\fh_{-Tate_{std}}\mod \to 
\Vect 
\]

\noindent denote the functor obtained by
restricting $C_{std,0}^{\sinf}(\fh,\fk;-)$ to 
$\fh_{-Tate_{std}}\mod^c$ and then ind-extending.

\subsection{}

We will now show 
that the canonical natural transformation:

\[
\eta:C_{std}^{\sinf}(\fh,\fk;-)|_{\fh_{-Tate_{std}}\mod^{+}} \to 
C_{std,0}^{\sinf}(\fh,\fk;-)
\]

\noindent of functors $\fh_{-Tate_{std}}\mod^+ \to \Vect$
is an isomorphism. 
(Combined with Theorem \ref{t:sinf-comparison}
below, this means that $C^{\sinf}(\fh,\fk;-)|_{\fh_{-Tate}\mod}$
may be calculated using the standard semi-infinite complex.)

First, if $\fh = \fk$, this follows immediately
from the fact that compact objects in $\fk\mod$
are closed under truncations 
(c.f. Example \ref{e:functor-ren-trun}).

In general, it is standard that for 
$M \in \fh_{-Tate_{std}}\mod^{\heart}$,  
$C_{std,0}^{\sinf}(\fh,\fk;M)$ has a canonical
increasing filtration indexed by $\bZ^{\geq 0}$ 
with associated graded terms:

\begin{equation}\label{eq:sinf-gr}
\gr_i C_{std,0}^{\sinf}(\fh,\fk;M) =
C^{\dot}(\fk,\Lambda^i(\fh/\fk) \otimes M)[i]
\end{equation}

\noindent for $C^{\dot}(\fk,-)$ denoting the
cohomological Chevalley complex
(see e.g. the proof of Lemma 19.8.1 from \cite{fg2}). 
This is functorial in $M$, so the functor
$C_{std,0}^{\sinf}(\fh,\fk;-):\fh_{-Tate_{std}}\mod^+ \to
\Vect$ upgrades to a functor valued in $\Fil_{\geq 0}\Vect$
the (DG) category of $\bZ^{\geq 0}$-filtered
vector spaces. 

By construction, the functor $C_{std}^{\sinf}(\fh,\fk;-)$
also upgrades to $\Fil_{\geq 0}\Vect$ compatibly
with the above and the map $\eta$. To 
check $\eta$ is an isomorphism, it is enough to do
so on the associated graded level. As the associated
graded functors above factor through the forgetful functor
$\fh_{-Tate_{std}}\mod \to \fk\mod^+$ (c.f. 
\eqref{eq:sinf-gr}), we conclude as above. 

\subsection{}

For the remainder of this section, $H$ is 
a formally smooth polarizable Tate group indscheme
with $\delta_{Tate} = 0$.

We have the following comparison theorem.

\begin{thm}\label{t:sinf-comparison}

For $H$ and $K$ as above,
there is a canonical identification of
$\fh_{Tate}$ with $\fh_{Tate_{std}}$ as central
extensions of $\fh$. Moreover, under this identification,
there is a canonical isomorphism
$C^{\sinf}(\fh,\fk;-) \simeq C_{std}^{\sinf}(\fh,\fk;-)$
of functors $\fh_{-Tate}\mod \simeq \fh_{-Tate_{std}}\mod \to
\Vect$. This pair of identifications is uniquely characterized
by compatibility with the given splittings of these
central extensions over $\fk$ and with the isomorphisms
Lemmas \ref{l:sinf-shapiro-std} and \ref{l:sinf-shapiro} as formulated
and proved below.

\end{thm}

We will prove this result in the remainder of this section.

\subsection{}

First, we review an important property of standard semi-infinite
infinite cohomology.

\begin{lem}\label{l:sinf-shapiro-std}

The composition:

\[
\fk\mod \xar{\ind_{\fk}^{\fh_{-Tate_{std}}}} 
\fh_{-Tate_{std}}\mod \xar{C_{std}^{\sinf}(\fh,\fk;-)}
\Vect 
\]

\noindent is isomorphic to the functor
$C^{\dot}(\fk,-) \coloneqq \ul{\Hom}_{\fk\mod}(k,-):\fk\mod \to \Vect$ of Lie algebra cohomology.

\end{lem}

\begin{proof}

By \cite{hitchin} Remark 7.13.30, there is a canonical
isomorphism of between the composite functor:

\[
\fk\mod^+ \xar{\ind_{\fk}^{\fh_{-Tate_{std}}}} 
\fh_{-Tate_{std}}\mod^+ 
\xar{C_{std,0}^{\sinf}(\fh,\fk;-)} \to \Vect.
\]

\noindent with $C^{\dot}(\fk,-)|_{\fk\mod^+}$. 
Now the result follows by construction of $C_{std}^{\sinf}$.

\end{proof}

\subsection{}

We now establish similar results for $C^{\sinf}(\fh,\fk;-)$.

First, observe that we have a duality functor:

\[
\bD_{\fh,K}^{\sinf}:\fh\mod^{\vee} \simeq \fh_{-Tate}\mod \in 
\DGCat_{cont}
\]

\noindent (depending on the choice of compact open 
subgroup $K$). Indeed, our choice of $K$ identifies 
$\fh_{-Tate}\mod \simeq \Vect_{\exp(\fh),w}$ 
(c.f. \S \ref{ss:sinf-lie}). 
Note that this category
is in fact dualizable as it is compactly generated,
and its dual is:

\[
\TwoHom_{\DGCat_{cont}}(\Vect_{\exp(\fh),w},\Vect) = 
\TwoHom_{\DGCat_{cont}}(\Vect,\Vect)^{\exp(\fh),w} = \fh\mod.
\]

Under this duality, the functor $C^{\sinf}(\fh,\fk;-)$
clearly 
corresponds to the trivial representation $k \in \fh\mod^{\heart}$.

We now have:

\begin{lem}\label{l:sinf-shapiro}

The composition:

\[
\fk\mod \xar{\ind_{\fk}^{\fh_{-Tate}}} 
\fh_{-Tate}\mod \xar{C^{\sinf}(\fh,\fk;-)}
\Vect 
\]

\noindent is canonically isomorphic to the functor
$C^{\dot}(\fk,-) \coloneqq \ul{\Hom}_{\fk\mod}(k,-):\fk\mod \to \Vect$ of Lie algebra cohomology.

\end{lem}

\begin{proof}

The induction functor 
$\ind_{\fk}^{\fh_{-Tate}}$ is dual to $\Oblv:\fh\mod \to \fk\mod$
by construction of $\bD_{\fh,K}^{\sinf}$. 
We now obtain the result by duality.

\end{proof}

\subsection{}

We now prove Theorem \ref{t:sinf-comparison}.

Let $\fh_{-Tate+Tate_{std}}$ denote the Baer
sum central extension of $\fh_{-Tate}$ and $\fh_{Tate_{std}}$.
We maintain our abuse of notation regarding modules
over central extensions: the category
$\fh_{-Tate+Tate_{std}}\mod$ is set up so the element 
$1 \in k \subset \fh_{-Tate+Tate_{std}}$ acts by
the identity on any object of it.\footnote{This does not
of course characterize the category. One can work with
group indschemes and central extensions by $\bG_m$ as above
to give one quick definition. Alternatively, one can note
that the centrality means $\QCoh(\bA^1) = k\mod$ (regarding
$k$ as an abelian Lie algebra in this notation)
acts canonically on the category of 
\emph{all} $\fh_{-Tate+Tate_{std}}$-modules, and we are
taking the fiber of that category at $1 \in \bA^1(k)$.}
Note that this central extension is canonically split over
$\fk$; in particular, we have a forgetful functor
$\Oblv:\fh_{-Tate+Tate_{std}}\mod \to \fk\mod$.

The functor 
$C_{std}^{\sinf}(\fh,\fk;-):\fh_{-Tate_{std}}\mod \to \Vect$
defines by the duality $\bD_{\fh,K}^{\sinf}$
an object $K \in \fh_{-Tate+Tate_{std}}\mod$.\footnote{As
always, this notation abusively indicates that the
central element $1 \in k \subset \fh_{-Tate+Tate_{std}}$
acts by the identity on our modules, understood in the
appropriately derived sense.}

Combining Lemmas \ref{l:sinf-shapiro-std} and \ref{l:sinf-shapiro},
we find that $\Oblv(K) = k \in \fk\mod$, where
$k \in \fk\mod$ indicates the trivial module. 
In particular, as $\Oblv$ is $t$-exact and conservative,
$K$ lies in $\fh_{-Tate+Tate_{std}}\mod^{\heart}$
and corresponds to a 1-dimensional representation.

Now observe that giving a 1-dimensional representation
of a central extension $0 \to k \to \fh^{\flat} \to \fh \to 0$
(on which the central element acts by the identity)
is equivalent to splitting the central extension:
the induced map $\fh^{\flat} \to \fh \times k$ is an
isomorphism of central extensions of $\fh$.
Under this splitting, the given 1-dimensional representation
of $\fh^{\flat}$ maps to the trivial representation of
$\fh$.

Therefore, we obtain a trivialization of the
central extension $\fh_{-Tate+Tate_{std}}$ of $\fh$
such that $K$ maps to the trivial object $k \in \fh\mod$.
This is equivalent to giving an isomorphism 
$\fh_{-Tate} \simeq \fh_{-Tate_{std}}$ of central
extensions such that the functor

\[
\fh_{-Tate}\mod \simeq \fh_{-Tate_{std}}\mod 
\xar{C_{std}^{\sinf}(\fh,\fk;-)}
\Vect
\]

\noindent matches under the duality $\bD_{\fh,K}^{\sinf}$
to the trivial module $k \in \fh\mod$. 
This gives the desired isomorphism 
$C^{\sinf}(\fh,\fk;-) \simeq C_{std}^{\sinf}(\fh,\fk;-)$.

Uniqueness follows as the map:

\[
\Aut_{\fh\mod}(k) \to \Aut_{\fk\mod}(k)
\]

\noindent is an isomorphism (both sides are $k^{\times}$,
considered as group objects in $\Set \subset \Gpd$).

\section{Harish-Chandra data}\label{s:hc}

\subsection{}

Suppose $H$ is an algebraic group and $A \in \Alg$ is
equipped with an action of $H$, giving a weak action of
$H$ on $A\mod$. It is not difficult to see in this setup that
upgrading\footnote{Although the forgetful functor
$A\mod \to \Vect$ is weakly $H$-equivariant, the ``upgrade"
in question does not interact with the forgetful functor.
For example, $H$ acts strongly on $\fh\mod$, but the forgetful
functor $\fh\mod \to \Vect$ is only weakly equivariant.} 
this action to a strong one is equivalent to
specifying a \emph{Harish-Chandra} datum in a suitable
derived sense. 

We remind that this means we are given an $H$-equivariant
map of Lie algebras $i:\fh \to A$ satisfying a number
of ``compatibilities" (which are actually extra
data in a derived setting), most notably, that the corresponding
adjoint action of $\fh$ on $A$ coincides with the infinitesimal
action of $H$ on $A$. 

\subsection{}

The goal for this section is to develop such ideas in the
setting where $H$ is a Tate group indscheme and $A$ is
an $\overset{\rightarrow}{\otimes}$-algebra. 

There are a number of subtleties compared to the
finite-dimensional setting discussed above related to the ideas
developed so far in this text.

First, $A$ needs to be
equipped with a renormalization datum compatible
with the $H$-action in the sense of \S \ref{ss:coact-ren},
and with the Harish-Chandra data in a suitable sense.

Second, we need to upgrade the naive action of $H$ on 
$A\mod_{ren}$ to a genuine one. We do this using the
theory of canonical renormalization from 
\S \ref{ss:can-renorm-tate}. 

With that said, 
the theory we develop has no\footnote{Although the 
data is 1-categorical in nature, checking that
an apparent Harish-Chandra datum actually defines
one in our sense involves non-trivial homological algebra
(as we will see).} homotopical complexity
for $A$ and $H$ classical.\footnote{In fact, the theory
essentially requires $H$ to be classical from the start.
More precisely, we require $H$ to be formally smooth, which
forces $H$ to be classical under mild countability assumptions;
see \cite{indschemes}.} 

The main example to have in mind is $A = U(\fh)$, the 
completed enveloping algebra of $\fh$ (i.e., 
the $\overset{\rightarrow}{\otimes}$-algebra assigned to the
$t$-exact functor $\fh\mod \coloneqq \Vect^{\exp(\fh),w} \to \Vect$
via Proposition \ref{p:alg-vs-cats}). 

For the above to make sense, 
we need a key technical result, Theorem \ref{t:gen-hc},
that (in particular) says 
that the genuine weak action of $H$ on $\fh\mod$
comes from canonical renormalization.
We need to impose 
two hypotheses on the group indschemes $H$ to obtain this result:
that $H$ is polarizable,
and that it \emph{has a prounipotent tail}, i.e., 
there exists a prounipotent compact open subgroup in $H$.
Therefore, these hypotheses trail us throughout this section.
We remark that they are satisfied in the main example
of interest: when $H$ is the loop group of a reductive group
(or a central extension of such). 

\subsection{}

As this section is lengthy, we begin with a brief
guide to its structure.

In \S \ref{ss:gen-act-alg}-\ref{ss:gen-ff-pf},
we introduce the notion of genuine $H$-action on 
an $\overset{\rightarrow}{\otimes}$-algebra $A$; roughly,
this means there is a genuine $H$-action
on $A\mod_{ren}$ defined by canonical renormalization.

In \S \ref{ss:gen-constr}, we formulate Theorem \ref{t:gen-hc},
which was mentioned above. 
The proof occupies \S \ref{ss:gen-constr}-\ref{ss:gen-hc-pf}.

In \S \ref{ss:hc-defin}, we formulate our definition of
Harish-Chandra data, which relies on Theorem \ref{t:gen-hc}.

Finally, in \S \ref{ss:cl-hc-1}-\ref{ss:cl-hc-2}, we discuss
Harish-Chandra data explicitly in the case where
$A$ is classical. 

\subsection{Genuine actions and $\overset{\rightarrow}{\otimes}$-algebras}\label{ss:gen-act-alg}

In what follows, let $H$ 
be an ind-affine Tate group indscheme. (These hypotheses
will be strengthened in \S \ref{ss:hc-start}.)

Recall the notation $\Alg_{ren}^{\overset{\rightarrow}{\otimes},H\actson}$
from \S \ref{ss:h-acts-ren}: this is the category of 
renormalized $\overset{\rightarrow}{\otimes}$-algebras with
naive $H$-actions that are compatible with the renormalization.

\begin{defin}

The 1-full subcategory:

\[
\presup{\prime}\Alg_{gen}^{\overset{\rightarrow}{\otimes},H\actson}
\subset \Alg_{ren}^{\overset{\rightarrow}{\otimes},H\actson}
\]

\noindent of \emph{$\overset{\rightarrow}{\otimes}$-algebras with
nearly genuine $H$-actions} has objects 
$A \in \Alg_{ren}^{\overset{\rightarrow}{\otimes},H\actson}$
such that the naive action of $H$ on $A\mod_{ren}$ 
(as in \S \ref{ss:h-acts-ren}) canonically renormalizes
(in the sense of \S \ref{ss:can-renorm-tate})
with respect to the given $t$-structure
on $A\mod_{ren}$.

Morphisms $A_1 \to A_2$ 
in $\presup{\prime}\Alg_{gen}^{\overset{\rightarrow}{\otimes},H\actson}$
are morphisms in $\Alg_{ren}^{\overset{\rightarrow}{\otimes},H\actson}$
with the property that for every $K \subset H$ a compact open
subgroup, the functor:

\[
A_2\mod_{ren}^{K,w} \to A_1\mod_{ren}^{K,w} 
\]

\noindent obtained by ind-extension from:

\[
A_2\mod_{ren}^{K,w,c} \to A_2\mod_{ren}^{K,w,naive,+} 
\to A_1\mod_{ren}^{K,w,naive,+} \simeq A_1\mod_{ren}^{K,w,+} \subset 
A_1\mod_{ren}^{K,w}
\]

\noindent is $t$-exact (equivalently, left $t$-exact).

Finally, we define:

\[
\Alg_{gen}^{\overset{\rightarrow}{\otimes},H\actson}
\subset
\presup{\prime}\Alg_{gen}^{\overset{\rightarrow}{\otimes},H\actson}
\]

\noindent of \emph{$\overset{\rightarrow}{\otimes}$-algebras with
genuine $H$-actions}
as the full subcategory consisting of
objects $A \in \presup{\prime}\Alg_{gen}^{\overset{\rightarrow}{\otimes},H\actson}$ such that the unit morphism
$k \to A \in \Alg_{ren}^{\overset{\rightarrow}{\otimes},H\actson}$
is a morphism in the 1-full subcategory
$\presup{\prime}\Alg_{ren}^{\overset{\rightarrow}{\otimes},H\actson}$. (In other words, the forgetful functor
$A\mod_{ren} \to \Vect$ induces a $t$-exact
functor $A\mod_{ren}^{K,w} \to \Rep(K)$ for any 
compact open subgroup $K\subset H$.)

\end{defin}

\begin{rem}\label{r:algs-1-ff}

By construction, each of the restriction functors:

\[
\Alg_{gen}^{\overset{\rightarrow}{\otimes},H\actson} \to 
\Alg_{ren}^{\overset{\rightarrow}{\otimes},H\actson} \to 
\Alg^{\overset{\rightarrow}{\otimes},H\actson}
\]

\noindent is 1-fully-faithful; the former is in addition conservative.
(This is an abstract way of saying that a genuine action of
$H$ on $A$ is equivalent to specifying a naive action and a renormalization
datum for $A$ satisfying some properties, and that genuinely equivariant
morphisms are naively $H$-equivariant morphisms satisfying some
properties.)

\end{rem}

\begin{defin}

For $A_1,A_2 \in \Alg_{gen}^{\overset{\rightarrow}{\otimes},H\actson}$,
we say a morphism $f:A_1 \to A_2$ in
$\Alg_{gen}^{\overset{\rightarrow}{\otimes},H\actson}$
is a \emph{genuinely $H$-equivariant morphism}.
We refer to a morphisms in 
$\Alg^{\overset{\rightarrow}{\otimes},H\actson}$
as \emph{naively $H$-equivariant}.

\end{defin}

\subsection{}

There is an evident functor:

\[
\begin{gathered}
(\Alg_{gen}^{\overset{\rightarrow}{\otimes},H\actson})^{op} \to
H\mod_{weak} \\
A \mapsto A\mod_{ren}
\end{gathered}
\]

\noindent given by canonical renormalization.
Following our standard abuses for genuine $H$-actions,
we denote this functor $A \mapsto A\mod_{ren}$.
Moreover, as $k \in \Alg_{gen}^{\overset{\rightarrow}{\otimes},H\actson}$ is an initial object
(by fiat in our definition of genuine $H$-action),
this functor upgrades to
a functor to the overcategory $(H\mod_{weak})_{/\Vect}$:
for $A \in \Alg_{gen}^{\overset{\rightarrow}{\otimes},H\actson}$,
the structural map $A\mod_{ren} \to \Vect$ is the
forgetful functor.

Let $\Alg_{conv,gen}^{\overset{\rightarrow}{\otimes},H\actson}
\subset \Alg_{gen}^{\overset{\rightarrow}{\otimes},H\actson}$
be the full subcategory consisting of those
objects whose underlying $\overset{\rightarrow}{\otimes}$-algebra
is convergent.

\begin{thm}\label{t:gen-ff}

For $H$ polarizable, the functor:

\[
(\Alg_{conv,gen}^{\overset{\rightarrow}{\otimes},H\actson})^{op} \to
(H\mod_{weak})_{/\Vect}
\]

\noindent is 1-fully-faithful and conservative.

\end{thm}

We defer the proof to \S \ref{ss:gen-ff-pf}.

\begin{rem}

Although the definition of the category
$\Alg_{conv,gen}^{\overset{\rightarrow}{\otimes},H\actson}$
is weighty, this result gives a way to convert
algebraic data 
in $\Alg^{\overset{\rightarrow}{\otimes}}$ that may be quite
concrete to abstract categorical data 
involving genuine $H$-actions.

\end{rem}

\subsection{}

To prove Theorem \ref{t:gen-ff}, we will need the following result.

\begin{prop}\label{p:gen-nat-trans}

Let $H$ be a polarizable Tate group indscheme. 
Let $\sC,\sD \in H\mod_{weak}$ be equipped with 
$t$-structures compatible with the weak $H$-actions.
Suppose the genuine $H$-actions on each of 
$\sC$ and $\sD$ are obtained by
canonical renormalization using these $t$-structures 
and the underlying naive $H$-actions 
(c.f. \S \ref{ss:can-renorm-tate}).

Let $F,G:\sC \to \sD \in H\mod_{weak}$ be two genuinely
$H$-equivariant functors, and suppose
that $G$ is left $t$-exact
(at the level of its underlying functor 
$\sC \to \sD \in \DGCat_{cont}$). 

Then the natural map:

\[
\ul{\Hom}_{\Hom_{H\mod_{weak}}}(F,G) \to 
\ul{\Hom}_{\Hom_{H\mod_{weak,naive}}}(F,G) \in \Vect
\]

\noindent (induced by $\Oblv_{gen}$) is an equivalence. 
In other words, giving a genuinely $H$-equivariant 
natural transformation between
$F$ and $G$ is equivalent to giving a naively $H$-equivariant
natural transformation between them.

\end{prop}

We will use the following lemma.

\begin{lem}\label{l:av-!-gen/naive}

Let $H$ be a Tate group indscheme and let $K \subset H$ be
a polarization of $H$. Then the forgetful functors:

\[
\begin{gathered} 
\sC^{H,w} \to \sC^{K,w} \\
\sC^{H,w,naive} \to \sC^{K,w,naive}
\end{gathered}
\]

\noindent admit left adjoints, denoted $\Av_!^w$ and
$\Av_!^{w,naive}$ respectively. Moreover,
the diagram:

\[
\xymatrix{
\sC^{K,w} \ar[r] \ar[d]^{\Av_!^w} & \sC^{K,w,naive} \ar[d]^{\Av_!^{w,naive}} \\
\sC^{H,w} \ar[r] & \sC^{H,w,naive}
}
\]

\noindent commutes (a priori, it commutes up to a natural
transformation).

\end{lem}

\begin{proof}

The existence of $\Av_!^w$, as we have appealed to at
various points earlier in this text, follows from
\eqref{eq:inv-coinv-hecke} and the Beck-Chevalley formalism.

Let $\Phi:H\mod_{weak,naive} \to H\mod_{weak}$ denote
the (non-continuous) right adjoint to $\Oblv_{gen}$.
Clearly $\Phi(-)^{H,w} = (-)^{H,w,naive}$.
Moreover, passing to right adjoints in Lemma \ref{l:induction}
\eqref{i:induction-comm} it follows that
$\Phi(-)^{K,w} = (-)^{K,w,naive}$. Now the commutativity
of the diagram follows
by rewriting it as:

\[
\xymatrix{
\sC^{K,w} \ar[r] \ar[d]^{\Av_!^w} & \Phi(\Oblv_{gen}(\sC))^{K,w} \ar[d]^{\Av_!^w} \\
\sC^{H,w} \ar[r] & \Phi(\Oblv_{gen}(\sC))^{H,w}. 
}
\]

(Alternatively, the base-change follows directly by applying the
Beck-Chevalley formalism in the naive setting
and comparing with \eqref{eq:inv-coinv-hecke}.)

\end{proof}

\begin{proof}[Proof of Proposition \ref{p:gen-nat-trans}]

In what follows, let $K \subset H$ be a polarization.

\step Let $\TwoHom(\sC,\sD) \in H\mod_{weak}$ denote
the inner $\Hom$ object between $\sC$ and $\sD$ in 
the symmetric monoidal category $H\mod_{weak}$. 

Note that $\Oblv_{gen}\TwoHom(\sC,\sD)$ is the
category $\TwoHom(\sC,\sD) \coloneqq \TwoHom_{\DGCat_{cont}}(\sC,\sD)$ of continuous DG functors between $\sC$ and $\sD$.
Indeed, this follows from the fact that $\Oblv_{gen}$
admits a $\DGCat_{cont}$-continuous 
left adjoint $- \otimes \IndCoh^*(H)$
(for $\IndCoh^*(H) \in H\mod_{weak}$ as in 
Example \ref{e:indcoh*-coreps}).

Similarly, formation of inner Homs is intertwined by
the forgetful functor $H\mod_{weak} \to K\mod_{weak}$:
this follows from the existence of the left adjoint 
$\ind_K^{H,w}$ from Lemma \ref{l:induction} and the
(evident) version of the projection formula for this left adjoint.

Clearly $\TwoHom(\sC,\sD)^{H,w} = \TwoHom_{H\mod_{weak}}(\sC,\sD)$.
By the above, we just as well have
$\TwoHom(\sC,\sD)^{K,w} = \TwoHom_{K\mod_{weak}}(\sC,\sD)$. 
Finally, because $\Oblv^{gen}\TwoHom(\sC,\sD)$ is the
category of functors between $\sC$ and $\sD$, we have:

\[
\begin{gathered}
\TwoHom(\sC,\sD)^{H,w} = \TwoHom_{H\mod_{weak,naive}}(\sC,\sD) \\
\TwoHom(\sC,\sD)^{K,w} = \TwoHom_{K\mod_{weak,naive}}(\sC,\sD).
\end{gathered}
\]

\step\label{st:left-t-comonad}

Consider $G \in \TwoHom(\sC,\sD)^{K,w}$ as above
(though it lifts to $H$-invariants).

Let $\Oblv:\TwoHom(\sC,\sD)^{K,w} \to \TwoHom(\sC,\sD) \in 
\DGCat_{cont}$ be the forgetful functor, and let
$\Av_*^{K,w}$ denote its right adjoint.
We claim that the natural map 
$G \to \Tot((\Av_*^{K,w}\Oblv)^{\dot + 1} G) \in 
\TwoHom(\sC,\sD)^{K,w}$
is an equivalence. Here we emphasize that the totalization
is calculated in the DG category 
$\TwoHom(\sC,\sD)^{K,w}$.

Indeed, we have:

\begin{equation}\label{eq:cpt-homs}
\begin{gathered}
\TwoHom(\sC,\sD)^{K,w} = 
\TwoHom_{K\mod_{weak}}(\sC,\sD) = 
\TwoHom_{\Rep(K)\mod}(\sC^{K,w},\sD^{K,w}) = \\
\TwoHom_{\Rep(K)^c\mod(\DGCat)}(\sC^{K,w,c},\sD^{K,w})
\end{gathered}
\end{equation}

\noindent for $\sC^{K,w,c} \subset \sC^{K,w}$ the subcategory
of compact objects: we remind that as part of the
definition of $\sC$ being obtained from canonical renormalization, 
$\sC^{K,w}$ is compactly generated with compact objects
being eventually coconnective.

The advantage of the last expression in \eqref{eq:cpt-homs} is that
limits are manifestly computed termwise.
So for $\sF \in \sC^{K,w,c}$, we have:

\[
\Tot((\Av_*^{K,w}\Oblv)^{\dot + 1} G)(\sF) = 
\Tot((\Av_*^{K,w}\Oblv)^{\dot + 1} G(\sF)) \in \sD^{K,w}.
\]

\noindent By Proposition \ref{p:gen-t-str} \eqref{i:gen-t-3},
the natural map from $G(\sF)$ to this limit is an equivalence.

\step\label{st:into-g}

As an immediate consequence of Step \ref{st:left-t-comonad},
note that for any $\widetilde{F}:\sC \to \sD$ genuinely 
$K$-equivariant, the map:

\[
\ul{\Hom}_{\Hom_{K\mod_{weak}}}(\widetilde{F},G) \to 
\ul{\Hom}_{\Hom_{K\mod_{weak,naive}}}(\widetilde{F},G)
\]

\noindent is an isomorphism. 

\step 

We now complete the argument. We again 
view $\Hom(\sC,\sD) \in H\mod_{weak}$ and
$F,G$ as objects in $\Hom(\sC,\sD)^{H,w}$.

Note that the forgetful functor $\Oblv:\Hom(\sC,\sD)^{H,w} \to 
\Hom(\sC,\sD)^{K,w}$ is conservative (by \eqref{eq:inv-coinv-hecke})
and admits a 
left adjoint $\Av_!^w$ (c.f. Lemma \ref{l:av-!-gen/naive}).
Therefore, this forgetful functor is monadic. The same
applies in the naively equivariant setting.

We obtain that $F$ is a geometric realization:

\[
|(\Av_!\Oblv)^{\dot+1}(F)| \isom F
\in \Hom(\sC,\sD)^{H,w} = \Hom_{H\mod_{weak}}(\sC,\sD).
\]

Therefore, we have:

\[
\begin{gathered}
\ul{\Hom}_{\Hom_{H\mod_{weak}}(\sC,\sD)}(F,G) \isom 
\Tot \ul{\Hom}_{\Hom_{K\mod_{weak}}(\sC,\sD)}(\Oblv
(\Av_!\Oblv)^{\dot}(F),G) \isom \\
\Tot \ul{\Hom}_{\Hom_{K\mod_{weak,naive}}(\sC,\sD)}(\Oblv
(\Av_!\Oblv)^{\dot}(F),G).
\end{gathered}
\]

\noindent Here we are using Step \ref{st:into-g} in the
second isomorphism, and we are implicitly
using Lemma \ref{l:av-!-gen/naive}
to intertwine $\Av_!$ functors in the genuine and naive settings.
By the same logic, 
the last term above computes $\ul{\Hom}_{H\mod_{weak,naive}}(F,G)$,
giving the claim.

\end{proof}

\subsection{}\label{ss:gen-ff-pf}

As promised, we now prove the above theorem.

\begin{proof}[Proof of Theorem \ref{t:gen-ff}]

First, let us verify that the functor is conservative.
The composition of this functor with the
forgetful functors:

\[
(H\mod_{weak})_{/\Vect} \to H\mod_{weak} \xar{\Oblv_{gen}} 
\DGCat_{cont}
\]

\noindent send 
$A \in \Alg_{gen}^{\overset{\rightarrow}{\otimes},H\actson}$ to 
$A\mod_{ren}$. This functor is conservative by 
Remark \ref{r:ren-cat}, giving the claim.

Now for $A_1,A_2 \in \Alg_{gen}^{\overset{\rightarrow}{\otimes},H\actson}$, we have the following commutative
diagram:

\begin{equation}\label{eq:gen-ff}
\vcenter{
\xymatrix{
\Hom_{\Alg_{gen}^{\overset{\rightarrow}{\otimes},H\actson}}(A_1,A_2)
\ar[r] \ar[d] & 
\Hom_{(H\mod_{weak})_{/\Vect}}^{\prime}(A_2\mod_{ren},A_1\mod_{ren}) 
\ar[d] \\
\Hom_{\Alg_{ren}^{\overset{\rightarrow}{\otimes},H\actson}}(A_1,A_2) 
\ar[r] &
\Hom_{(H\mod_{weak,naive})_{/\Vect}}^{\prime}(A_2\mod_{ren},A_1\mod_{ren}).
}}
\end{equation}

\noindent Here the decoration $\prime$ on the bottom left term
indicates the subcategory of those functors that are $t$-exact
after applying $(-)^{K,w}$ for any compact open subgroup
$K$, while the similar notation on the bottom right term
indicates the subcategory of $t$-exact functors.
We wish to show that the top arrow in 
\eqref{eq:gen-ff} is fully-faithful.
We will do so by showing that the other three arrows are
fully-faithful.

The left arrow of \eqref{eq:gen-ff} is fully-faithful by definition
of genuine $H$-actions.

The right arrow of \eqref{eq:gen-ff} is fully-faithful by Proposition
\ref{p:gen-nat-trans}.

Finally, the bottom arrow of \eqref{eq:gen-ff} is 
an equivalence by Remark \ref{r:ren-cat}; indeed,
by \emph{loc. cit}. (or more precisely, by
Proposition \ref{p:alg-vs-cats} and Theorem \ref{t:tens-ren}), 
the functor:

\[
\begin{gathered}
(\Alg_{conv,ren}^{\overset{\rightarrow}{\otimes}})^{op} \to
(\DGCat_{cont})_{/\Vect} \\
A \mapsto A\mod_{ren}
\end{gathered}
\]

\noindent is symmetric monoidal and 
fully-faithful.\footnote{Explicitly, the
essential image of this functor
consists of those compactly generated
DG categories $\sC$ equipped with $F: \sC \to \Vect$
such that for the $t$-structure on $\sC$ defined by
having $\sC^{\leq 0}$ generated under colimits by
compact objects $\sF \in \sC^c$ such that 
$F(\sF) \in \Vect^{\leq 0}$, the functor $F$ is $t$-exact
and conservative on $\sC^+$, and such that compact objects
in $\sC$ are eventually coconnective.}

\end{proof}

\subsection{A construction of genuine $H$-actions}\label{ss:gen-constr}

We now formulate a key result that allows us to construct
many genuine $H$-actions on $\overset{\rightarrow}{\otimes}$-
algebras.

Suppose $A \in \Alg_{conv,gen}^{\overset{\rightarrow}
{\otimes},H\actson}$ be given. 
We have the corresponding object $A\mod_{ren} \in H\mod_{weak}$.

We define: 

\[
A\# U(\fh)\mod_{ren} \coloneqq
\Oblv^{str\to w}(A\mod_{ren}^{\exp(\fh),w}) \in H\mod_{weak}.
\]

\noindent (The notation will be justified in what follows.)
Note that this category has a canonical 
genuinely $H$-equivariant functor to $\Vect$:

\[
A\# U(\fh)\mod_{ren} =
\Oblv^{str\to w}(A\mod_{ren}^{\exp(\fh),w}) \to 
A\mod_{ren} \to \Vect
\]

\noindent where the first functor is the counit for the 
adjunction and the second arrow is the standard forgetful
functor for $A\mod_{ren}$.

\begin{thm}\label{t:gen-hc}

Suppose that $H$ is formally smooth Tate group indscheme
that is polarizable,
and has a prounipotent tail. Then 
$A\# U(\fh)\mod_{ren} \in (H\mod_{weak})_{/\Vect}$ 
lies in the essential image of the functor
from Theorem \ref{t:gen-ff}.

\end{thm}

\begin{example}

Taking $A = k$, this result is already quite non-trivial:
it says that under the above hypotheses, 
the naive $H$-action on $\fh\mod$ canonically
renormalizes (with respect to the standard $t$-structure), 
and that $\fh\mod \coloneqq 
\Oblv^{str\to w}(\Vect^{\exp(\fh),w}) \in H\mod_{weak}$
is its canonical renormalization. 

\end{example}

The proof of Theorem \ref{t:gen-hc} is involved, so
is deferred to \S \ref{ss:gen-hc-pf} so that we may first give 
some preliminary results.

\begin{rem}

To orient the reader, let us briefly discuss the
importance of Theorem \ref{t:gen-hc} for defining
Harish-Chandra data.
Assume the result for now.
Of course, we should let $A\# U(\fh)$ denote
the corresponding object of 
$\Alg_{conv,gen}^{\overset{\rightarrow}
{\otimes},H\actson}$. 

As the notation indicates, 
$A\# U(\fh)$ should be understood as the usual
smash product. Then a morphism
$A\# U(\fh) \to A$ restricting to the identity
along the canonical embedding $A \to A\# U(\fh)$
is the ``main" part of a Harish-Chandra datum
(i.e., the rest of the data is interpreted as homotopy
compatibilities). We refer to \S \ref{ss:cl-hc-1}-\ref{ss:cl-hc-2}
for more explicit discussion.

\end{rem}

\begin{rem}

Before proving the theorem, we do not make explicit reference
to the $\overset{\rightarrow}{\otimes}$-algebra
$A\# U(\fh)$, only its category of modules.
That is, we treat $A\# U(\fh)\mod_{ren}$ as
alternative notation to 
$\Oblv^{str\to w}(A\mod_{ren}^{\exp(\fh),w})$.
We let $\Oblv:A\# U(\fh)\mod_{ren} \to \Vect$ denote
the forgetful functor constructed above.

\end{rem}

\subsection{$t$-structures}

As Theorem \ref{t:gen-hc} concerns canonical renormalization
and therefore $t$-structures,
it is convenient to have some convenient language
regarding $t$-structures in the presence of $H$-actions.

Therefore, we begin with an extended digression on this 
subject. The reader may safely skip ahead to 
\S \ref{ss:smash-str} and refer back as needed.

\subsection{}\label{ss:tate-t-str}

Suppose $H$ is a Tate group indscheme and 
$\sC \in H\mod_{weak,naive}$. 

Note that the
action functor:

\[
\act_{\sC}:\IndCoh^*(H) \otimes \sC \to \sC
\]

\noindent lifts canonically as:

\[
\xymatrix{
\IndCoh^*(H) \otimes \sC \ar@{..>}[rr]^{\alpha_{\sC}} 
\ar[dr]^{\act_{\sC}}
& & \IndCoh^*(H) \otimes \sC 
\ar[dl]_{\Gamma^{\IndCoh}(H,-) \otimes \id_{\sC}} \\
& \sC
}
\]

\noindent with $\alpha_{\sC}$ an equivalence. Indeed,
viewing $\IndCoh^*(H)$ as a coalgebra in $\DGCat_{cont}$
(via pushforwards along diagonal maps, as works for any
strict indscheme), there is a unique map of $\IndCoh^*(H)$-comodules
$\alpha_{\sC}$ fitting into a diagram as above. That this functor
is an equivalence follows from the case $\sC = \IndCoh^*(H)$, 
where it is follows from strictness of $H$.

Suppose now that $\sC$ is equipped with a $t$-structure. 
We say this $t$-structure
is \emph{compatible} with the (naive, weak) 
action of $H$ on $\sC$ if it is compatible with
filtered colimits and $\alpha_{\sC}$
is $t$-exact when both sides are equipped with the tensor
product $t$-structures (as in Lemma \ref{l:cpt-tstr}).

\begin{example}\label{e:h-ren-t-str}

Suppose that $A \in \Alg_{ren}^{\overset{\rightarrow}{\otimes},H\actson}$. Then the induced
naive, weak $H$-action on $A\mod_{ren}$ is compatible
with the $t$-structure. Indeed, this is a restatement of the 
definition of compatibility between an $H$-action
and a renormalization datum.

\end{example}

\subsection{} We now move to discuss $t$-structures
in the presence of \emph{strong} $H$-actions.

We will need the following result.

\begin{lem}\label{l:str-weak-inv}

Suppose $K$ is a classical affine group scheme.
Suppose $\sC \in K\mod$ is a DG category 
equipped with a strong $K$-action.

Let $K = \lim_j K/K_j$ with $K/K_j$ an algebraic group
(so $K_j \subset K$ is a normal compact open subgroup).

Below, we also write $\sC$ for the induced object
of $K\mod_{weak}$ under $\Oblv^{str \to w}$.

\begin{enumerate}

\item The natural functor
$\colim_j (\sC^{K_j})^{K/K_j,w} \to \sC^{K,w} \in \DGCat_{cont}$ 
is an equivalence.

\item Each of the structural functors in this colimit
admits a continuous right adjoint.

\end{enumerate}

\end{lem}

\begin{proof}

By construction, we have:

\[
\fk\mod = \underset{i}{\colim} \, (\fk/\fk_i)\mod = 
\underset{i}{\colim} \, D(K/K_i)^{K/K_i,w} \in \DGCat_{cont}
\]

\noindent as $(D(K),\Rep(K))$-bimodules. Moreover, each
structural functor in this colimit clearly admits
a continuous right adjoint (calculated as 
Lie algebra cohomology for
the appropriate finite-dimensional Lie algebra), 
and that right adjoint is
a morphism of $(D(K),\Rep(K))$-bimodules. 
This gives the claim by construction of $\Oblv^{str\to w}$.

\end{proof}

\subsection{}\label{ss:str-t-str}

Now suppose $H$ is a Tate group indscheme with a prounipotent 
tail, and that $\sC \in H\mod$ is acted on strongly
by $H$ and is equipped with a $t$-structure.

We say that this $t$-structure is \emph{strongly compatible} with
the $H$-action if it compatible with the underlying 
weak, naive action of $H$ and for any 
prounipotent compact open subgroup $K \subset H$, the
subcategory $\sC^K \subset \sC$ is closed under truncations. 

\begin{rem}\label{r:weak-str-t-str}

For $H$ an algebraic group, this condition is clearly equivalent
to the $t$-structure being compatible with the underlying
weak action.
As discussed in \cite{whit} \S B.4, this is equivalent
to any other notion of a $t$-structure being compatible
with a strong action essentially because the forgetful functor 
$\Oblv:D(H) \to \QCoh(H)$ is
conservative and $t$-exact up to shift. 

From here, one deduces that in general, if 
a $t$-structure is compatible with the weak naive action 
of a Tate group indscheme $H$, it is strongly compatible
if and only if the above condition holds for some
prounipotent subgroup. 

\end{rem}

\begin{rem}

We do not know of an example of $\sC \in H\mod$ 
as above and a $t$-structure
that is compatible with the weak, naive action of $H$
but not strongly compatible.

\end{rem}

\subsection{}\label{ss:str-t-str-2}

Suppose $\sC \in H\mod$ is equipped with a $t$-structure
that is strongly compatible with the $H$-action.

Observe that for $K \subset H$ any compact open, $\sC^K$ 
inherits a unique $t$-structure
such that $\sC^K \to \sC$ is $t$-exact.
Indeed, if $K$ is prounipotent, this is
true by design; in general, choose $K^u \subset K$ a normal, 
prounipotent compact open subgroup, and then 
observe that the action of the algebraic group 
$K/K^u$ on $\sC^{K^u}$ is compatible with the $t$-structure
there in the sense of \cite{whit} Appendix B. 
By \emph{loc. cit}., we obtain the claim.

More generally, in the above notation, 
each $(\sC^{K_j})^{K/K_j,w}$ inherits a canonical $t$-structure,
and each of the structural functors in the colimit
from Lemma \ref{l:str-weak-inv} is $t$-exact.

Therefore, $\sC^{K,w}$ admits a unique $t$-structure
such that each functor $(\sC^{K_j})^{K/K_j,w} \to \sC^{K,w}$
is $t$-exact: see \cite{whit} Lemma 5.4.3.

\begin{lem}\label{l:t-cons-k}

In the above setting, 
the forgetful functor $\sC^{K,w,+} \to \sC^+$ is conservative.

\end{lem}

\begin{proof}

For every $j$, let 
$\alpha_j:\sC^{K,w} \to (\sC^{K_j})^{K/K_j,w}$
be the right adjoint to the structural functor 
$\beta_j:(\sC^{K_j})^{K/K_j,w} \to \sC^{K,w}$.

By $t$-exactness, it suffices to show that this
functor is conservative on the heart.
For $\sF \in \sC^{K,w,\heart}$, we have
$\sF = \colim_j \beta_j\alpha_j(\sF)$. 
Note that $\beta_j$ is $t$-exact, so $\alpha_j$ is
left $t$-exact. On $H^0$, each structural map
in this colimit is a monomorphism: indeed, 
$H^0(\beta_j\alpha_j(\sF))$ is the maximal subobject
of $\sF$ lying in $(\sC^{K_j})^{K/K_j,w,\heart}$
(which is a full subcategory of $\sC^{K,w,\heart}$ because
we are working with abelian categories).

Now if $\sF \neq 0$, then $H^0(\alpha_j(\sF)) \neq 0$
for some $j$, and as the composition
$(\sC^{K_j})^{K/K_j,w} \to \sC^{K,w} \to \sC$ is
clearly conservative, we obtain the claim.

\end{proof}
 
\begin{cor}\label{c:t-naive}

In the above setting, the natural functor
$\sC^{K,w,+} \to \sC^{K,w,naive,+}$ is an equivalence.

\end{cor}

\begin{proof}

The $t$-exact conservative forgetful functor
$\sC^{K,w,+} \to \sC^+$ is comonadic 
by Lemma \ref{l:cosimp-summand}. The comonads on $\sC$
defined by $\sC^{K,w}$ and $\sC^{K,w,naive}$ always coincide,
so we obtain the claim.

\end{proof}

\subsection{Some results on smash products}\label{ss:smash-str}

We now suppose we are in the setup of Theorem \ref{t:gen-hc}.

\begin{lem}\label{l:smash-str}

\begin{enumerate}

\item\label{i:smash-str-1} $A\# U(\fh)\mod_{ren}$ is compactly generated. 

\item\label{i:smash-str-3} There exists a unique compactly generated $t$-structure on 
$A\# U(\fh)\mod_{ren}$ such that compact objects
are eventually coconnective and such that the forgetful functor
$A\# U(\fh)\mod_{ren} \to \Vect$ is $t$-exact and
conservative on $A\# U(\fh)\mod_{ren}^+$.

\item\label{i:smash-str-4} The above $t$-structure
is strongly compatible with the $H$-action on 
$A\# U(\fh)\mod_{ren}$.

\item\label{i:smash-str-2} 
For any compact open subgroup $K \subset H$,
$A\# U(\fh)\mod_{ren}^{K,w}$ is compactly generated.

\item\label{i:smash-str-5} 

By \eqref{i:smash-str-4} and \S \ref{ss:str-t-str-2},
$A\# U(\fh)\mod_{ren}^{K,w}$ is equipped with a canonical
$t$-structure. This $t$-structure is compactly generated,
and compact objects are bounded
from below. Moreover, the forgetful functor 
$A\# U(\fh)\mod_{ren}^{K,w} \to \Rep(K)$ is $t$-exact.

\item\label{i:smash-str-6}

Suppose $f:A \to B$ is a morphism of 
$\overset{\rightarrow}{\otimes}$-algebras with
genuine $H$-actions. 

Then for any $K \subset H$ compact open, the functor:

\[
B\# U(\fh)\mod_{ren}^{K,w} \to A\# U(\fh)\mod_{ren}^{K,w}
\]

\noindent is $t$-exact (for the $t$-structures from
\eqref{i:smash-str-5}).

\end{enumerate}

\end{lem}

\begin{proof}

We proceed in steps.

\step\label{st:smash-1} 

Recall that by construction, the (genuine) weak $H$-action on
$A\# U(\fh)\mod_{ren}$ canonically upgrades to a strong 
action. We use the same notation $A\# U(\fh)\mod_{ren}$
for the corresponding object of $H\mod$;
in particular, for $K \subset H$ compact open,
we let $A\# U(\fh)\mod_{ren}^K$ denote the strong
$K$-invariants for the action.

We begin by proving that $A\# U(\fh)\mod_{ren}^K$
is compactly generated for any compact open
subgroup $K \subset H$.

Note that by definition, we have: 

\[
A\# U(\fh)\mod_{ren}^K =
(A\mod_{ren}^{\exp(\fh),w})^K = 
A\mod_{ren}^{H_K^{\wedge},w}
\]

\noindent where $H_K^{\wedge}$ is the formal completion
of $H$ along $K$. The forgetful functor:

\[
A\# U(\fh)\mod_{ren}^K = 
A\mod_{ren}^{H_K^{\wedge},w} \to A\mod_{ren}^{K,w}
\]

\noindent is conservative and admits a left adjoint;
as this forgetful functor is manifestly continuous, 
it is monadic.

Because the $H$-action on $A\mod_{ren}$ arises by
canonical renormalization (by assumption on $A$),
$A\mod_{ren}^{K,w}$ is compactly generated.
By the above, we immediately deduce the same for
$A\# U(\fh)\mod_{ren}^K$.

\step\label{st:smash-2}

Note that the monad on $A\mod_{ren}^{K,w}$ constructed
in Step \ref{st:smash-1} is $t$-exact for the
$t$-structure on $A\mod_{ren}^{K,w}$ (coming from
Proposition \ref{p:gen-t-str}). Indeed, as 
$H$ is formally smooth, $H_K^{\wedge}$ is too, so
$\omega_{H_K^{\wedge}/K} \in \IndCoh(H_K^{\wedge}/K)^{K,w}$
lies in the heart of the $t$-structure. The monad in question
is given by convolution with this object.
Convolution by an object in $\Coh(H_K^{\wedge})^{K,w,\heart}$
defines a functor $A\mod_{ren}^{K,w} \to A\mod_{ren}^{K,w}$
that is left $t$-exact up to shift by definition
of canonical renormalization, and it is $t$-exact
by the compatibility of the naive
action of $H$ on $A\mod_{ren}$ with the $t$-structure;
therefore, the same applies to convolution by arbitrary objects
of $\IndCoh(H_K^{\wedge})^{K,w,\heart}$, giving the claim.

It follows that 
$A\mod_{ren}^{H_K^{\wedge},w} = A\# U(\fh)\mod_{ren}^K$
admits a unique $t$-structure such that the (monadic,
with monad the one in question) forgetful functor
to $A\mod_{ren}^{K,w}$ is $t$-exact.

\step\label{st:smash-2.5}

We now deduce \eqref{i:smash-str-1} and \eqref{i:smash-str-3}.

We have:

\[
A\# U(\fh)\mod_{ren} = 
\underset{K \subset H \text{ compact open}}{\colim}
A\# U(\fh)\mod_{ren}^K \in \DGCat_{cont}.
\]

Each structural functor in this colimit admits a
continuous right adjoint so preserves compact objects.
We obtain that the colimit is compactly generated
since each term is by Step \ref{st:smash-1}.  

Moreover, we claim that each of the structural functors
in this colimit is
$t$-exact for the $t$-structures from Step \ref{st:smash-2}.
For $K^{\prime} \subset K \subset H$ compact open
subgroups, we have a commutative diagram of forgetful functors:

\[
\xymatrix{
A\# U(\fh)\mod_{ren}^K \ar[d] \ar[r] & 
A\# U(\fh)\mod_{ren}^{K^{\prime}} \ar[d] \\
A\mod_{ren}^{K,w} \ar[r] & 
A\mod_{ren}^{K^{\prime},w}.
}
\]

\noindent The vertical functors are $t$-exact and 
conservative by construction, and the bottom functor
is clearly $t$-exact, so the claim follows.

Therefore, our (filtered) colimit 
inherits a canonical $t$-structure
such that each functor 
$A\# U(\fh)\mod_{ren}^{K} \to A\# U(\fh)\mod_{ren}$ is
$t$-exact. Let us show that this $t$-structure has
the desired properties from \eqref{i:smash-str-3}.

It is clear from the construction that
this $t$-structure is compactly generated and that
compact objects are eventually coconnective.

For $K$ any compact open subgroup of $H$, the
functor:

\[
A\# U(\fh)\mod_{ren}^{K} \to 
A\mod_{ren}^{K,w} \to A\mod_{ren} \to \Vect
\]

\noindent calculates the forgetful functor.
The first functor in this sequence is $t$-exact and
conservative by design, while the remaining functors
are $t$-exact and conservative on eventually coconnective
subcategories by assumption.
It remains to show this conservativeness survives
passage to the colimit in $K$.

First, suppose $K$ is a fixed prounipotent compact open.
Observe that $A\# U(\fh)\mod_{ren}^{K,\heart}$
admits an explicit description: it is the abelian
category of discrete $H^0(A)$-modules equipped with a suitably
compatible smooth $\fh$-action such that $\fk = \Lie(K)$ acts
locally nilpotently. Indeed, this follows
from the description of the category as modules
over a certain monad on $A\mod_{ren}^{K,w,\heart}$.

In particular, the abelian category
$A\# U(\fh)\mod_{ren}^{K,\heart} \subset
A\# U(\fh)\mod_{ren}^{\heart}$ is closed under
taking subobjects. 

It follows that for $\sF \in A\# U(\fh)\mod_{ren}^{\heart}$,
the map $\sF \to \Oblv\Av_*(\sF)$ induces a monomorphism
on $H^0$; here e.g. $\Av_*$ indicates the functor
of strong $K$-averaging. 
For such $\sF$, note that:

\[
\sF = \underset{K \subset H \text{ compact open}}{\colim}
\, \Oblv\Av_*^K(\sF) = 
\underset{K \subset H \text{ compact open}}{\colim}
\, H^0\Oblv\Av_*^K(\sF)
\]

\noindent with each structural map in the latter colimit
a monomorphism in $A\# U(\fh)\mod_{ren}^{\heart}$.
Therefore, if $\sF$ 
is non-zero, $H^0(\Av_*^K(\sF))$ is non-zero for some $K$.
Now the desired conservativeness follows
from the similar result for $A\# U(\fh)\mod_{ren}^{K}$.

\step\label{st:smash-3}

We now mildly generalize our earlier constructions.

Let $K$ be as before, and let $K_0 \subset K$ be a normal
compact open subgroup.
We will study the category:  

\[
(A\# U(\fh)\mod_{ren}^{K_0})^{K/K_0,w}.
\]

As before, this category identifies with
$(A\mod_{ren}^{H_{K_0}^{\wedge},w})^{K/K_0,w}$.
There is a canonical forgetful functor:

\[
(A\# U(\fh)\mod_{ren}^{K_0})^{K/K_0,w} = 
(A\mod_{ren}^{H_{K_0}^{\wedge},w})^{K/K_0,w} \to 
(A\mod_{ren}^{K_0,w})^{K/K_0,w} = A\mod_{ren}^{K,w}
\]

\noindent which is again monadic. Clearly the corresponding
monad on $A\mod_{ren}^{K,w}$ is again $t$-exact.

Therefore, we once again find that 
$(A\# U(\fh)\mod_{ren}^{K_0})^{K/K_0,w}$ is again
compactly generated, and that it admits a unique
$t$-structure for which the forgetful functor
to $A\mod_{ren}^{K,w}$ is $t$-exact.

\step\label{st:smash-4}

We now show \eqref{i:smash-str-2} from the statement
of the lemma. 

By Lemma \ref{l:str-weak-inv}, and using the
notation of \emph{loc. cit}., we have:

\[
A\# U(\fh)\mod_{ren}^{K,w} = 
\underset{j}{\colim} \, 
(A\# U(\fh)\mod_{ren}^{K_j})^{K/K_j,w} \in \DGCat_{cont}.
\]

Each of the structural functors in this colimit
admits a continuous right adjoint. Therefore,
compact generation of the colimit follows
from compact generation of each term, which we
have shown in Step \ref{st:smash-3}.

\step\label{st:smash-5}

Let us be in the general setup of \S \ref{ss:tate-t-str},
with $H$ acting naively on $\sC$, which is equipped with
a $t$-structure.  
Suppose that the $t$-structure
on $\sC$ is compactly generated. 

We claim that the $t$-structure on $\sC$ 
is compatible with the naive $H$-action if and only if
for every $\sF \in \Coh^*(H)^{\leq 0}$
and $\sG \in \sC^{\leq 0}$ compact, 
$\alpha_{\sC}(\sF \boxtimes \sG)$ is connective.

Indeed, this condition clearly implies 
$\alpha_{\sC}$ is right $t$-exact. 
It suffices to show that  
$(\alpha_{\sC})^{-1}$ is similarly right $t$-exact.
Note that $\alpha_{\sC}^{-1}$ is obtained by conjugating
$\alpha_{\sC}$ by the automorphism
$\on{inv}_*^{\IndCoh} \otimes \id_{\sC}$ for
$\on{inv}:H \isom H$ the inversion map. As this automorphism
is $t$-exact (since $\on{inv}_*^{\IndCoh}$ clearly is), 
we obtain the result.

\step\label{st:smash-6}

We will show that the $t$-structure
on $\sC = A\# U(\fh)\mod_{ren}$ is compatible with the
naive $H$-action using the criterion of Step \ref{st:smash-5}.
Let $\alpha = \alpha_{\sC}$ in what follows.

Let $K \subset H$ be a fixed compact open subgroup and
let $\Av_! = \Av_!^{K \to H_K^{\wedge},w}:A\mod_{ren}^{K,w} \to 
A\# U(\fh)\mod_{ren}^K$ denote the left adjoint to the
forgetful functor. Let $\sG_0 \in A\mod_{ren}^{K,w,\leq 0}$
be compact and let $\sG \coloneqq \Av_!(\sG_0)$.
Note that objects of this form compactly generate
$A\# U(\fh)\mod_{ren}^{\leq 0}$ (letting $K$ vary, of course).

Let $\sF \in \Coh(H)^{\heart}$. 
By Step \ref{st:smash-5} and the above remarks,
it suffices to show that 
$\alpha(\sF \boxtimes \sG)$
is connective.

We will show this in Step \ref{st:smash-8} after
some preliminary constructions.

Let us just note one special case. Suppose $\sF = k_h$
is the skyscraper sheaf at a $k$-point $h$ of $H$.
Then the object in question is 
$\Av_!^{\Ad_h(K) \to H_K^{\wedge},w}(k_h \star \sG_0)$.
Clearly $k_h \star \sG_0 \in A\mod_{ren}^{\Ad_h(K),w,\leq 0}$,
so the same is true after $!$-averaging.

The general argument is similar, but has additional
complications due to working in families.

\step\label{st:smash-7}

We need some constructions involving standard Chevalley-style
constructions. 

Let $K$ be any classical affine group scheme and
let $K_0 \subset K$ be a compact open subgroup.
Suppose $\sC$ is equipped with a genuine $K$-action.

Let $\sH \in \sC^{K_{K_0}^{\wedge},w}$ be given.
We will construct a canonical
\emph{Chevalley} filtration on $\sH$, which is an increasing
filtration with $\gr_i \sH = 
\Av_!^{K_0 \to K_{K_0}^{\wedge},w}
(\Lambda^i(\fk/\fk_0) \star \Oblv(\sH))[i]$.
For clarity: the notation means we forget
$\sH$ down to $\sC^{K_0,w}$, act by
$\Lambda^i(\fk/\fk_0)[i] \in \Rep(K_0)$, then $!$-average
back.

Indeed, there is an action of 
the symmetric monoidal category $\Rep(K_{K_0}^{\wedge}) = 
\fk\mod^{K_0}$ on $\sC^{K_{K_0}^{\wedge}}$.
The trivial representation $k \in \fk\mod^{K_0}$
has a standard filtration with 
$\gr_i(k) = \ind_{\fk_0}^{\fk}(\Lambda^i(\fk/\fk_0))[i]$
(see \cite{grbook} \S IV.5.2 for a much more general
context for such constructions). As $k$ is the
unit for the monoidal structure here,
we obtain a filtration of the desired type
by functoriality.

We will actually need a more general, parametrized 
version of this construction. We sketch the ideas below.

Suppose $S$ is an affine scheme. 
Let $\sK$ be a \emph{compact}\footnote{The terminology
is admittedly bad: it is meant to evoke \emph{compact open},
nothing about properness.}
group scheme over $S$,
meaning an affine group scheme
over $S$ that can be realized as a projective limit
under smooth surjective structure maps of affine group schemes that
are smooth over $S$.

Let $\sK_0 \subset \sK$ be an ($S$-family of)
compact open subgroups of $\sK$, 
meaning $\sK_0$ is compact in the above sense
and we are given 
$\sK_0 \to \sK$
a homomorphism of group schemes over $S$ that is a closed embedding
such that $\sK/\sK_0$ is smooth over $S$ (in particular, of
finite presentation over $S$).

For such $\sK$, there is a symmetric monoidal
category $\Rep(\sK)$, defined as the evident
colimit as in the case where $S$ 
is a point. For example, $\Rep(K \times S) = \Rep(K) \otimes \QCoh(S)$ for $K$ a classical affine group scheme.
Note that $\Rep(\sK)$ is a $\QCoh(S)$-module
category, is a symmetric monoidal category
over $\QCoh(S)$. 

Let $\sC$ be a genuine $K$-category, meaning
we are given $\sC^{K,w}$ a $\Rep(K)$-module category.
We obtain $\sC^{\sK_0,w} \coloneqq \sC^{K,w}
\otimes_{\Rep(\sK)} \Rep(\sK_0)$.

The ideas of \S \ref{s:weak-ind} translate into this
setting in an evident way.
In particular, we 
a parametrized category of Harish-Chandra modules
$\Rep(\sK_{\sK_0}^{\wedge})$, which is equipped with 
a monadic forgetful functor to $\Rep(\sK_0)$.
This allows us to make sense of $\sC^{\sK_{\sK_0}^{\wedge},w}$,
again by tensoring.

In particular, 
the construction of Chevalley filtrations goes through
in this setting.

\step\label{st:smash-8}

We now conclude the argument. We use the notation
from Step \ref{st:smash-6}.

Let $S \subset H$ be a classical affine subscheme on 
which $\sF$ is scheme-theoretically supported.

In what follows, we use a subscript $S$ to indicate
a product with $S$. For example, $H_S = H \times S$.

Let $\sK$ be the group scheme $K \times S$ over $S$;
in what follows,
we \emph{always} 
regard $\sK$ as a (family of) compact open subgroup(s)
of $H_S$ via the map:

\[
\sK = K \times S \xar{(k,h) \mapsto (\Ad_h(k),h)} H \times S = H_S.
\]

By a standard Noetherian descent argument,
there exists a compact open subgroup $K_0 \subset H$
such that $K_{0,S} \subset \sK \subset H_S$.
Note that $\sK/K_{0,S}$ is smooth over $S$ in this case.

Let $A\mod_{ren,S} \coloneqq \QCoh(S) \otimes A\mod_{ren}$.
We similarly have $A\mod_{ren,S}^{K_S,w} = 
A\mod_{ren}^{K,w} \otimes \QCoh(S)$ and
$A\mod_{ren}^{\sK,w}$, with $\alpha_{A\mod_{ren}}$ inducing an 
isomorphism:

\[
A\mod_{ren,S}^{K_S,w} \isom A\mod_{ren,S}^{\sK,w}.
\] 

Regarding $\sF$ as a coherent sheaf on $S$,
$\sF \boxtimes \sG_0 \in A\mod_{ren,S}^{K,w}$ by definition.
Let $\sH \coloneqq 
\alpha_{A\mod_{ren}}(\sF \boxtimes \sG_0) \in 
A\mod_{ren,S}^{\sK,w}$. 

We clearly have:

\[
\alpha(\sF \boxtimes \Av_!^{K \to H_K^{\wedge},w}(\sG_0)) =
\Av_!^{\sK \to H_{K,S}^{\wedge},w}(\sH)
\]

\noindent where the left hand side is what we wish to show
is connective. 

Applying Step \ref{st:smash-7}, we obtain an increasing filtration
on $\sH$ with:

\[
\gr_i \sH = 
\Av_!^{K_{0,S} \to \sK_{K_{0,S}}^{\wedge},w}
(\Lambda^i(\Lie(\sK)/\Lie(K_{0,S})) \star \Oblv(\sH))[i].
\]

\noindent We remark that $\Lie(\sK)/\Lie(K_{0,S})$
is a finite-rank vector bundle on $S$. As in the previous
step,
$\Oblv(\sH)$ indicates that we forget down to
weak $K_{0,S}$-invariants.

Therefore, $\Av_!^{\sK \to H_{K,S}^{\wedge},w}(\sH)$
inherits a filtration in 
$\IndCoh^*(H) \otimes A\mod_{ren}^{H_{K_0}^{\wedge},w}$
with $i$th associated graded term:

\begin{equation}\label{eq:gri-h}
(\id \otimes \Av_!^{K_0 \to H_{K_0}^{\wedge},w})
(\Lambda^i(\Lie(\sK)/\Lie(K_{0,S})) \star \Oblv(\sH))[i]
\in \QCoh(S) \otimes A\mod_{ren}^{H_{K_0}^{\wedge},w}.
\end{equation}

Now observe that $\Oblv(\sH) \in A\mod_{ren,S}^{K_{0,S},w}$
is connective: 
it suffices\footnote{Unlike coconnectivity.}
to check this after
applying the forgetful functor to $A\mod_{ren,S}$
where it is clear.

This clearly that implies \eqref{eq:gri-h} is connective
(for the tensor product $t$-structure),
giving our claim.

\step To complete the proof of \eqref{i:smash-str-4},
it remains to show that
the $t$-structure on $A\# U(\fh)\mod_{ren}$ is
\emph{strongly} compatible with the $H$-action.

Fortunately, this is evident from our constructions:
the $t$-structure on $A\# U(\fh)\mod_{ren}$ was
defined so that:

\[
A\# U(\fh)\mod_{ren}^K = A\mod_{ren}^{H_K^{\wedge},w} \to 
A\# U(\fh)\mod_{ren}
\]

\noindent is $t$-exact for every compact open subgroup $K$.

\step\label{st:smash-9} We now show \eqref{i:smash-str-5}.

The analysis of Step \ref{st:smash-3} clearly shows
that for any $K_0 \subset K$ a normal compact open subgroup,
the $t$-structure on
$(A\# U(\fh)\mod_{ren}^{K_0})^{K/K_0,w}$ is compactly generated
and compact objects are bounded from below. By
construction, this implies the same for
$A\# U(\fh)\mod_{ren}^{K,w}$.

To show that $A\# U(\fh)\mod_{ren}^{K,w} \to \Rep(K)$
is $t$-exact, we claim that it suffices to show this for its
restriction to $(A\# U(\fh)\mod_{ren}^{K_0})^{K/K_0,w}$.
Indeed, right $t$-exactness of this functor is evident
(as the forgetful functor $A\# U(\fh)\mod_{ren} \to \Vect$ 
is $t$-exact and all our forgetful functors are conservative
on eventually coconnective subcategories).

For left $t$-exactness, 
note that if $\sF \in 
A\# U(\fh)\mod_{ren}^{K,w}$, then:

\[
\sF = \underset{j}{\colim} \, L_j(\sF)
\]

\noindent where $L_j$ is the comonad on
$A\# U(\fh)\mod_{ren}^{K,w}$ defined by
the adjunction:

\[
(A\# U(\fh)\mod_{ren}^{K_j})^{K/K_j,w} \rightleftarrows
A\# U(\fh)\mod_{ren}^{K,w}
\]

\noindent using the notation of Lemma \ref{l:str-weak-inv}.
Because the left adjoint
in this adjunction is $t$-exact, $L_j$ is left $t$-exact.
As this colimit is filtered, we obtain the claim.

Now observe that the composition:

\[
(A\# U(\fh)\mod_{ren}^{K_0})^{K/K_0,w} =
(A\mod_{ren}^{H_{K_0}^{\wedge}})^{K/K_0,w} \to 
A\mod_{ren}^{K,w} \to \Rep(K)
\]

\noindent calculates the forgetful functor in question.
The second arrow is $t$-exact by assumption on $A$. 
The same holds for the first arrow is $t$-exact because
$A\mod_{ren}^{H_{K_0}^{\wedge}} \to 
A\mod_{ren}^{K_0,w}$ is $t$-exact by construction, and 
forgetting from $K$-invariants to $K_0$-invariants
is $t$-exact and conservative (because $K/K_0$ is finite type).

\step Finally, it remains to show \eqref{i:smash-str-6}.

First, note that the functor:

\[
B\# U(\fh)\mod_{ren}^K \to A\# U(\fh)\mod_{ren}^K
\]

\noindent is $t$-exact. Indeed, we have a commutative diagram:

\[
\xymatrix{
B\# U(\fh)\mod_{ren}^K \ar[r] \ar[d] & A\# U(\fh)\mod_{ren}^K 
\ar[d] \\
B\mod^{K,w} \ar[r] & A\mod^{K,w}
}
\]

\noindent where, as always, the vertical arrows are given by 
rewriting e.g. $A\# U(\fh)\mod_{ren}^K$
as $A\mod_{ren}^{H_K^{\wedge},w}$. These vertical
arrows are conservative and $t$-exact by construction,
and the bottom horizontal arrow is $t$-exact by definition
of morphism in 
$\Alg_{gen}^{\overset{\rightarrow}{\otimes},H\actson}$.

More generally, we find that for any $K_0 \subset K$
a compact open normal subgroup, the functor:

\[
(B\# U(\fh)\mod_{ren}^{K_0})^{K/K_0,w} \to 
(A\# U(\fh)\mod_{ren}^{K_0})^{K/K_0,w}
\]

\noindent is $t$-exact. We now conclude the result
by the same logic as in Step \ref{st:smash-9}.

\end{proof}

\subsection{Some results on canonical renormalization}

We now develop some general results on canonical 
renormalization in the setting of Tate group indschemes.
The ultimate result is Corollary \ref{c:canon-renorm-criterion},
which gives a convenient way of checking the hypotheses
for canonical renormalization.

\begin{prop}\label{p:canon-recog}

Suppose that $H$ is a Tate group indscheme with a prounipotent
tail. 

Suppose $\sC \in H\mod$ is acted on strongly by $H$ and
equipped with a $t$-structure strongly compatible with the 
$H$-action. 

Suppose that for every compact open subgroup 
$K \subset H$, $\sC^{K,w}$ is compactly generated
with compact objects lying in $\sC^{K,w,+}$
(with respect to the $t$-structure of \S \ref{ss:str-t-str-2}).

Then for every compact open subgroup $K \subset H$,
an object $\sF \in \sC^{K,w}$ is compact if and only
if $\sF \in \sC^{K,w,+}$ and 
$\Oblv(\sF)$ is compact in $\sC$.

\end{prop}

\begin{rem}

We remark that the result is clearly about $K$, and that
$H$ is a bit of a red herring.

Also, note that the conclusion of the lemma may be stated
as $\sC \in K\mod_{weak}$ is obtained by canonical
renormalization (in the sense of \S \ref{ss:can-renorm})
from the underlying naive weak $K$-action
on $\sC$.

\end{rem}

\begin{proof}[Proof of Proposition \ref{p:canon-recog}]

Clearly compact objects satisfy this property,
as $\Oblv$ admits the continuous right adjoint $\Av_*^w$.
Therefore, suppose $\sF \in \sC^{K,w,+}$ with
$\Oblv(\sF)$ compact;
we wish to show that $\sF$ is compact.

Choose $N \gg 0$ such that $\sF \in \sC^{K,w,\geq -N}$.
Note that $\sF$ is compact in $\sC^{K,w,\geq -N}$:
see Step \ref{st:acpt-pres} from the proof of
Lemma \ref{l:coh-t-str}.

Write $\sF$ as a filtered 
colimit $\colim_i \sF_i$ with $\sF_i \in \sC^{K,w}$ compact.
As the $t$-structure is compatible with filtered colimits,
we obtain $\sF = \colim_i \tau^{\geq -N} \sF_i$. 
Because $\sF$ is compact in $\sC^{K,w,\geq -N}$, we see
that $\sF$ is a summand of $\tau^{\geq -N} \sF_i$ for
some $i$.

By Lemma \ref{l:str-weak-inv} (and in the notation
of \emph{loc. cit}.), the map:

\[
\underset{j}\colim \, (\sC^{K_j})^{K/K_j,w,c} \to 
\sC^{K,w,c} \in \DGCat
\]

\noindent is an equivalence, where $K_j$ runs over
compact open prounipotent subgroups of $H$ that are normal in $K$;
as is usual, the superscript $(-)^{c}$ indicates the subcategory
of compact objects. Therefore, $\sF_i$
lifts to $(\sC^{K_j})^{K/K_j,w,c}$ for some index $j$.
As the forgetful functor $(\sC^{K_j})^{K/K_j,w} \to 
\sC^{K,w}$ is $t$-exact, the same is true of 
$\sF$ itself. We abuse notation in letting 
$\sF$ also denote a lift to $(\sC^{K_j})^{K/K_j,w}$.

As the forgetful functor $(\sC^{K_j})^{K/K_j,w} \to 
\sC^{K,w}$ admits a continuous right adjoint,
it suffices to show that $\sF$ is compact as an
object of $(\sC^{K_j})^{K/K_j,w}$. 
Moreover, by Lemma \ref{l:cpt-recognition}, it suffices
to show that $\sF$ is compact after forgetting
to $\sC^{K_j}$. As $K_j$
is prounipotent by assumption, the forgetful functor 
$\sC^{K_j} \to \sC$ is fully-faithful, giving the claim.

\end{proof}

Above, we used the following result.

\begin{lem}\label{l:cpt-recognition}

Let $H$ be an affine algebraic group (in particular,
of finite type) acting weakly
on $\sC$. Then $\sF \in \sC^{H,w}$ is compact if
and only if $\Oblv(\sF) \in \sC$ is compact.

\end{lem}

\begin{proof}

Clearly if $\sF \in \sC^{H,w}$ is compact, then
$\Oblv(\sF) \in \sC$ is compact. Suppose the converse.

Recall from the proof of Lemma \ref{l:recognition} 
that for $\sG \in \sC^{H,w}$, 
$\sG$ is a summand of 
$\Tot^{\leq n}(\Av_*^w\Oblv)^{\dot+1}(\sG)$ for some $n$;
moreover, this is functorial in $\sG$ by the construction
of \emph{loc. cit}. 

Therefore, the functor
$\ul{\Hom}_{\sC^{H,w}}(\sF,-):\sC^{H,w} \to \Vect$
is a summand of:

\[
\Tot^{\leq n} \ul{\Hom}_{\sC}(\Oblv(\sF),
\Oblv(\Av_*^w\Oblv)^{\dot}(-).
\]

\noindent A summand of a
finite limit of continuous functors
is continuous, giving the claim.

\end{proof}

\subsection{}

Next, we show the following result.

\begin{prop}\label{p:canon-left-t}

Suppose that $H$ is a polarizable Tate group indscheme
with a prounipotent tail.
Suppose that $\sC \in H\mod$ is equipped with
a compactly generated 
$t$-structure strongly compatible with the action
of $H$ on $\sC$.

Then for every compact open subgroup $K \subset H$
and every $\sF \in \sH_{H,K}^{w,\geq 0}$, the functor
$\sF \star -: \sC^{K,w} \to \sC^{K,w}$ is left $t$-exact.

\end{prop}

\begin{proof}

\step\label{st:conv-bounds}
 
First, we claim that the conclusion of the
for a compact open subgroup $K$ is equivalent
to asking that $\sF \in \sH_{H,K}^w$ compact acts
on $\sC^{K,w}$ by a functor that is left $t$-exact up to shift.
This property is clearly weaker than that in the
statement of the proposition, so suppose it is satisfied.

We remind
that $\Coh(K\backslash H/K) \subset \sH_{H,K}^w$ is the
subcategory of compact objects and is closed under truncations.
Therefore, it suffices to show that 
for $\sF \in \Coh(K\backslash H/K)^{\geq 0}$,
the functor $\sF \star -:\sC^{K,w} \to \sC^{K,w}$ is left
$t$-exact. Indeed, any object of $\sH_{H,K}^{w,\geq 0}$ is
a filtered colimit of such coconnective compact objects.
Fix $\sF \in \Coh(K\backslash H/K)^{\geq 0}$ in what follows.

Under our assumption, 
$\sF \star -$ maps $\sC^{K,w,\geq 0}$ into
$\sC^{K,w,+}$. 
Therefore, by Lemma \ref{l:t-cons-k},
it suffices to show that the composition:

\[
\sC^{K,w} \xar{\sF \star -} \sC^{K,w} \xar{\Oblv} \sC
\]

\noindent is left $t$-exact.
Moreover, we can clearly replace $\sC^{K,w}$ with
$\sC^{K,w,naive}$ here.
Then the corresponding functor may be calculated as the composition:

\[
\begin{gathered}
\sC^{K,w,naive} \xar{\sF \boxtimes -}
\IndCoh(H/K) \otimes \sC^{K,w,naive} \to 
(\IndCoh(H) \otimes \sC)^{K,w,naive} 
\xar{\overset{\alpha_{\sC}}{\simeq}} \\
\IndCoh(H/K) \otimes \sC 
\xar{\Gamma^{\IndCoh}(H/K,-)} \sC.
\end{gathered}
\]

\noindent There are some things to explain in the above
manipulations: we are regarding $\sF \in \IndCoh(H/K)$
by forgetting the left $K$-equivariance;
in the third term, the $K$-equivariance is taken for the
diagonal $K$-action mixing the given action on $\sC$
and the right action of $H$ on $\sC$; and the implicit 
commuting of weak $K$-equivariance with the tensor product
by $\sC$ in the fourth term is justified by 
the fact that $\sC$ is assumed compactly generated
and therefore is dualizable (or one may use
Lemma \ref{l:torsor-coinv}). 
The first functor is left $t$-exact by assumption
on $\sF$; the second is clearly $t$-exact; the
third is $t$-exact because the naive $H$-action
is compatible with $t$-structures; and the fourth
by Lemma \ref{l:cpt-tstr} \eqref{i:tstr-2}, using that
$\Gamma^{\IndCoh}(H/K,-)$ is left $t$-exact.
This gives the claim.

\step\label{st:pol-t-exact}

Next, we check the above hypothesis in the case
where $K$ is a polarization of $H$.
In fact, a little more generally, we will show that if 
$K$ is a polarization and $K_0 \subset H$ is any other
compact open subgroup, then for any
$\sF \in \Coh(K_0\backslash H/K)$, the
functor: 

\[
\sF \star -:\sC^{K_0,w} \to \sC^{K,w}
\]

\noindent is left $t$-exact up to shift. 

Let $\sG \in \Coh(K\backslash H/K_0)$ be obtained by
applying Serre duality on $H/K$ 
to $\sF$ (considered with its natural $K_0$-equivariant structure) 
and then pulling back along 
the inversion map $K\backslash H/K_0 \isom K_0\backslash H/K$.

As is standard, 
$\sG \star -:\sC^{K,w} \to \sC^{K_0,w}$ is left adjoint
to $\sF \star -$ (by ind-properness of $H/K$). 
Therefore, it suffices to show
$\sG \star -$ is right $t$-exact up to shift.

This is straightforward: it suffices to show
the composite with $\sC^{K_0,w} \to \sC$ 
is right $t$-exact up to shift by Lemma \ref{l:t-cons-k},
and this follows by a similar (in fact, simpler)
argument to Step \ref{st:conv-bounds},
using that $\sG$ is supported on a finite type subscheme
of $H/K$.

\step 

Next, suppose $K$ is a compact open subgroup of
$H$ that admits an embedding $K \subset K_{pol} \subset H$
with $K_{pol}$ a polarization of $H$ and 
$K$ normal in $K_{pol}$ (so $K_{pol}/K$ is an affine
algebraic group). 
We will prove the result for $K$ in this case.

Let $\sF \in \sH_{H,K}^{w,\geq 0}$ and 
$\sG \in \sC^{K,w,\geq 0}$ be given. We need to show
that $\sF \star \sG \in \sC^{K,w,\geq 0}$.

As the functor $\Av_*^w:\sC^{K,w} \to \sC^{K_{pol},w}$ 
of averaging from $K$ to $K_{pol}$ is conservative and 
$t$-exact (by the normality assumption), 
it suffices to show that 
$\Av_*^w(\sF \star \sG) \in \sC^{K_{pol},w,\geq 0}$.

This term may clearly be calculated by averaging
$\sF \in \sH_{H,K}^w$ on the left 
to obtain ${\widetilde{\sF} \in 
\IndCoh_{ren}(K_{pol}\backslash H/K)^{\geq 0}}$, and then 
convolving with $\widetilde{\sF}$.
By the previous step, that object is eventually coconnective,
and by Step \ref{st:conv-bounds} it is honestly
coconnective.

\step 

Finally, we prove the claim for 
$K$ a general compact open subgroup of $H$.

By the previous step, there exists a compact open
subgroup $K_0 \subset K$ for which the
conclusion of the proposition holds. 

We again let
$\sF \in \sH_{H,K}^{w,\geq 0}$ and 
$\sG \in \sC^{K,w,\geq 0}$ denote given objects,
and we aim to show that their convolution is eventually
coconnective.

By the proof of Lemma \ref{l:recognition}, 
$\sG$ is a direct summand of 
$\Tot^{\leq n}(\Av_*^w\Oblv)^{\dot+1}(\sG)$ for some
$n$; here our functors denote the adjoint pair
$\Oblv:\sC^{K,w} \rightleftarrows \sC^{K_0,w}:\Av_*^w$.
Each term in this finite limit lies in
$\Av_*^w(\sC^{K_0,w,\geq 0})$, so we may assume 
$\sG = \Av_*^w(\sG_0)$ for $\sG_0 \in \sC^{K_0,w,\geq 0}$.

It suffices to check that
$\Oblv(\sF \star \sG) = \Oblv(\sF \star \Av_*^w(\sG_0)) \in 
\sC^{K_0,w}$ is eventually coconnective,
since $\Oblv:\sC^{K,w} \to \sC^{K_0,w}$ is conservative.

But the above object may be calculated by
mapping $\sF$ along the functor $\sH_{H,K}^w \to \sH_{H,K_0}^w$
of forgetting equivariance on both sides and then 
convolving with $\sG_0$; by assumption on $K_0$, this
object is coconnective as desired.

\end{proof}

\begin{rem}

The careful reader will see that we never really used the
hypothesis that the $H$-action on $\sC$ is strong.
Here are the actually relevant hypotheses.
First, we need a genuine $H$-action 
on $\sC$ and a $t$-structure on $\sC$ naively compatible with
the $H$-action. In addition, for every compact open subgroup $K \subset H$,
we need a $t$-structure on $\sC^{K,w}$ for which 
$\sC^{K,w} \to \sC^{K,w,naive}$ is $t$-exact and an equivalence on
eventually coconnective subcategories. Finally, we need that
for $K_1 \subset K_2$ compact open subgroups,
$\sC^{K_2,w} \to \sC^{K_1,w}$ is $t$-exact. 

\end{rem}

\begin{cor}\label{c:canon-renorm-criterion}

Suppose $H$ is polarizable with a prounipotent tail.
Suppose $\sC \in H\mod$ is equipped with a $t$-structure
strongly compatible with the weak $H$-action.
Suppose that for every $K \subset H$ compact open,
$\sC^{K,w}$ is compactly generated by objects
lying in $\sC^{K,w,+} \cap \sC^{K,w,\leq 0}$. 

Then the naive weak action of $H$ on $\sC$
canonically renormalizes, and 
$\Oblv^{str\to w}(\sC) \in H\mod_{weak}$ 
is its canonical renormalization.

\end{cor}

\begin{proof}

Immediate from the definition of canonical renormalization
and from Proposition \ref{p:canon-recog} and
Proposition \ref{p:canon-left-t} (and Example 
\ref{e:module-cpts-trun} as applied to $\sA_{ren} = \sH_{H,K}^w$).

\end{proof}

\subsection{Conclusion}\label{ss:gen-hc-pf}

We now combine the various ingredients 
above to conclude the proof of Theorem \ref{t:gen-hc}.

By Lemma \ref{l:smash-str} \eqref{i:smash-str-1},
$A\# U(\fh)\mod_{ren}$ is compactly generated.
Moreover, this category has a canonical 
compactly generated $t$-structure
by Lemma \ref{l:smash-str} \eqref{i:smash-str-3}.
The forgetful functor 
$A\# U(\fh)\mod_{ren} \to \Vect$ from 
\S \ref{ss:gen-constr} is $t$-exact and conservative
on eventually coconnective objects
by Lemma \ref{l:smash-str} \eqref{i:smash-str-4}.
Therefore, by Remark \ref{r:ren-cat}, there
is a corresponding 
connective $\overset{\rightarrow}{\otimes}$-algebra
$A\# U(\fh) \in \Alg_{ren}^{\overset{\rightarrow}{\otimes}}$.

Moreover, the naive action of $H$ on 
$A\# U(\fh)\mod_{ren}$, its naive compatibility with 
the $t$-structure (Lemma \ref{l:smash-str} \eqref{i:smash-str-4}),
and the naive $H$-equivariance of the
forgetful functor $A\# U(\fh)\mod_{ren} \to \Vect$
define a naive $H$-action on $A\# U(\fh)$ compatible 
with its renormalization datum. Indeed, this
follows from Remark \ref{r:ren-cat} (c.f. the
end of the proof of Theorem \ref{t:gen-ff}).
This upgrades $A\# U(\fh)$ to an object
of $\Alg_{ren}^{\overset{\rightarrow}{\otimes},H\actson}$.
We claim that this action is genuine.

First, we need to show that the genuine $H$-action
on $A\# U(\fh)\mod_{ren}$ is obtained by canonical
renormalization.
We do this by applying Corollary \ref{c:canon-renorm-criterion}
to $\sC = A\# U(\fh)\mod_{ren}$. We check that the various
conditions from that corollary are satisfied.

By construction,
the naive weak $H$-action on $A\# U(\fh)\mod_{ren}$ upgrades to a 
strong action. 

By Lemma \ref{l:smash-str} \eqref{i:smash-str-2} and
\eqref{i:smash-str-5},
$A\# U(\fh)\mod_{ren}^{K,w}$ is compactly generated,
and its $t$-structure is as well, 
and compact objects are eventually coconnective.

Therefore, the corollary applies, and
we find that the $H$-action on $A\# U(\fh)$ is
nearly genuine (in the sense of \S \ref{ss:gen-act-alg}).
It is genuine by Lemma \ref{l:smash-str} \eqref{i:smash-str-5}.

\subsection{Harish-Chandra data}\label{ss:hc-start}

In the remainder of this section, we suppose that 
$H$ is a formally smooth 
polarizable ind-affine Tate group indscheme with 
prounipotent tail. In particular, Theorem \ref{t:gen-hc}
applies.

\subsection{}\label{ss:comonad-over}

The reader may prefer to skip this material
and refer back to it as needed.

Let $\sC$ be a category, and suppose
$T: \sC \to \sC$ is a comonad.
Let $\sF \in \sC$ be a fixed object.
We claim that $T$ canonically upgrades to a comonad
on the overcategory $\sC_{/\sF}$.

For $\sG \in \sC_{/\sF}$, we have the map
$T(\sG) \to \sG \to \sF$ where the first map is the
counit for $T$ and the second map is the structure
map for $\sG$; this makes $T(\sG)$ into an object of
$\sC_{/\sF}$. We denote this functor by
$T_{/\sF}:\sC_{/\sF} \to \sC_{/\sF}$.

We claim $T_{/\sF}$ has a natural comonad structure.
Consider $T\comod_{/\sF}$, the
category of $T$-comodules $\sG$ in $\sC$ equipped with a
map $\alpha:\sG \to \sF \in \sC$ (with no hypotheses on
how $\alpha$ interacts with the comodule structure).
The forgetful functor $T\comod_{/\sF} \to \sC_{/\sF}$
is obviously conservative; we claim that it is actually
comonadic. Indeed, $\sC_{/\sF} \to \sC$ clearly commutes with 
contractible limits,\footnote{See \cite{htt} Proposition 4.4.2.9 for
a complete proof.}, so the claim follows from Barr-Beck.
It is clear the underlying endofunctor of this comonad on
$\sC_{/\sF}$ is given by $T_{/\sF}$.

Note that by construction, the data of
a $T_{/\sF}$-comodule structure on $\sG \in \sC_{/\sF}$ is
equivalent to a $T$-comodule structure on the 
underlying object $\sG \in \sC$.

\subsection{}\label{ss:hc-defin}

We apply the above for the comonad 
$\Oblv^{str \to w} \circ (-)^{\exp(\fh),w}:H\mod_{weak} \to 
H\mod_{weak}$ and $\Vect \in H\mod_{weak}$.
This defines a comonad on $H\mod_{weak,/\Vect}$.

By Theorem \ref{t:gen-hc} and 
Lemma \ref{l:smash-str} \eqref{i:smash-str-6}, this
comonad preserves the 1-full subcategory:

\[
(\Alg_{conv,gen}^{\overset{\rightarrow}{\otimes},H\actson})^{op} 
\overset{Thm. \ref{t:gen-ff}}{\subset} 
(H\mod_{weak})_{/\Vect}.
\] 

\noindent This induces a monad
on $\Alg_{conv,gen}^{\overset{\rightarrow}{\otimes},H\actson}$,
which we denote by $A \mapsto A\# U(\fh)$.

We can now make the following definition.

\begin{defin}

A \emph{Harish-Chandra datum} for 
$A \in \Alg_{conv,gen}^{\overset{\rightarrow}{\otimes},H\actson}$
is a structure of module for the above monad.

\end{defin}

\begin{rem}\label{r:hc-map}

By definition, a Harish-Chandra datum gives rise to an ``action" map
$A \# U(\fh) \to A$.

\end{rem}

\subsection{}\label{ss:hc-to-strong}

We now make the following observation.

\begin{lem}

The functor $\Oblv^{str \to w}:H\mod \to H\mod_{weak}$
is comonadic.

\end{lem}

\begin{proof}

This functor is conservative as 
the composition:

\[
H\mod \to H\mod_{weak} \xar{\Oblv_{gen}} \DGCat_{cont}
\]

\noindent computes the forgetful functor for
$H\mod$, which is conservative.

To conclude, we simply note that
as $H$ is polarizable, 
$\Oblv^{str \to w}$ admits a left
adjoint by Proposition \ref{p:inv-coinv-w/str} and 
therefore commutes with all limits.

\end{proof}

Therefore, by the discussion of \S \ref{ss:comonad-over},
a Harish-Chandra datum for
$A \in \Alg_{conv,gen}^{\overset{\rightarrow}{\otimes},H\actson}$
is equivalent to upgrading the genuine $H$-action on
$A\mod_{ren}$ to a strong $H$-action with the 
property\footnote{This encodes the fact that $A\mod_{ren}$
is a comodule in the 1-full subcategory 
$(\Alg_{conv,gen}^{\overset{\rightarrow}{\otimes},H\actson})^{op} 
{\subset} 
(H\mod_{weak})_{/\Vect}$, i.e., it has to do with the
fact that this is a 1-full subcategory and not an 
actual subcategory.}
that the coaction functor $A\mod_{ren} \to 
\Oblv^{str \to w}(A\mod_{ren}^{\exp(\fh),w}) = 
A\# U(\fh)\mod_{ren}$ come from a genuinely $H$-equivariant morphism
$A\# U(\fh) \to A$.

More precisely, let 
$\Alg_{HC}^{\overset{\rightarrow}{\otimes},H\actson}$
denote the category of 
$A \in \Alg_{conv,gen}^{\overset{\rightarrow}{\otimes},H\actson}$
equipped with a Harish-Chandra datum (i.e., the category
of modules for the monad $A \mapsto A\# U(\fh)$).
Then we have:

\begin{lem}\label{l:hc-cat}

The above functor:

\[
(\Alg_{HC}^{\overset{\rightarrow}{\otimes},H\actson})^{op} 
\xar{ A \mapsto A\mod_{ren}}
H\mod \underset{H\mod_{weak}}{\times} H\mod_{weak,/\Vect}
\]

\noindent is 1-fully-faithful.

\end{lem}

\subsection{The classical case}\label{ss:cl-hc-1}

Let $A$ be a classical, renormalized $\overset{\rightarrow}{\otimes}$-algebra
equipped with a genuine $H$-action. As in Theorem \ref{t:gen-hc},
we can form the smash product $A \# U(\fh)$.
We wish to explicitly describe the category 
$A \# U(\fh)\mod_{ren}^{\heart} (= A \# U(\fh)\mod_{naive}^{\heart})$. 

Suppose $V$ is an object of this abelian category.
The canonical map $A \to A \# U(\fh)$ (coming as the unit for the
monad structure) makes $V$ into a (discrete) $H^0(A)$-module.
Also, the fact that the unit map $k \to A$ is $H$-equivariant
gives a map $k \# U(\fh) = U(\fh) \to A \# U(\fh)$, so 
$V$ also acquires an $\fh$-module structure. 

To describe the compatibility between these two actions, we need 
the following digression.
Any $\xi \in \fh$ defines a \emph{derivation}
$\delta_{\xi}:A \to A$. In detail, $\xi$ defines a homomorphism
$\Fun(H) \to k[\vareps]/\vareps^2$ extending the augmentation on
$\Fun(H)$, so we obtain a map:

\[
A \xar{\coact} A \overset{!}{\otimes} \Fun(H) \to 
A \overset{!}{\otimes} k[\vareps]/\vareps^2 \in 
\Alg^{\overset{\rightarrow}{\otimes}}
\]

\noindent giving $\id_A$ mod $\vareps$. 
If we write the underlying map of pro-vector spaces
as $\id_A \times \delta_{\xi} \vareps$,
the map $\delta_{\xi}:A \to A \in \Pro\Vect^{\heart}$ 
is by definition our derivation.

Then we claim that the difference between the two maps:

\begin{equation}\label{eq:xi-act-bracket}
\begin{gathered}
A \overset{\rightarrow}{\otimes} V \xar{\act} V \xar{(\xi \cdot -)}
V, \\
A \overset{\rightarrow}{\otimes} V \xar{\id \overset{\rightarrow}{\otimes} (\xi \cdot -)}
A \overset{\rightarrow}{\otimes} V \xar{\act} V
\end{gathered} 
\end{equation}

\noindent is:

\begin{equation}\label{eq:delta-xi}
A \overset{\rightarrow}{\otimes} V \xar{\delta_{\xi} \overset{\rightarrow}{\otimes} \id_V} 
A \overset{\rightarrow}{\otimes} V \xar{\act} V.
\end{equation}

\noindent More symbolically:\footnote{This 
latter formula is only sufficient when
$A$ is a \emph{topological vector 
space}. By definition, this means that there
is a (non-derived) topological vector space $A^{\sim}$
with a complete, separated, linear topology such that $A$
is the associated pro-vector space,
i.e., $A = \lim A^{\sim}/U \in \Pro\Vect$
where $U$ runs over open subspaces of $A^{\sim}$.
In this case, it is typically sufficient
to work with elements of $A$.

Note that $\fh$ is necessarily 
a topological vector space, justifying working directly with its elements in
some of these formulae.}

\[
\xi \cdot f \cdot v - f \cdot \xi \cdot v = \delta_{\xi}(f) \cdot v, 
\hspace{.5cm} 
\xi \in \fh, f \in A, v \in V.
\]

Indeed, for $K \subset H$ compact open, we have the canonical
equivalence:

\[
A \# U(\fh)\mod_{ren}^{K,\heart} = A\mod_{ren}^{H_K^{\wedge},w,\heart} = 
A\mod_{naive}^{H_K,w,naive,\heart}.
\] 

\noindent This verifies the above identity for 
$V$ strongly $K$-equivariant. Every object in 
$A \# U(\fh)\mod_{ren}^{K,\heart} $ is a filtered colimit
of objects strongly equivariant for some congruence subgroup, so 
we obtain the claim in general.

In addition, this same logic implies the converse. That is,
we have the following result.

\begin{prop}\label{p:smash-mods}

For $V \in \Vect^{\heart}$, lifting $V$ to an object of
$A \# U(\fh)\mod_{ren}^{\heart}$ is equivalent via the above
constructions to specifying an action of $A$ on $V$ and
an action of $\fh$ on $V$ such that the difference
between the two maps in \eqref{eq:xi-act-bracket}
is \eqref{eq:delta-xi} for any $\xi \in \fh$.

\end{prop}

We then obtain:

\begin{cor}\label{c:smash-cl-maps}

Suppose $B \in \Alg^{\overset{\rightarrow}{\otimes}}$ is classical.
Then specifying a map $A \# U(\fh) \to B \in 
\Alg^{\overset{\rightarrow}{\otimes}}$ is equivalent to 
giving maps $\vph:A \to B \in \Alg^{\overset{\rightarrow}{\otimes}}$
and $i:\fh \to B$ compatible with brackets\footnote{If 
$B$ is a topological vector space, then in the language of 
\cite{beilinson-top-alg}, we would say $i$ is a map of
topological Lie algebras.}
such for any $\xi \in \fh$,
the difference between the two maps:

\[
\begin{gathered}
A \xar{\vph} B \xar{i(\xi) \cdot -} B \\
A \xar{\vph} B \xar{- \cdot i(\xi)} B 
\end{gathered}
\]

\noindent is the composition:

\[
A \xar{\delta_{\xi}} A \xar{\vph} B.
\]

\end{cor}

\begin{proof}

Immediate from Proposition \ref{p:alg-vs-cats}
(and \cite{higheralgebra} Theorem 1.3.3.2). 

\end{proof}

\subsection{}\label{ss:cl-hc-2}

In the above setting, we now show:

\begin{lem}\label{l:smash-cl}

Suppose that for any compact open subgroup $K \subset H$,
$K$ is the spectrum of a countably generated ring.
Then $A \# U(\fh)$ is classical.

\end{lem}

\begin{proof}

By Proposition \ref{p:alg-vs-cats}, it suffices to show
$A \# U(\fh)\mod_{ren}^+$ is the bounded derived category of its
underlying abelian category. By \cite{whit} Lemma 5.4.3
and our countability assumption, it suffices to show
that for any compact open subgroup $K \subset H$,
$A \# U(\fh)\mod_{ren}^{K,+} = A \mod_{ren}^{H_K^{\wedge},w,+}$
is the bounded derived category of its underlying abelian category.
This category admits a monadic, $t$-exact restriction functor to 
$A\mod_{ren}^{K,w,+}$, and the corresponding monad is $t$-exact;
so Lemma \ref{l:monad-der-ab}
reduces to showing that $A\mod_{ren}^{K,w,+}$ is the bounded
derived category of its heart. As this category is comonadic
over $A\mod_{ren}^+$ with $t$-exact comonad, that claim follows
from the similar one for $A\mod_{ren}^+$.

\end{proof}

\begin{rem}

It seems likely that 
the above result is true without the countability hypothesis.

\end{rem}

By this lemma (and Yoneda), Corollary \ref{c:smash-cl-maps} gives a complete
description of $A \# U(\fh)$. In particular, using the 1-categorical
nature of our setup, we find the following result.

\begin{cor}\label{c:hc-cl}

Under the above assumptions,
a Harish-Chandra datum for $A$ (in the sense of \S \ref{ss:hc-defin})
is equivalent to specifying an $H$-equivariant map $i:\fh \to A$ 
compatible with brackets such that for $\xi \in \fh$,
$[i(\xi),-] = \delta_{\xi}$ as maps $A \to A \in \Pro\Vect^{\heart}$,
and such that the induced (naively $H$-equivariant) map:

\[
A \# U(\fh) \to A
\]

\noindent is genuinely $H$-equivariant.

\end{cor}

\begin{rem}

The last condition in the corollary can be difficult to check in practice.
There is one simple case though: if compact objects in 
$A\mod_{ren}$ are closed under truncations, then 
for any $B \in \Alg_{conv,gen}^{\overset{\rightarrow}{\otimes},H\actson}$,
any naively $H$-equivariant morphism $B \to A$ is genuinely $H$-equivariant.
Indeed, this follows from the definition of canonical renormalization
and from the discussion of \S \ref{ss:functors-ren}.

\end{rem}

We record the following consequence of the
above discussion for later reference.

\begin{cor}\label{c:cl-hc-maps}

Suppose $A,B \in \Alg_{HC}^{\overset{\rightarrow}{\otimes},H\actson}$
are classical $\overset{\rightarrow}{\otimes}$-algebras equipped
with genuine $H$-actions. Then giving a morphism 
$f:A \to B \in \Alg_{HC}^{\overset{\rightarrow}{\otimes},H\actson}$
is equivalent to giving a genuinely $H$-equivariant 
morphism $f: A \to B \in \Alg_{conv,gen}^{\overset{\rightarrow}{\otimes},H\actson}$
such that the diagram:

\[
\xymatrix{
& \fh \ar[dl] \ar[dr] & \\
A \ar[rr]^f & & B
}
\]

\noindent commutes in $\Pro\Vect^{\heart}$, where the diagonal
morphisms encode the Harish-Chandra data as above.

\end{cor}

\section{Application to the critical level}\label{s:critical}

\subsection{} In this section, we prove Theorem \ref{t:opers}, providing
a large class of symmetries for Kac-Moody representations
at critical level. We also show a compatibility result, 
Theorem \ref{t:whit-linear}, with our previous work \cite{whit}.
The arguments are straightforward applications of the methods developed
in \S \ref{s:hc}. 

\subsection{}

Let us describe the contents of this section in more detail.

Let $G$ be a split reductive group over $k$. Let\footnote{We abuse
notation in letting $K$ denote both Laurent series and compact open
subgroups of $G(K)$. This abuse should never cause confusion, and we
prefer it to various alternatives.} 
$K \coloneqq k((t))$.

Recall that for any $\Ad$-invariant symmetric bilinear form
$\kappa$ on $\fg$, we have the corresponding Kac-Moody central extension
$0 \to k \to \widehat{\fg}_{\kappa} \to \fg((t))$. As we will
discuss in \S \ref{ss:central-extns}, $\kappa$ defines a \emph{twisted}
notion of strong $G(K)$-actions; we denote the corresponding category by
$G(K)\mod_{\kappa}$. The theory is barely different from the untwisted
one. A basic object is $\widehat{\fg}_{\kappa}\mod \in G(K)\mod_{\kappa}$.

Let $crit \coloneqq -\frac{1}{2} \kappa_{\fg}$, for $\kappa_{\fg}$
the Killing form on $\fg$.

Let $\Op_{\ld{G}}$ denote the indscheme of $\ld{G}$-opers on 
the punctured disc; we take as the definition of
opers what are called \emph{marked} opers
in \cite{bezrukavnikov-travkin} \S A.
(which is a slight modification of the definition in 
\cite{hitchin}).

Our goal for this section is to construct a coaction of 
$\IndCoh^*(\Op_{\ld{G}})$ on $\widehat{\fg}_{crit}\mod \in G(K)\mod_{crit}$
via the Feigin-Frenkel isomorphism. In other words, we wish to show
that in a suitable sense, $\widehat{\fg}_{crit}\mod$ is tensored
over its center compatibly with the (critical level) strong $G(K)$-action
on it. This result appears as Theorem \ref{t:opers}. 

\subsection{Central extensions}

We begin by generalizing some material from \S \ref{s:strong}
in the presence of central extensions and twisted $D$-modules.

We outline the main ideas and leave the verification that certain
constructions generalize to the reader.

\subsection{}\label{ss:tw-start}

Fix $c \in k$ once and for all; we refer to $c$ as the \emph{twisting
parameter}. Let $\bB \bG_m$ denote the Zariski sheafified version
of the classifying space.

Observe that $\AffSch_{/\bB \bG_m}^{cl}$ of
classical affine schemes equipped with a line bundle embeds
as a full subcategory of $\Pro(\AffSch_{f.t.,/\bB \bG_m})$,
the pro-category of (classically) finite type classical affine schemes
with a line bundle; this follows by standard Noetherian approximation
(specifically, \cite{egaiv} Theorem 8.5.2). 
Then the procedure from \cite{dmod} \S 2 produces functors:

\[
\xymatrix{
D_c^*:\PreStk_{/\bB\bG_m} \to \DGCat_{cont} \\
D_c^!:\PreStk_{/\bB\bG_m}^{op} \to \DGCat_{cont}.
}
\]

\noindent These functors are given by suitable Kan extensions from
the finite type setup (c.f. \emph{loc. cit}.), and for 
$S$ affine, finite type, and equipped with a line bundle
$\sL$, they each assign to $S$ the category of $(\sL,c)$-twisted
$D$-modules on $S$ (as defined e.g. in \cite{crystals} \S 5).

\begin{rem}

We generally omit the line bundle from the notation since it can usually
be taken for granted.

\end{rem}

As in \S \ref{ss:indcoh-dmod}, there is a canonical natural transformation:

\[
\IndCoh^* \to D_c^* \in \TwoHom(\IndSch_{reas,/\bB\bG_m},\DGCat_{cont})
\]

\noindent defined by a formal extension process from the finite type case.

\subsection{}\label{ss:tw-finish}

Because $\bB \bG_m$ is a commutative group, there is a canonical
symmetric monoidal structure on $\PreStk_{/\bB\bG_m}$
for which the forgetful functor to $\PreStk$ is symmetric monoidal
for the Cartesian monoidal product.

Explicitly, for $(S,\sL_S)$ and $(T,\sL_T)$ in $\PreStk_{/\bB\bG_m}$,
$S \times T$ is equipped with the line bundle $\sL_S \boxtimes \sL_T$.

Then each of the functors $D_c^*$ and $D_c^!$ are naturally
lax symmetric monoidal for this symmetric monoidal structure,
meaning that we have \emph{external products} in either setup.
Indeed, this lax symmetric monoidal structure arises from the
finite type setup by Kan extension.
As in Remark \ref{r:indcoh-dmod-lax}, the natural transformation
$\IndCoh^* \to D_c^*$ canonically upgrades (via our same extension procedure)
to a natural transformation of lax symmetric monoidal functors.

\subsection{}\label{ss:central-extns}

Now let $H$ be a Tate group indscheme and let
$\lambda:H \to \bB \bG_m$ be a homomorphism; equivalently, we have
a central extension:

\[
1 \to \bG_m \to H^{\prime} \to H \to 1.
\]

\noindent We assume that there exists a compact open subgroup $K \subset H$
on which $\lambda$ is trivial as a homomorphism. Then
note that $H^{\prime}$ is also a Tate group indscheme. If $K$ can be taken
to be a polarization of $H$, then $H^{\prime}$ is polarizable.

We obtain the category $D_c^*(H)$ of twisted $D$-modules on $H$
for the underlying line bundle defined by $\lambda$.
Because $(H,\lambda)$ is an algebra object in $\PreStk_{/\bB\bG_m}$
(for the symmetric monoidal structure described above),
$D_c^*(H)$ is canonically a monoidal DG category, 
i.e., an algebra object in $\DGCat_{cont}$.

We let $H\mod_c$ denote the category of modules for $D_c^*(H)$ in
$\DGCat_{cont}$, and we refer to these as DG categories
equipped with \emph{(strong) $c$-twisted $H$-actions}.

\subsection{}\label{ss:tw-invariants}

Before proceeding, we will need the following digression on
twisted invariants and coinvariants.

Suppose $\sC \in H\mod_{weak}$. As in \S \ref{ss:tate-tw-inv},
we have a certain 
full subcategory $\sC_{(1)}^{H^{\prime},w} \subset \sC^{H_{Tate},w}$

Now suppose that $\lambda$ factors as 
$H \xar{\widehat{\lambda}} \bB \bG_a^{\wedge} \xar{\exp} \bB \bG_m$.
In this case, we can define a new homomorphism:

\[
\lambda^c \coloneqq \exp(c \cdot \widehat{\lambda}):H \to 
\bB \bG_m.
\]

\noindent We let $\sC^{H,w,\chi_c}$ denote the corresponding category
of \emph{twisted invariants}, i.e., we
take $\sC^{H,w,\chi_c} \coloneqq 
\sC_{(1)}^{H_c^{\prime},w}$
for $H_c^{\prime}$ the central extension defined
by $\lambda^c$.

This construction can also be understood as follows.
We obtain a homomorphism 
$\IndCoh^*(H) \to \IndCoh(\bB \bG_a^{\wedge}) \simeq \QCoh(\bA^1)$
where the right hand side is equipped with its usual tensor product
structure (as opposed to convolution). Our twisting parameter
$c \in k$ defines a homomorphism $\QCoh(\bA^1) \to \Vect$ (taking the
fiber at\footnote{The minus sign here matches the sign conventions
of \S \ref{s:sinf}.} $-c$), so an object $\chi_c \in H\mod_{weak,naive}$.
It is easy to see that this naive weak 
$H$-action canonically renormalizes, defining $\chi_c \in H\mod_{weak}$.
Then by Proposition \ref{p:chi-tate-inv-gen},
$\sC^{H,w,\chi_c} = (\sC \otimes \chi_c)^{H,w}$.

Similarly, we have a twisted coinvariants functor $\sC_{H,w,\chi_c}$, defined
by tensoring with $\chi_c$ and taking coinvariants.

\subsection{}

Combining the material of \S \ref{ss:tw-start}-\ref{ss:tw-finish} 
with the methods of \S \ref{s:strong},
we obtain a forgetful functor:

\[
\Oblv^{str \to w}:H\mod_c \to H\mod_{weak}.
\]

\noindent This functor admits left and right adjoints, which we
denote by $(-)_{\exp(\fh),w,\chi_c}$ and $(-)^{\exp(\fh),w,\chi_c}$.

There are explicit formulae for these functors, similar
to \S \ref{ss:oblv-str-adjs}. Before giving them,
suppose that $K \subset H$ is a compact open subgroup
on which $\lambda$ is trivial. In this case, $\lambda|_{H_K^{\wedge}}$
clearly factors through $\bB \bG_m^{\wedge} \overset{\log}{\isom} 
\bB \bG_a^{\wedge}$, so the discussion of \S \ref{ss:tw-invariants}
applies.

Now we have:

\[
\begin{gathered}
\sC^{\exp(\fh),w} \coloneqq
\underset{\lambda \text{ trivial on K}}{\underset{K \subset H \text{ compact open}}{\colim}}
\sC^{H_K^{\wedge},w,\chi_c} \in \DGCat_{cont} \\
\sC_{\exp(\fh),w} \coloneqq
\underset{\lambda \text{ trivial on K}}{\underset{K \subset H \text{ compact open}}{\lim}}
\sC_{H_K^{\wedge},w,\chi_c} \in \DGCat_{cont}.
\end{gathered}
\]

\noindent under the obvious structural functors.

Proposition \ref{p:inv-coinv-w/str} has
an immediate counterpart in this setting:
simply change $(\exp(\fh),w)$-invariants
and coinvariants in \emph{loc. cit}.
to the corresponding $c$-twisted versions.

\subsection{}

Now suppose that $H$ is formally smooth.
Let $\fh^{\prime}$ denote the Lie algebra
of $H^{\prime}$, considered as a 
central extension of the Tate Lie algebra $\fh$ by $k$.
Let $\fh_c^{\prime}$ denote the central extension
obtained by Baer-scaling by our twisting
parameter $c \in k$.

We let $\fh_c^{\prime}\mod \coloneqq 
\Vect^{\exp(\fh),w,c}$. 

Note that the notation is abusive: this category
should be understood not as modules of the
abstract Tate Lie algebra $\fh_c^{\prime}$,
but as modules over the central extension, i.e.,
modules on which $1 \in k \subset \fh_c^{\prime}$
acts by the identity (in a suitable derived
sense).

An evident 
version of Proposition \ref{l:hmod} applies;
this shows that $\fh_c^{\prime}\mod$ has a
natural $t$-structure with the expected heart,
and is the ``renormalized"
DG category of representations 
considered in \cite{dmod-aff-flag} \S 23.

\begin{rem}

By Proposition \ref{p:inv-coinv-w/str} (or its appropriate
version here), $\fh_c^{\prime}\mod$ is dualizable with dual
$\fh_{-c-Tate}^{\prime}\mod$, where the notation indicates
the Baer sum of the inverse central extension to 
$\fh_c^{\prime}$ and $\fh_{-Tate}$. 
By Theorem \ref{t:sinf-comparison},
the pairing:

\[
\fh_c^{\prime}\mod \otimes \fh_{-c-Tate}^{\prime}\mod \to \Vect 
\]

\noindent is given on eventually coconnective objects by tensoring
and applying semi-infinite cohomology.
Note that the equivalence 
$(\fh_{-c-Tate}^{\prime}\mod)^{\vee} \simeq \fh_c^{\prime}\mod$
is of categories acted on by $H$ strongly with $c$-twist.\footnote{A priori,
$(\fh_{-c-Tate}^{\prime}\mod)^{\vee}$ may look like it is
acted on strongly with twist $c+Tate$. But since the Tate extension is
by definition integral, we can canonically identify $H\mod_{c+Tate}$
with $H\mod_c$ (though not compatibly with the forgetful functors
to $H\mod_{weak}$!).}

\end{rem}

\subsection{Critical level}

We now apply the above to $H = G(K)$ the loop group. Let 
$G(O) \subset G(K)$ be the subgroup, which is a compact open subgroup.

Let $\lambda:G(K) \to \bB \bG_m$ be the map
defining the Tate central extension of
$G(K)$, as in \S \ref{s:sinf}. 
Let $c = -\frac{1}{2}$ above. 
Apply Theorem \ref{t:sinf-comparison}
relative to the compact open subgroup
$K = G(O) \subset G(K)$ 
and \cite{chiral} \S 2.7.5, the corresponding
central extension is
the critical level Kac-Moody extension. Note
that this Tate central extension is canonically trivialized over $G(O)$.

We let $G(K)\mod_{crit}$ denote $H\mod_c$ and
$\widehat{\fg}_{crit}\mod$ denote $\fh_c^{\prime}\mod$ in this setting.
We let $U(\widehat{\fg}_{crit})$ denote the \emph{twisted} enveloping
algebra, i.e., the $\overset{\rightarrow}{\otimes}$-algebra
defined by $\Oblv:\widehat{\fg}_{crit}\mod^+ \to \Vect$
(in \cite{hitchin}, it is usually denoted
$\ol{U}^{\prime}(\fg \otimes K)$).

We let $\bV_{crit} \in \widehat{\fg}_{crit}\mod^{\heart}$ denote
the vacuum representation $\ind_{\fg[[t]]}^{\widehat{\fg}_{crit}}(k)$.
More generally, for $n \geq 0$, we let 
$\bV_{crit,n} = \ind_{t^n\fg[[t]]}^{\widehat{\fg}_{crit}}(k) \in 
\widehat{\fg}_{crit}\mod^{\heart}$. By construction, these objects
compactly generated $\widehat{\fg}_{crit}\mod$.

\begin{rem}

The above clearly applies as is to define
$G(K)\mod_{\kappa}$, the category of categories with 
level $\kappa$ (strong) $G(K)$-actions, 
$\kappa$ any scalar multiple of
the Killing form. 
A simple modification allows for arbitrary
levels $\kappa$ (for possibly non-simple $G$): 
see \cite{whit} \S 1.29-30. 

\end{rem}

\subsection{The center}\label{ss:center-intro}

Define $\fZ \in \ComAlg(\Pro\Vect^{\heart},\overset{!}{\otimes})$ as
the \emph{non-derived} center of $U(\widehat{\fg}_{crit})$.

More precisely, $\fZ$ is the pro-vector space corresponding
to the the topological vector space 
$H^0(U(\widehat{\fg}_{crit})^{G(K)})$
for the adjoint action of $G(K)$ on 
$U(\widehat{\fg}_{crit})$;\footnote{See \cite{hitchin} 
Theorem 3.7.7 (ii). In fact,
$\fZ$ is the non-derived invariants of
$U(\widehat{\fg}_{crit})$ with respect to 
\emph{any} subgroup of $G(K)$ containing a compact open
subgroup; see \cite{hitchin} 3.7.11. We will not need this fact
here.} the topology is the subspace topology.
It is straightforward to see that the multiplication on
$\fZ$ is commutative, and therefore its evident
$\overset{\rightarrow}{\otimes}$-algebra structure
extends canonically to a commutative 
$\overset{!}{\otimes}$-algebra structure
(c.f. \cite{beilinson-top-alg} \S 1.5).

In what follows, we treat $\fZ$ and $U(\fg_{crit})$
interchangeably as pro-vector spaces and as topological
vector spaces, following our custom from \S 

\subsection{}\label{ss:z-str}

We now recall the finer structure of $\fZ$.

Define $I_n \subset \fZ$ as the ideal 
$\fZ \cap U(\fg_{crit})\cdot t^n\fg[[t]] \subset 
U(\fg_{crit})$. By construction of the topology
of $\fZ$, this ideal is open in $\fZ$, and the ideals
$I_n$ as $n$ varies provide a neighborhood basis of $0$.
We let $\fz_n = \fZ/I_n$.\footnote{The notation
follows \cite{hitchin}, which denotes
our $\fz_0$ by $\fz$.}

We let $\sZ_n = \Spec(\fz_n)$ and define
$\sZ = \colim_n \sZ_n \in \IndSch$ (so 
$\sZ = \Spf(\fZ)$).

By \cite{hitchin} \S 3.7.9-10,
$\sZ$ is\footnote{Mildly non-canonically: there are
various choices needed to obtain these coordinates.}
isomorphic to $\prod_{i=1}^r \bA^1(K)$ where $r$ is the
rank of $\fg$. Moreover, \emph{loc. cit}. constructs 
such an isomorphism such that $\sZ_n \subset \sZ$ 
corresponds to the closed subscheme
$\prod_{i=1}^r t^{-d_i}\bA^1(O) \subset \prod_{i=1}^r \bA^1(K)$
for $t$ our coordinate on the formal disc and
$d_1,\ldots,d_r$ the degrees of the invariant polynomials
on $\fg$.

In particular, $\sZ$ is a reasonable indscheme.
Therefore, $\fZ\mod_{ren} \coloneqq \IndCoh^*(\sZ)$ defines a renormalization
datum for $\fZ$; the forgetful functor to $\Vect$
is $\Gamma^{\IndCoh}(\sZ,-)$. (See also Example
\ref{e:pro-algebras}.) In what follows, 
we always consider $\sZ$ as equipped with this renormalization
datum.

\begin{rem}

Using the isomorphism above, one can construct an 
equivalence $\QCoh(\sZ) \simeq \IndCoh^*(\sZ)$:
this is the unique morphism of $\QCoh(\sZ)$-module categories
sending $\sO_{\sZ}$ to the $*$-pullback of
$\omega_{\prod_{i=1}^r \bA^1(K)/\bA^1(O)} \in 
\IndCoh(\prod_{i=1}^r \bA^1(K)/\bA^1(O))$ along the evident
projection (using the additive structure on $\bA^1$ to form
the quotient).
However, this construction is highly non-canonical and does
not behave well with respect to the changes of 
coordinates (which are non-linear).

\end{rem}

\subsection{}

Being the center of $U(\fg_{crit})$, $\fZ$ acts on
$U(\fg_{crit})$. 

More precisely, 
we can consider $\fZ$ as a commutative algebra object
in the symmetric monoidal category
$(\Alg^{\overset{\rightarrow}{\otimes}}, \overset{!}{\otimes})$;
then $U(\fg_{crit})$ is a module for it in the usual
sense of monoidal categories.
For example, the action map 
$\act:\fZ \overset{!}{\otimes} U(\fg_{crit}) \to U(\fg_{crit})
\in \Alg^{\overset{\rightarrow}{\otimes}}$
sends $z \otimes \xi \mapsto z\cdot \xi$.

\begin{lem}\label{l:act-ren}

The morphism $\act$ is compatible with renormalization data,
i.e., it upgrades (necessarily uniquely) to a morphism
in $\Alg_{ren}^{\overset{\rightarrow}{\otimes}}$.

\end{lem}

\begin{proof}

Let:

\[
\coact_{naive}:\widehat{\fg}_{crit}\mod_{naive}^+ \to 
\fZ \overset{!}{\otimes} U(\fg_{crit}) \mod_{naive}^+ 
\]

\noindent denote the standard forgetful functor.\footnote{The notation is
motivated by Remark \ref{r:z-coact-1}.}
Define:

\[
\coact_{naive}^L:\fZ \overset{!}{\otimes} U(\fg_{crit}) \mod_{naive}^+ \to 
\Pro(\widehat{\fg}_{crit}\mod_{naive}^+)
\]

\noindent as its pro-left adjoint. 

Let $n,m \geq 0$.
We claim that: 

\begin{equation}\label{eq:oblv-naive-l}
\coact_{naive}^L(\fz_n \otimes \bV_{crit,m}) = 
\underset{r \geq n,m}{\lim} \, 
\fz_n \underset{\fz_r}{\otimes} \bV_{crit,m} \in 
\Pro(\widehat{\fg}_{crit}\mod_{naive}^+).
\end{equation}

\noindent Here the limit is formed in this pro-category,
and the terms make sense because the action of
$\fZ$ on $\bV_{crit,m}$ factors
through $\fz_m$ by definition. We emphasize that this
is a derived tensor product, i.e., we view 
$\fz_n \otimes \bV_{crit,m}$ as an 
$\fz_r$-bimodule in $\fg_{crit} \mod_{naive}$ 
and then form its Hochschild homology
in this category; because the kernel of 
$\fz_r \to \fz_n$ is generated by a regular sequence
(by \S \ref{ss:z-str}), this tensor product does honestly
lie in $\widehat{\fg}_{crit}\mod_{naive}^+$.
We note that this reasoning also shows that this 
object actually lies in 
$\widehat{\fg}_{crit}\mod^c \subset \widehat{\fg}_{crit}\mod^+$.

Indeed, there is a canonical map from the
left hand side of \eqref{eq:oblv-naive-l} to the
right hand side. It suffices to show this map is
an isomorphism when evaluated on objects in the heart
of the $t$-structure. We can explicitly calculate both
sides using Koszul and Chevalley complexes, giving the claim.

Now because the objects $\fz_n \otimes \bV_{crit,m}$
compactly generate 
$\fZ \overset{!}{\otimes} U(\fg_{crit}) \mod_{ren}^{\leq 0}$
by definition (see \S \ref{ss:alg-tensor-ren}), 
the fact that $\coact_{naive}^L(\fz_n \otimes \bV_{crit,m})$
is pro-(compact and connective) immediately implies
that $\coact_{naive}$ 
renormalizes (i.e., is left $t$-exact), 
i.e., it gives the conclusion of the lemma.

\end{proof}

Therefore, we find that 
$U(\fg_{crit}) \in \Alg_{ren}^{\overset{\rightarrow}{\otimes}}$
is canonically a module for $\fZ \in 
\Alg_{ren}^{\overset{\rightarrow}{\otimes}}$.

\begin{rem}\label{r:z-coact-1}

Applying the symmetric monoidal functor 
$\Alg_{ren}^{\overset{\rightarrow}{\otimes},op} 
\xar{A \mapsto A\mod_{ren}} \DGCat_{cont}$, we find that
$\widehat{\fg}_{crit}\mod$ is canonically a comodule for 
$\fZ\mod_{ren} = \IndCoh^*(\sZ)$.\footnote{In geometric terms,
note that diagonal pushforward equips $\IndCoh^*(\sZ)$ with
a canonical coalgebra structure in $\DGCat_{cont}$,
using the fact that $\sZ$ is a strict indscheme.} 

\end{rem}

\subsection{}

Next, we include $G(K)$-actions. 

For $H$ an
ind-affine group indscheme the category 
$\Alg^{\overset{\rightarrow}{\otimes},H\actson}$
from \S \ref{ss:gp-setting}. Note that this category is
canonically a module category for the (symmetric) monoidal category
$(\Alg^{\overset{\rightarrow}{\otimes}},\overset{!}{\otimes})$:
this is a general feature for co/module categories in
monoidal categories. The same discussion applies verbatim 
in the setting of renormalized 
$\overset{\rightarrow}{\otimes}$-algebras.

Now take $H = G(K)$. Recall that $G(K)$ has an adjoint action on 
$U(\widehat{\fg}_{crit})$ as a renormalized
$\overset{\rightarrow}{\otimes}$-algebra 
encoding the naive, weak $G(K)$-action
on $\widehat{\fg}_{crit}\mod$. 

The action of 
$\fZ \in \ComAlg(\Alg^{\overset{\rightarrow}{\otimes}},\overset{!}{\otimes})$
on $U(\widehat{\fg}_{crit}) \in \Alg^{\overset{\rightarrow}{\otimes}}$ 
above clearly upgrades to an action in the symmetric monoidal
category $\Alg^{\overset{\rightarrow}{\otimes},G(K)\actson}$.
Moreover, tracing the 
definitions, Lemma \ref{l:act-ren} immediately implies
that this action upgrades to an action of 
$\fZ \in \ComAlg(\Alg_{ren}^{\overset{\rightarrow}{\otimes}},\overset{!}{\otimes})$
on $U(\widehat{\fg}_{crit}) \in \Alg_{ren}^{\overset{\rightarrow}{\otimes}}$.

\begin{rem}\label{r:z-coact-2}

In parallel to Remark \ref{r:z-coact-1}, the above discussion
implies that $\IndCoh^*(\sZ)$ coacts on 
$\widehat{\fg}\mod_{crit} \in G(K)\mod_{weak,naive}$,
where the latter category is considered as a module category
for $(\DGCat_{cont},\otimes)$.

\end{rem}

\subsection{}

We now extend the discussion to the setting of genuine actions.

Suppose now that $H$ is an ind-affine Tate group indscheme.
Recall from \S \ref{ss:gen-act-alg} that we have the
1-full subcategory 
$\Alg_{gen}^{\overset{\rightarrow}{\otimes},H\actson}$
of $\Alg_{ren}^{\overset{\rightarrow}{\otimes},H\actson}$.
It is immediate from the definitions and
Lemma \ref{l:cpt-tstr} that this 1-full subcategory is closed
under the 
$(\Alg_{ren}^{\overset{\rightarrow}{\otimes}},\overset{!}{\otimes})$-action.
Therefore, there is a unique action of
$(\Alg_{ren}^{\overset{\rightarrow}{\otimes}},\overset{!}{\otimes})$
on $\Alg_{gen}^{\overset{\rightarrow}{\otimes},H\actson}$
compatible with the embedding into $\Alg_{ren}^{\overset{\rightarrow}{\otimes},H\actson}$.

We remark that the functor 
$\Alg_{gen}^{\overset{\rightarrow}{\otimes},H\actson,op} 
\xar{A\mapsto A\mod_{ren}} H\mod_{weak}$
is $(\Alg_{ren}^{\overset{\rightarrow}{\otimes}},\overset{!}{\otimes})$-linear,
where $(\Alg_{ren}^{\overset{\rightarrow}{\otimes}},\overset{!}{\otimes})$
acts on $H\mod_{weak}$ through its canonical symmetric monoidal functor to
$\DGCat_{cont}$.

\subsection{}

Recall from Theorem \ref{t:gen-hc} that the $G(K)$-action on 
$U(\widehat{\fg}_{crit}) \in \Alg_{ren}^{\overset{\rightarrow}{\otimes},G(K)\actson}$ 
is genuine.

We now have the following upgraded version of Lemma \ref{l:act-ren}.

\begin{lem}\label{l:act-gen}

The morphism $\act:\fZ \overset{!}{\otimes} U(\fg_{crit}) \to U(\fg_{crit})
\in \Alg_{ren}^{\overset{\rightarrow}{\otimes},H\actson}$
is genuinely $H$-equivariant, i.e., it is a morphism 
in $\Alg_{gen}^{\overset{\rightarrow}{\otimes},H\actson}$.

\end{lem}

\begin{proof}

As in Remark \ref{r:z-coact-2}, $\widehat{\fg}_{crit}\mod$ has commuting\footnote{In
homotopically precise terms, we should say
this category is an $(\IndCoh^*(G(K)),\IndCoh^!(\sZ))$-bimodule,
where $\IndCoh^!(\sZ) \coloneqq \IndCoh^*(\sZ)^{\vee}$ as usual.}
$\IndCoh^*(G(K))$-module and $\fZ\mod_{ren}$-comodule structures.

For $K \subset G(K)$ a compact open subgroup, we need to show 
that the coaction functor:\footnote{We can commute the weak invariants
with the tensor product as $\IndCoh^*(\sZ)$ is dualizable.}

\[
\coact^{K,w,naive}:\widehat{\fg}_{crit}\mod^{K,w,naive} \to 
(\IndCoh^*(\sZ) \otimes \widehat{\fg}_{crit}\mod )^{K,w,naive} = 
\IndCoh^*(\sZ) \otimes \widehat{\fg}_{crit}\mod^{K,w,naive}
\]

\noindent renormalizes to a left $t$-exact functor:

\[
\coact^{K,w}:\widehat{\fg}_{crit}\mod^{K,w} \to 
\IndCoh^*(\sZ) \otimes \widehat{\fg}_{crit}\mod^{K,w}
\]

\noindent (where the right hand side is equipped with the usual
tensor product $t$-structure). 

As restriction from genuine weak $K$-invariants to invariants
for a small compact open subgroup above is $t$-exact and conservative,
we can assume for simplicity that $K$ is prounipotent and contained
in $G(O)$.

Now the argument is essentially the same as in Lemma \ref{l:act-ren}.
Consider
$\fz_n \boxtimes \bV_{crit,m} \in 
\IndCoh^*(\sZ) \otimes \widehat{\fg}_{crit}\mod^{K,w,naive,+}$.
Note that these objects compactly generate
$\big(\IndCoh^*(\sZ) \otimes \widehat{\fg}_{crit}\mod^{K,w}\big)^{\leq 0}$
as $K$ is prounipotent. Moreover, the pro-left adjoint (on eventually
coconnective subcategories): 

\[
\coact^{K,w,+,L}:
(\IndCoh^*(\sZ) \otimes \widehat{\fg}_{crit}\mod^{K,w,naive})^+
\Pro(\widehat{\fg}_{crit}\mod^{K,w,naive,+})
\]

\noindent sends $\fz_n \boxtimes \bV_{crit,m}$ to 
$\underset{r \geq n,m}{\lim} \, 
\fz_n \underset{\fz_r}{\otimes} \bV_{crit,m}$
(the limit being formed in the pro-category), which is again
seen using Koszul/Chevalley resolutions. As 
each
$\fz_n \underset{\fz_r}{\otimes} \bV_{crit,m}$ 
is connective and compact in $\widehat{\fg}_{crit}\mod^{K,w}$,
we obtain the claim. 

\end{proof}

The above result implies that
 $\fZ \in \ComAlg(\Alg_{ren}^{\overset{\rightarrow}{\otimes}},\overset{!}{\otimes})$
on $U(\widehat{\fg}_{crit}) \in \Alg_{gen}^{\overset{\rightarrow}{\otimes}}$.

\begin{rem}\label{r:z-coact-3}

As with Remarks \ref{r:z-coact-1} and \ref{r:z-coact-2}, 
the above discussion
implies that $\IndCoh^*(\sZ)$ coacts on 
$\widehat{\fg}\mod_{crit} \in G(K)\mod_{weak}$.

\end{rem}

\subsection{}\label{ss:hc-tensoring}

We now include Harish-Chandra data to extend to strong actions.

Suppose $H$ is a Tate group indscheme satisfying the hypotheses
of Theorem \ref{t:gen-hc}. 

Recall that in \S \ref{ss:hc-defin}, we defined a 
monad $A \mapsto A \# U(\fh)$ on 
$\Alg_{conv,gen}^{\overset{\rightarrow}{\otimes},H\actson}$.
We claim that this monad canonically upgrades to a 
$(\Alg_{conv,ren}^{\overset{\rightarrow}{\otimes}},\overset{!}{\otimes})$-linear monad, i.e., it canonically
lifts along the forgetful map:

\[
\Alg\Big(\TwoEnd_{\Alg_{conv,ren}^{\overset{\rightarrow}{\otimes}}\mod}
(\Alg_{conv,gen}^{\overset{\rightarrow}{\otimes},H\actson})\Big) 
\to \Alg\Big(\TwoEnd(\Alg_{conv,gen}^{\overset{\rightarrow}{\otimes},H\actson})\Big) =
\{\text{monads on }\Alg_{conv,gen}^{\overset{\rightarrow}{\otimes},H\actson}\}.
\]

Indeed, this follows immediately from the constructions
and the observation that the
functor $(-)^{\exp(\fh),w}$ is $\DGCat_{cont}$-linear
for $H$, as is clear from 
Proposition \ref{p:inv-coinv-w/str}.

\subsection{}

Let $\Alg_{HC,crit}^{\overset{\rightarrow}{\otimes},G(K)\actson}$
be defined as the category of $\overset{\rightarrow}{\otimes}$-algebras
with genuine $G(K)$-actions and critical level Harish-Chandra data
(as in \S \ref{ss:hc-to-strong}). 
By \S \ref{ss:hc-tensoring}, 
$\Alg_{HC,crit}^{\overset{\rightarrow}{\otimes},G(K)\actson}$ is 
canonically a module category for 
$(\Alg_{conv,ren}^{\overset{\rightarrow}{\otimes}},\overset{!}{\otimes})$.

We claim that our earlier constructions upgrade to an action
of $\fZ \in 
\ComAlg(\Alg_{conv,ren}^{\overset{\rightarrow}{\otimes}},\overset{!}{\otimes})$
on $U(\widehat{\fg}_{crit}) \in \Alg_{HC,crit}^{\overset{\rightarrow}{\otimes},G(K)\actson}$
(where $U(\widehat{\fg}_{crit})$ is equipped with its
tautological critical level Harish-Chandra datum).

Indeed, as all of the $\overset{\rightarrow}{\otimes}$-algebras
here are classical, by Corollary 
\ref{c:cl-hc-maps}, this amounts to 
the evident commutativity
of the diagram:

\[
\xymatrix{
& \widehat{\fg}_{crit} \ar[dl]_{\xi\mapsto 1 \otimes \xi} \ar[dr] & \\
\fZ \overset{!}{\otimes} U(\widehat{\fg}_{crit}) 
\ar[rr]^{\act} & &
U(\widehat{\fg}_{crit}).
}
\]

Therefore, we obtain:

\begin{thm}\label{t:opers}

There is a canonical coaction of $\IndCoh^*(\sZ)$ on 
$\widehat{\fg}_{crit}\mod$ considered as an object
of the ($\DGCat_{cont}$-enriched category) $G(K)\mod_{crit}$. 

Recalling the Feigin-Frenkel isomorphism $\sZ \simeq \Op_{\ld{G}}$ 
(see \cite{hitchin} \S 3.7 for this form of the isomorphism),
we in particular obtain a coaction of $\IndCoh^*(\Op_{\ld{G}})$
on $\widehat{\fg}_{crit}\mod \in G(K)\mod_{crit}$.
 
\end{thm}

\subsection{A compatibility}

We now establish a compatibility with a related result from \cite{whit}.
We use the notation from \emph{loc. cit}. without further mention.

\begin{thm}\label{t:whit-linear}

The equivalence:

\[
\Whit(\widehat{\fg}_{crit}\mod) \isom \IndCoh^*(\Op_{\ld{G}})
\]

\noindent from \cite{whit} Corollary 7.8.1 canonically upgrades
to an equivalence of $\IndCoh^*(\Op_{\ld{G}})$-comodules,
where the comodule structure on the left hand side 
comes from Theorem \ref{t:opers} and the comodule structure on the
right hand side is the tautological one.

\end{thm}

\begin{rem}

Despite the appearance of the Langlands dual group in the statement,
this is essentially notational: we are just choosing to write
$\Op_{\ld{G}}$ instead of $\sZ$.

\end{rem}

\begin{proof}[Proof of Theorem \ref{t:whit-linear}]

In \cite{whit} \S 5, we showed that
$\Whit(\widehat{\fg}_{crit}\mod)$ is compactly generated,
and we equipped it with a certain canonical 
$t$-structure for which compact objects are bounded.
By \emph{loc. cit}. Corollary 7.8.1, compact objects
are even closed under truncation functors in 
$\Whit(\widehat{\fg}_{crit}\mod)$ (this is special to critical
level). 

Moreover, \emph{loc. cit}. shows that the natural functor 
$\Psi:\Whit(\widehat{\fg}_{crit}\mod)
\to \Vect$ (induced by Drinfeld-Sokolov reduction)
is $t$-exact, and $\Psi|_{\Whit(\widehat{\fg}_{crit}\mod)^+}$
is shown to be conservative.

By Proposition \ref{p:alg-vs-cats}, there is a canonical
convergent, connective 
$\overset{\rightarrow}{\otimes}$-algebra $\sW_{crit}$
with $\Whit(\widehat{\fg}_{crit}\mod)^+ \isom 
\sW_{crit}\mod^+$. By \cite{whit} \S 5, this
algebra identifies with the usual critical level
affine $\sW$-algebra (and in particular, it is classical);
this isomorphism is canonical, and has to do with the
description of the functor $\Psi$ via Drinfeld-Sokolov
reduction.

Now by Theorem \ref{t:opers} and functoriality,
$\Whit(\widehat{\fg}_{crit}\mod)$
is canonically a $\IndCoh^*(\Op_{\ld{G}})$-comodule.
It follows that the functor $\Psi$ factors through a
unique morphism:

\[
\xymatrix{
\Whit(\widehat{\fg}_{crit}\mod) \ar@{..>}[rr]^{\Psi^{enh}} \ar[dr]^{\Psi}
& &
\IndCoh^*(\Op_{\ld{G}}) \ar[dl]^{\Gamma^{\IndCoh}(\Op_{\ld{G}},-)}
\\
& \Vect
}
\]

\noindent with $\Psi^{enh}$ a morphism of
$\IndCoh^*(\Op_{\ld{G}})$-comodule categories. 
Note that $\Psi^{enh}$ is $t$-exact: 
as compact objects
in $\Whit(\widehat{\fg}_{crit}\mod)$ are closed under
truncations, this follows as $\Psi$ and
$\Gamma^{\IndCoh}(\Op_{\ld{G}},-)$ are $t$-exact and
conservative on eventually coconnective subcategories.

Therefore, there is a canonical map $\alpha:\fZ \to \sW_{crit}$
corresponding to the functor $\Psi^{enh}$.
By construction, $\alpha$ is the standard map arising
by functoriality of Drinfeld-Sokolov reduction.
As in \cite{feigin-frenkel-duality},
$\alpha$ is an isomorphism. 
Moreover, the equivalence
$\fZ \simeq \Fun(\Op_{\ld{G}})$ from \cite{hitchin} \S 3.7
corresponds to the identification $\alpha:\fZ \isom \sW_{crit}$
and the identification $\sW_{crit} = \Fun(\Op_{\ld{G}})$
from Feigin-Frenkel duality \cite{feigin-frenkel-duality}.

It follows that $\Psi^{enh}$
is an equivalence on bounded below subcategories.
As compact objects on both sides are exactly almost compact
objects, $\Psi^{enh}$ is actually an equivalence.
Moreover, by the above discussion, $\Psi^{enh}$ is
canonically isomorphic to the equivalence \cite{whit} Corollary 7.8.1.

As $\Psi^{enh}$ was a morphism of 
$\IndCoh^*(\Op_{\ld{G}})$-comodules
by construction, we obtain the claim.

\end{proof}

\bibliography{bibtex}{}
\bibliographystyle{alphanum}

\end{document}